\documentclass[11pt]{article}
\usepackage{amssymb}
\usepackage{amsbsy}
\usepackage[latin1]{inputenc}
\usepackage{amsthm}
\usepackage{graphicx} 
\usepackage{subfigure}
\usepackage{pst-eucl}
\usepackage[latin1]{inputenc}
\usepackage[english]{babel}
\usepackage{amsmath,amssymb,graphics,mathrsfs}
\usepackage{amsmath,amssymb,latexsym}
\usepackage{float}
\usepackage{graphicx,color}
\usepackage[T1]{fontenc}
\usepackage[active]{srcltx}
\usepackage{multicol}
\usepackage[latin1]{inputenc}
\usepackage{pst-all}
\usepackage{enumerate}
\usepackage{pstricks}
\usepackage{pstricks-add}
\usepackage{setspace}
\usepackage{soul}
\usepackage{cancel}
\usepackage{epsfig}
\usepackage[margin=1in,text={6.5in,9in}]{geometry}

\usepackage{nonfloat}

\usepackage[colorlinks=true,citecolor=red,linkcolor=blue,urlcolor=RubineRed,pdfpagetransition=Blinds,pdftoolbar=false,pdfmenubar=false]{hyperref}

\newcommand{\R}{\hbox{\rm I \kern -5pt R}}     
\newcommand{\p} {\hbox{\rm I \kern -5pt P}}

\def\u        {{\textbf{u}}}
\def\zero        {{\textbf{0}}}

\def\x  {\boldsymbol x}

\def\n        {{\textbf{n}}}

\def \hueco{\noalign{\medskip}}
\def \pato{\forall\,}
\def \beq{\begin{equation}}
\def \eeq{\end{equation}}
\def \ba{\begin{array}}
\def \ea{\end{array}}
\def \dis{\displaystyle}

\newtheorem{prop}{Proposition}[section]
\newtheorem{defi}[prop]{Definition}
\newtheorem{cor}[prop]{Corollary}
\newtheorem{obs}[prop]{Remark}
\newtheorem{lem}[prop]{Lemma}

\begin{document}

\title{Structure preserving finite element schemes for the Navier-Stokes-Cahn-Hilliard system with degenerate mobility} 
\author{F.~Guill\'en-Gonz\'alez\thanks{Dpto. Ecuaciones Diferenciales y An\'alisis Num\'erico and IMUS, 
Universidad de Sevilla, Facultad de Matem\'aticas, C/ Tarfia, S/N, 41012 Sevilla (SPAIN). Email: guillen@us.es},
~and
 G.~Tierra\thanks{Department of Mathematics, University of North Texas, Denton TX (USA). Email:  gtierra@unt.edu}}

\maketitle


\begin{abstract}
In this work we present two new numerical schemes to approximate the Navier-Stokes-Cahn-Hilliard system with degenerate mobility using finite differences in time and finite elements in space. The proposed schemes are conservative, energy-stable and preserve the maximum principle approximately (the amount of the phase variable being outside of the interval $[0,1]$ goes to zero in terms of a truncation parameter).
Additionally, we present several numerical results to illustrate the accuracy and the well behavior of the proposed schemes, as well as a comparison
with the behavior of the Navier-Stokes-Cahn-Hilliard model with constant mobility.

\end{abstract}
%
%
%
%

\section{Introduction}

Many scientific, engineering, and industrial applications related with hydrodynamics and materials science are based on the understanding of the evolution in time of the interface between two or more immiscible fluids.
In particular, the phase-field/diffuse-interface approach has been widely considered in recent times due to its ability to describe topological transitions like droplet coalescence or droplet break-up in a natural way.
The basic idea behind this approach is that the structure of the interface is determined by the competition of the tendencies for mixing and de-mixing, which are balanced through a non-local mixing energy. The most classical model in this direction is the Cahn-Hilliard equation \cite{Cahn1958}, which was introduced to model the thermodynamic forces driving phase separation.
In \cite{HH77}, the authors presented the so-called Model H, which models the behavior of a mixture of two incompressible, viscous Newtonian fluids with the same constant density. Interestingly, by using the rational continuum mechanics framework the authors in \cite{GPV96} arrived at the same model and showed that it satisfies the second law of thermodynamics. The obtained model is usually called the Navier-Stokes-Cahn-Hilliard system.

In this work we consider the Navier-Stokes-Cahn-Hilliard system with the mobility term being a non-linear degenerate function, in order that the flux only acts away of the pure phases. The authors in \cite{Abeletal} established the existence and regularity of weak solutions of a more general system that represents mixtures of fluids with different densities, which includes the particular case of considering matching constant densities, as in the original Navier-Stokes-Cahn-Hilliard system. Additionally, equivalent analytical results have been stablished for nonlocal models associated with the Navier-Stokes-Cahn-Hilliard system \cite{Frigerietal19,Frigerietal15,Frigerietal20}.

The design and study of numerical schemes for both the Cahn-Hilliard and Navier-Stokes-Cahn-Hilliard systems has become very popular these days, although most of the research has targeted the constant mobility case (check \cite{Tierra} for an overview and comparison of different approaches). In recent years there has been also a growing interest in studying numerically these problems using degenerate mobilities. 
The case of logarithmic potentials together with degenerate mobilities were studied using mixed finite element approximations in \cite{Barrett98,Copetti92}, while in \cite{Chenetal19} the authors focus on the same problem but using finite differences. Discontinuous Galerkin methods have also become very popular to approximate phase field systems with degenerate mobilities \cite{DKR23,DKR22,Kay2009,LiuFrankRiviere,XiaXuShu}.

In this work we extend the ideas presented in \cite{KGCHDM}, where two numerical schemes using finite differences in time and finite elements in space to approximate the Cahn-Hilliard equation with degenerate mobility where derived, showing that they are energy-stable and they preserve the maximum principle approximately (the amount of the solution being outside of the interval $[0,1]$ goes to zero in terms of a truncation parameter), where the proofs rely on the introduction and estimation of two singular potentials.
In fact, the ideas of achieving approximate maximum principles have recently proved very useful in the context of chemotaxis models \cite{KMAD19, KMAD21, KMAD22, GuillenTierra}.
The main challenge to be able to extend those ideas to the Navier-Stokes-Cahn-Hilliard system consist on being able to handle the convection term that transport the phase field function $\phi$. To this end, we present a completely new approach for approximating those terms which will lead to numerical schemes that are energy-stable and satisfy approximate maximum principles. Up to our knowledge, these are the first finite element schemes for the Navier-Stokes-Cahn-Hilliard system preserving these properties at discrete level.

This work is organized as follows: In Section~\ref{sec:model} we present the Navier-Stokes-Cahn-Hilliard model with degenerate mobility. After that, we introduce a truncated version of the model and two truncated functionals G$_\varepsilon(\phi)$ and J$_\varepsilon(\phi)$ and their associated estimates. Then, two new numerical schemes together with the 
properties that they satisfy are presented in Section~\ref{sec:schemes}. In particular we show that both schemes are conservative, energy-stable and they satisfy approximate maximum principles. In order to illustrate the well behavior of the proposed numerical schemes and their applicability to simulate complex situations we present several numerical results in Section~\ref{sec:simulations}. Finally we state the conclusions of our work in Section~\ref{sec:conclusions}.

\section{Model}\label{sec:model}
Let $\Omega\subset\mathbb{R}^d$ (with $d=1,2,3$) be a bounded spatial domain with polyhedral boundary $\partial\Omega$ and $[0,T]$ a finite time interval.
The Navier-Stokes-Cahn-Hilliard \cite{GPV96,HH77} system reads:
\beq\label{eq:NSCHgeneral}
\left\{\ba{rcl}
\u_t 
+ (\u\cdot\nabla)\u 
- \nabla\cdot(\nu(\phi)\mathbb{D}(\u)) 
+ \nabla P 
+ \lambda\nabla\cdot(\nabla\phi\otimes\nabla\phi)
&=&\zero\,,
\\ \hueco
\nabla\cdot\u
&=&0\,,
\\ \hueco\dis
\phi_t + \u\cdot\nabla\phi- \nabla \cdot\left[\gamma M(\phi)\nabla\left(\frac{\delta E}{\delta\phi}\right)\right] 
&=&0\,,
\ea\right.
\eeq
where $\u(\x,t)$ denotes the velocity field, $P(\x,t)$ the pressure, $\phi(\x,t)$ the phase field function (with pure phases located at $\phi=0$ and $\phi=1$), $\nu(\phi)$ is the kinematic viscosity of the mixture (depending on the phase), $\mathbb{D}(\u)$ the symmetric gradient of the velocity ($\mathbb{D}(\u)=\frac12(\nabla\u + (\nabla\u)^T)$), $\lambda>0$ a capillarity parameter, $\gamma>0$ being a relaxation time parameter and $\frac{\delta E}{\delta\phi}$ the variational derivative of the total energy of the system:
\beq
E(\u,\phi)
\,:=\,
E_{\rm kin}(\u)
+ \lambda \, E_{\rm mix}(\phi)
\,=\,
\int_\Omega
\frac12|\u|^2 
d\x
+ \lambda\int_\Omega\left(
\frac12|\nabla\phi|^2 + F(\phi)
\right)d\x\,.
\eeq
Moreover, $M(\phi)$ represents the degenerate mobility term
\beq\label{eq:defM}
M(\phi)\,:=\, \phi \, (1-\phi)\,.
\eeq
The thickness of the interface between the two fluids is associated to the (small) parameter $\eta>0$, which is used to define the double well potential $F(\phi)$ as
$$
F(\phi)
\,:=\,
\frac1{4\eta^2}\phi^2(1-\phi)^2
\,=\,
\frac1{4\eta^2}\left(\phi^4 - 2\phi^3 + \frac32\phi^2\right)
-
\frac1{8\eta^2}\phi^2
\,=:\,
F_c(\phi) + F_e(\phi)\,,
$$
where we have separated the potential into a convex function $F_c(\phi)$ and a concave one $F_e(\phi)$. In particular,
$$
F''_c(\phi)=\frac3{\eta^2}\left(\phi - \frac12\right)^2\geq 0
\quad
\mbox{ and }
\quad
F''_e(\phi)=- \frac1{4\eta^2}\leq 0\,.
$$
We now introduce the chemical potential as a new variable $\mu(\x,t)$:
\beq
\mu
\,:=\,
\frac{\delta E}{\delta \phi}
\,=\,
\lambda\frac{\delta E_{\rm mix}}{\delta \phi}
\,=\,
\lambda\big(-\Delta\phi + F'(\phi)\big)\,,
\eeq
then, 
defining 
${p}=P + \frac\lambda2|\nabla\phi|^2 + \lambda F(\phi) - \phi\mu$ and taking advantage of the definition of $\mu$, we arrive to the relation 
$$
\nabla P + \lambda\nabla\cdot(\nabla\phi\otimes\nabla\phi)
\,=\,
\nabla P -\mu\nabla\phi + \lambda\nabla\left(\frac12|\nabla\phi|^2 + F(\phi)\right)
\,=\,
\nabla{p} +\phi\nabla \mu\,,
$$
 hence, using the incompressibility of $\u$, we can rewrite system \eqref{eq:NSCHgeneral} as
\beq\label{eq:NSCHb}
\left\{\ba{rcl}
\u_t 
+ (\u\cdot\nabla)\u 
- \nabla\cdot(\nu(\phi)\mathbb{D}(\u))
+ \nabla  p 
+ \phi\nabla\mu
&=&\zero\,,
\\ \hueco
\nabla\cdot\u
&=&0\,,
\\ \hueco
\phi_t + \nabla\cdot(\phi\u)- \nabla \cdot\left[\gamma M(\phi)\nabla\mu\right] 
&=&0\,,
\\ \hueco
\lambda\big(-\Delta\phi + F'(\phi)\big) - \mu 
&=&0\,.
\ea\right.
\eeq
System \eqref{eq:NSCHb} is supplemented with initial conditions 
\beq\label{eq:ICs}
\u|_{t=0}=\u_0
\quad\quad
\mbox{ and }
\quad\quad
\phi|_{t=0}=\phi_0
\quad \mbox{ in } \Omega\,,
\eeq
and boundary conditions
\beq\label{eq:BCs}
\u|_{\partial\Omega}
\,=\,
\textbf{0}
\quad\quad
\mbox{ and }
\quad\quad
\nabla \phi\cdot \n|_{\partial\Omega}
\,=\,
\nabla \mu\cdot \n|_{\partial\Omega}
\,=\,
0\,,
\eeq
with $\n$ denoting the outward normal vector to the boundary $\partial\Omega$. 

\begin{lem}
Any $(\u,p,\phi,\mu)$ regular enough solution of \eqref{eq:NSCHb}-\eqref{eq:BCs} satisfies the conservation property 
$$
\frac{d}{dt}\left(\int_\Omega \phi \,d\x\right)\,=\,0\,,
$$
and  the following (dissipative) energy law
\beq\label{eq:energylaw}
\dis\frac{d}{dt}
E(\u,\phi)
+\dis 
\big\|\sqrt{\nu(\phi)}\mathbb{D}(\u)\big\|^2_{L^2}
+\gamma \big\|\sqrt{M(\phi)}\nabla\mu\big\|^2_{L^2}
=0\,.
\eeq
\end{lem}
\begin{proof}
The conservation holds just by testing equation \eqref{eq:NSCHb}$_3$ by $1$. 
Testing \eqref{eq:NSCHb} by $(\u,p,\mu,\phi_t)$ we obtain
\eqref{eq:energylaw}, accounting that the velocity convective term $(\u\cdot\nabla)\u $ vanish (using the incompressibility of $\u$) and the force term $\phi\nabla\mu$ cancel with the phase convective term $\nabla\cdot(\phi\u)$.
\end{proof}
\subsection{Truncated non-degenerated problem}

Now we follow the ideas presented for the well-posedness of the Cahn-Hilliard equation with degenerate mobility \cite{ElliotGarcke1996} and for a general version of Navier-Stokes-Cahn-Hilliard with variable density \cite{Abeletal}, by introducing  a truncated non-degenerated version of problem \eqref{eq:NSCHb}, whose solution converges to the solution of the original problem \eqref{eq:NSCHb} when the truncation parameter goes to zero. The key point is to replace the mobility term $M(\phi)$  defined in \eqref{eq:defM} by a truncated (and non-degenerated) version $M_\varepsilon(\phi)$, which depends on a small truncation parameter $\varepsilon>0$. To this end, we first define the truncation functional $T^{1-\varepsilon}_\varepsilon$ such that 
$$
T^{1-\varepsilon}_\varepsilon(\phi)
\,=\,
\left\{\ba{ll}
\varepsilon & \mbox{ if } \phi\leq\varepsilon\,,
\\ \hueco
\phi & \mbox{ if } \varepsilon\leq\phi\leq1-\varepsilon\,,
\\ \hueco
1 - \varepsilon & \mbox{ if } \phi\geq1 - \varepsilon\,.
\ea\right.
$$
In particular, we have the following property
\beq\label{eq:truncprop}
1 - T^{1-\varepsilon}_\varepsilon(\phi)
= T^{1-\varepsilon}_\varepsilon(1 - \phi)\,.
\eeq
Now we can take advantage of the previous truncation function and define the truncated (non-degenerate) mobility $M_\varepsilon(\phi)$ as 
$$
M_\varepsilon(\phi)
\,=\,
T^{1-\varepsilon}_\varepsilon(\phi)
T^{1-\varepsilon}_\varepsilon(1 - \phi)
\,=\,
\left\{\ba{ll}
\varepsilon(1 - \varepsilon) & \mbox{ if } \phi\leq\varepsilon\,,
\\ \hueco
\phi(1-\phi) & \mbox{ if } \varepsilon\leq\phi\leq1-\varepsilon\,,
\\ \hueco
\varepsilon(1 - \varepsilon) & \mbox{ if } \phi\geq1-\varepsilon\,,
\ea\right.
$$
Then the truncated version of model \eqref{eq:NSCHb} reads:

\beq\label{eq:NSCHG2}
\left\{\ba{rcl}
\u_t + (\u\cdot\nabla)\u - \nabla\cdot(\nu(\phi)\mathbb{D}(\u)) + \nabla  p+ T^{1-\varepsilon}_\varepsilon(\phi)\nabla\mu
 &=&\zero\,,
\\ \hueco
\nabla\cdot\u
&=&0\,,
\\ \hueco\dis
\phi_t + \nabla\cdot(T^{1-\varepsilon}_\varepsilon(\phi)\u) 
- \nabla \cdot\left[\gamma M_\varepsilon(\phi)\nabla\mu\right] 
&=&0\,,
\\ \hueco
\lambda\big(-\Delta\phi + F'(\phi)\big) - \mu 
&=&0\,,
\ea\right.
\eeq
supplemented with the initial and boundary conditions presented in \eqref{eq:ICs} and \eqref{eq:BCs}, respectively.
\begin{lem}
Any $(\u,p,\phi,\mu)$ regular enough solution of \eqref{eq:NSCHG2} is conservative, that is, 
$$
\frac{d}{dt}\left(\int_\Omega \phi\, d\x\right)\,=\,0\,,
$$
and satisfies the (dissipative) energy law
\beq\label{eq:energylawtrunc}
\dis\frac{d}{dt}
E(\u,\phi)
+ \big\|\sqrt{\nu(\phi)}\mathbb{D}(\u)\big\|^2_{L^2}
+\gamma\big\|\sqrt{M_\varepsilon(\phi)}\nabla\mu\big\|^2_{L^2}
\,=\,
0\,.
\eeq
\end{lem}
\begin{proof}
The conservation holds just by testing equation \eqref{eq:NSCHG2}$_3$ by $1$.
 Testing \eqref{eq:NSCHG2} by $(\u,p,\mu,\phi_t)$ we obtain \eqref{eq:energylawtrunc}.
\end{proof}

The next step is to follow the ideas presented in \cite{ElliotGarcke1996,KGCHDM} and introduce two singular functionals (that we  denote $G_\varepsilon(\phi)$ and $J_\varepsilon(\phi)$) that will be used to show the additional boundedness of the solution of \eqref{eq:NSCHG2}. Moreover, discrete versions of these functionals will be used to show approximate maximum principle properties for the corresponding numerical schemes.

\subsubsection{Functional $G_\varepsilon(\phi)$}
We define a new functional $G_\varepsilon(\phi)$ such that
$$
G_\varepsilon''(\phi)
\,:=\,
\frac1{M_\varepsilon(\phi)}
\,=\,
\frac1{T^{1-\varepsilon}_\varepsilon(\phi)T^{1-\varepsilon}_\varepsilon(1-\phi)}
\,=\,
\frac1{T^{1-\varepsilon}_\varepsilon(\phi)} + \frac1{T^{1-\varepsilon}_\varepsilon(1-\phi)}.
$$
Then, we define the auxiliary potential $H_\varepsilon(\phi)$ such that 
$$H_\varepsilon''(\phi):=\frac1{T^{1-\varepsilon}_\varepsilon(\phi)} ,
$$
hence
$$
G_\varepsilon''(\phi)
\,=\,
H_\varepsilon''(\phi) + H_\varepsilon''(1 - \phi).
$$
In particular, one has the equalities 
\beq\label{eq:truncpropH-bis}
G_\varepsilon'(\phi)
\,=\,
H_\varepsilon'(\phi) - H_\varepsilon'(1 - \phi),
\eeq
and
\beq\label{eq:truncpropH}
 T^{1-\varepsilon}_\varepsilon(\phi) \nabla H_\varepsilon'(\phi)=\nabla \phi
 =
 -T^{1-\varepsilon}_\varepsilon(1-\phi) \nabla H_\varepsilon'(1-\phi)\,.
\eeq

\begin{lem}\label{lem:estG}
Any regular enough solution of \eqref{eq:NSCHG2} satisfies the equality
\beq\label{eq:estimateGtrunc}
\begin{array}{l}
  \dis
\frac{d}{dt}\left(\int_\Omega G_\varepsilon(\phi) \,d\x\right)
+\dis
\lambda\gamma\int_\Omega (\Delta\phi)^2 d\x
+
\lambda\gamma\int_\Omega F''_c(\phi)|\nabla\phi|^2 d\x
\\
\dis
\qquad =
-\lambda\gamma\int_\Omega F''_e(\phi)|\nabla\phi|^2 d\x
=
\dis
\frac{\lambda\gamma}{4\eta^2}\int_\Omega|\nabla\phi|^2 d\x\,,
\end{array}
\eeq
\end{lem}
\begin{proof}
We test \eqref{eq:NSCHG2}$_3$ by $G_\varepsilon'(\phi)$. Using \eqref{eq:truncprop}, \eqref{eq:truncpropH-bis} and \eqref{eq:truncpropH},  
$$
\ba{rcl}
 T^{1-\varepsilon}_\varepsilon(\phi)\nabla G_\varepsilon'(\phi)
&=&\dis
 T^{1-\varepsilon}_\varepsilon(\phi)\nabla H_{\varepsilon}'(\phi)
- T^{1-\varepsilon}_\varepsilon(\phi)\nabla H_{\varepsilon}'(1-\phi)
 \\ \hueco
 &=&\dis
 T^{1-\varepsilon}_\varepsilon(\phi)\nabla H_{\varepsilon}'(\phi)
+ T^{1-\varepsilon}_\varepsilon(1-\phi)\nabla H_{\varepsilon}'(1-\phi)
- \nabla H_{\varepsilon}'(1-\phi)
   \\ \hueco
&=&\dis
- \nabla H_{\varepsilon}'(1-\phi)
 \,,
 \ea
$$
hence,  taking into account  $\nabla\cdot\u=0$ in $\Omega$ and $\u|_{\partial\Omega}=\textbf{0}$, one has 
$$
\int_\Omega \nabla\cdot(T^{1-\varepsilon}_\varepsilon(\phi)\u) \, G_\varepsilon'(\phi)
=
\int_\Omega  T^{1-\varepsilon}_\varepsilon(\phi)\nabla G_\varepsilon'(\phi) \cdot \u 
=
- \int_\Omega \nabla (H_{\varepsilon}'(1-\phi)) \cdot \u 
=0.
$$ 
On the other hand, using that $M_\varepsilon(\phi)\, G_\varepsilon''(\phi) = 1$, one has 
$$
- \int_\Omega  \nabla \cdot\left[\gamma M_\varepsilon(\phi)\nabla\mu\right] \, G_\varepsilon'(\phi)
= \int_\Omega \gamma M_\varepsilon(\phi)\nabla\mu \cdot \nabla\phi \, G_\varepsilon''(\phi)
= \gamma\int_\Omega \nabla\mu\cdot \nabla\phi \, .
$$
Accounting previous computations,  we obtain 
\beq\label{eq:proofG2trunc}
\dis\frac{d}{dt}\left(\int_\Omega G_\varepsilon(\phi) \,d\x\right)
+ \gamma\int_\Omega \nabla\mu\cdot \nabla\phi \,d\x
=0\,.
\eeq
By testing \eqref{eq:NSCHG2}$_4$ by $-\Delta\phi$, we can write  
\beq\label{eq:proofG3trunc}
\int_\Omega \nabla\mu\cdot \nabla\phi \,d\x
=
\lambda\left(\int_\Omega (\Delta\phi)^2 d\x
+\int_\Omega  F_c''(\phi)|\nabla\phi|^2 \,d\x
+\int_\Omega  F_e''(\phi)|\nabla\phi|^2 \,d\x
\right)\,.
\eeq
Finally combining \eqref{eq:proofG2trunc} with \eqref{eq:proofG3trunc} we obtain \eqref{eq:estimateGtrunc}.

\end{proof}

\subsubsection{Functional $J_\varepsilon(\phi)$}
We define a new functional $J_\varepsilon(\phi)$ such that
$$
J_\varepsilon''(\phi)
\,:=\,
\frac1{\sqrt{M_\varepsilon(\phi)}}
\,=\,
\frac1{\sqrt{T^{1-\varepsilon}_\varepsilon(\phi)T^{1-\varepsilon}_\varepsilon(1-\phi)}}.
$$
Using equality \eqref{eq:truncprop}, one has 
$$
J_\varepsilon''(\phi)
\,:=\,
\frac{\sqrt{T^{1-\varepsilon}_\varepsilon(1-\phi)}}{\sqrt{T^{1-\varepsilon}_\varepsilon(\phi)}} + \frac{\sqrt{T^{1-\varepsilon}_\varepsilon(\phi)}}{\sqrt{T^{1-\varepsilon}_\varepsilon(1-\phi)}}.
$$
Then, we define the auxiliary potential $K_\varepsilon(\phi)$ such that 
$$
K_\varepsilon''(\phi)
\,:=\,
\frac{\sqrt{T^{1-\varepsilon}_\varepsilon(1-\phi)}}{\sqrt{T^{1-\varepsilon}_\varepsilon(\phi)}}\,, 
$$
hence
$$
J_\varepsilon''(\phi)
\,=\,
K_\varepsilon''(\phi) + K_\varepsilon''(1 - \phi).
$$
In particular, one has 
\beq\label{eq:truncpropK-bis}
J_\varepsilon'(\phi)
\,=\,
K_\varepsilon'(\phi) - K_\varepsilon'(1 - \phi),
\eeq
and
\beq\label{eq:truncpropK}
 T^{1-\varepsilon}_\varepsilon(\phi) 
 \nabla K_\varepsilon'(\phi)
 =
 \sqrt{T^{1-\varepsilon}_\varepsilon(\phi)T^{1-\varepsilon}_\varepsilon(1-\phi)}\nabla \phi
 =
-  T^{1-\varepsilon}_\varepsilon(1-\phi) 
 \nabla K_\varepsilon'(1-\phi)
 \,.
\eeq

\begin{lem}\label{lem:estJ}
Any regular enough solution of system \eqref{eq:NSCHG2} satisfies the inequality
\beq\label{eq:estimateJtrunc}
\frac{d}{dt}\left(\int_\Omega J_\varepsilon(\phi) \,d\x\right)
\,\leq\,
\gamma\left(
\int_\Omega M_\varepsilon(\phi)|\nabla\mu|^2 d\x
+
\int_\Omega |\nabla\phi|^2 d\x\right)\,.
\eeq
\end{lem}
\begin{proof}
Using \eqref{eq:truncprop}, \eqref{eq:truncpropK-bis} and \eqref{eq:truncpropK} we can write
$$
\ba{rcl}
 T^{1-\varepsilon}_\varepsilon(\phi)\nabla J_\varepsilon'(\phi)
&=&\dis
 T^{1-\varepsilon}_\varepsilon(\phi)\nabla K_{\varepsilon}'(\phi)
- T^{1-\varepsilon}_\varepsilon(\phi)\nabla K_{\varepsilon}'(1-\phi)
 \\ \hueco
 &=&\dis
 T^{1-\varepsilon}_\varepsilon(\phi)\nabla K_{\varepsilon}'(\phi)
+ T^{1-\varepsilon}_\varepsilon(1-\phi)\nabla K_{\varepsilon}'(1-\phi)
- \nabla K_{\varepsilon}'(1-\phi)
   \\ \hueco
&=&\dis
- \nabla K_{\varepsilon}'(1-\phi)
 \,.
 \ea
$$
Hence, testing \eqref{eq:NSCHG2}$_3$ by $J_\varepsilon'(\phi)$ and taking into account  $\nabla\cdot\u=0$ in $\Omega$, $\u|_{\partial\Omega}=\textbf{0}$ 
and  $M_\varepsilon(\phi)\, J_\varepsilon''(\phi) = \sqrt{M_\varepsilon(\phi)}$, 
 we obtain 
\beq\label{eq:proofJ2trunc}
\dis\frac{d}{dt}\left(\int_\Omega J_\varepsilon(\phi) \,d\x\right)
=- \gamma\int_\Omega \sqrt{M_\varepsilon(\phi)}\nabla\mu\cdot \nabla\phi \,d\x
\,.
\eeq
Finally, bounding the right hand side of last expression by using $ab\leq a^2/2 + b^2/2$ we obtain \eqref{eq:estimateJtrunc}.

\end{proof}

\section{Numerical schemes}\label{sec:schemes}
The numerical schemes proposed in this section are designed as approximations of the corresponding weak formulation of the system \eqref{eq:NSCHG2}. Hereafter $(\cdot,\cdot)$ denotes the $L^2(\Omega)$-scalar product.
For all numerical schemes we consider a partition of the time interval $[0,T]$ into $N$ subintervals, with constant time step $\Delta t = T /N$ and we denote by $\delta_t$ the (backward) discrete time derivative
$$
\delta_t f\,:=\,\frac{f^{n+1} - f^n}{\Delta t}\,.
$$
We consider structured triangulations $\{\mathcal{T}_h\}$ of the domain $\Omega$ (which implies that each element has a side parallel to the $\kappa$-th axis for each $\kappa=1,\dots,d$ with length bounded by the mesh size parameter $h>0$) and with its elements denoted by $I_i$, that is $\{\mathcal{T}_h\}=\bigcup_{i} I_i$. The unknowns $(\u,p,\phi,\mu)$ are approximated using the $C^0$-Finite Elements spaces of order $k\ge1$ (denoted by $\mathbb{P}_k$):
$$
(\u,p,\phi,\mu)\in \textbf{U}_h\times P_h\times\Phi_h\times M_h\,,
$$ 
in such a way that $\Phi_h=\mathbb{P}_1$,  $M_h= \mathbb{P}_k$ and the pair $\textbf{U}_h\times P_h$ satisfies a discrete version of the \textit{inf-sup} condition \cite{GiraultRaviart86}.
Moreover, we are going to take advantage of \textit{mass-lumping} ideas (see \cite{CiarletRaviart73}) to help us achieve boundedness of the unknown $\phi$. To this end we introduce the discrete semi-inner product on $C^0(\overline{\Omega})$ and its induced discrete semi-norm:
$$
(f,g)_h\,:=\,\int_\Omega I_h(f g)d\x
\quad\quad
\mbox{ and }
\quad\quad
|f|_h\,:=\,\sqrt{(f,f)_h}\,,
$$
with $I_h(f(\x))$ denoting the nodal $\mathbb{P}_1$-interpolation of the function $f(\x)$. We use $f_+$ or $f_-$ to denote the positive or negative part of a function $f$ ($f_+(\x)=\max\{f(\x),0\}$ and $f_-(\x)=\min\{f(\x),0\}$). 

\begin{defi}
A numerical scheme to approximate system \eqref{eq:NSCHG2} is energy-stable if the following relation holds:
\beq\label{eq:defenergystability}
\dis\delta_t E(\u^{n+1},\phi^{n+1})
+ \big\|\sqrt{\nu(\phi^{n+1})}\mathbb{D}(\u^{n+1})\big\|^2_{L^2}
+\big\|\sqrt{M_\varepsilon(\phi^{n+1})}\nabla\mu^{n+1}\big\|^2_{L^2}
\,\leq\,
0\,.
\eeq
In particular, energy-stable numerical schemes satisfy the energy decreasing in time property
$$
E(\u^{n+1},\phi^{n+1})
\leq
E(\u^{n},\phi^{n})\, , \quad \forall n\ge 0\,.
$$

\end{defi}
In the following we propose two new numerical schemes, called
$G_\varepsilon$-scheme and $J_\varepsilon$-scheme, designed to take advantage of discrete versions of Lemmas~\ref{lem:estG} and \ref{lem:estJ} to derive approximated discrete maximum principles. Additionally both schemes are conservative and energy-stable. The energy stability relies on using the ideas of Eyre \cite{Eyre} for approximating the potential $F'(\phi)$ (also known as the \textit{convex splitting} approach), which consists on taking implicitly the convex part and explicitly the concave part of the potential because the following relation holds:
\beq\label{eq:eyre}
\frac1{\Delta t}
\int_\Omega \Big(F_c'(\phi^{n+1}) + F_e'(\phi^n)\Big)(\phi^{n+1} - \phi^n) d\x
\geq
\frac1{\Delta t} \int_\Omega \Big(F(\phi^{n+1}) - F(\phi^n)\Big)d\x\,.
\eeq

\subsection{G$_\varepsilon$-Scheme}
We first define a $\mathbb{P}_0$ approximation of $T_\varepsilon^{1-\varepsilon}(\phi)$, called $T^G_h(\phi)$, such that
\beq\label{def:ThG}
 T^G_h(\phi) \nabla I_h(G_\varepsilon'(\phi))
 \,=\,
 -\nabla I_h(H_\varepsilon'(1- \phi)),
 \quad \pato \phi\in \Phi_h\,.
\eeq
In fact, $T^G_h(\phi)$ is a $\mathbb{P}_0$ diagonal matrix function (related to $T^{1-\varepsilon}_\varepsilon(\phi)$),
denoting its diagonal elements $T^G_h(\phi)_{\kappa\kappa}$
for $\kappa=1,\dots,d$. In particular, in each element $I_i$ of the triangulation $\{\mathcal{T}_h\}$ (with nodes $\x_0,\x_\kappa$ in the $\kappa$-th axis) we define
$$
(T^G_h)_{\kappa\kappa}\Big|_{I_i} 
\,:=\,
\left\{\ba{cc}\dis
-\frac{ H_\varepsilon'(1 - \phi(\x_\kappa))  -  H_\varepsilon'(1 - \phi(\x_{0}))}{ G_\varepsilon'(\phi(\x_\kappa))  -  G_\varepsilon'(\phi(\x_{0}))}
& 
 \mbox{ if } \phi(\x_{0}) \neq \phi(\x_\kappa)\,,
\\ \hueco\hueco\dis
 P^h_0 (T^{1-\varepsilon}_\varepsilon(\phi))& 
\mbox{ if } \phi(\x_{0}) = \phi(\x_\kappa)\,,
\ea\right.
$$
where $P^h_0(\phi)|_{I_i}=( \int_{I_i}\phi \, d\x)/|I_i|$ denotes the $\mathbb{P}_0$-projection of $\phi$. It is easy to check that \eqref{def:ThG} holds.
\begin{obs}
There are other possibilities for defining $(T^G_h)_{\kappa\kappa}|_{I_i} $ when $\phi(\x_{0}) = \phi(\x_\kappa)$, for instance the expression $  H_\varepsilon''(1 - \phi(\x_{0})) /  G_\varepsilon''(\phi(\x_{0}))$  can also be considered.
\end{obs}
We also define a $\mathbb{P}_0$ approximation of $M_\varepsilon(\phi)$, called $M_\varepsilon^G(\phi)$ such that
\beq\label{def:MhG}
M_\varepsilon^G(\phi)
 \nabla I_h(G_\varepsilon'(\phi))
 \,=\,\nabla \phi ,
 \quad \pato \phi\in \Phi_h\, .
\eeq
In fact, since $\Phi_h=\mathbb{P}_1$, then  $M^G_\varepsilon(\phi)$ is a $\mathbb{P}_0$ diagonal matrix function (related to $M_\varepsilon(\phi)$) defined by elements by 
$$
(M_\varepsilon^G)_{\kappa\kappa}\Big|_{I_i} 
\,:=\,
\left\{\ba{cc}\dis
\frac{  \phi(\x_{\kappa}) - \phi(\x_0)}{ G_\varepsilon'(\phi(\x_\kappa))  -  G_\varepsilon'(\phi(\x_{0}))}
& 
 \mbox{ if } \phi(\x_{0}) \neq \phi(\x_\kappa)\,,
\\ \hueco\hueco\dis
M\big(P^h_0 (T^{1-\varepsilon}_\varepsilon(\phi))\big)
& 
\mbox{ if } \phi(\x_{0}) = \phi(\x_\kappa)\,.
\ea\right.
$$
\begin{obs}
There are other possibilities for defining $(M_\varepsilon^G)_{\kappa\kappa}|_{I_i} $ when $\phi(\x_{0}) = \phi(\x_\kappa)$, for instance the expression $G_\varepsilon''(\phi(\x_{0}))^{-1}$  was  considered in \cite{KGCHDM}.
\end{obs}

The $G_\varepsilon$-scheme reads:
Find $(\u^{n+1},p^{n+1},\phi^{n+1},\mu^{n+1})\in \bold{U}_h\times P_h \times \Phi_h \times M_h$ such that 
\beq\label{eq:NSCHschemeG}
\left\{\ba{rcl}\dis
\frac1{\Delta t}(\u^{n+1} - \u^n,\bar{\u})
+ \left((\u^n\cdot\nabla)\u^{n+1},\bar{\u} \right)
+ \frac12\left((\nabla\cdot\u^n)\u^{n+1},\bar{\u} \right) 
&&
\\ \hueco
+ (\nu(\phi^{n+1})\mathbb{D}(\u^{n+1}),\mathbb{D}(\bar{\u}))
 + (\nabla p^{n+1},\bar\u)
+ (T^G_h(\phi^{n+1})\nabla\mu^{n+1},\bar{\u})
&=&0\,,
\\ \hueco
(\nabla\cdot\u^{n+1},\bar{p})
&=&0\,,
\\ \hueco\dis
\frac1{\Delta t}\big(\phi^{n+1} - \phi^n, \bar\mu\big)_h 
- (T^G_h(\phi^{n+1}) \u^{n+1},\nabla\bar\mu)
+\gamma\big(M_\varepsilon^G(\phi^{n+1})\nabla\mu^{n+1},\nabla\bar\mu\big) 
&=&0\,,
\\ \hueco
\lambda(\nabla\phi^{n+1},\nabla\bar\phi)
+ \lambda\big(I_h(F_c'(\phi^{n+1})) + I_h(F_e'(\phi^{n})),\bar\phi\big)_h
-(\mu^{n+1},\bar\phi)_h 
&=&
0\,,
\ea\right.
\eeq
for all $(\bar\u, \bar p, \bar\phi, \bar\mu)\in \bold{U}_h\times P_h \times \Phi_h \times M_h$.

\subsubsection{Conservation of volume and energy-stability}

\begin{lem}\label{lem:vol_sta_G}
Any solution 
of the  G$_\varepsilon$-scheme \eqref{eq:NSCHschemeG} is conservative, that is, 
$$
\int_\Omega \phi^{n+1} d\x
\,=\,
\int_\Omega \phi^{n} d\x
\,=\,
\dots
\,=\,
\int_\Omega \phi^{0} d\x\,,
$$
and it is energy stable in the sense that it satisfies the following discrete version of \eqref{eq:energylawtrunc}:  
\beq\label{eq:discreteenergylawG}
\ba{c}
\dis\delta_t
\left( E_{\rm kin}(\u^{n+1})
+ \frac{\lambda}2\|\nabla\phi^{n+1}\|^2_{L^2}
 +\lambda\int_\Omega  I_h(F(\phi^{n+1}))d\x\right)
 \\ \hueco
\dis
+\left\|\sqrt{\nu(\phi^{n+1})}\mathbb{D}(\u^{n+1})\right\|^2_{L^2}
+\gamma\left\|\sqrt{M_\varepsilon^G(\phi^{n+1})}\nabla\mu^{n+1}\right\|^2_{L^2}
\\ \hueco
+\dis
\frac1{2\Delta t}\|\u^{n+1} - \u^n\|^2_{L^2}
+ \frac{\lambda}{2\Delta t}\|\nabla(\phi^{n+1} - \phi^n)\|^2_{L^2}
\leq\dis
0
\,.
\ea
\eeq

\end{lem}
\begin{proof}
The conservation holds just by testing equation \eqref{eq:NSCHschemeG}$_3$ by $\bar \mu =1$. Testing \eqref{eq:NSCHschemeG}$_1$ by $\u^{n+1}$ and \eqref{eq:NSCHschemeG}$_2$ by $p^{n+1}$, we obtain 
\beq\label{eq:proofG1}
\delta_tE_{\rm kin}(\u^{n+1})
+ \frac1{2\Delta t}\left\|\u^{n+1} - \u^n\right\|^2_{L^2}
+\left\|\sqrt{\nu(\phi^{n+1})}\mathbb{D}(\u^{n+1})\right\|^2_{L^2}
+ \left(T^G_h(\phi^{n+1})\nabla\mu^{n+1},\u^{n+1}\right)
\,=\,0\,.
\eeq
Testing \eqref{eq:NSCHschemeG}$_3$ by $\bar\mu=\mu^{n+1}$ and \eqref{eq:NSCHschemeG}$_4$ by $\bar\phi=\delta_t\phi^{n+1}$ we obtain 
\beq\label{eq:proofG2}
\ba{c}\displaystyle
\frac\lambda2\delta_t\|\nabla\phi^{n+1}\|_{L^2}^2 
+ \frac{\lambda}{2\Delta t}\|\nabla(\phi^{n+1} - \phi^n)\|^2_{L^2}
+ \frac{\lambda}{\Delta t}\big(I_h(F_c'(\phi^{n+1})) + I_h(F_e'(\phi^{n})),\phi^{n+1} - \phi^n\big)_h
\\ \hueco\displaystyle
- (T^G_h(\phi^{n+1}) \u^{n+1},\nabla\mu^{n+1})
+\gamma\left\|\sqrt{M_\varepsilon^G(\phi^{n+1})}\nabla\mu^{n+1}\right\|^2_{L^2}
\,=\,0\,.
\ea
\eeq
Finally, taking into account \eqref{eq:eyre} and adding together \eqref{eq:proofG1} with \eqref{eq:proofG2} we obtain \eqref{eq:discreteenergylawG}.
\end{proof}

\subsubsection{Approximated maximum principle}

\begin{lem}\label{le:G_eps}
If $\mathbb{P}_1\subseteq P_h$,   $\Phi_h=\mathbb{P}_1$ and $\mathbb{P}_1\subset M_h$,  then any solution of the  G$_\varepsilon$-scheme \eqref{eq:NSCHschemeG} satisfies the following discrete version of  \eqref{eq:estimateGtrunc}:
\beq\label{eq:discretestimateG}
\ba{c}
\dis\delta_t\left(\int_\Omega I_h(G_\varepsilon(\phi^{n+1})) \,d\x\right)
+\dis
 \gamma\int_\Omega I_h\Big( (\omega^{n+1})^2\Big) \,d\x
+  \gamma\lambda\int_\Omega \Big|\sqrt{R^h(\phi^{n+1})} \nabla\phi^{n+1}\Big|^2 \,d\x 
\\ \hueco
\leq\dis
\frac{\gamma\lambda}{8\eta^2}\left(
\int_\Omega |\nabla\phi^{n}|^2 \,d\x
+ \int_\Omega |\nabla\phi^{n+1}|^2 \,d\x\right)
\,,
\ea
\eeq
\\where $\omega^{n+1}=\mu^{n+1} - \lambda\big(I_h(F_c'(\phi^{n+1})) + I_h(F_e'(\phi^{n}))\big)$ and $R^h(\phi)$ being a $\mathbb{P}_0$ 
diagonal matrix function (related to $F''_c(\phi)$), defined by
$$
R^h_{\kappa\kappa}\Big|_{I_i} 
\,:=\,
\left\{\ba{cc}\dis
\frac{ F'_c(\phi(\x_\kappa))  -  F'_c(\phi(\x_{0}))}{ \phi(\x_\kappa) - \phi(\x_{0}) }
& 
 \mbox{ if } \phi(\x_{0}) \neq \phi(\x_\kappa)\,,
\\ \hueco
\dis{F''_c(\phi(\x_{0}))} 
& 
\mbox{ if } \phi(\x_{0}) = \phi(\x_\kappa)\,.
\ea\right.
$$
Note that $R^h_{\kappa\kappa}\Big|_{I_i}\ge 0$ owing to the convexity of $F_c(\phi)$.

\end{lem}

\begin{proof}

Using \eqref{def:ThG} and taking into account  that $(\nabla\cdot\u^{n+1},\bar p)=0$  for any $\bar p \in P_h$ and the requirement $\mathbb{P}_1\subseteq P_h$, we can write
$$
\ba{rcccl}\dis
\Big(T^G_h(\phi^{n+1})\u^{n+1},\nabla I_h(G_{\varepsilon}'(\phi^{n+1}))\Big)
&=&
\Big(T^G_h(\phi^{n+1})\nabla I_h(G_{\varepsilon}'(\phi^{n+1})) , \u^{n+1}\Big)
&& 
\\ \hueco\dis
&=&
-\Big(\nabla I_h(H_{\varepsilon}'(1 - \phi^{n+1})) , \u^{n+1}\Big)
&=&
0
\,.
\ea
$$
Moreover, using \eqref{def:MhG} we have
$$
\ba{rcl}\dis
\Big(M_\varepsilon^h(\phi^{n+1})\nabla\mu^{n+1},\nabla I_h(G_{\varepsilon}'(\phi^{n+1}))\Big)
&=&
\Big(M_\varepsilon^h(\phi^{n+1})\nabla I_h(G_{\varepsilon}'(\phi^{n+1})) , \nabla\mu^{n+1}\Big) 
\\ \hueco\dis
&=&
\Big(\nabla \phi^{n+1} , \nabla\mu^{n+1}\Big)
\,.
\ea
$$
Finally, testing \eqref{eq:NSCHschemeG}$_3$ by $\bar\mu=I_h(G_{\varepsilon}'(\phi^{n+1}))$ (it is possible because $\mathbb{P}_1\subset M_h$), \eqref{eq:NSCHschemeG}$_4$ by $\bar\phi=\phi^{n+1}$ and working as in \cite{KGCHDM} we obtain \eqref{eq:discretestimateG}.
\end{proof}

\begin{cor}\label{lem:Gapprxbounds}
Under hypotheses of Lemma~\ref{le:G_eps}, any solution 
 of the G$_\varepsilon$-scheme \eqref{eq:NSCHschemeG} satisfies the following estimates:
\beq\label{eq:Glowerbound}
\int_\Omega \big(I_h(\phi_-^{n})\big)^2 d\x 
\leq
C \, \frac{\varepsilon(1 - \varepsilon)}{\eta^2}
\leq
C \, \frac{\varepsilon}{\eta^2}
\eeq
and 
\beq\label{eq:Gupperbound}
\int_\Omega \Big(I_h\big((\phi^{n} - 1)_+\big)\Big)^2 d\x 
\leq
C \, \frac{\varepsilon(1 - \varepsilon)}{\eta^2}
\leq
C \, \frac{\varepsilon}{\eta^2}\,,
\eeq
where $C$ depends on the initial energy $E(\phi^0)$ and $\int_\Omega I_h(G_\varepsilon(\phi^0)) d\x$. 
\end{cor}

\begin{proof}
Following the same arguments presented in \cite{KGCHDM}.
\end{proof}

\begin{obs}\label{rem:boundG}
G$_\varepsilon$-scheme \eqref{eq:NSCHschemeG} satisfies an approximate maximum principle for $\phi^{n}$, because estimates \eqref{eq:Glowerbound} and \eqref{eq:Gupperbound} imply in particular  that
$$
I_h(\phi^{n}_-) \rightarrow 0
\quad\mbox{ and }\quad
I_h \big((\phi^{n} - 1)_+\big) \rightarrow 0
\quad \mbox{ in }  L^2(\Omega),
\quad \mbox{ as }  \quad\varepsilon\rightarrow0\,,
$$
with $\mathcal{O}\left(\sqrt{\varepsilon}/\eta\right)$ accuracy rate.
\end{obs}

\subsection{J$_\varepsilon$-Scheme}
We first define a $\mathbb{P}_0$ approximation of $T_\varepsilon^{1-\varepsilon}(\phi)$, called $T^J_h(\phi)$ such that\beq\label{def:ThJ}
 T^J_h(\phi) \nabla I_h(J_\varepsilon'(\phi))
 \,=\,
 -\nabla I_h(K_\varepsilon'(1- \phi)),
 \quad \pato \phi\in \Phi_h\,.
\eeq
In fact, $T^J_h(\phi)$ is a $\mathbb{P}_0$ diagonal matrix function defined by
$$
(T^J_h)_{\kappa\kappa}\Big|_{I_i} 
\,:=\,
\left\{\ba{cc}\dis
-\frac{ K_\varepsilon'(1 - \phi(\x_\kappa))  -  K_\varepsilon'(1 - \phi(\x_{0}))}{ J_\varepsilon'(\phi(\x_\kappa))  -  J_\varepsilon'(\phi(\x_{0}))}
& 
 \mbox{ if } \phi(\x_{0}) \neq \phi(\x_\kappa)\,,
\\ \hueco\hueco\dis
 P^h_0 (T^{1-\varepsilon}_\varepsilon(\phi))
 & 
\mbox{ if } \phi(\x_{0}) = \phi(\x_\kappa)\,.
\ea\right.
$$
\begin{obs}
There are other possibilities for defining $(T^J_h)_{kk}|_{I_i} $ when $\phi(\x_{0}) = \phi(\x_\kappa)$, for instance the expression $  K_\varepsilon''(1 - \phi(\x_{0})) / J_\varepsilon''(\phi(\x_{0}))$  can also be considered.
\end{obs}
We also define a $\mathbb{P}_0$ approximation of $M_\varepsilon(\phi)$, called $M_\varepsilon^J(\phi)$ such that
\beq\label{def:MhJ}
M^J_\varepsilon(\phi)
 \nabla I_h(J_\varepsilon'(\phi))
 \,=\,\sqrt{M^J_\varepsilon(\phi)}\nabla \phi ,
 \quad \pato \phi\in \Phi_h\,.
\eeq
In fact, $M^J_\varepsilon(\phi)$ is a $\mathbb{P}_0$ diagonal matrix function 
defined as 
$$
(M^J_\varepsilon)_{\kappa\kappa}\Big|_{I_i} 
\,:=\,
\left\{\ba{cc}\dis
\left(\frac{  \phi(\x_{\kappa}) - \phi(\x_0)}{ J_\varepsilon'(\phi(\x_\kappa))  -  J_\varepsilon'(\phi(\x_{0}))}\right)^2
& 
 \mbox{ if } \phi(\x_{0}) \neq \phi(\x_\kappa)\,,
\\ \hueco\hueco\dis
M\big(P^h_0 (T^{1-\varepsilon}_\varepsilon(\phi))\big)
& 
\mbox{ if } \phi(\x_{0}) = \phi(\x_\kappa)\,.
\ea\right.
$$
\begin{obs}
There are other possibilities for defining $(M_\varepsilon^J)_{\kappa\kappa}|_{I_i} $ when $\phi(\x_{0}) = \phi(\x_\kappa)$, for instance the expression $J_\varepsilon''(\phi(\x_{0}))^{-2}$  was  considered in \cite{KGCHDM}.
\end{obs}

The $J_\varepsilon$-scheme reads:
Find $(\u^{n+1},p^{n+1},\phi^{n+1},\mu^{n+1})\in \bold{U}_h\times P_h \times \Phi_h \times M_h$ such that
\beq\label{eq:NSCHschemeJ}
\left\{\ba{rcl}\dis
\frac1{\Delta t}(\u^{n+1} - \u^n,\bar{\u})
+ \left((\u^n\cdot\nabla)\u^{n+1},\bar{\u} \right)
+ \frac12\left((\nabla\cdot\u^n)\u^{n+1},\bar{\u} \right) 
&&
\\ \hueco
\dis
+ (\nu(\phi^{n+1})\mathbb{D}(\u^{n+1}),\mathbb{D}(\bar \u))
 + (\nabla p^{n+1},\bar\u)
+ (T^J_h(\phi^{n+1})\nabla\mu^{n+1},\bar{\u})
&=&0\,,
\\ \hueco
(\nabla\cdot\u^{n+1},\bar{p})
&=&0\,,
\\ \hueco\dis
\frac1{\Delta t}\big(\phi^{n+1} - \phi^n, \bar\mu\big)_h 
- (T^J_h(\phi^{n+1}) \u^{n+1},\nabla\bar\mu)
+\gamma\big(M^J_\varepsilon(\phi^{n+1})\nabla\mu^{n+1},\nabla\bar\mu\big) 
&=&0\,,
\\ \hueco
\lambda(\nabla\phi^{n+1},\nabla\bar\phi)
+ \lambda\big(F_c'(\phi^{n+1}) + F_e'(\phi^{n}),\bar\phi\big)
-(\mu^{n+1},\bar\phi)_h 
&=&
0\,,
\ea\right.
\eeq
for all $(\bar\u, \bar p, \bar\phi, \bar\mu)\in \bold{U}_h\times P_h \times \Phi_h \times M_h$.

\subsubsection{Conservation of volume and energy-stability}
\begin{lem}
Scheme \eqref{eq:NSCHschemeJ} is conservative, that is, 
$$
\int_\Omega \phi^{n+1} d\x
\,=\,
\int_\Omega \phi^{n} d\x
\,=\,
\dots
\,=\,
\int_\Omega \phi^{0} d\x\,,
$$
and it is energy stable in the sense that it satisfies a discrete version of \eqref{eq:energylawtrunc}:  
\beq\label{eq:discreteenergylaw}
\ba{c}
\dis\delta_t E(\u^{n+1},\phi^{n+1})
+\left\|\sqrt{\nu(\phi^{n+1})}\mathbb{D}(\u^{n+1})\right\|^2_{L^2}
+\gamma\left\|\sqrt{M^J_\varepsilon(\phi^{n+1})}\nabla\mu^{n+1}\right\|^2_{L^2}
\\ \hueco\dis
+
\frac1{2\Delta t}\|\u^{n+1} - \u^n\|^2_{L^2}
+ \frac{\lambda}{2\Delta t}\|\nabla(\phi^{n+1} - \phi^n)\|^2_{L^2}
\leq\dis
0
\,.
\ea
\eeq

\end{lem}
\begin{proof}
The conservation holds just by testing equation \eqref{eq:NSCHschemeJ}$_3$ by $\bar \mu =1$. 
\\
Testing \eqref{eq:NSCHschemeJ} by $(\u^{n+1},p^{n+1},\mu^{n+1},\delta_t\phi^{n+1})$ and  working  as in Lemma~\ref{lem:vol_sta_G}, we obtain  \eqref{eq:discreteenergylaw}.
\end{proof}
\subsubsection{Approximated maximum principle}

\begin{lem}\label{le:J_eps}
If $\mathbb{P}_1\subseteq P_h$, $\Phi_h=\mathbb{P}_1$ and $\mathbb{P}_1\subset M_h$, then any solution of the  J$_\varepsilon$-scheme \eqref{eq:NSCHschemeJ} satisfies the following discrete version of estimate \eqref{eq:estimateJtrunc}:
\beq\label{eq:discretestimateJ}
\delta_t\left(\int_\Omega I_h (J_\varepsilon(\phi^{n+1}))\right)
\,\leq\,
\gamma\left(
\int_\Omega  M^J_\varepsilon(\phi ^{n+1})|\nabla\mu^{n+1}|^2 d\x
+ \dis\int_\Omega |\nabla\phi^{n+1}|^2 d\x\right)\,.
\eeq
\end{lem}

\begin{proof}

Using \eqref{def:ThJ} and taking into account  that $(\nabla\cdot\u^{n+1},\bar p)=0$ holds for any $\bar p \in P_h$ and the requirement $\mathbb{P}_1\subseteq P_h$, we can write

$$
\ba{rcccc}\dis
\Big(T^J_h(\phi^{n+1})\u^{n+1},\nabla I_h(J_{\varepsilon}'(\phi^{n+1}))\Big)
&=&
\Big(T^J_h(\phi^{n+1})\nabla I_h(J_{\varepsilon}'(\phi^{n+1})) , \u^{n+1}\Big)
&&
\\ \hueco\dis
&=&
-\Big(\nabla I_h(K_{\varepsilon}'(1 - \phi^{n+1})) , \u^{n+1}\Big)
&=&
0\,.
\ea
$$
Moreover, using \eqref{def:MhJ} we have
$$
\ba{rcl}\dis
\Big(M^J_\varepsilon(\phi^{n+1})\nabla\mu^{n+1},\nabla I_h(J_{\varepsilon}'(\phi^{n+1}))\Big)
&=&\dis
\int_\Omega \sqrt{M^J_\varepsilon(\phi^{n+1})}\nabla\mu^{n+1}\cdot \nabla\phi^{n+1}\,d\x
\\ \hueco\dis
&\leq&\dis
\int_\Omega M^J_\varepsilon(\phi^{n+1})|\nabla\mu^{n+1}|^2\,d\x
+
\int_\Omega |\nabla\phi^{n+1}|^2\,d\x\,.
\ea
$$
Finally, testing \eqref{eq:NSCHschemeJ}$_3$ by $\bar\mu=I_h(J_{\varepsilon}'(\phi^{n+1}))$ and working as in \cite{KGCHDM} we obtain \eqref{eq:discretestimateJ}.
\end{proof}

\begin{cor}\label{lem:Japprxbounds}
Under hypotheses of Lemma~\ref{le:J_eps}, any  solution 
of the J$_\varepsilon$-scheme \eqref{eq:NSCHschemeJ} satisfies the  following estimates:
\beq\label{eq:Jlowerbound}
\int_\Omega \big(I_h(\phi_-^{n})\big)^2 d\x 
\leq
C \, \sqrt{\varepsilon(1 - \varepsilon)}
\leq
C \, \sqrt{\varepsilon}
\eeq
and
\beq\label{eq:Jupperbound}
\int_\Omega \Big(I_h\big((\phi^n - 1)_+\big)\Big)^2 d\x 
\leq
C \, \sqrt{\varepsilon(1 - \varepsilon)}
\leq
C \, \sqrt{\varepsilon}\,,
\eeq
where $C$ depends on the initial energy $E(\phi^0)$ and $\int_\Omega I_h (J_\varepsilon(\phi^0))$. 
\end{cor}
\begin{proof}
Following the same arguments presented in \cite{KGCHDM}.
\end{proof}

\begin{obs}\label{rem:boundJ}
J$_\varepsilon$-scheme \eqref{eq:NSCHschemeJ} satisfies an approximated minimum principle for $\phi^{n}$, because estimates \eqref{eq:Jlowerbound} and \eqref{eq:Jupperbound} imply in particular  that
$$
I_h(\phi^{n}_-) \rightarrow 0
\quad\mbox{ and }\quad
I_h \big((\phi^{n} - 1)_+\big) \rightarrow 0
\quad \mbox{ in }  L^2(\Omega)
\quad \mbox{ as } \quad \varepsilon\rightarrow0\,,
$$
with $\mathcal{O}\left(\sqrt[4]{\varepsilon}\right)$ accuracy rate.
\end{obs}

\section{Numerical simulations}\label{sec:simulations}

In this section we present numerical experiments to illustrate the effectiveness of the proposed numerical schemes. Although the schemes are valid in $\mathbb{R}^d$ $(d = 1,2,3)$, we will present only results in $2D$ for the sake of simplicity. All the simulations have been carried out using \textbf{FreeFem++}  \cite{freefem}, the images generated with \textbf{Paraview} \cite{paraview} and the plots with \textbf{MATLAB} \cite{matlab}.

In some examples we will compare the behavior of the numerical schemes presented in this work with the behavior of the Navier-Stokes-Cahn-Hilliard model with constant mobility, denoting the numerical scheme to approximate this system as CM-scheme. In particular, the CM-scheme reads:
Find $(\u^{n+1},p^{n+1},\phi^{n+1},\mu^{n+1})\in \bold{U}_h\times P_h \times \Phi_h \times M_h$ such that
\beq\label{eq:NSCHschemeCM}
\left\{\ba{rcl}\dis
\frac1{\Delta t}(\u^{n+1} - \u^n,\bar{\u})
+ \left((\u^n\cdot\nabla)\u^{n+1},\bar{\u} \right)
+ \frac12\left((\nabla\cdot\u^n)\u^{n+1},\bar{\u} \right)
&&
\\ \hueco
\dis
+ (\nu(\phi^{n+1})\mathbb{D}(\u^{n+1}),\mathbb{D}(\bar \u)) 
 + (\nabla p^{n+1},\bar\u)
+ (\phi^{n+1}\nabla\mu^{n+1},\bar{\u})
&=&0\,,
\\ \hueco
(\nabla\cdot\u^{n+1},\bar{p})
&=&0\,,
\\ \hueco\dis
\frac1{\Delta t}\big(\phi^{n+1} - \phi^n, \bar\mu\big)_h 
- (\phi^{n+1} \u^{n+1},\nabla\bar\mu)
+\gamma\big(\nabla\mu^{n+1},\nabla\bar\mu\big) 
&=&0\,,
\\ \hueco
\lambda(\nabla\phi^{n+1},\nabla\bar\phi)
+ \lambda\big(F_c'(\phi^{n+1}) + F_e'(\phi^{n}),\bar\phi\big)
-(\mu^{n+1},\bar\phi)_h 
&=&
0\,,
\ea\right.
\eeq
for all $(\bar\u, \bar p, \bar\phi, \bar\mu)\in \bold{U}_h\times P_h \times \Phi_h \times M_h$.
\begin{lem}
Scheme \eqref{eq:NSCHschemeCM} is conservative, that is, 
$$
\int_\Omega \phi^{n+1} d\x
\,=\,
\int_\Omega \phi^{n} d\x
\,=\,
\dots
\,=\,
\int_\Omega \phi^{0} d\x\,,
$$
and it is energy stable because it satisfies the following relation:  
\beq\label{eq:discreteenergylaw-bis}
\ba{c}
\dis\delta_t E(\u^{n+1},\phi^{n+1})
+\left\|\sqrt{\nu(\phi^{n+1})}\mathbb{D}(\u^{n+1})\right\|^2_{L^2}
+\gamma\left\|\nabla\mu^{n+1}\right\|^2_{L^2}
\\ \hueco\dis
+
\frac1{2\Delta t}\|\u^{n+1} - \u^n\|^2_{L^2}
+ \frac{\lambda}{2\Delta t}\|\nabla(\phi^{n+1} - \phi^n)\|^2_{L^2}
\leq\dis
0
\,.
\ea
\eeq
\end{lem}
\begin{proof}
Following the same arguments in Lemma~\ref{lem:vol_sta_G}.
\end{proof}

In the first part of the section we present the iterative algorithms that we use for approximating each of the nonlinear schemes. Then we present the results of several examples, which have been chosen to illustrate three different aspects of the model and numerical schemes: (a)  the accuracy; (b)  the dissipation of energy when no external forces are applied to the system; (c) the different behavior of the system (and the bounds of $\phi$) when using degenerate or constant mobility; (d)  the model and numerical schemes can be considered to represent situations with external forces
applied to the system.

For all the simulations we have considered a square domain $\Omega=[0,1]^2$, constant viscosity $\nu=1$ and the parameters values $\gamma=10^{-3}$, $\lambda=10^{-1}$, $\eta=10^{-2}$ and $\varepsilon=10^{-8}$ and a structured mesh of size $h=1/100$. 
The choice of the discrete spaces considered is: $\mathbb{P}_1-bubble \times \mathbb{P}_1$ for the pair $(\textbf{U}_h\times P_h)$ (mini-element, that is known to be a stable pair for Navier-Stokes \cite{GiraultRaviart86}) and $\mathbb{P}_1 \times \mathbb{P}_1$  for the pair $\Phi_h\times M_h$. In particular, the discrete spaces satisfy the requirements to achieve the approximate maximum principles results presented in Corollaries \ref{lem:Gapprxbounds} and \ref{lem:Japprxbounds}.

\subsection{Iterative methods}
The  schemes presented in this paper are all coupled and nonlinear. In the following we detail the iterative algorithms that we have considered to approximate the nonlinear $G_\varepsilon$-scheme \eqref{eq:NSCHschemeG}, $J_\varepsilon$-scheme \eqref{eq:NSCHschemeJ} and CM-scheme \eqref{eq:NSCHschemeCM}, which are based on decoupling the fluid part and the phase field part to reduce the size of the linear systems to be solved, together with using Newton's method in the nonlinear terms for which we are able to compute their derivative. Let $(\u^{n},p^{n},\phi^{n},\mu^{n})$ and $(\u^{\ell},p^{\ell},\phi^{\ell},\mu^{\ell})$ be known (we consider initially $(\u^{\ell=0},p^{\ell=0},\phi^{\ell=0},\mu^{\ell=0})=(\u^{n},p^{n},\phi^{n},\mu^{n})$). In all cases we iterate until the relative $L^2$-errors reach the required tolerance \texttt{TOL}:
$$
\frac{\|(\u^{\ell+1} - \u^{\ell},p^{\ell+1} - p^{\ell},\phi^{\ell+1} - \phi^{\ell},\mu^{\ell+1} - \mu^{\ell})\|_{L^2\times L^2\times L^2\times L^2}}{\|(\u^{\ell+1} ,p^{\ell+1} ,\phi^{\ell+1} ,\mu^{\ell+1})\|_{L^2\times L^2\times L^2\times L^2}}
\,\leq\,
\texttt{TOL}\,.
$$
In particular, in all the examples the tolerance will be fixed as $\texttt{TOL}=10^{-4}$.
\subsubsection{Iterative algorithm to approximate G$_{\varepsilon}$-Scheme} 
\begin{itemize}
\item[Step 1]: Find $(\phi^{\ell+1},\mu^{\ell+1})\in \Phi_h \times M_h$ such that for all $(\bar\phi, \bar\mu)\in \Phi_h \times M_h$
\begin{equation*}\label{eq:NSCHschemeGIter1}
\left\{\ba{rcl}\dis
\frac1{\Delta t}\big(\phi^{\ell+1} , \bar\mu\big)_h 
+\big(M_\varepsilon^G(\phi^{\ell})\nabla\mu^{\ell+1},\nabla\bar\mu\big) 
&=&\dis
\frac1{\Delta t}\big(\phi^n, \bar\mu\big)_h 
+ (T^G_h(\phi^{\ell}) \u^{\ell},\nabla\bar\mu)\,,
\\ \hueco
\lambda(\nabla\phi^{\ell+1},\nabla\bar\phi)
+ \lambda(I_h(F_c''(\phi^{\ell}))\phi^{\ell+1} ,\bar\phi)_h
-(\mu^{\ell+1},\bar\phi)_h 
&=&
- \lambda\big(I_h(F_c'(\phi^{\ell})) 
-I_h(F_c''(\phi^{\ell}))\phi^\ell ,\bar\phi\big)_h
\\ \hueco
&&
- \lambda( I_h(F_e'(\phi^{n})),\bar\phi\big)_h
\,.
\ea\right.
\end{equation*}

\item[Step 2]: Find $(\u^{\ell+1},p^{\ell+1})\in \bold{U}_h\times P_h$ such that for all $(\bar\u, \bar p)\in \bold{U}_h\times P_h$
\begin{equation*}\label{eq:NSCHschemeGIter2}
\left\{\ba{rcl}\dis
\frac1{\Delta t}(\u^{\ell+1},\bar{\u})
+ \left((\u^n\cdot\nabla)\u^{\ell+1},\bar{\u} \right)
+ \frac12\left((\nabla\cdot\u^n)\u^{\ell+1},\bar{\u} \right)
&&
\\ \hueco
+ \big(\nu(\phi^{\ell +1})\mathbb{D}(\u^{\ell+1}),\mathbb{D}(\bar\u)\big) 
 + (\nabla p^{\ell+1},\bar\u)
&=&
\dis\frac1{\Delta t}(\u^n,\bar{\u})
- (T^G_h(\phi^{\ell})\nabla\mu^{\ell+1},\bar{\u})
\,,
\\ \hueco
(\nabla\cdot\u^{\ell+1},\bar{p})
&=&0\,.
\ea\right.
\end{equation*}
\end{itemize}

\subsubsection{Iterative algorithm to approximate J$_{\varepsilon}$-Scheme} 
\begin{itemize}
\item[Step 1]: Find $(\phi^{\ell+1},\mu^{\ell+1})\in \Phi_h \times M_h$ such that for all $(\bar\phi, \bar\mu)\in \Phi_h \times M_h$
\begin{equation*}\label{eq:NSCHschemeJIter1}
\left\{\ba{rcl}\dis
\frac1{\Delta t}\big(\phi^{\ell+1} , \bar\mu\big)_h 
+\big(M_\varepsilon^J(\phi^{\ell})\nabla\mu^{\ell+1},\nabla\bar\mu\big) 
&=&\dis
\frac1{\Delta t}\big(\phi^n, \bar\mu\big)_h 
+ (T^J_h(\phi^{\ell}) \u^{\ell},\nabla\bar\mu)\,,
\\ \hueco
\lambda(\nabla\phi^{\ell+1},\nabla\bar\phi)
+ \lambda\big(F_c''(\phi^{\ell})\phi^{\ell +1},\bar\phi\big)
-(\mu^{\ell+1},\bar\phi)_h 
&=&
-\lambda\big(F_c'(\phi^{\ell})- F_c''(\phi^{\ell})\phi^{\ell} + F_e'(\phi^{n}),\bar\phi\big)
\,.
\ea\right.
\end{equation*}

\item[Step 2]: Find $(\u^{\ell+1},p^{\ell+1})\in \bold{U}_h\times P_h$ such that for all $(\bar\u, \bar p)\in \bold{U}_h\times P_h$
\begin{equation*}\label{eq:NSCHschemeJIter2}
\left\{\ba{rcl}\dis
\frac1{\Delta t}(\u^{\ell+1},\bar{\u})
+ \left((\u^n\cdot\nabla)\u^{\ell+1},\bar{\u} \right)
+ \frac12\left((\nabla\cdot\u^n)\u^{\ell+1},\bar{\u} \right)
&&
\\ \hueco
+ \big(\nu(\phi^{\ell +1})\mathbb{D}(\u^{\ell+1}),\mathbb{D}(\bar\u)\big) 
 + (\nabla p^{\ell+1},\bar\u)
&=&
\dis\frac1{\Delta t}(\u^n,\bar{\u})
- (T^J_h(\phi^{\ell})\nabla\mu^{\ell+1},\bar{\u})
\,,
\\ \hueco
(\nabla\cdot\u^{\ell+1},\bar{p})
&=&0\,.
\ea\right.
\end{equation*}
\end{itemize}

\subsubsection{Iterative algorithm to approximate CM-Scheme} 
\begin{itemize}
\item[Step 1]: Find $(\phi^{\ell+1},\mu^{\ell+1})\in \Phi_h \times M_h$ such that for all $(\bar\phi, \bar\mu)\in \Phi_h \times M_h$
\begin{equation*}\label{eq:NSCHschemeJIter1}
\left\{\ba{rcl}\dis
\frac1{\Delta t}\big(\phi^{\ell+1} , \bar\mu\big)_h 
+\big(M\nabla\mu^{\ell+1},\nabla\bar\mu\big) 
&=&\dis
\frac1{\Delta t}\big(\phi^n, \bar\mu\big)_h 
+ (\phi^{\ell} \u^{\ell},\nabla\bar\mu)\,,
\\ \hueco
\lambda(\nabla\phi^{\ell+1},\nabla\bar\phi)
+ \lambda\big(F_c''(\phi^{\ell})\phi^{\ell +1},\bar\phi\big)
-(\mu^{\ell+1},\bar\phi)_h 
&=&
-\lambda\big(F_c'(\phi^{\ell})- F_c''(\phi^{\ell})\phi^{\ell} + F_e'(\phi^{n}),\bar\phi\big)
\,.
\ea\right.
\end{equation*}

\item[Step 2]: Find $(\u^{\ell+1},p^{\ell+1})\in \bold{U}_h\times P_h$ such that for all $(\bar\u, \bar p)\in \bold{U}_h\times P_h$
\begin{equation*}\label{eq:NSCHschemeJIter2}
\left\{\ba{rcl}\dis
\frac1{\Delta t}(\u^{\ell+1},\bar{\u})
+ \left((\u^n\cdot\nabla)\u^{\ell+1},\bar{\u} \right)
+ \frac12\left((\nabla\cdot\u^n)\u^{\ell+1},\bar{\u} \right)
&&
\\ \hueco
+ \big(\nu(\phi^{\ell +1})\mathbb{D}(\u^{\ell+1}),\mathbb{D}(\bar\u)\big) 
 + (\nabla p^{\ell+1},\bar\u)
&=&
\dis\frac1{\Delta t}(\u^n,\bar{\u})
- (\phi^{\ell}\nabla\mu^{\ell+1},\bar{\u})
\,,
\\ \hueco
(\nabla\cdot\u^{\ell+1},\bar{p})
&=&0\,.
\ea\right.
\end{equation*}
\end{itemize}

\subsection{Example I. Accuracy study}
In this example we estimate numerically the order of convergence in time  of the two presented numerical schemes. We now introduce some additional notation. The individual errors using discrete norms and the convergence rate between two 
consecutive 
time steps of size $\Delta t$ and $\widetilde{\Delta t}$ are defined as
$$
e_2(\psi):=\|\psi_{exact} - \psi_s\|_{L^2(\Omega)},
e_1(\psi):=\|\psi_{exact} - \psi_s\|_{H^1(\Omega)}
\mbox{ and }
r_i(\cdot):=\left[\log\left(\frac{e_i(\cdot)}{\tilde{e}_i(\cdot)}\right)\right]\Big/ \left[\log\left(\frac{\Delta t}{\widetilde{\Delta t}}\right)\right]\,.
$$
We compute the EOC (Experimental Order of Convergence) using as reference (or exact) solution the one obtained by solving the system using the initial condition
\begin{equation*}\label{eq:initialcondex1}
\phi^0(x,y)=\frac12\left(\cos(5\pi x)\sin\left(3\pi y+\frac{\pi}2\right) + 1\right)
\end{equation*}
and the discretization parameters $h = 1/100$ and $\Delta t = 10^{-8}$. This configuration is presented in Figure~\ref{fig:Ex1FinalConfig}.

\begin{figure}[H]
\begin{center}
\includegraphics[width=0.32\textwidth]{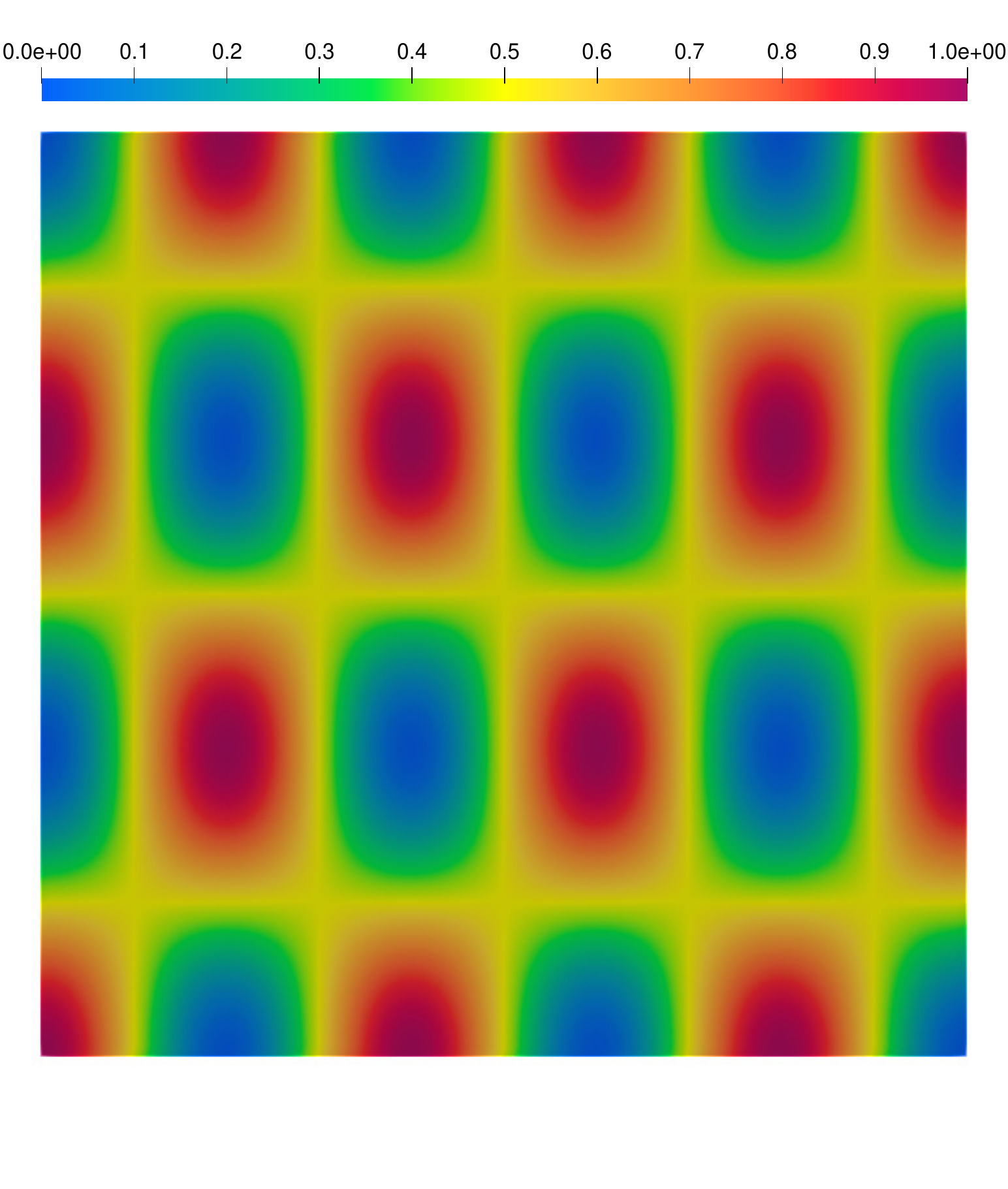}
\includegraphics[width=0.32\textwidth]{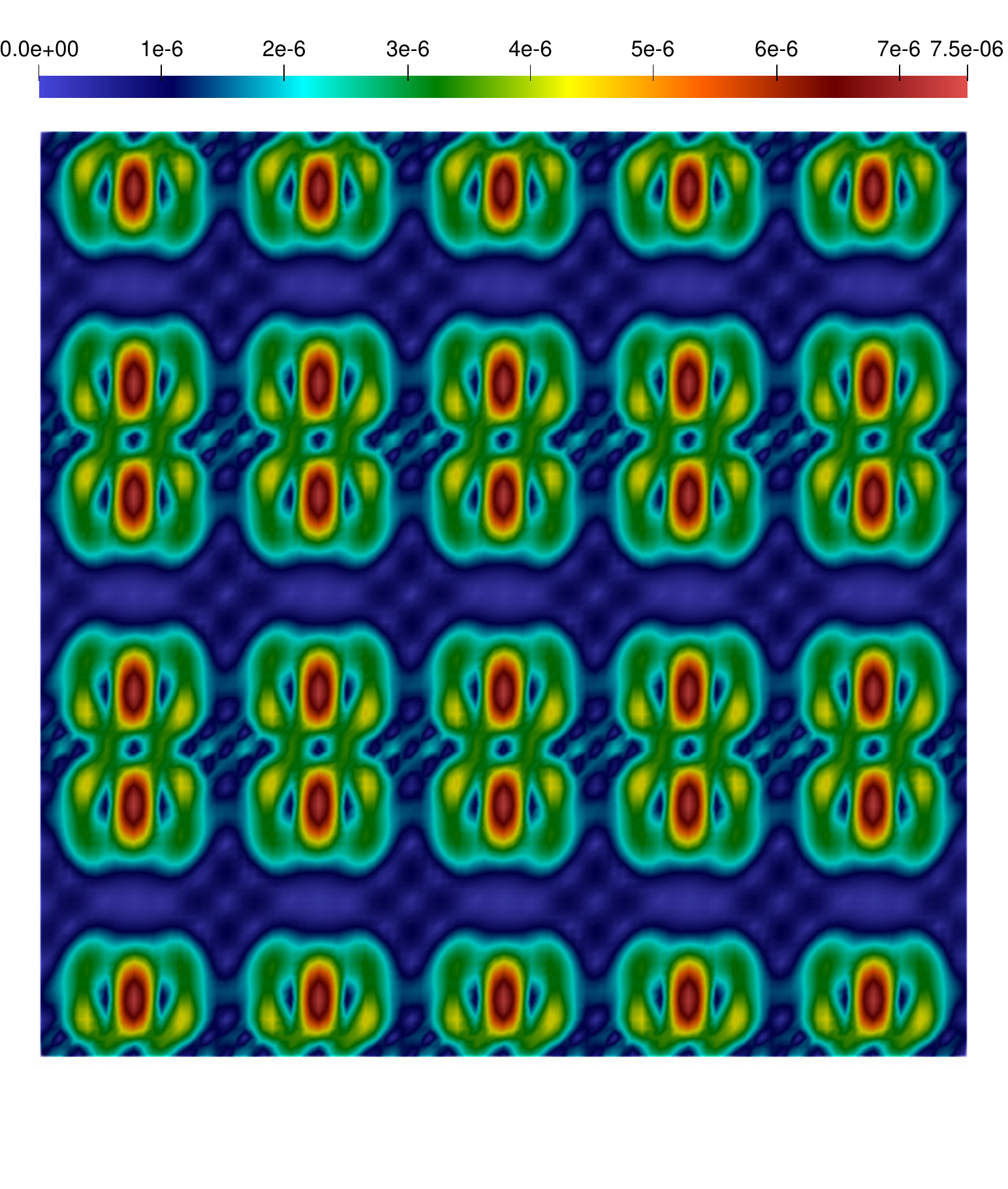}
\includegraphics[width=0.32\textwidth]{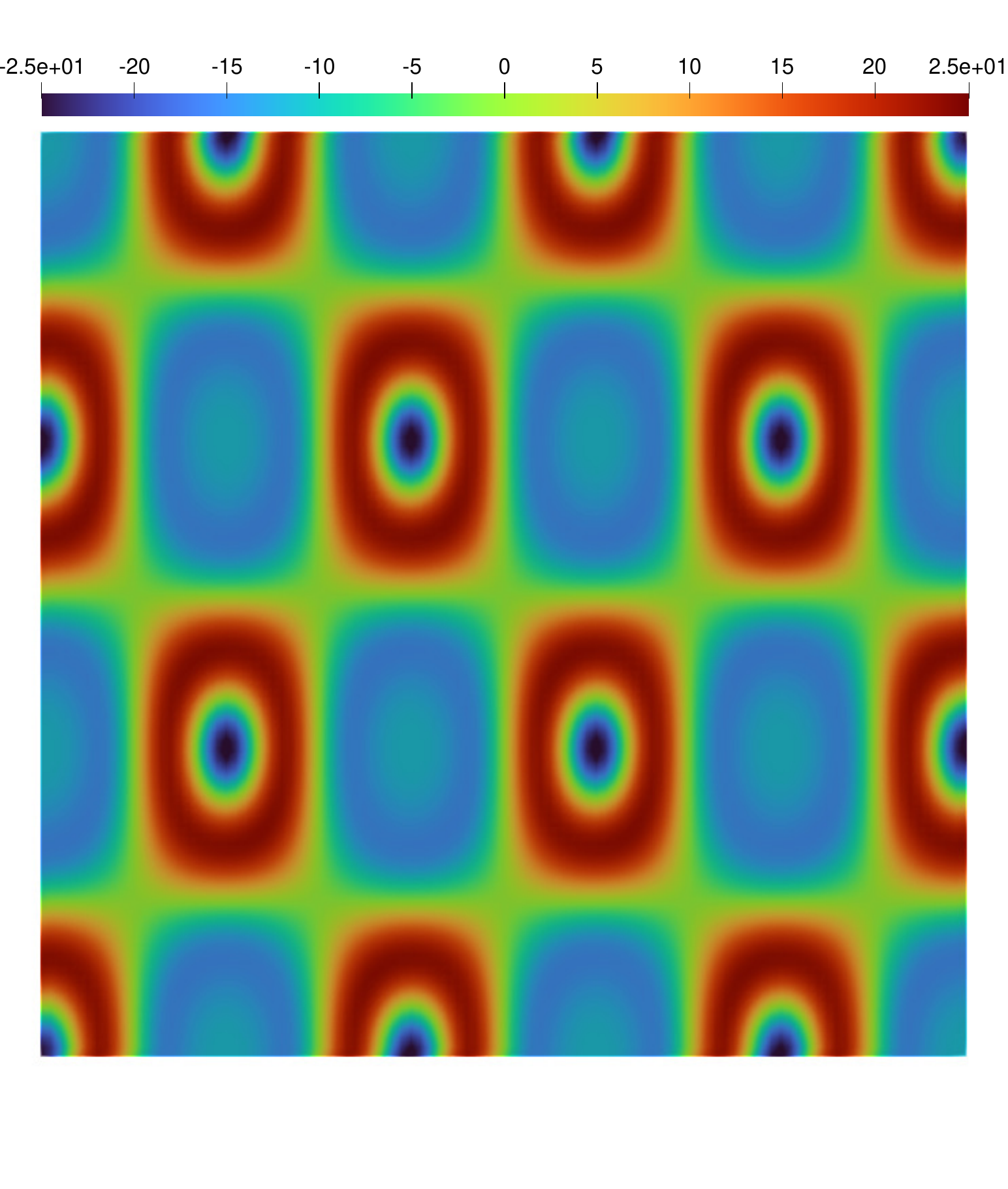}
\caption{Example I. Experimental Order of Convergence. Reference solution for $\phi$ (left), $|\u|$ (center) and $p$ (right).} \label{fig:Ex1FinalConfig}
\end{center}
\end{figure}

\begin{table}
\begin{tabular}{|c|cc|cc|cc|cc|}                      
\hline                                                                                                                                                     
$\Delta t $& $e_2(\phi)$  & $r_2(\phi)$  & $e_1(\phi)$ & $r_1(\phi)$ & $e_2(\mu)$ & $r_2(\mu)$ & $e_1(\mu)$ & $r_1(\mu)$ 
\\ \hline
$1\times10^{-6}$ & $0.135\times10^{-7}$ & $ - $ & $0.104\times10^{-5}$ & $ - $ & $0.145\times10^{-2}$ & $ - $ & $0.766\times10^{-1}$ & $ - $ 
\\                         
$2\times10^{-6}$ & $0.067\times10^{-7}$ & $ 1.0162 $ & $0.051\times10^{-5} $ & $ 1.0148 $ & $0.071\times10^{-2}$ & $ 1.0145 $ & $0.379\times10^{-1}$ & $ 1.0145 $ 
\\
$3\times10^{-6}$ & $0.044\times10^{-7}$ & $ 1.0268 $ & $ 0.034\times10^{-5}$ & $ 1.0240 $ & $0.047\times10^{-2}$ & $ 1.0252 $ & $0.250\times10^{-1}$ & $ 1.0252 $ 
\\
$4\times10^{-6}$ & $0.032\times10^{-7}$ & $ 1.0362 $ & $0.025\times10^{-5}$ & $ 1.0316 $ & $0.035\times10^{-2}$ & $ 1.0360 $ & $ 0.185\times10^{-1}$ & $  1.0359$ 
\\
$5\times10^{-6}$ & $0.026\times10^{-7}$ & $ 1.0463 $ & $0.020\times10^{-5}$ & $ 1.0414 $ & $0.027\times10^{-2}$ & $ 1.0469 $ & $0.147\times10^{-1}$ & $ 1.0468 $ 
\\
$6\times10^{-6}$ & $0.021\times10^{-7}$ & $ 1.0550 $ & $0.016\times10^{-5}$ & $ 1.0525 $ & $ 0.023\times10^{-2}$ & $ 1.0580 $ & $0.121\times10^{-1}$ & $ 1.0580 $ 
\\
$7\times10^{-6}$ & $0.018\times10^{-7}$ & $ 1.0750 $ & $ 0.014\times10^{-5}$ & $ 1.0683 $ & $ 0.019\times10^{-2}$ & $ 1.0694 $ & $0.102\times10^{-1}$ & $ 1.0693 $ 
\\ \hline
$\Delta t$ & $e_2(\u)$  & $r_2(\u)$  & $e_1(\u)$ & $r_1(\u)$ & $e_2(p)$ & $r_2(p)$ 
\\ \hline    
$1\times10^{-6}$ &  $0.173\times10^{-7}$  & $ - $  & $0.130\times10^{-5}$ & $ - $ & $0.798\times10^{-3}$ & $ - $ 
\\ 
$2\times10^{-6}$ &  $0.084\times10^{-7}$  & $ 1.0280 $  & $0.063\times10^{-5}$ & $ 1.0148 $ & $0.395\times10^{-3}$ & $ 1.0145 $ 
\\    
$3\times10^{-6}$ &  $0.055\times10^{-7}$  & $ 1.0330 $  & $0.042\times10^{-5}$ & $ 1.0240 $ & $  0.260\times10^{-3}$ & $ 1.0252 $ 
\\   
$4\times10^{-6}$ &  $0.041\times10^{-7}$  & $ 1.0414 $  & $0.031\times10^{-5}$ & $ 1.0316 $ & $0.193\times10^{-3}$ & $ 1.0360 $ 
\\   
$5\times10^{-6}$ &  $0.032\times10^{-7}$  & $ 1.0511 $  & $0.024\times10^{-5}$ & $ 1.0414 $ & $0.153\times10^{-3}$ & $ 1.0469 $ 
\\   
$6\times10^{-6}$ &  $0.027\times10^{-7}$  & $ 1.0614 $  & $0.020\times10^{-5}$ & $ 1.0525 $ & $0.126\times10^{-3}$ & $ 1.0580 $ 
\\
$7\times10^{-6}$ &  $0.022\times10^{-7}$  & $ 1.0723 $  & $0.017\times10^{-5}$ & $ 1.0683 $ & $0.107\times10^{-3}$ & $ 1.0694 $ 
\\ \hline
\end{tabular}
\caption{Example I. G$_\varepsilon$-Scheme. Experimental absolute errors and order of convergences with reference solution computed using $\Delta t=10^{-8}$. }\label{tab:Ex1_orderG}
\end{table}

The convergence history for a sequence of time steps using $h = 1/100$ is presented in Tables~\ref{tab:Ex1_orderG} and \ref{tab:Ex1_orderJ} ($G_\varepsilon$-scheme and $J_\varepsilon$-scheme, respectively).
We can observe that all the unknowns for the two schemes show a good performance from the order of convergence point of view, because all of them achieve the expected order one in time.

\begin{table}
\begin{tabular}{|c|cc|cc|cc|cc|}                      
\hline                                                                                                                                                                                                                                          
$\Delta t $& $e_2(\phi)$  & $r_2(\phi)$  & $e_1(\phi)$ & $r_1(\phi)$ & $e_2(\mu)$ & $r_2(\mu)$ & $e_1(\mu)$ & $r_1(\mu)$ 
\\ \hline                                                                                                                                                        
$1\times10^{-6}$ & $0.732\times10^{-7}$ & $ - $ & $0.224\times10^{-4}$ & $ - $ & $0.138\times10^{-2}$ & $ - $ & $1.034\times10^{-1}$ & $ - $ 
\\ 
$2\times10^{-6}$ & $0.365\times10^{-7}$ & $ 1.0029 $ & $ 0.112\times10^{-4}$ & $ 1.0004 $ & $0.068\times10^{-2}$ & $ 1.0141 $ & $0.499\times10^{-1}$ & $ 1.0050 $ 
\\ 
$3\times10^{-6}$ & $0.241\times10^{-7}$ & $ 1.0177 $ & $ 0.074\times10^{-4}$ & $ 1.0170 $ & $0.045\times10^{-2}$ & $ 1.0250 $ & $0.330\times10^{-1}$ & $ 1.0197 $ 
\\ 
$4\times10^{-6}$ & $0.179\times10^{-7}$ & $ 1.0306$ & $0.055\times10^{-4}$ & $ 1.0301 $ & $0.033\times10^{-2}$ & $ 1.0358 $ & $0.245\times10^{-1} $ & $ 1.0321 $ 
\\ 
$5\times10^{-6}$ & $0.142\times10^{-7}$ & $ 1.0427 $ & $0.043\times10^{-4}$ & $ 1.0423 $ & $0.026\times10^{-2}$ & $ 1.0468 $ & $0.194\times10^{-1}$ & $ 1.0439 $ 
\\ 
$6\times10^{-6}$ & $0.117\times10^{-7}$ & $ 1.0544 $ & $0.036\times10^{-4}$ & $ 1.0541 $ & $0.021\times10^{-2}$ & $ 1.0579 $ & $ 0.160\times10^{-1}$ & $ 1.0555 $ 
\\ 
$7\times10^{-6}$ & $0.099\times10^{-7}$ & $ 1.0673 $ & $ 0.030\times10^{-4}$ & $ 1.0671 $ & $ 0.018\times10^{-2}$ & $ 1.0693 $ & $0.136\times10^{-1}$ & $ 1.0673 $ 
\\ \hline
$\Delta t$ & $e_2(\u)$  & $r_2(\u)$  & $e_1(\u)$ & $r_1(\u)$ & $e_2(p)$ & $r_2(p)$ 
\\ \hline                            
$1\times10^{-6}$ &  $0.483\times10^{-7}$  & $ - $  & $0.160\times10^{-4}$ & $ - $ & $0.7727$ & $ - $ 
\\     
$2\times10^{-6}$ &  $0.232\times10^{-7}$  & $ 1.0588 $  & $0.074\times10^{-4}$ & $ 1.1119 $ & $0.3734$ & $ 1.0492 $ 
\\     
$3\times10^{-6}$ &  $0.153\times10^{-7}$  & $ 1.0239 $  & $0.048\times10^{-4}$ & $ 1.0540 $ & $ 0.2464$ & $ 1.0250 $ 
\\     
$4\times10^{-6}$ &  $0.113\times10^{-7}$  & $ 1.0348 $  & $0.035\times10^{-4}$ & $ 1.0558 $ & $ 0.1829$ & $ 1.0358 $ 
\\     
$5\times10^{-6}$ &  $0.090\times10^{-7}$  & $ 1.0458 $  & $0.028\times10^{-4}$ & $ 1.0620 $ & $0.1448$ & $ 1.0468 $ 
\\     
$6\times10^{-6}$ &  $0.074\times10^{-7}$  & $ 1.0570 $  & $0.023\times10^{-4}$ & $ 1.0703 $ & $0.1194$ & $ 1.0579 $ 
\\ 
$7\times10^{-6}$ &  $0.063\times10^{-7}$  & $ 1.0687 $  & $0.019\times10^{-4}$ & $ 1.0796 $ & $0.1013$ & $ 1.0693 $ 
\\ \hline
\end{tabular}
\caption{Example I. J$_\varepsilon$-Scheme. Experimental absolute errors and order of convergences with reference solution computed using $\Delta t=10^{-8}$. }\label{tab:Ex1_orderJ}
\end{table}

\subsection{Example II. Energy stability and bound preservation}

In this example we are interested in studying the dissipative character of the proposed numerical schemes as well as their ability of achieving the expected boundedness preservation of the variable $\phi$. In all the simulations we consider the time step $\Delta t=10^{-4}$.

\subsubsection{Merging droplets}

In this case we have designed an initial condition resembling two droplets of a fluid immersed in another fluid, whose boundaries are slightly touching each other. The two droplets will merge arriving to an equilibrium configurations with a single larger droplet.

The dynamic of the two merging drops is presented in Figure~\ref{fig:Ex2_merging_dyn_G} ($G_\varepsilon$-scheme), Figure~\ref{fig:Ex2_merging_dyn_J} ($J_\varepsilon$-scheme) and Figure~\ref{fig:Ex2_merging_dyn_CM} (CM-scheme), where we can observe how the movement of the interface produces motion in the fluid (that was originally at rest) to produce a single circular drop while maintaining the volume of the two fluids. It is interesting to note that the dynamic of the constant mobility case (CM-scheme) is much faster than the one exhibited by the non-constant mobility. 

\begin{figure}[H]
\begin{center}
\includegraphics[scale=0.1125]{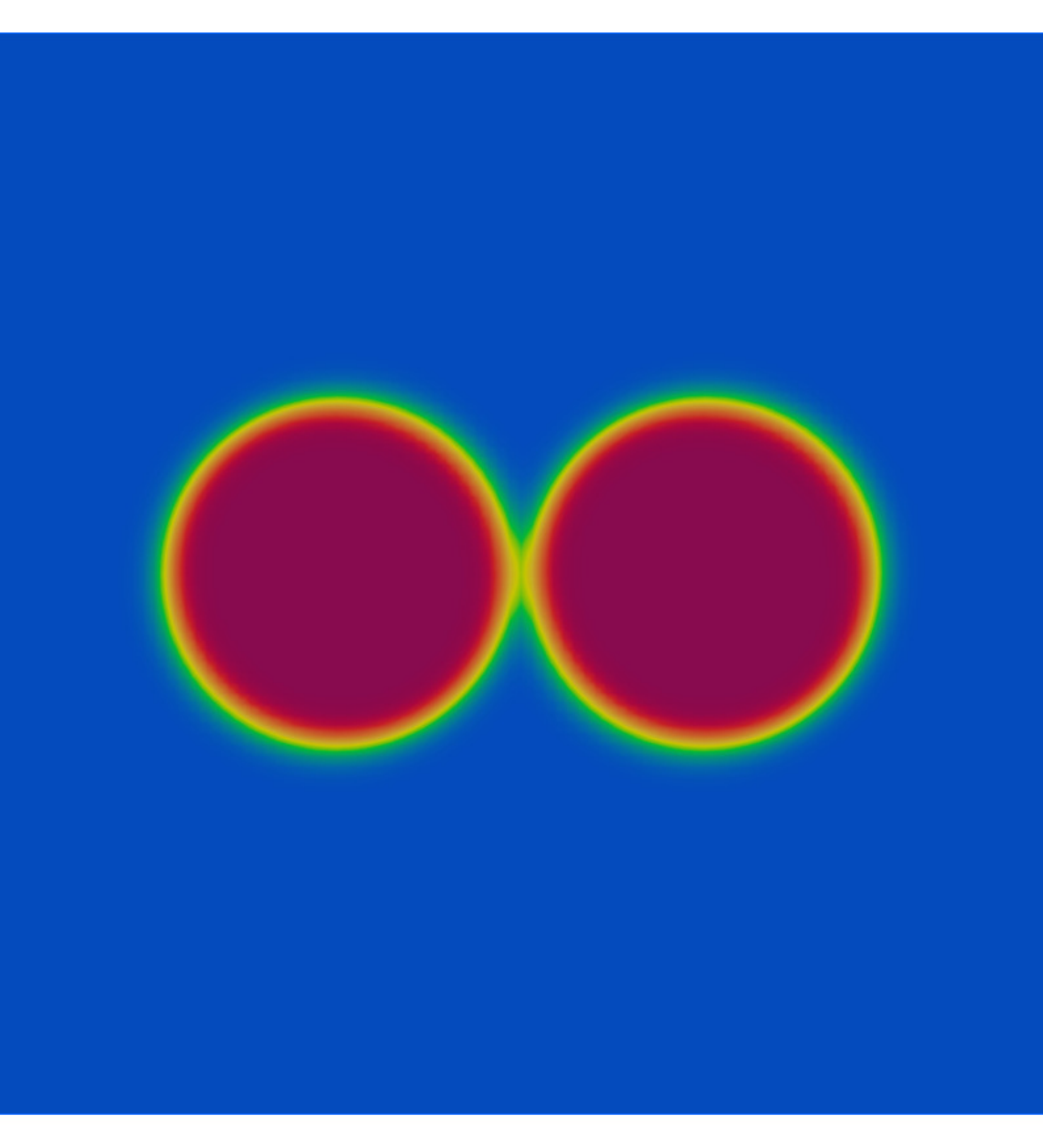}
\includegraphics[scale=0.1125]{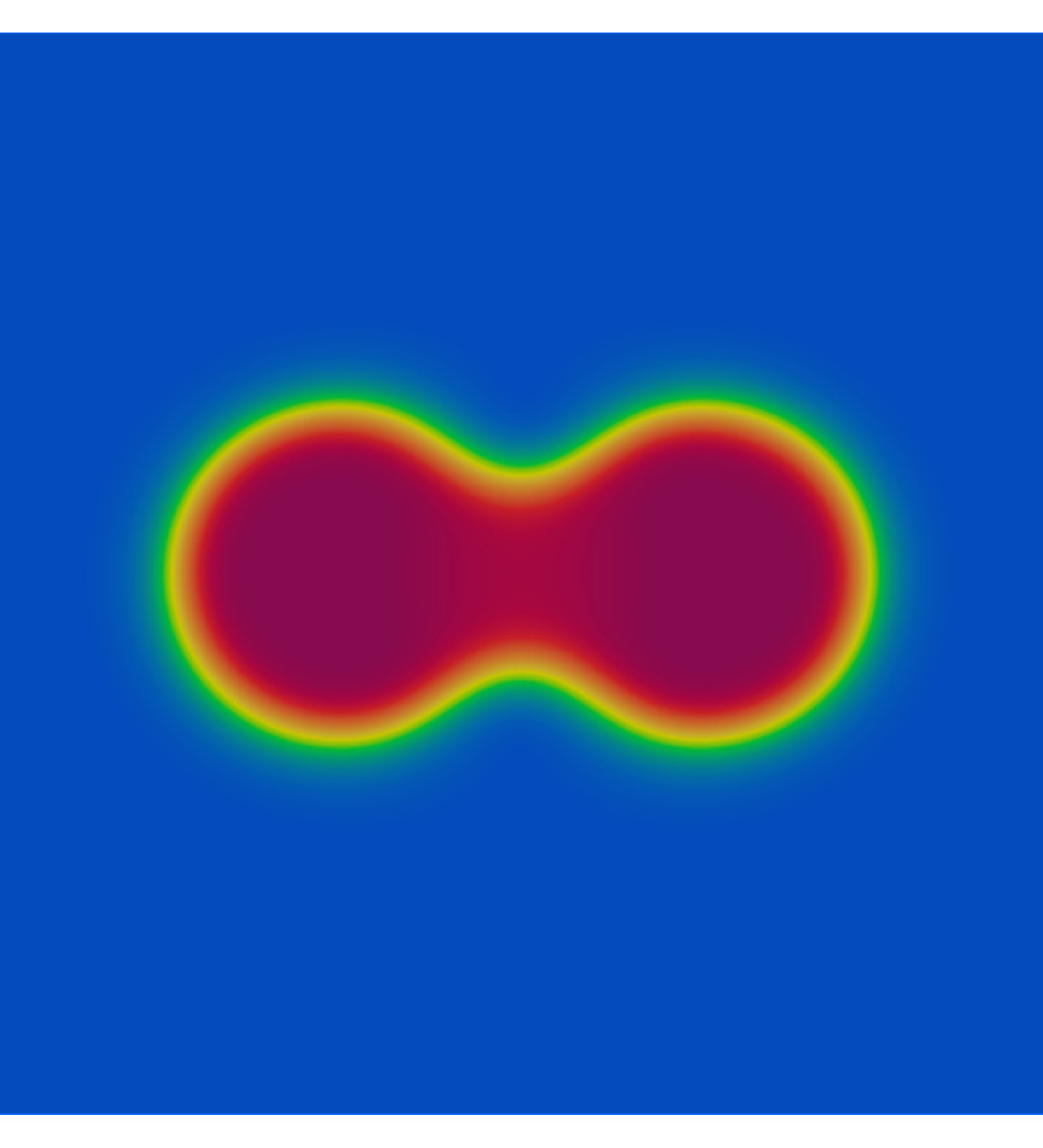}
\includegraphics[scale=0.1125]{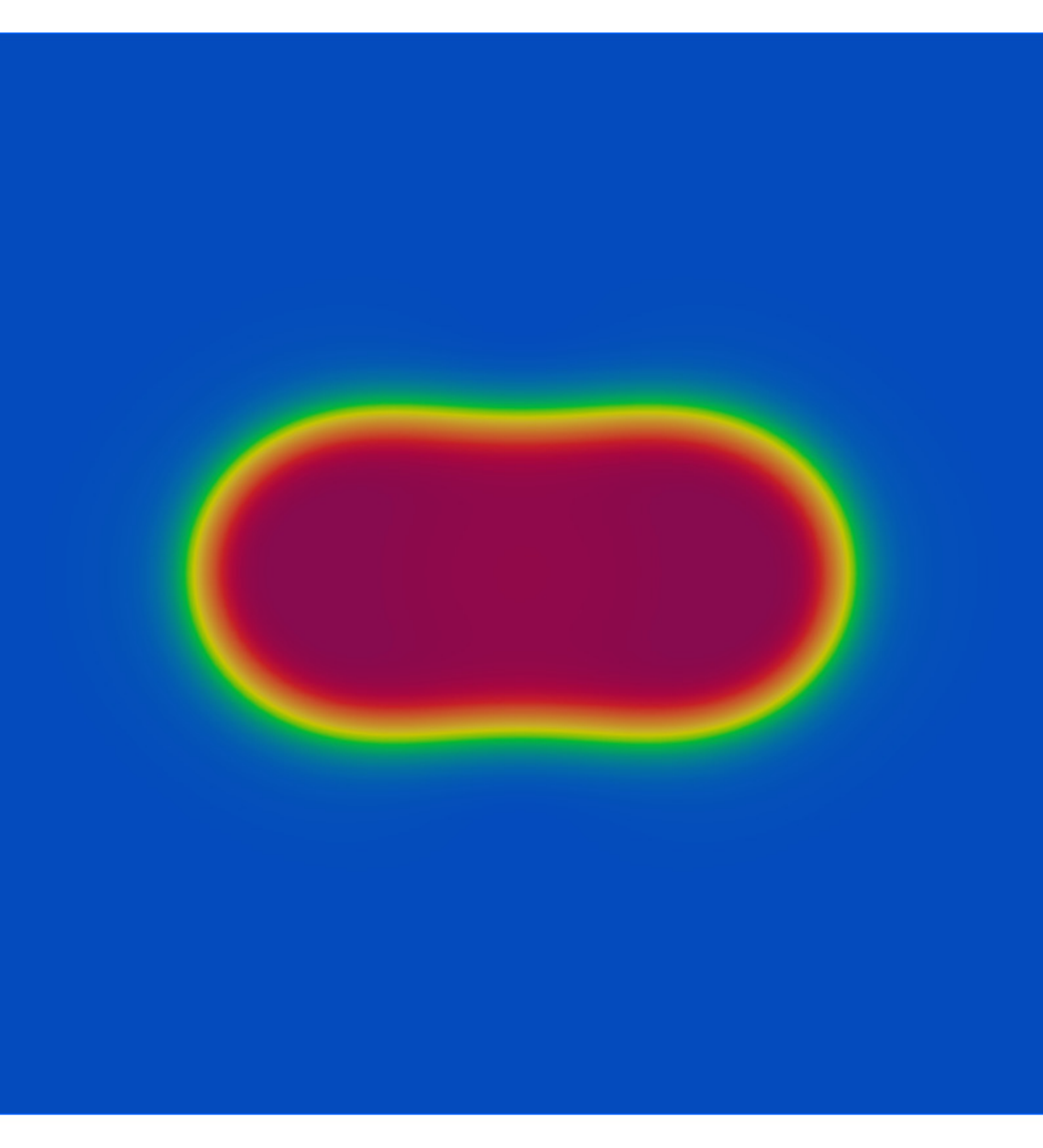}
\includegraphics[scale=0.1125]{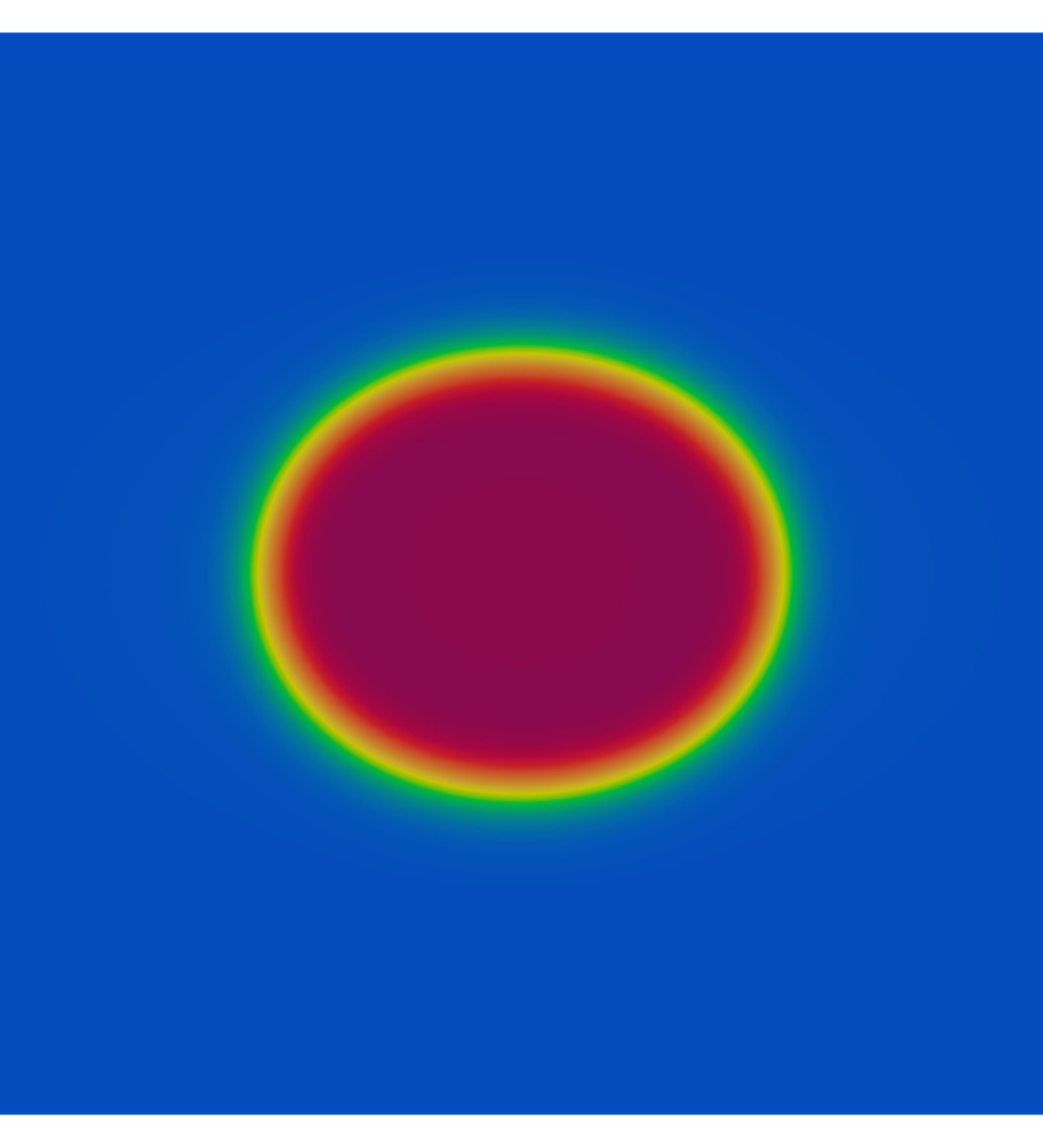}
\includegraphics[scale=0.1125]{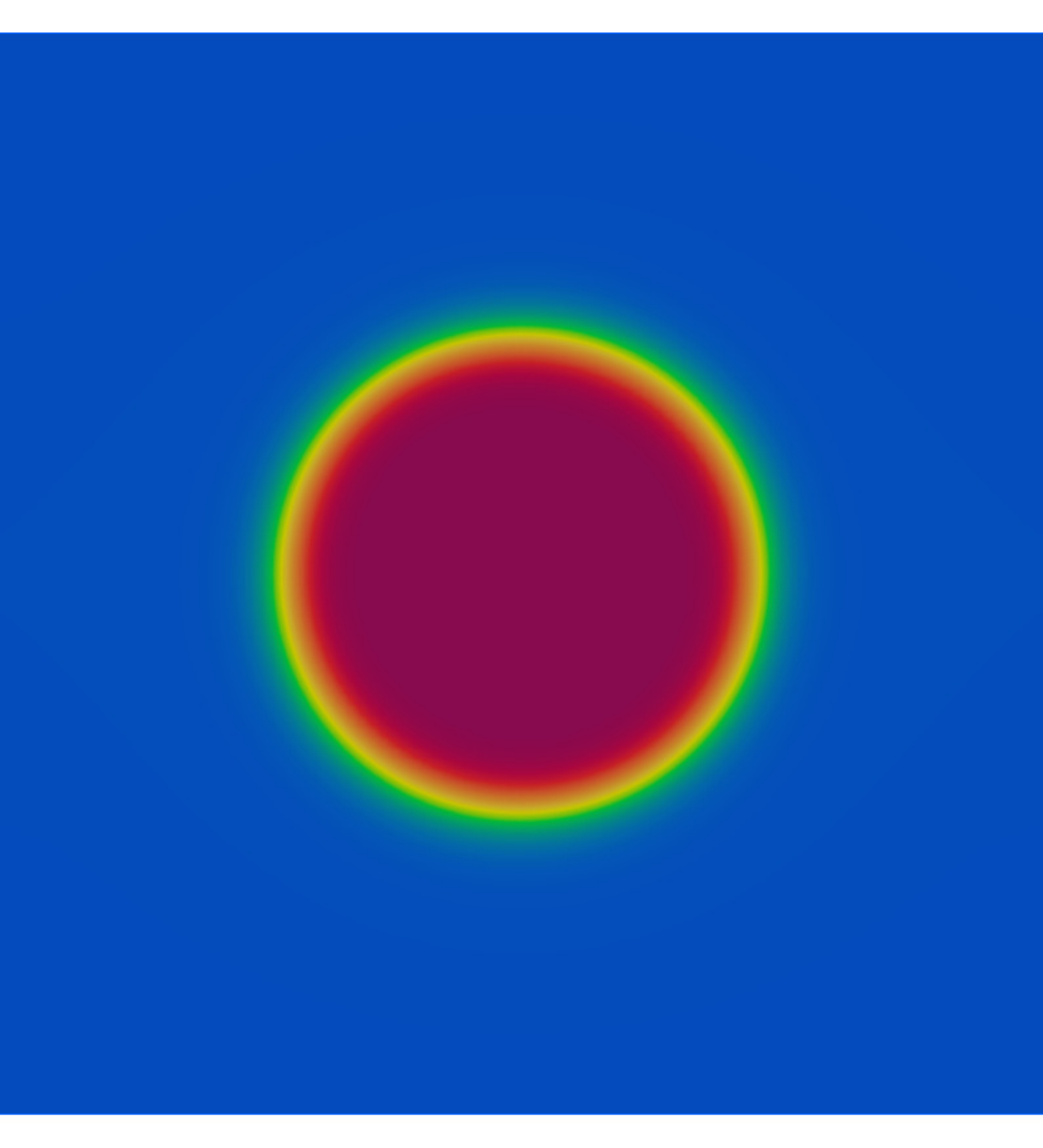}
\\
\includegraphics[scale=0.1125]{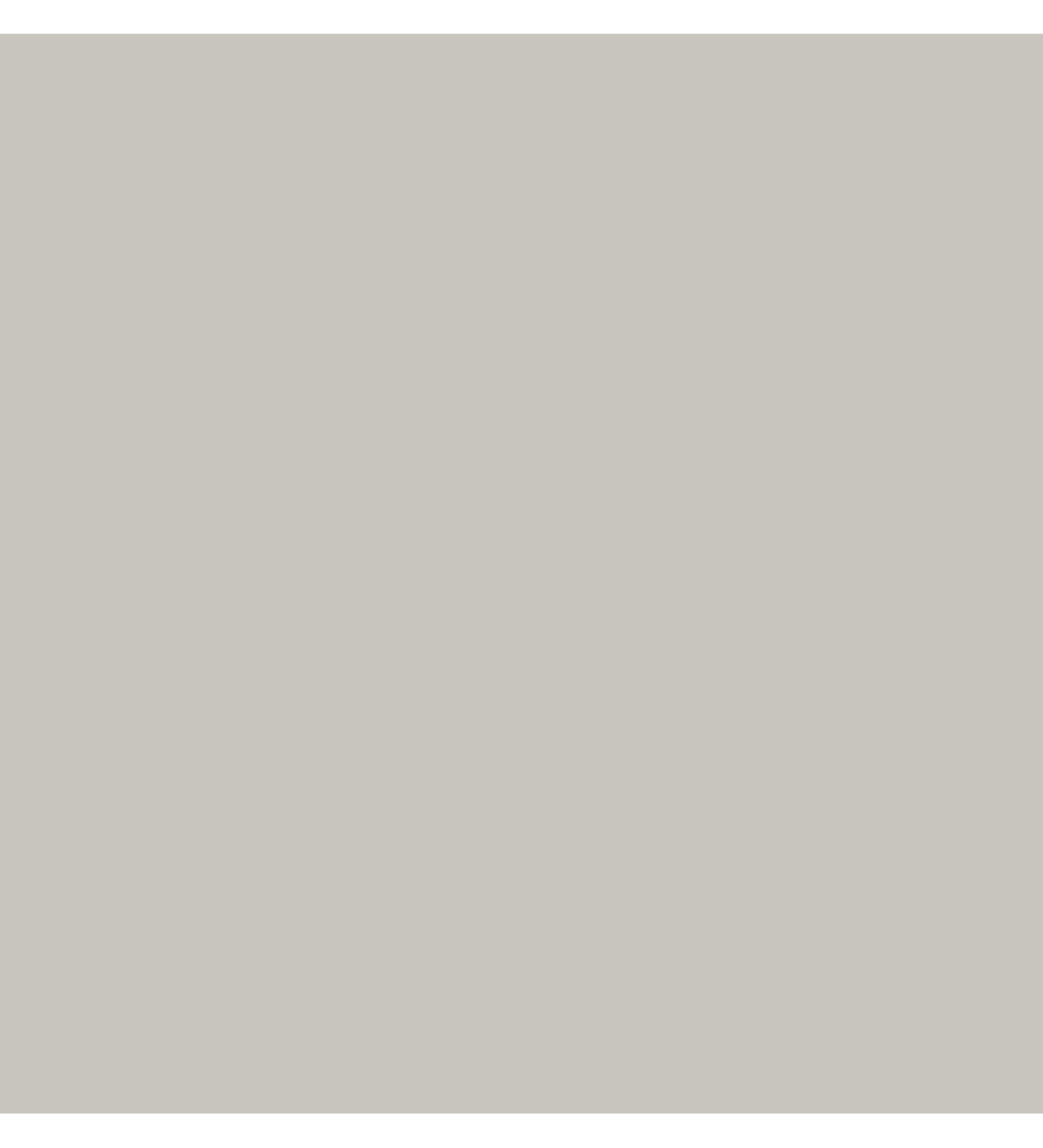}
\includegraphics[scale=0.1125]{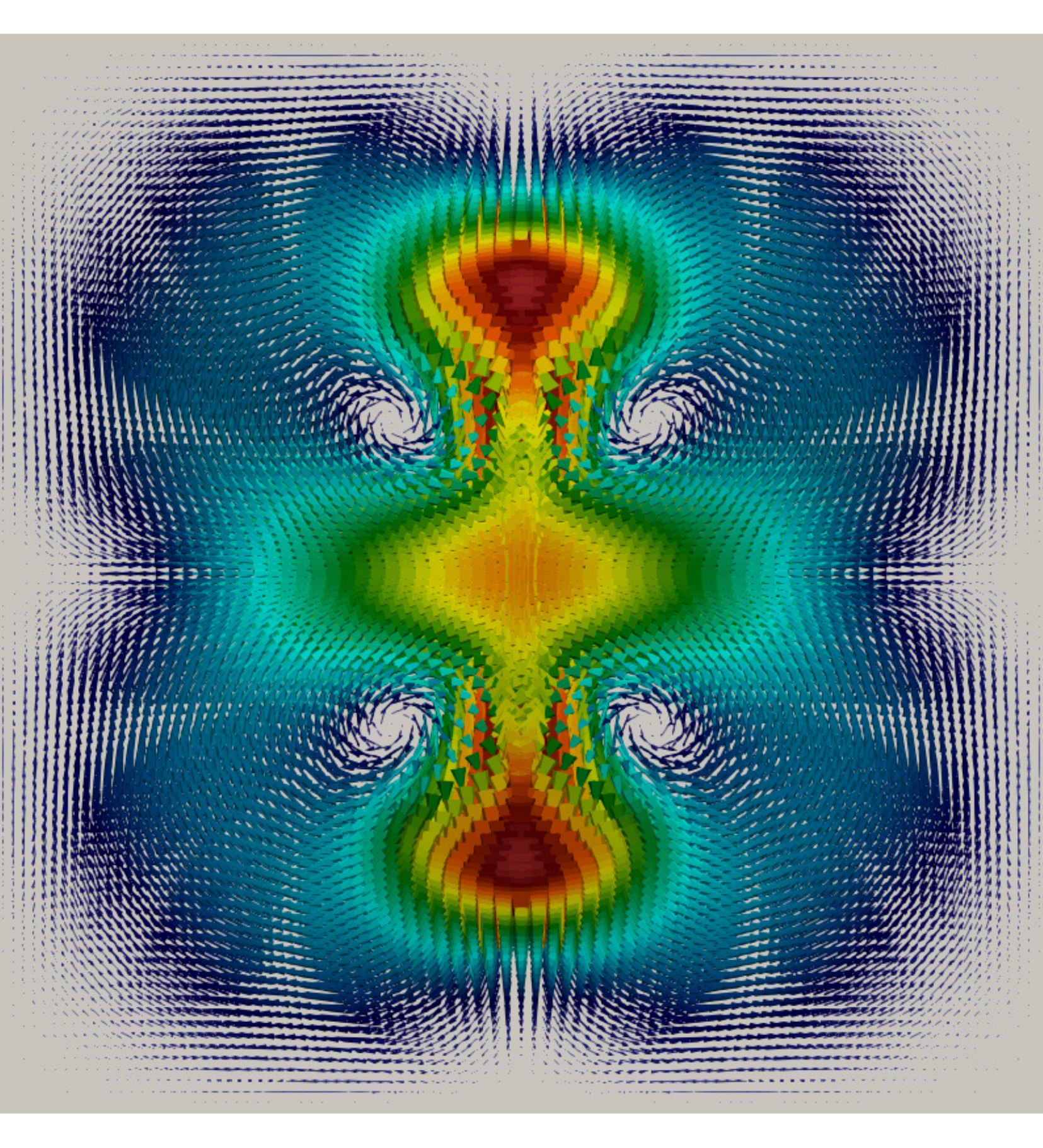}
\includegraphics[scale=0.1125]{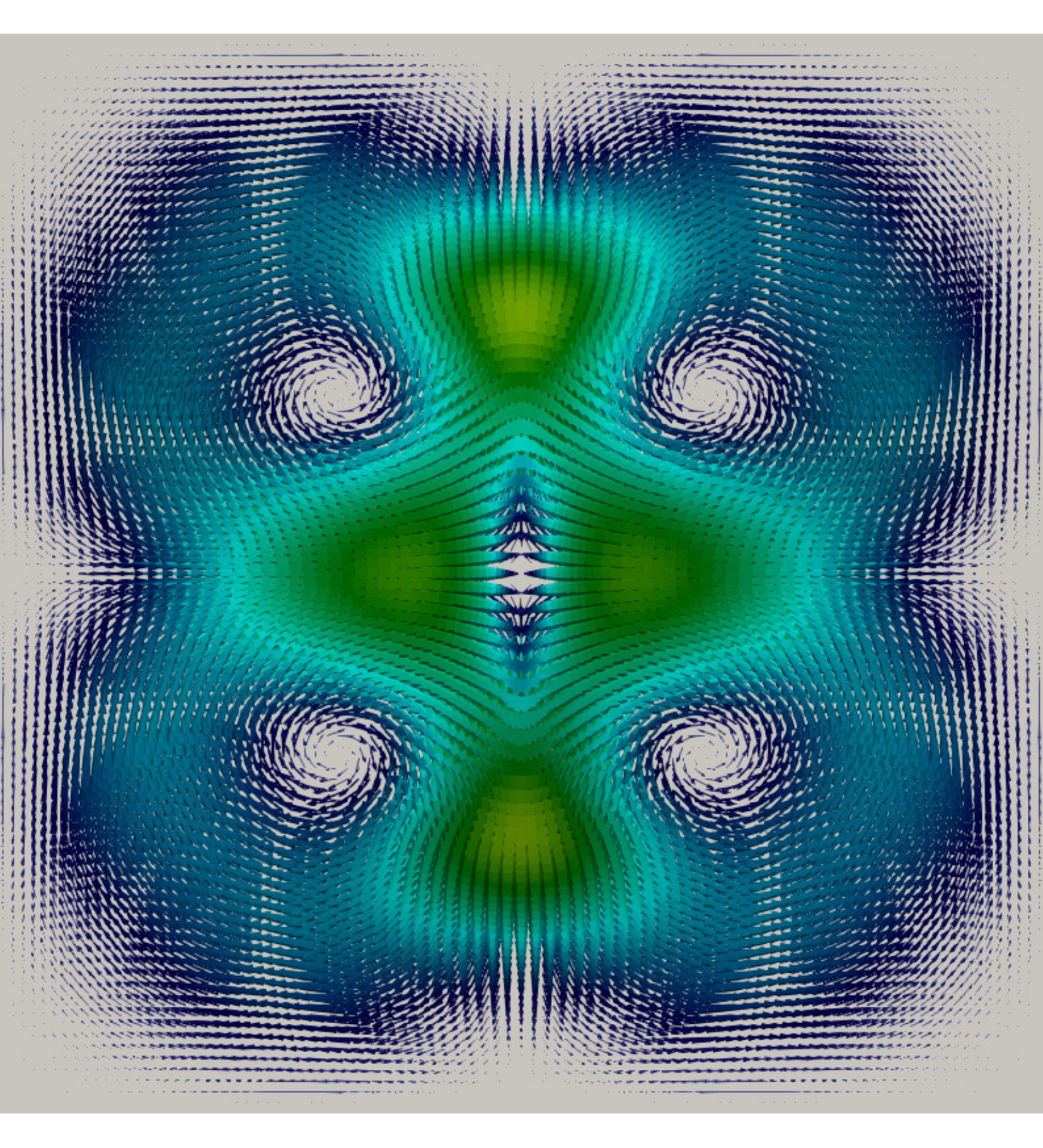}
\includegraphics[scale=0.1125]{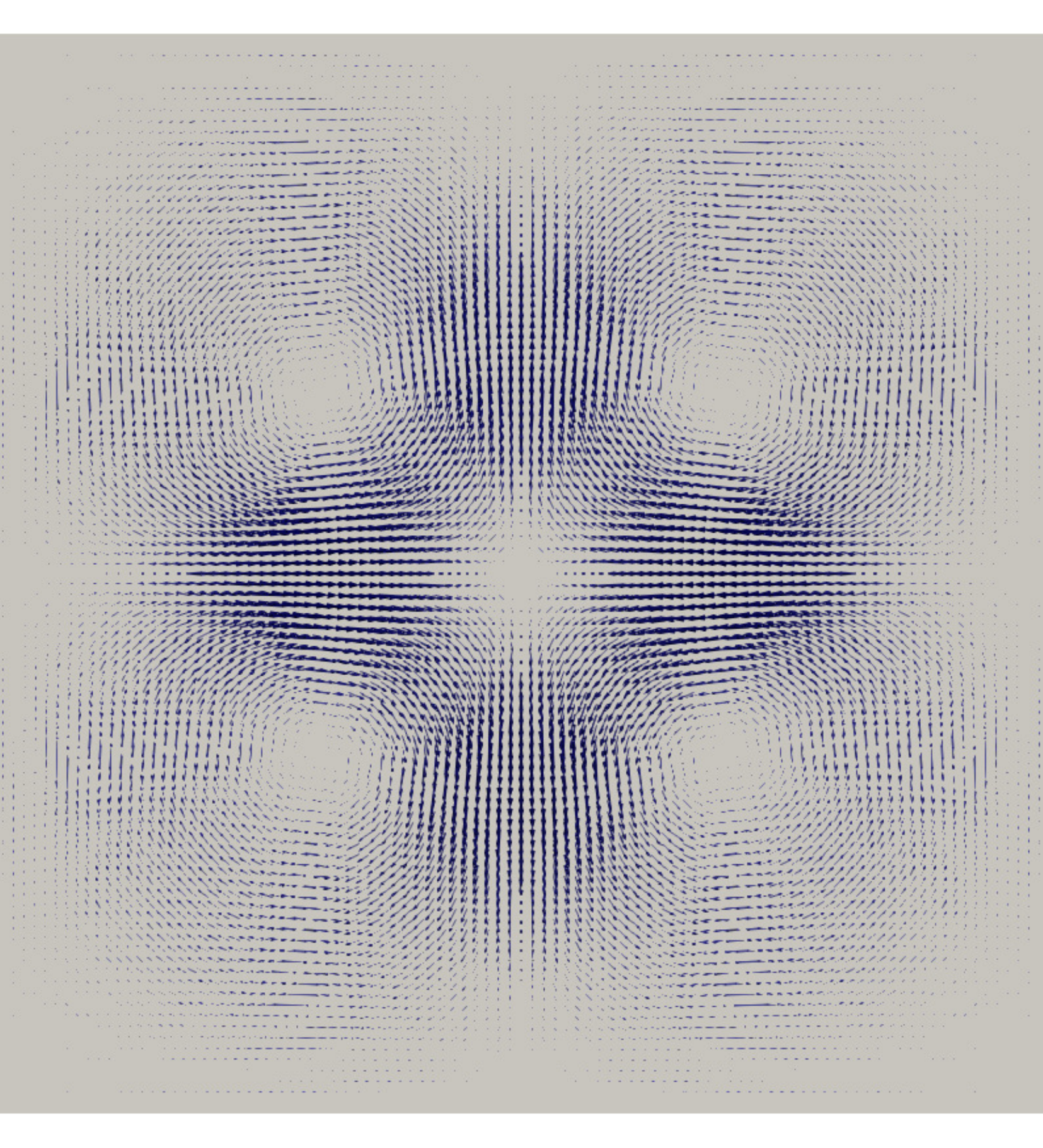}
\includegraphics[scale=0.1125]{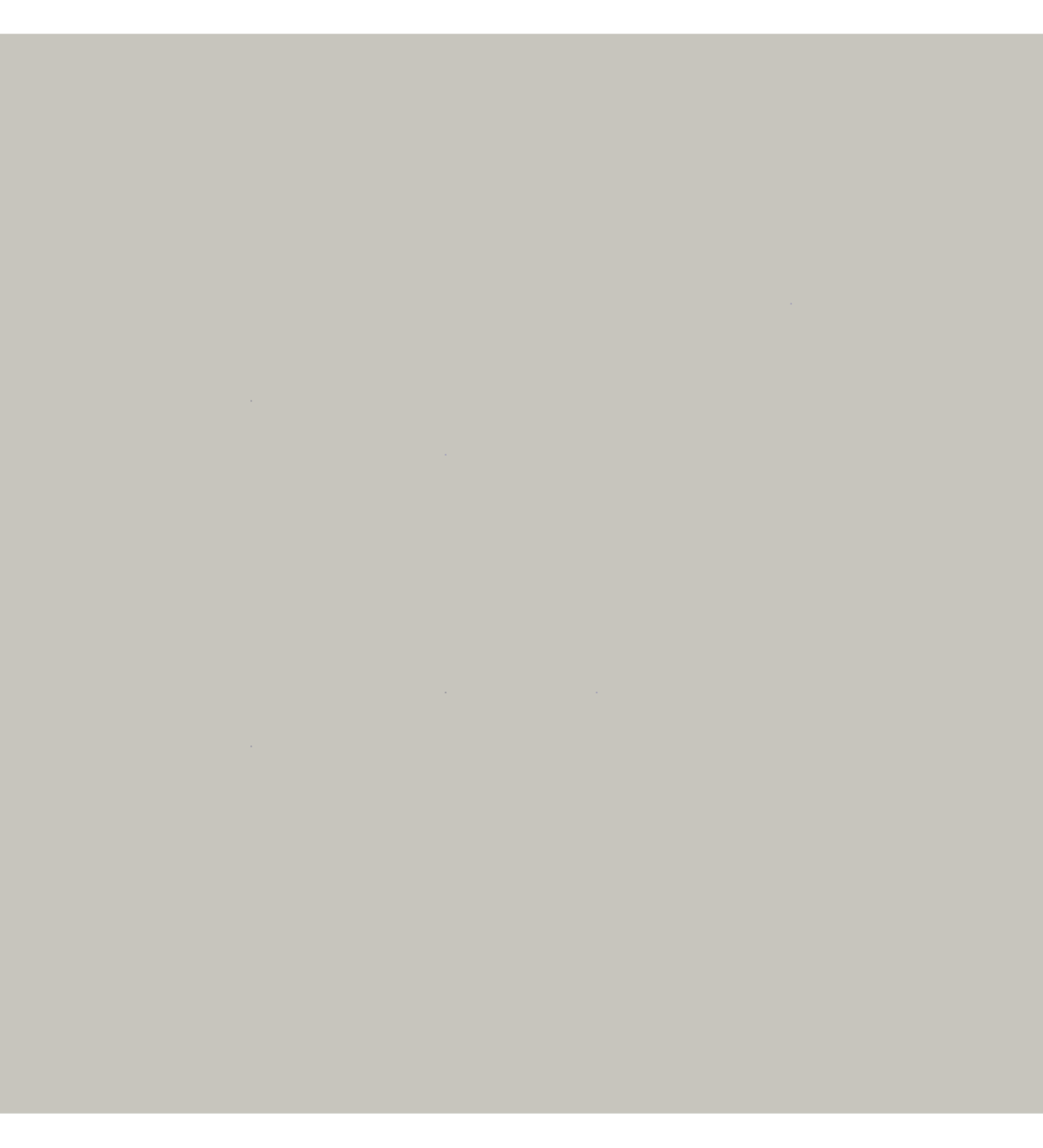}
\end{center}
\caption{Example II. Merging droplets. Evolution in time of $\phi$ and $\u$ for $G_\varepsilon$-scheme at times $t=0, 0.2, 0.5, 1.5$ and $5$.}\label{fig:Ex2_merging_dyn_G}
\end{figure}

\begin{figure}[H]
\begin{center}
\includegraphics[scale=0.1125]{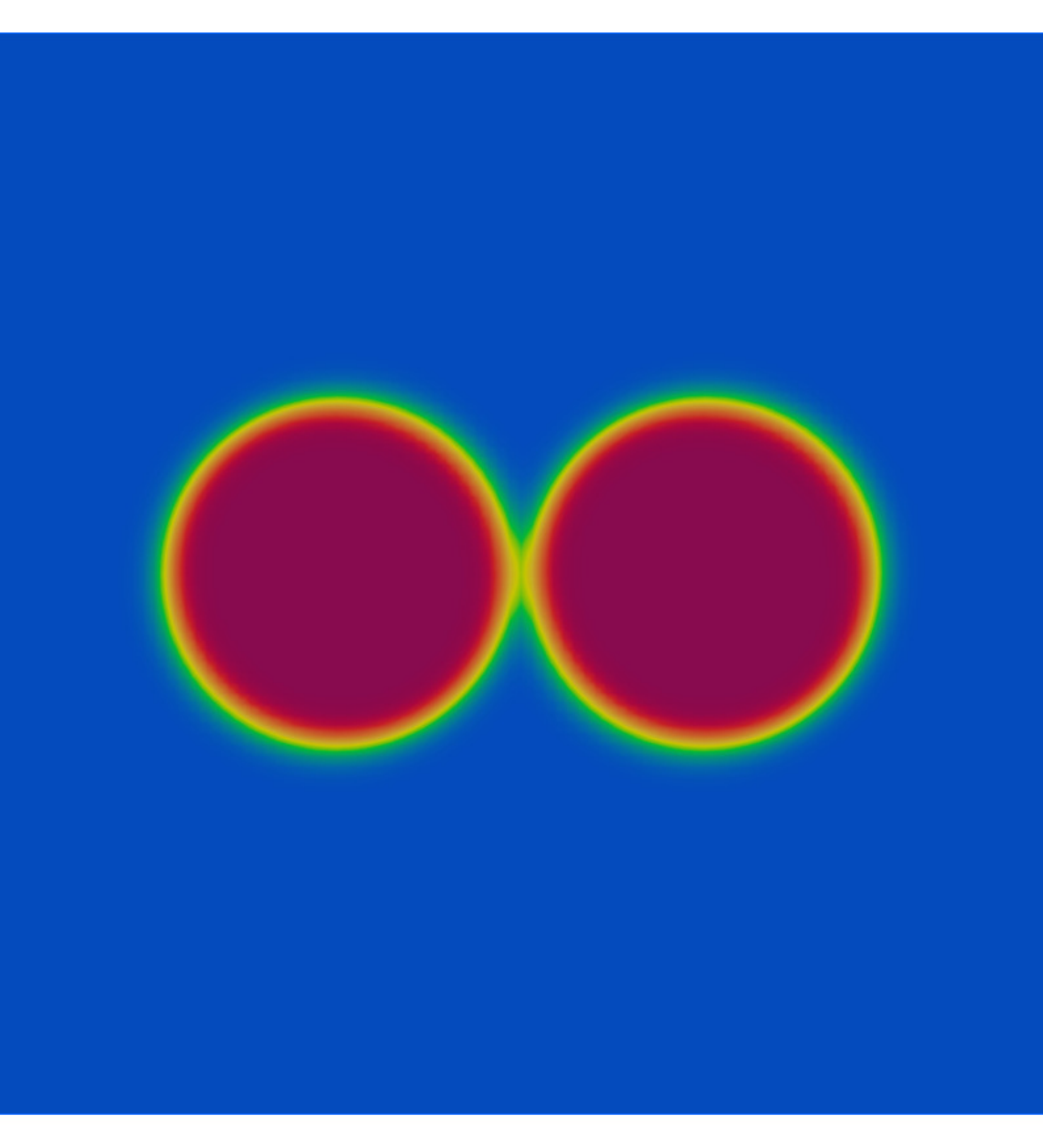}
\includegraphics[scale=0.1125]{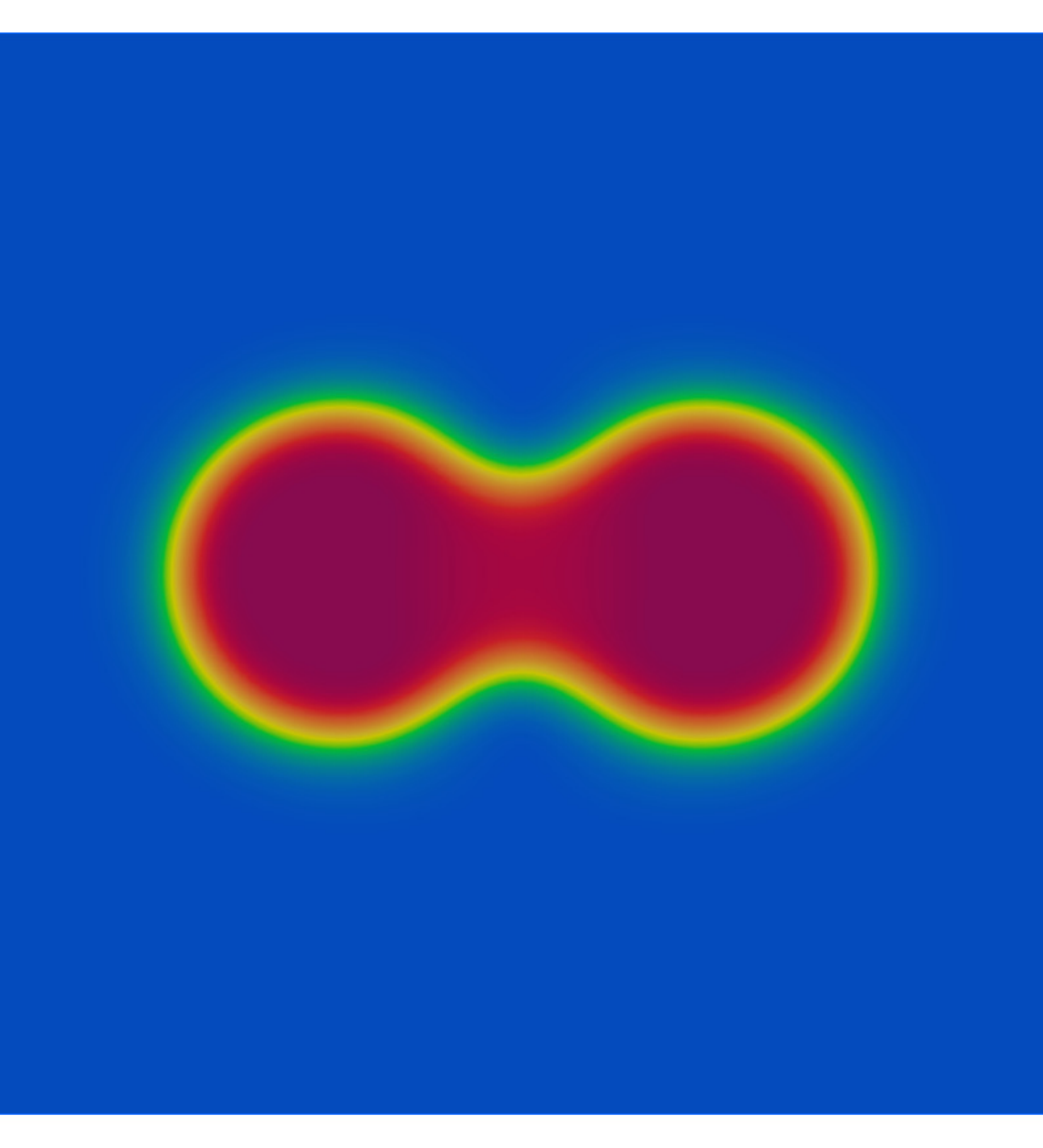}
\includegraphics[scale=0.1125]{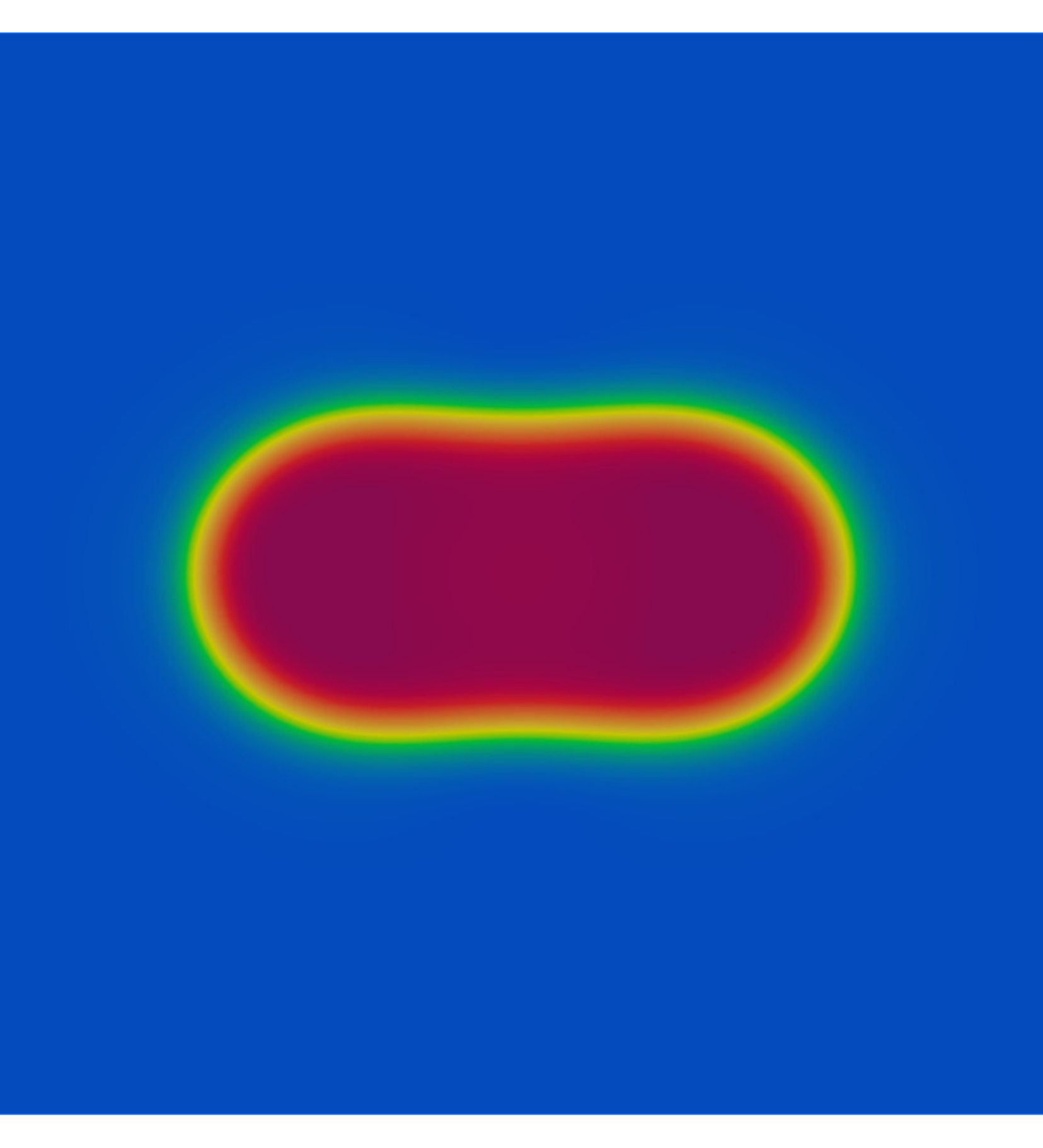}
\includegraphics[scale=0.1125]{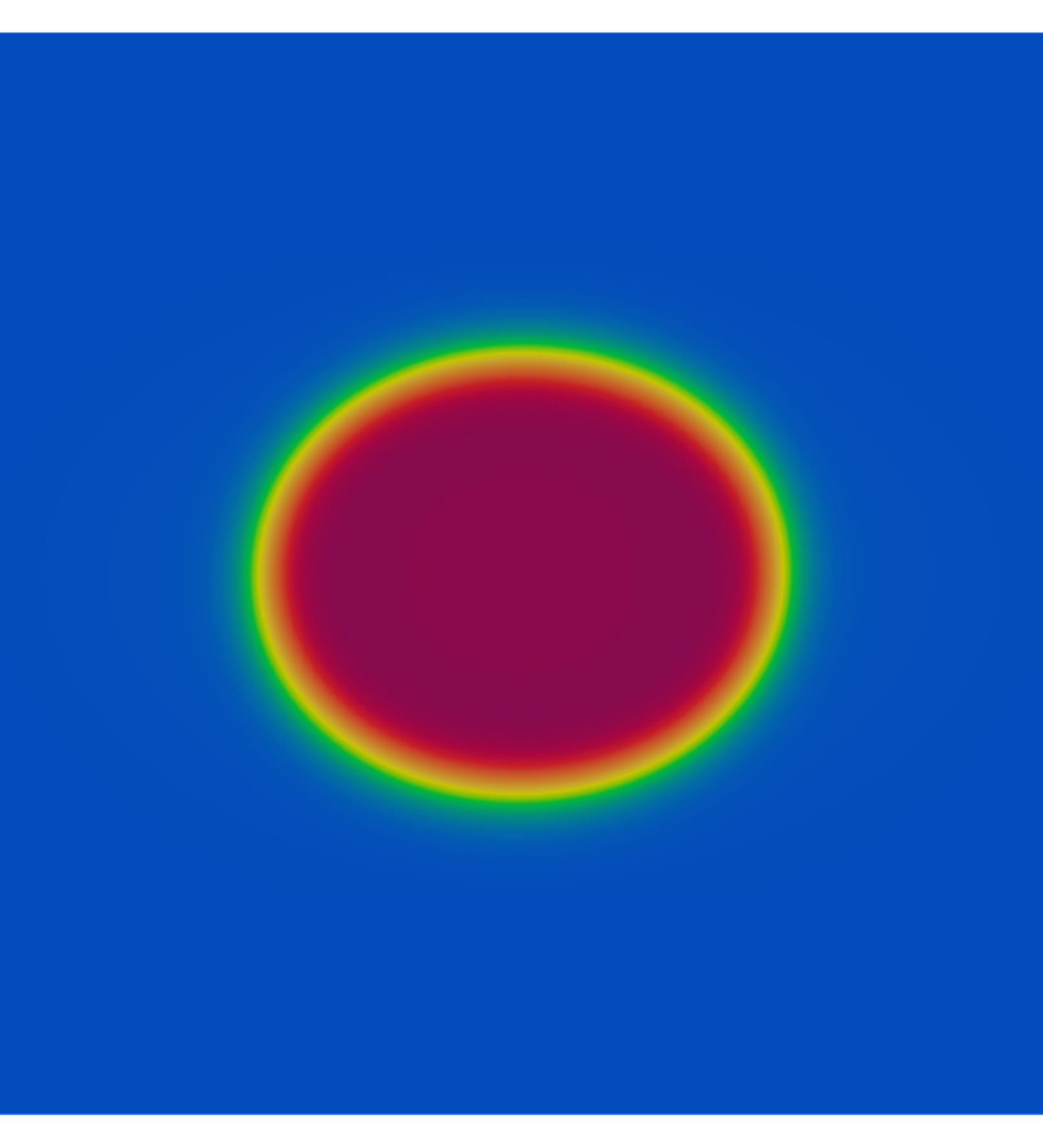}
\includegraphics[scale=0.1125]{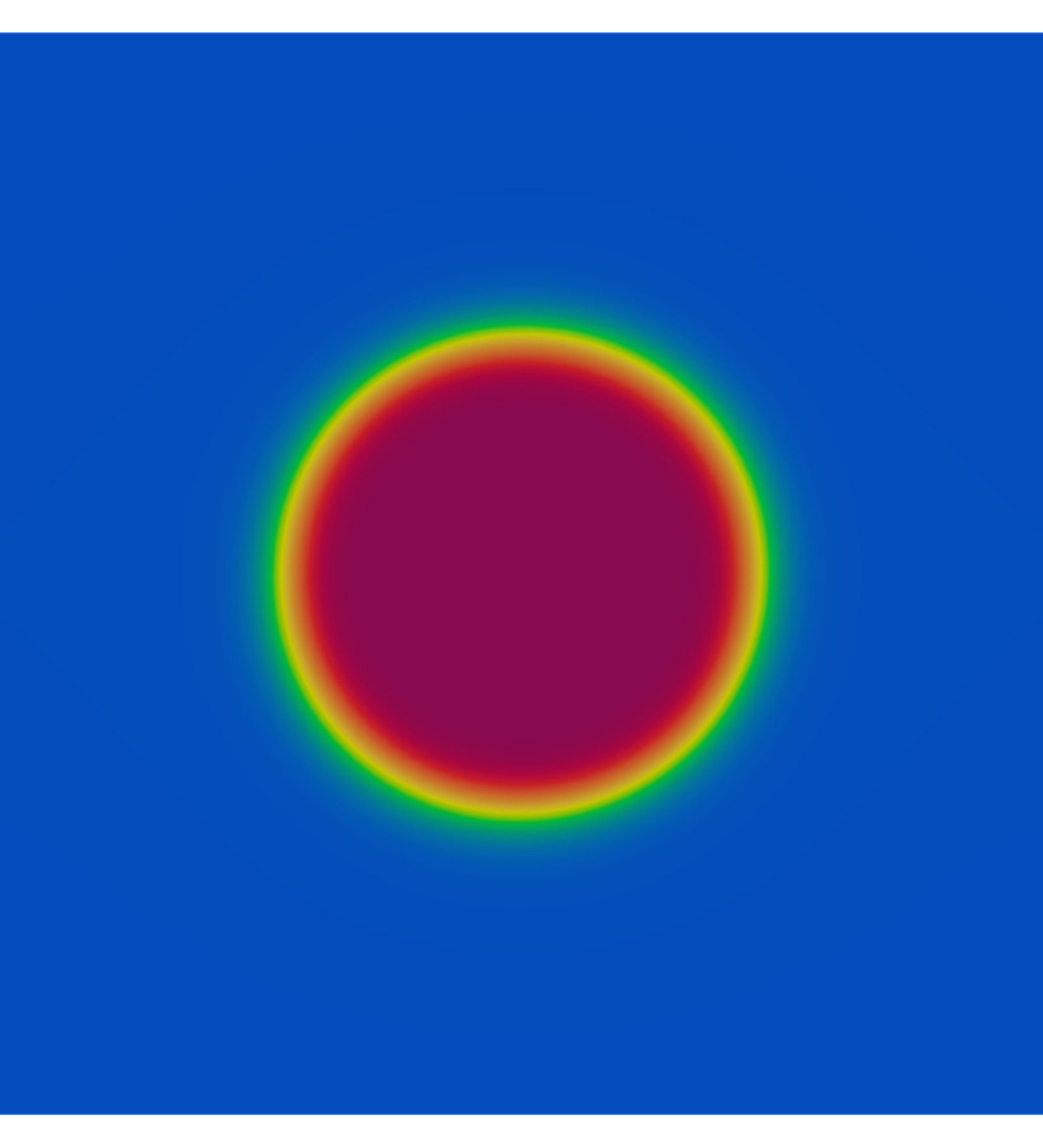}
\\
\includegraphics[scale=0.1125]{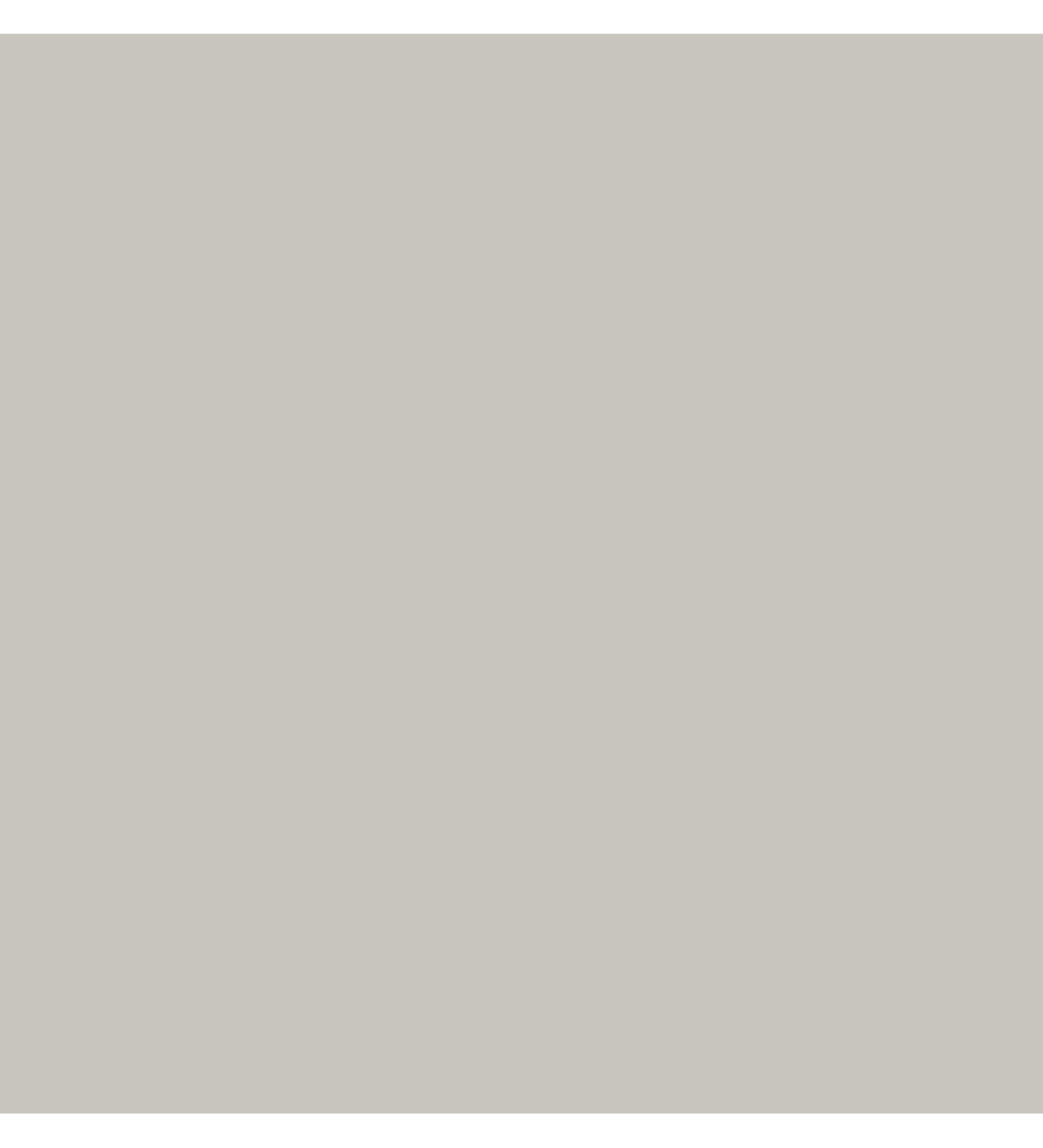}
\includegraphics[scale=0.1125]{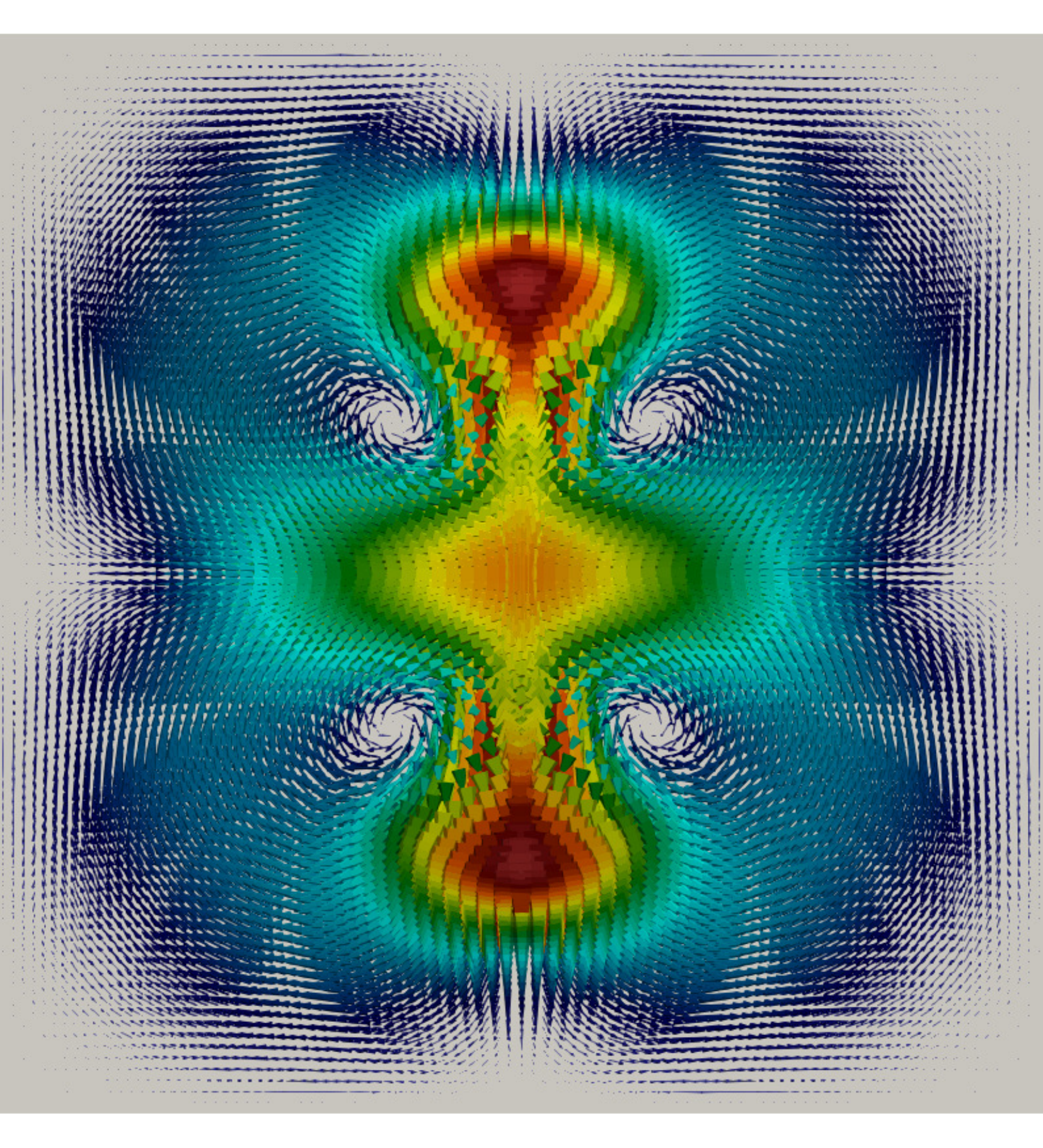}
\includegraphics[scale=0.1125]{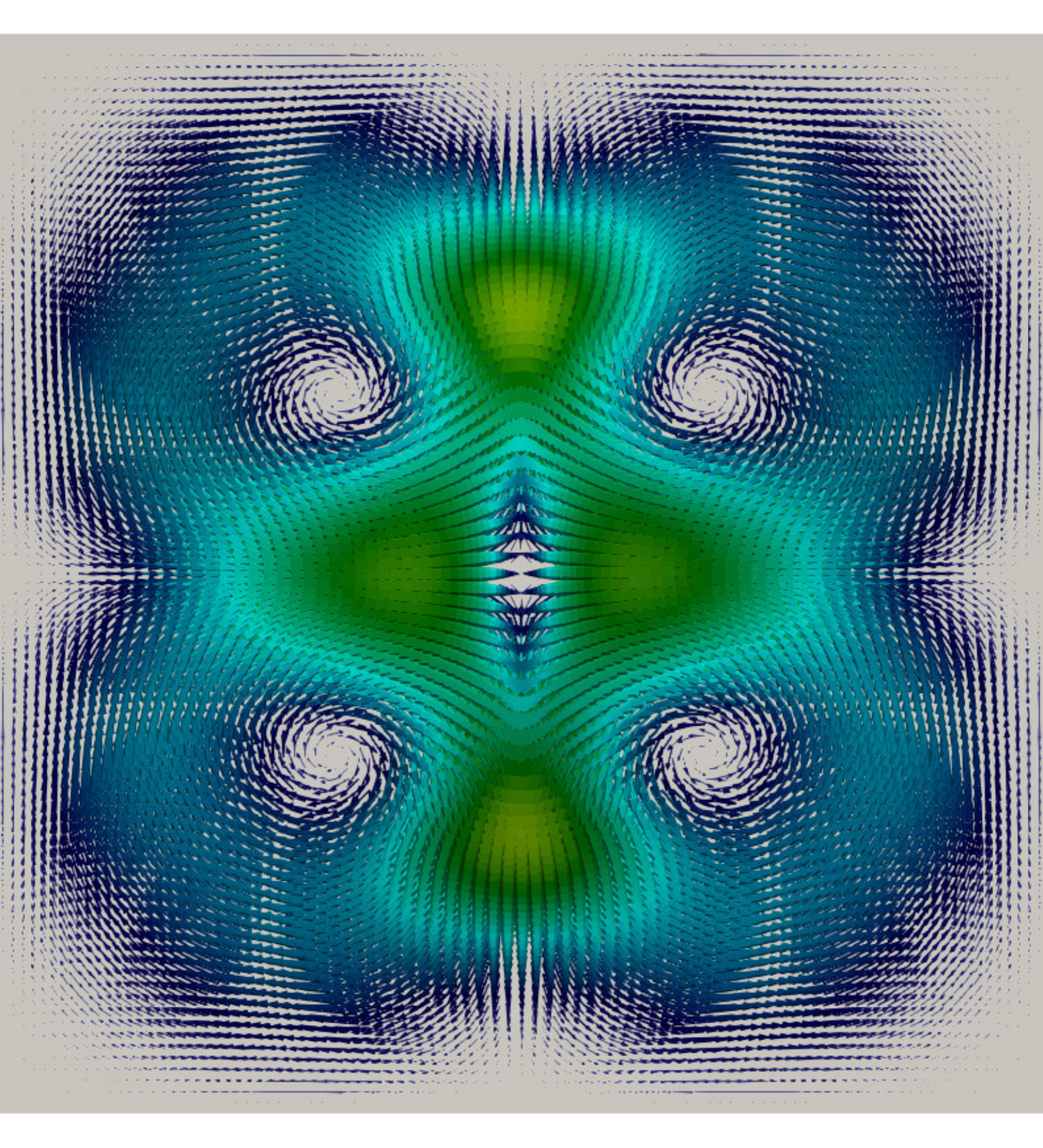}
\includegraphics[scale=0.1125]{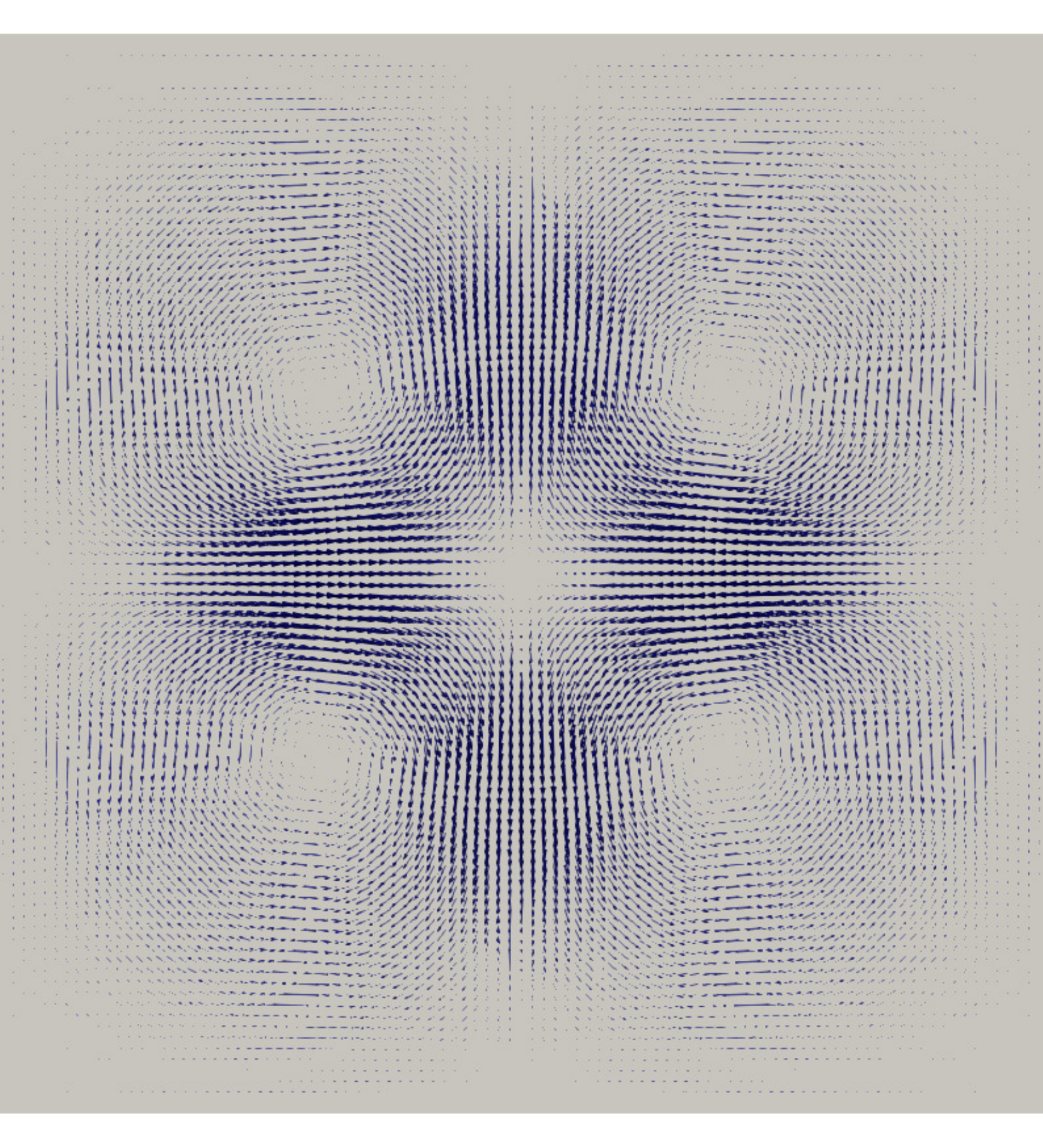}
\includegraphics[scale=0.1125]{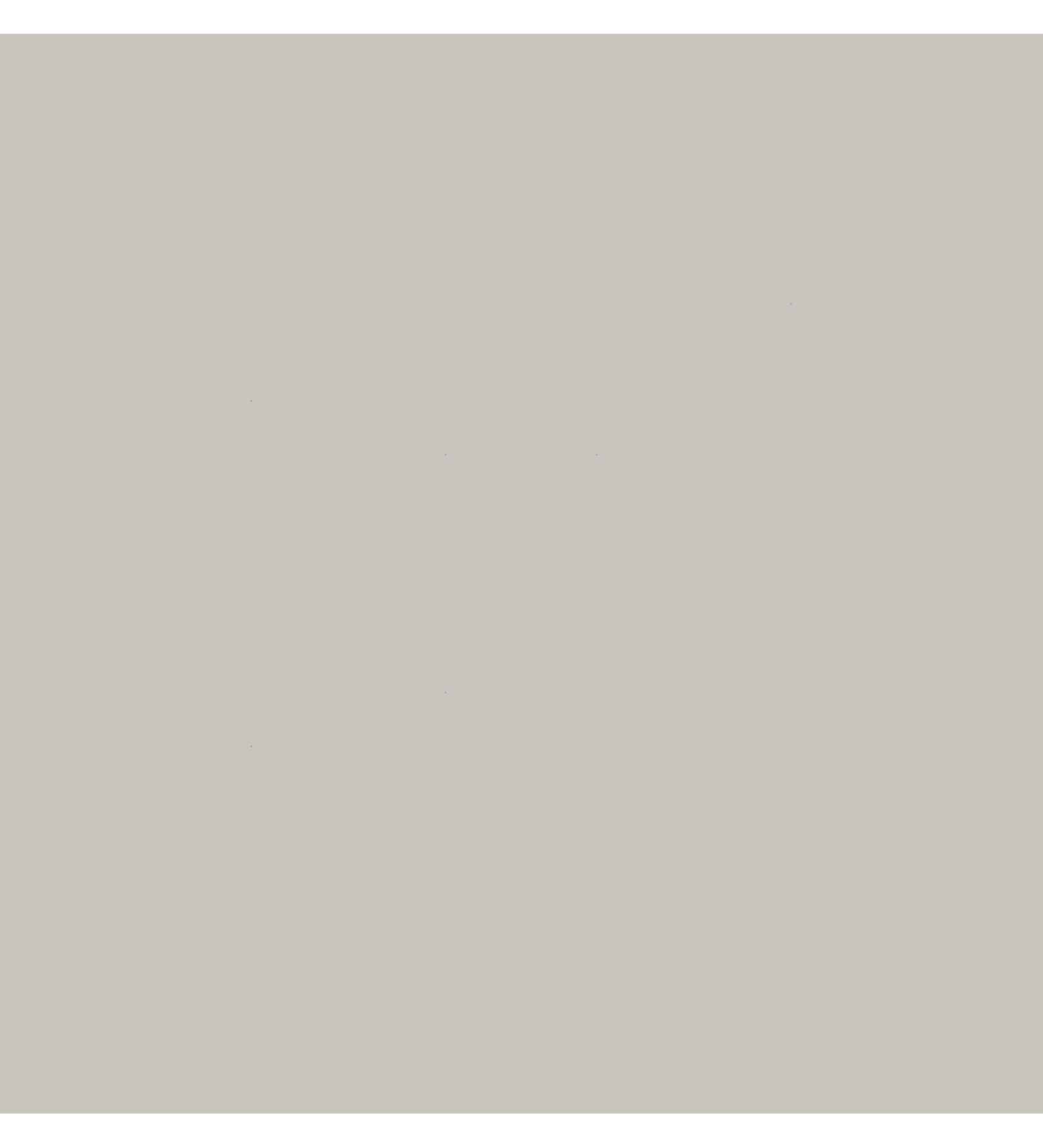}
\end{center}
\caption{Example II. Merging droplets. Evolution in time of $\phi$ and $\u$ for $J_\varepsilon$-scheme at times $t=0, 0.2, 0.5, 1.5$ and $5$.}\label{fig:Ex2_merging_dyn_J}
\end{figure}

\begin{figure}[H]
\begin{center}
\includegraphics[scale=0.1125]{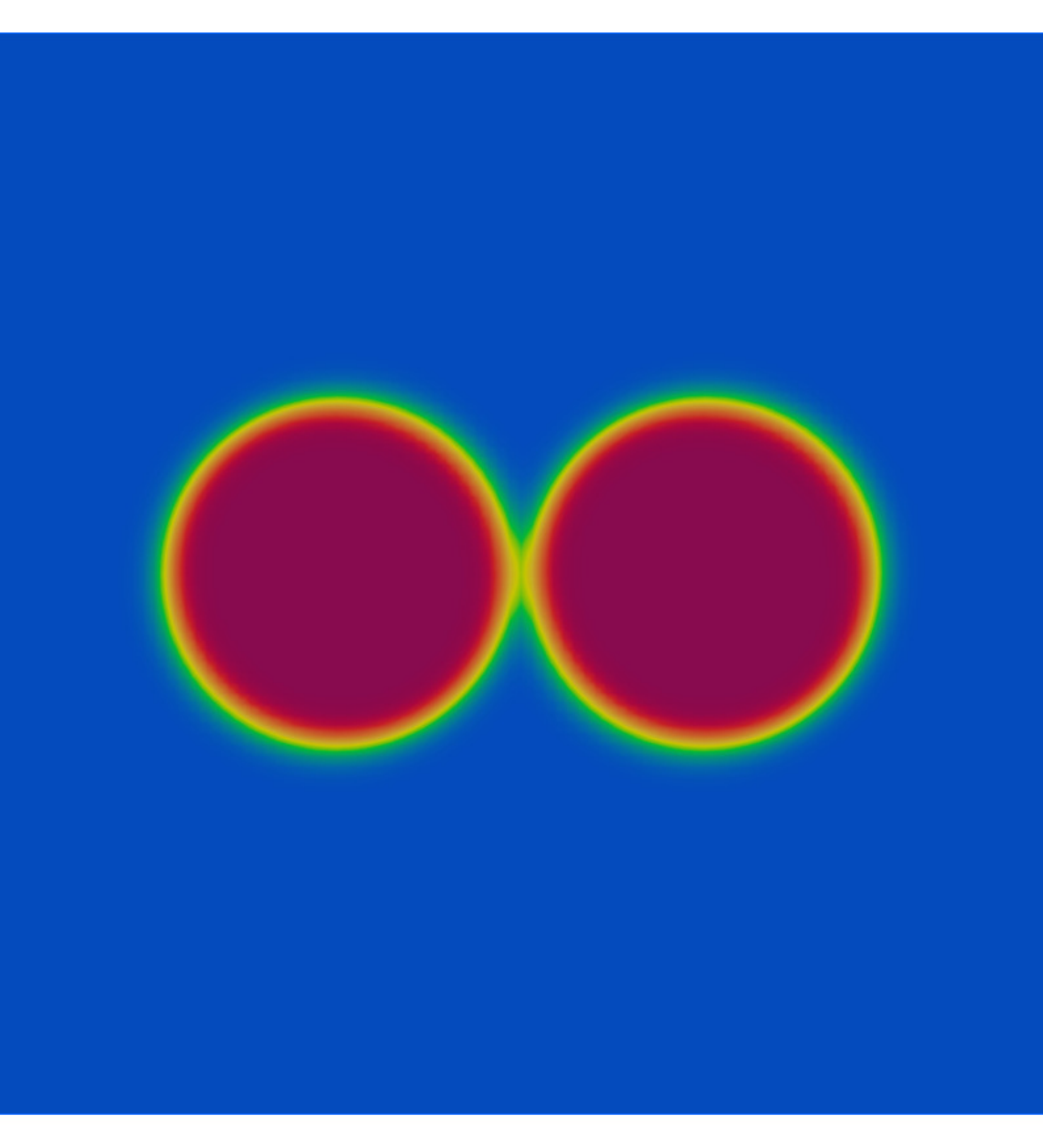}
\includegraphics[scale=0.1125]{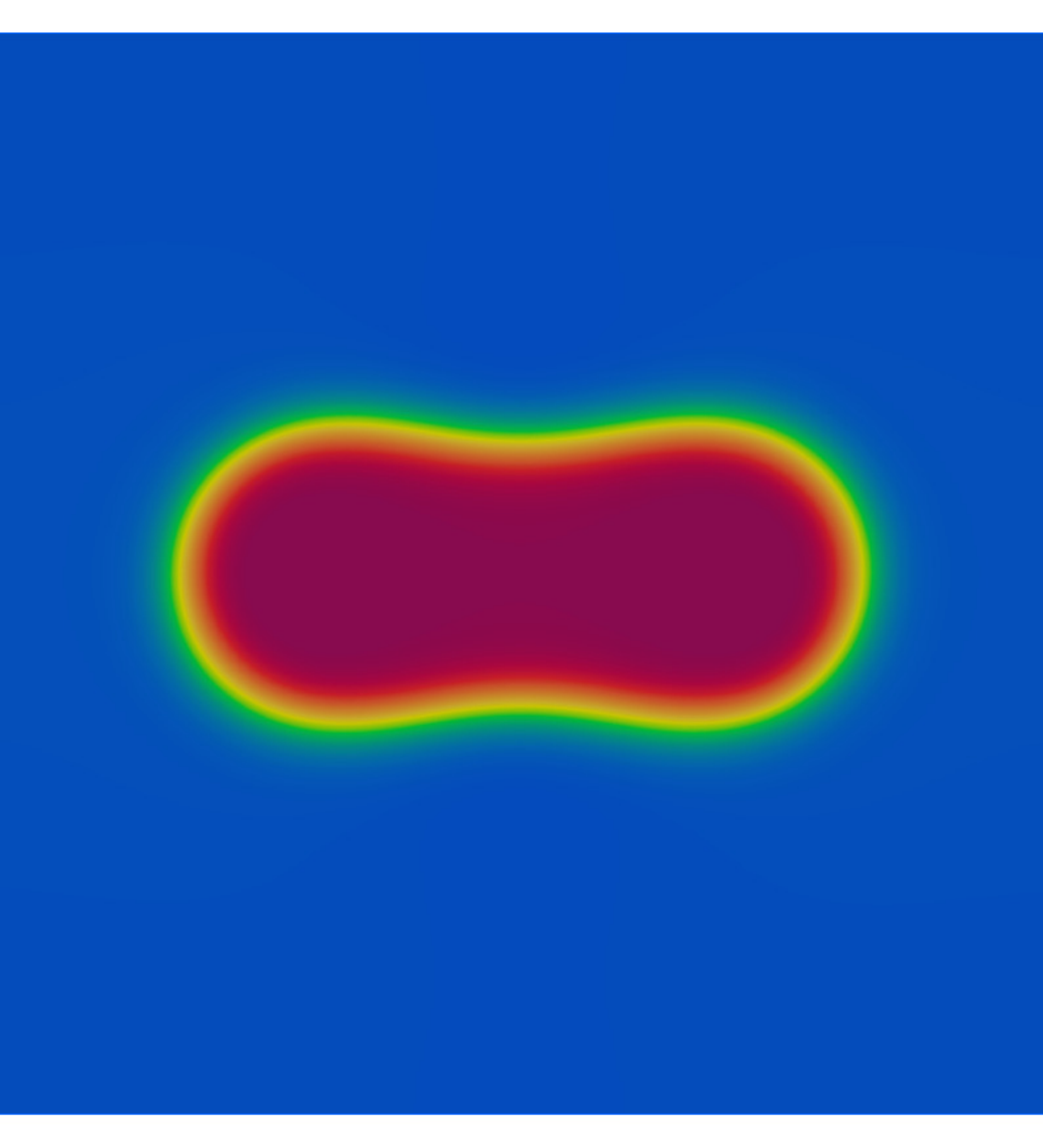}
\includegraphics[scale=0.1125]{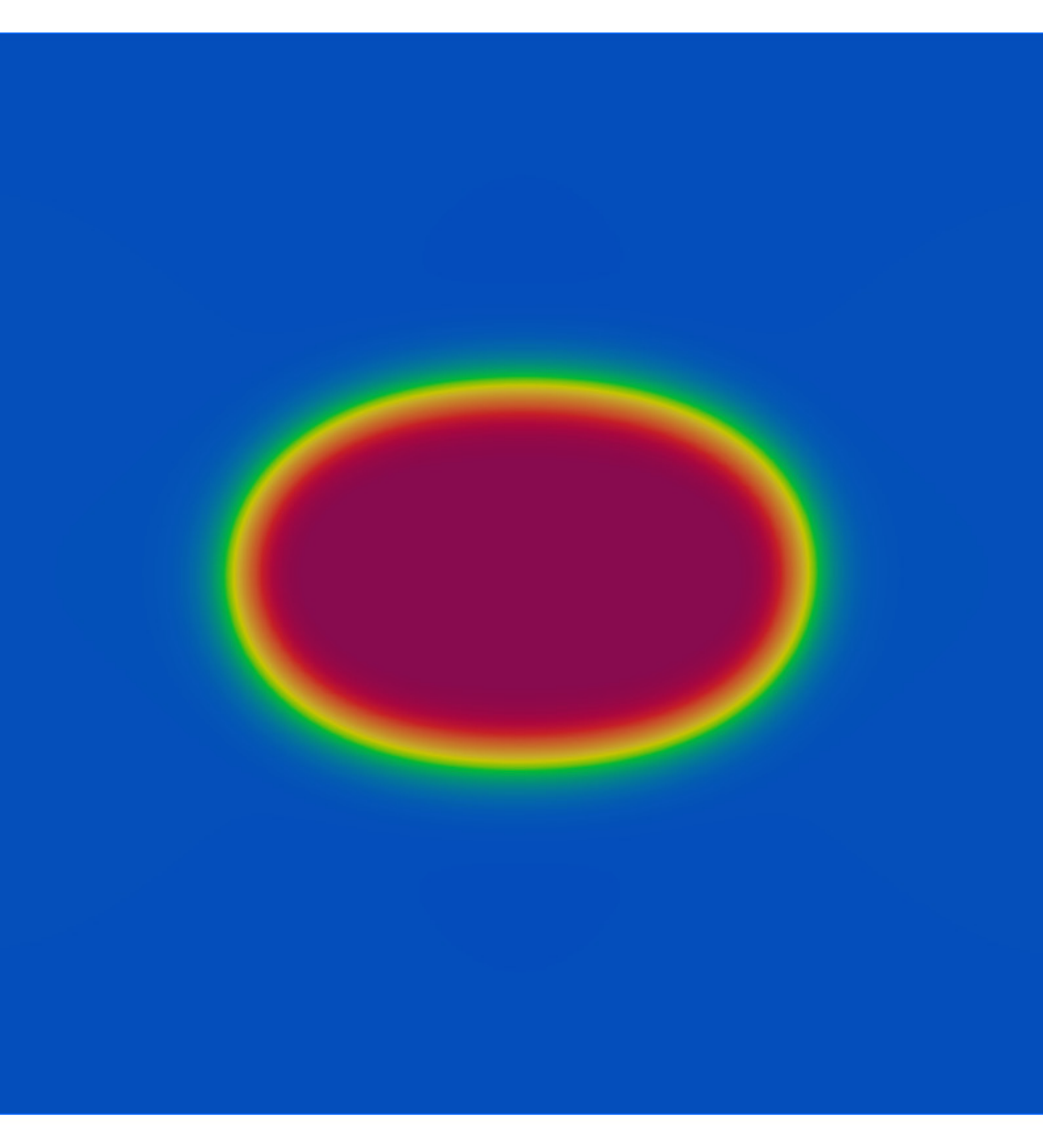}
\includegraphics[scale=0.1125]{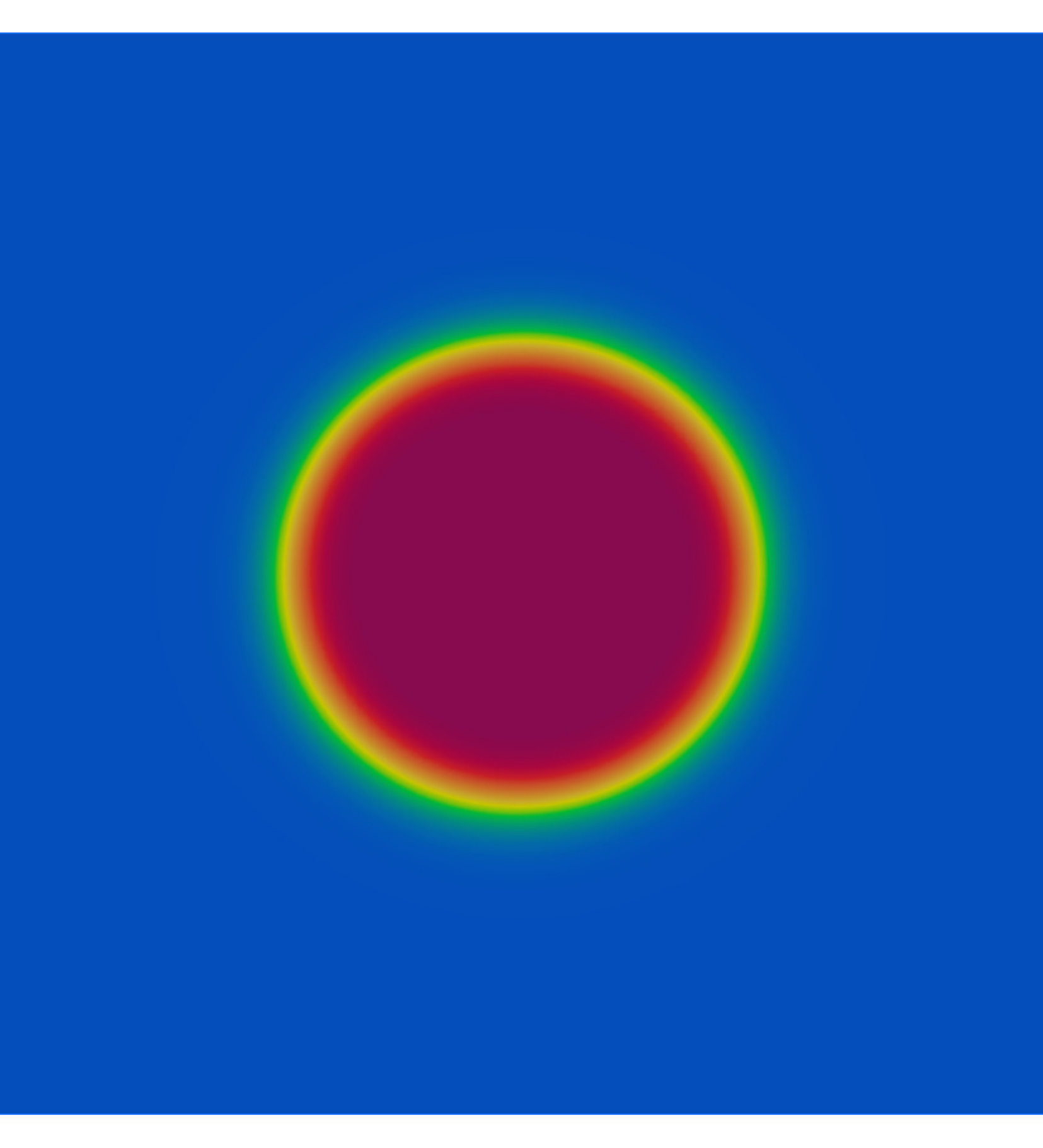}
\includegraphics[scale=0.1125]{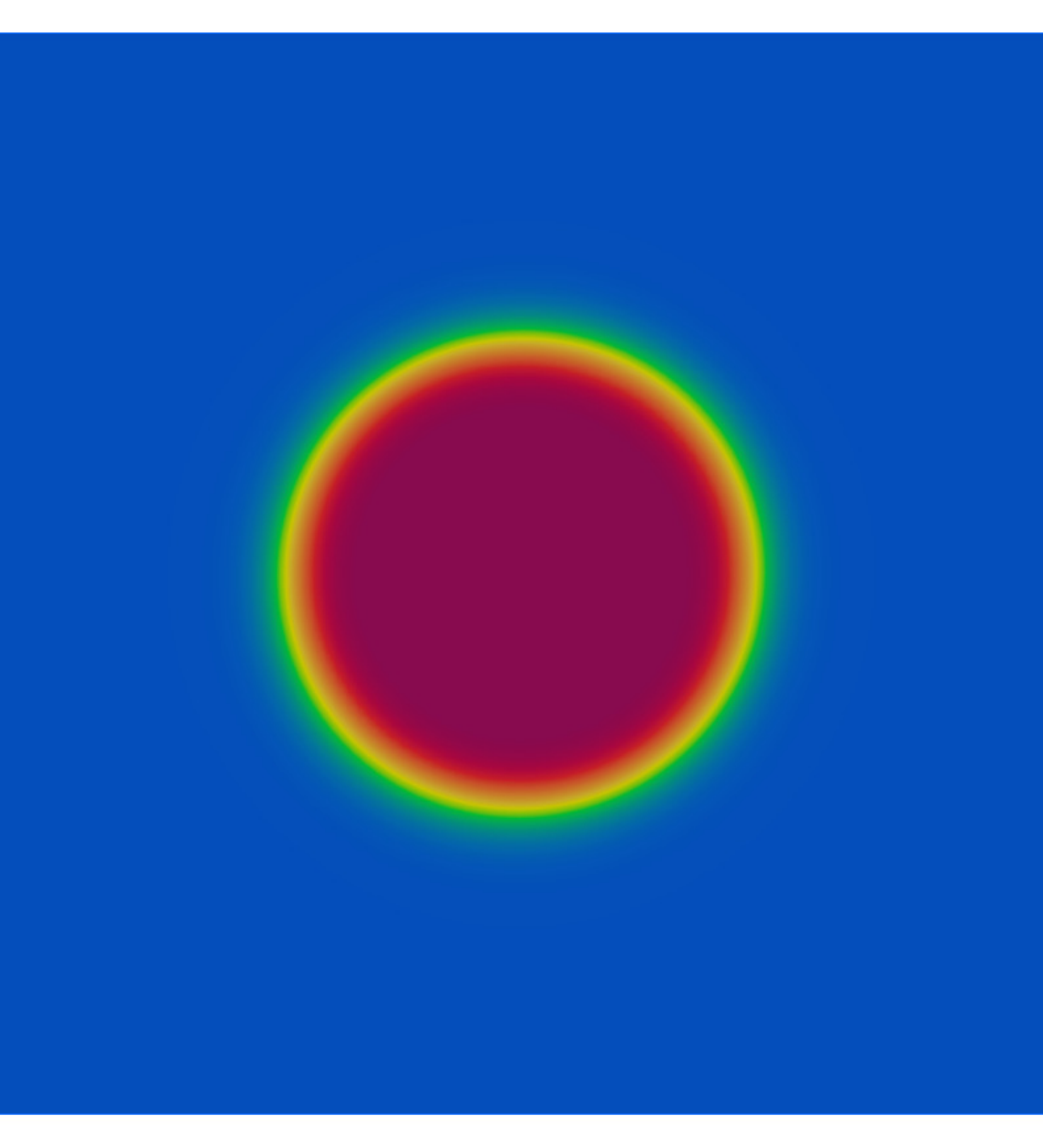}
\\
\includegraphics[scale=0.1125]{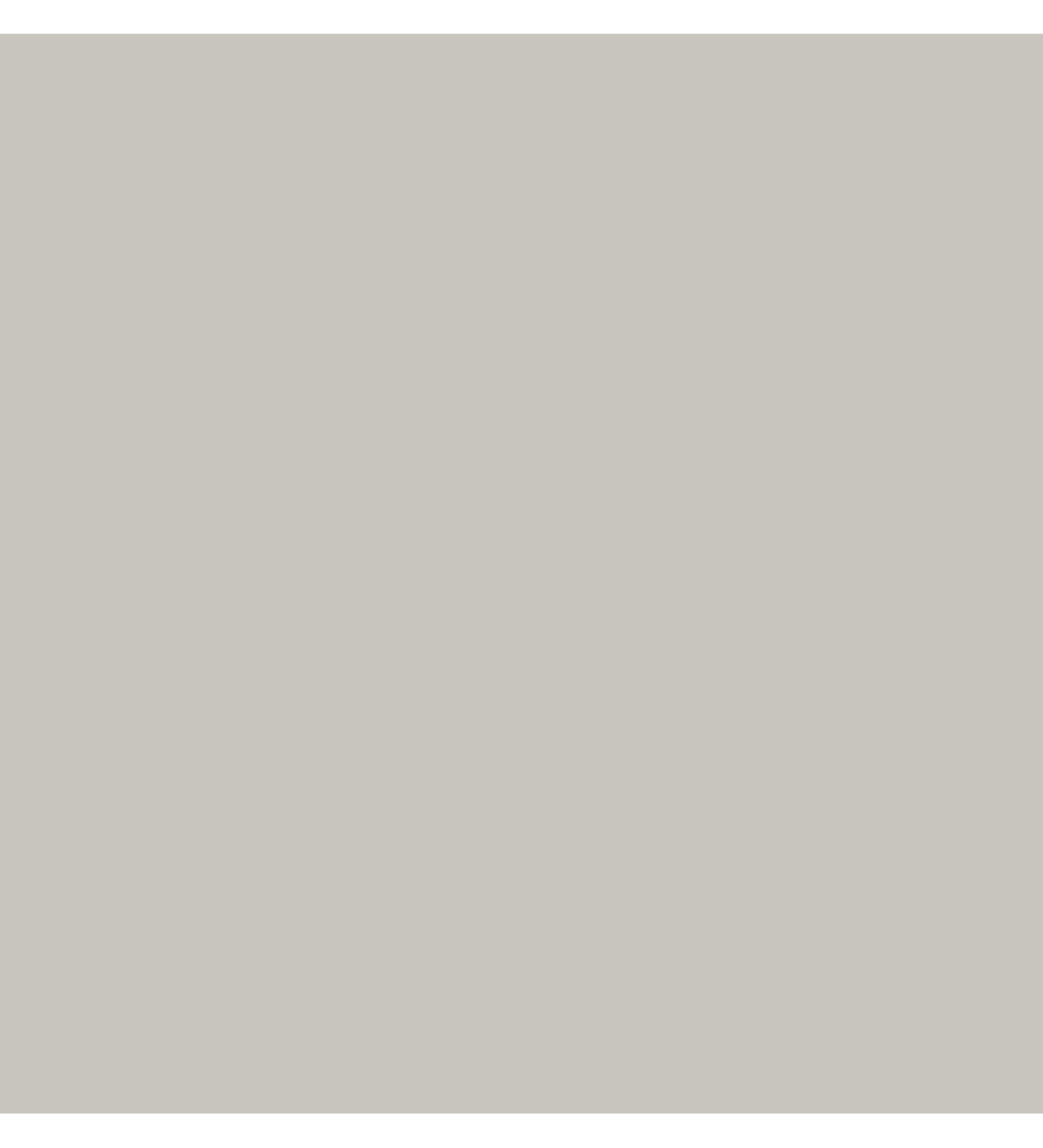}
\includegraphics[scale=0.1125]{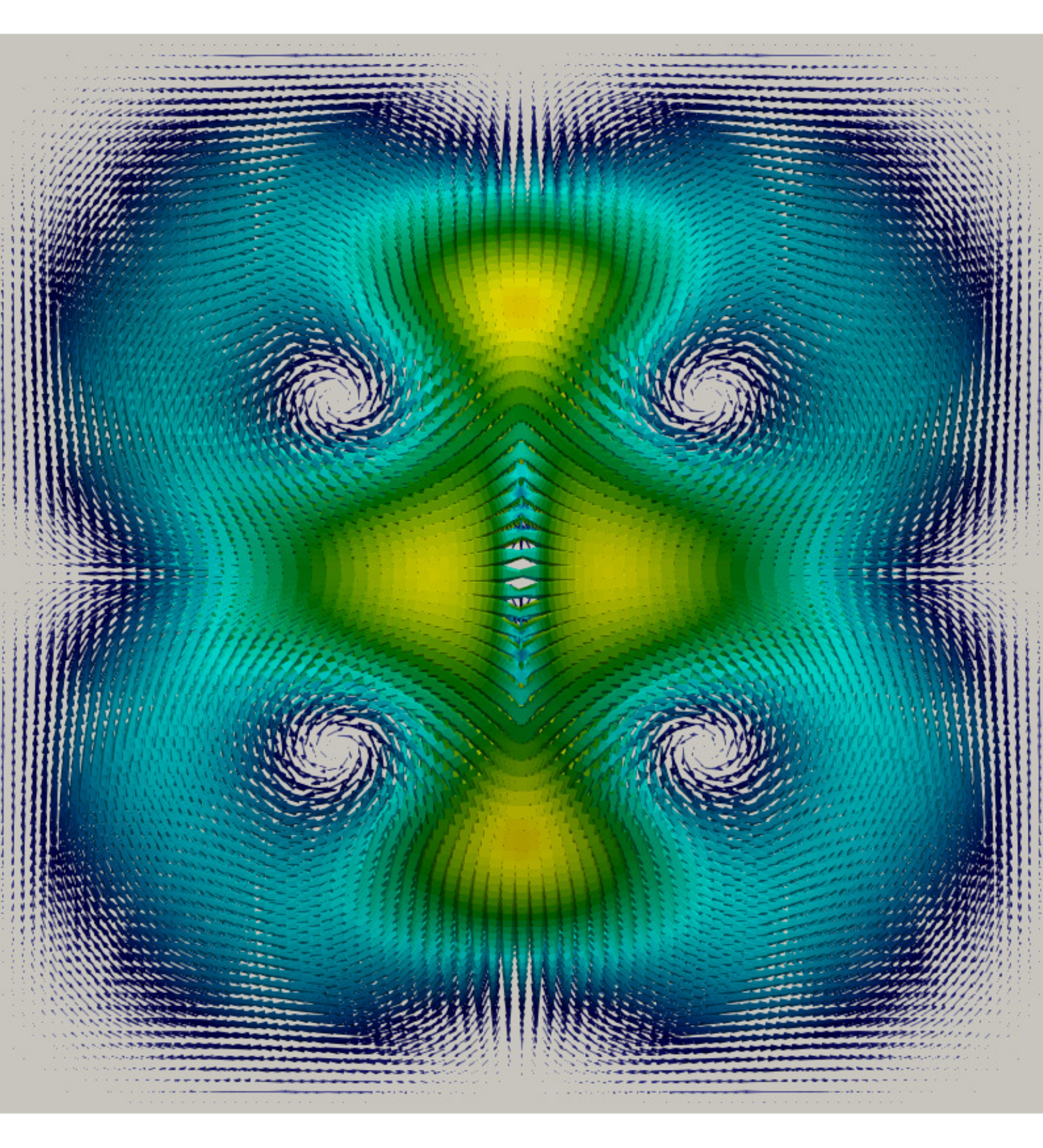}
\includegraphics[scale=0.1125]{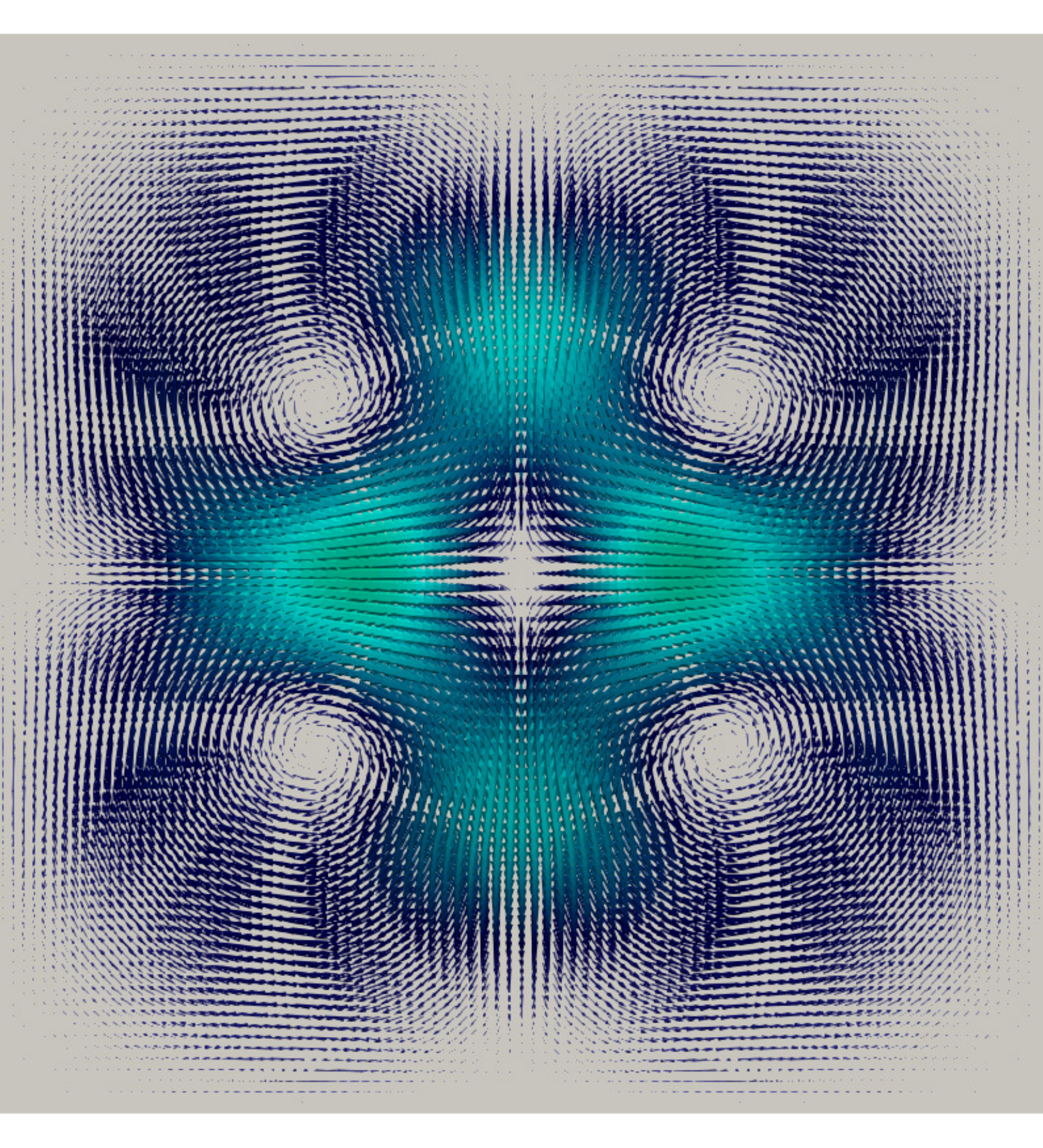}
\includegraphics[scale=0.1125]{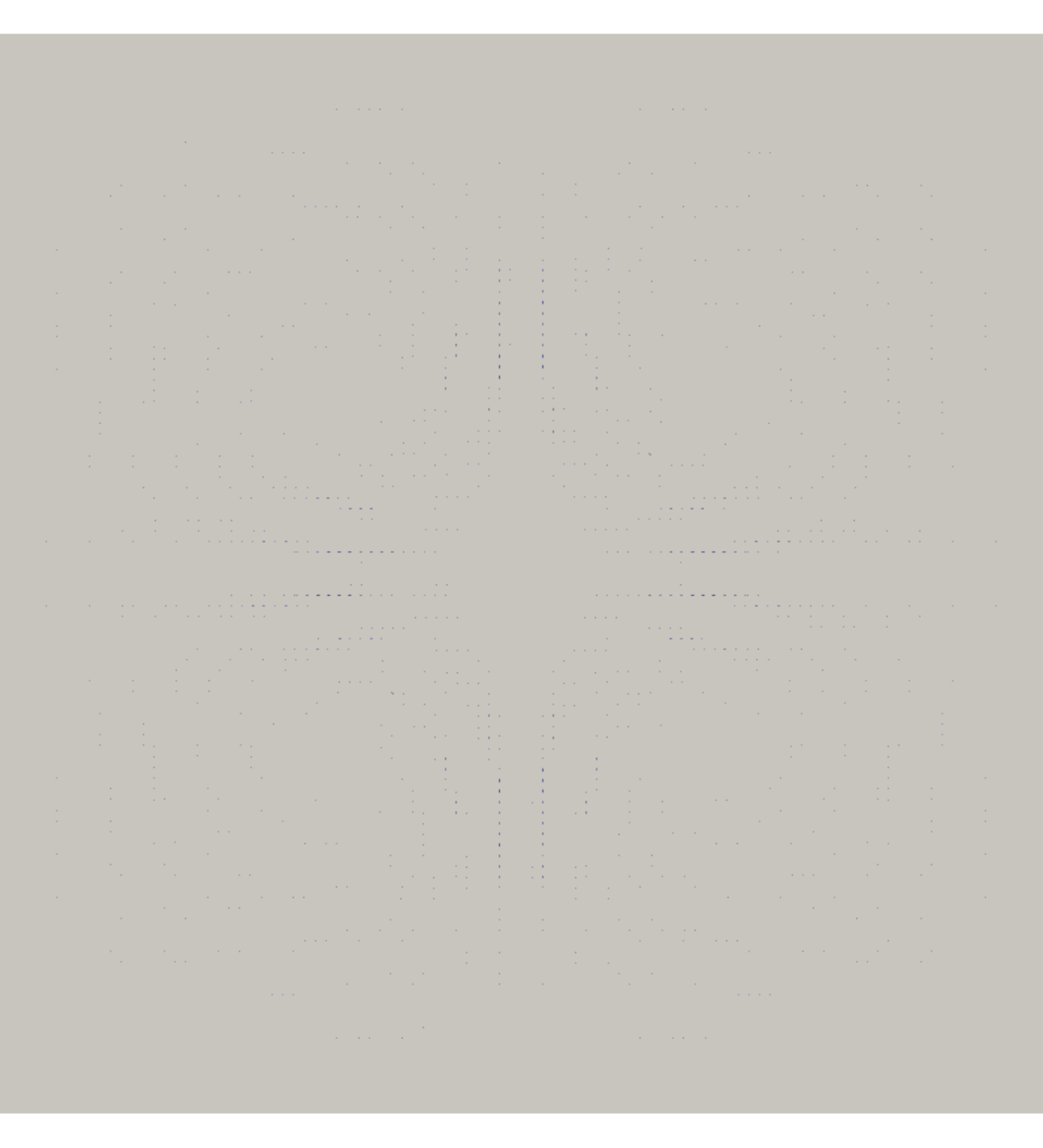}
\includegraphics[scale=0.1125]{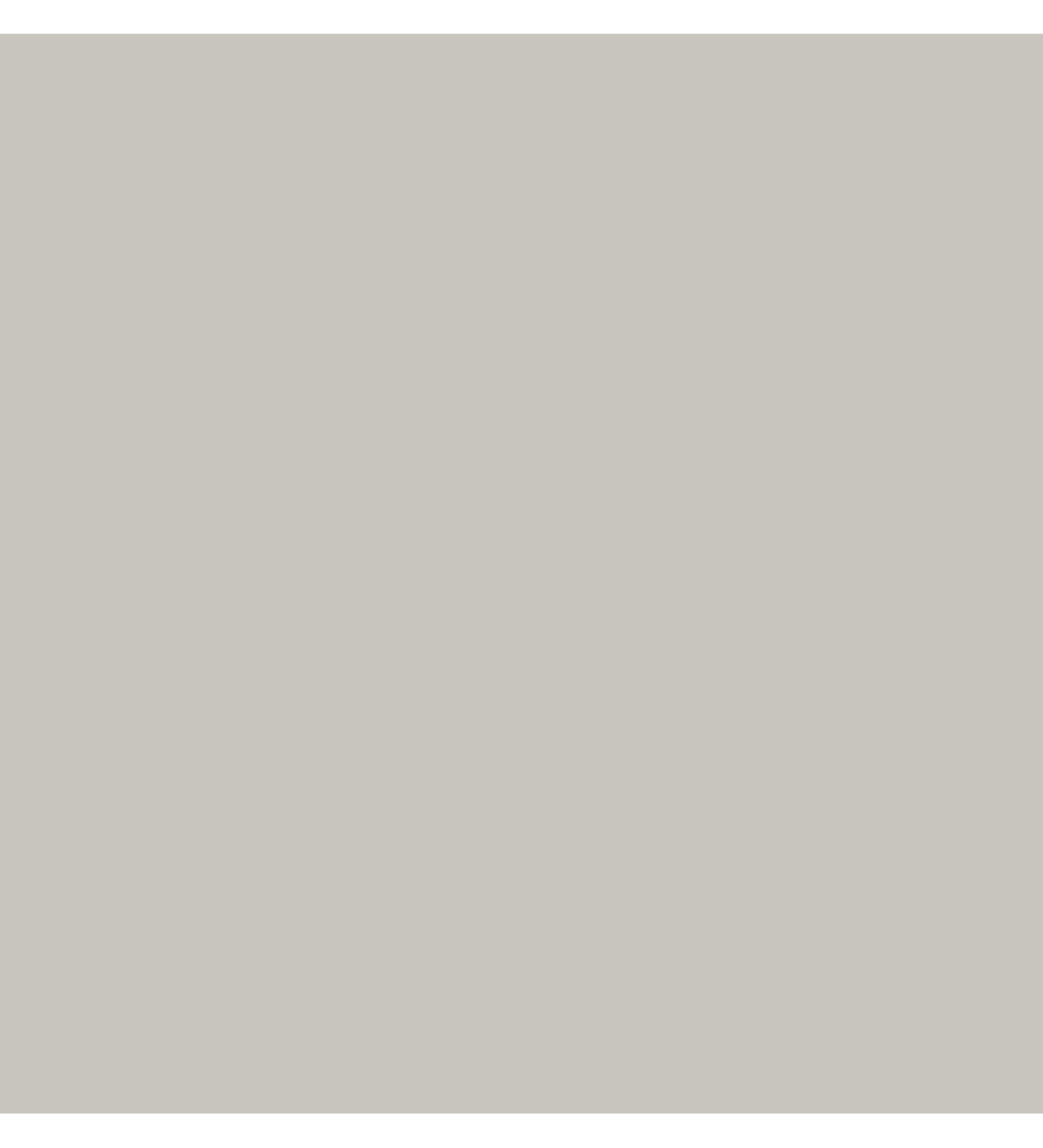}
\end{center}
\caption{Example II. Merging droplets. Evolution in time of $\phi$ and $\u$ for CM-scheme at times $t=0, 0.2, 0.5, 1.5$ and $5$.}\label{fig:Ex2_merging_dyn_CM}
\end{figure}

In Figure~\ref{fig:Ex2_Merging_energies} we present the evolution in time of the different energies and the evolution of the volume.  As expected, all the schemes achieve the dissipation of the free energy (as predicted) as well as the conservation of volume. Moreover in Figure~\ref{fig:Ex2_Merging_bounds} we present the evolution in time of the maximum and minimum of $\phi$ in the domain $\Omega$ and the evolution of $\int_\Omega(\phi_-)^2 d\x$ and $\int_\Omega((\phi-1)_+)^2 d\x$. In fact, we know that these two terms are bounded in terms of $\varepsilon$ due to Remark~\ref{rem:boundG} ($G_\varepsilon$-scheme) and Remark~\ref{rem:boundJ} ($J_\varepsilon$-scheme). As expected, CM-scheme does not enforce the variable $\phi$ to remain in the interval $[0,1]$ while the other two schemes are very close to achieving it. This is a consequence of the evolution $\int_\Omega(\phi_-)^2 d\x$ and $\int_\Omega((\phi-1)_+)^2 d\x$ which are clearly bounded by $\varepsilon$ and as expected the obtained bound for $G_\varepsilon$-scheme is slightly better than the one for $J_\varepsilon$-scheme (as expected from Remarks~\ref{rem:boundG} and \ref{rem:boundJ}).

\begin{figure}[H]
\begin{center}
\includegraphics[scale=0.1125]{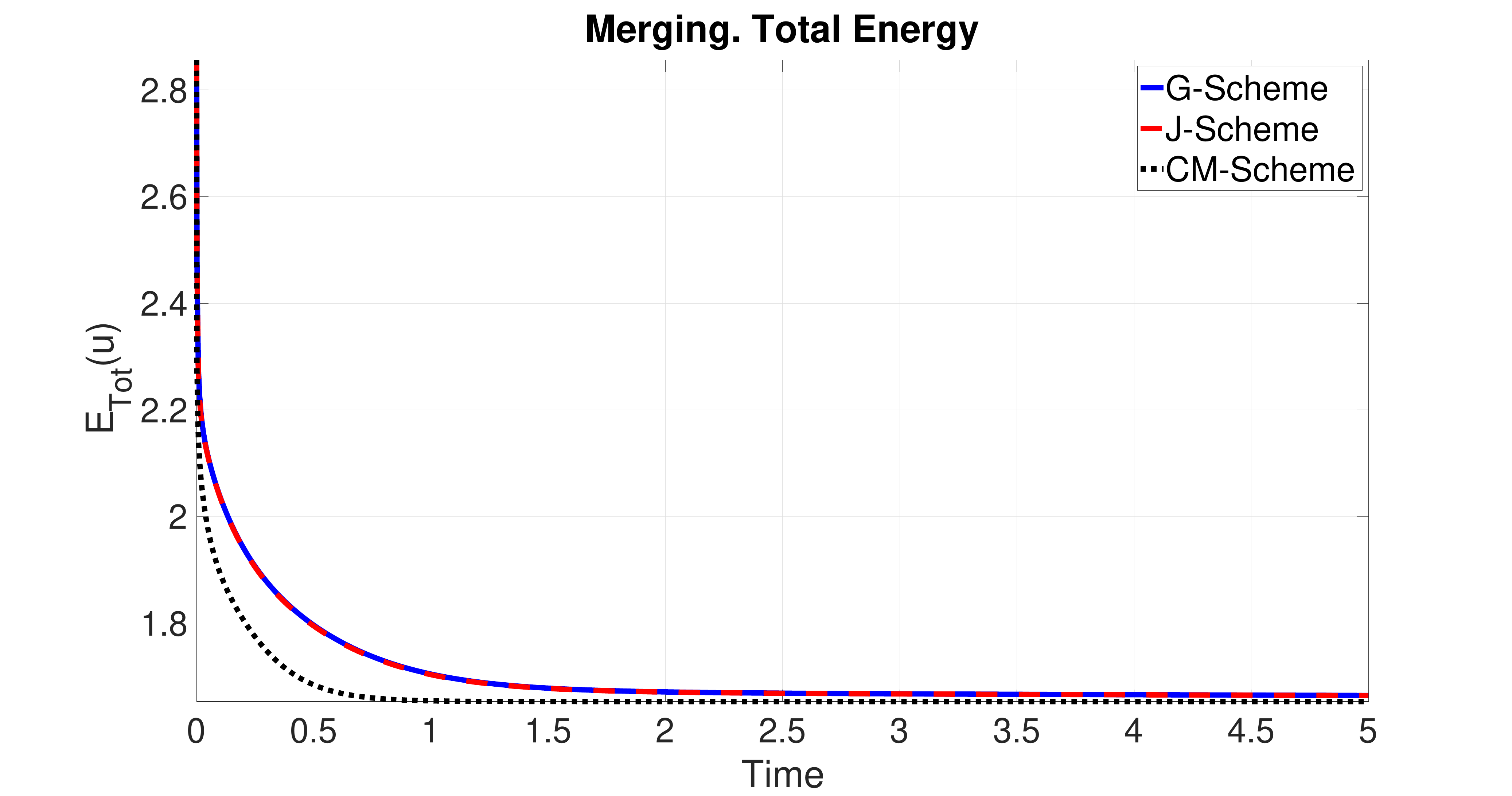}
\includegraphics[scale=0.1125]{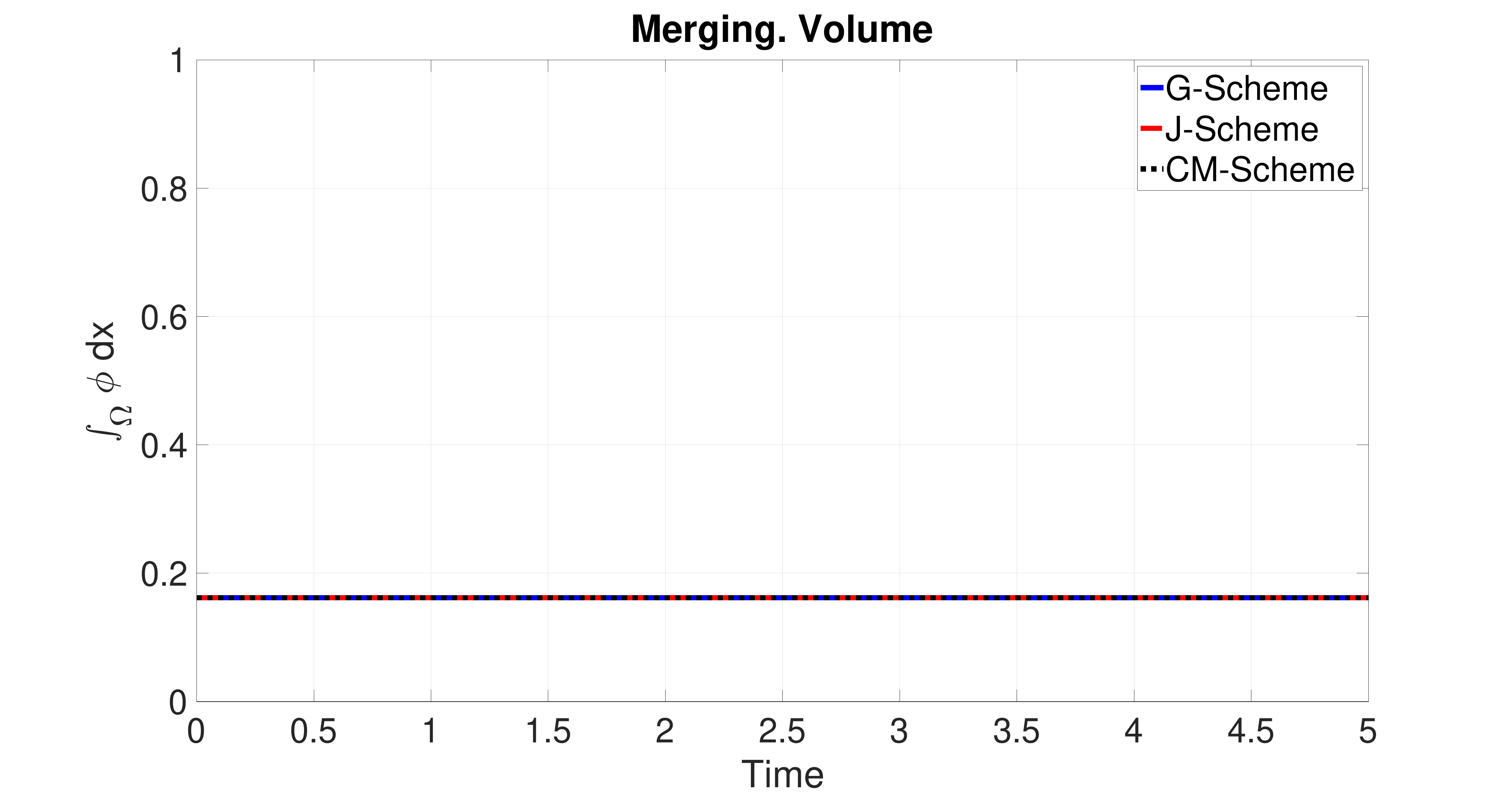}
\\
\includegraphics[scale=0.1125]{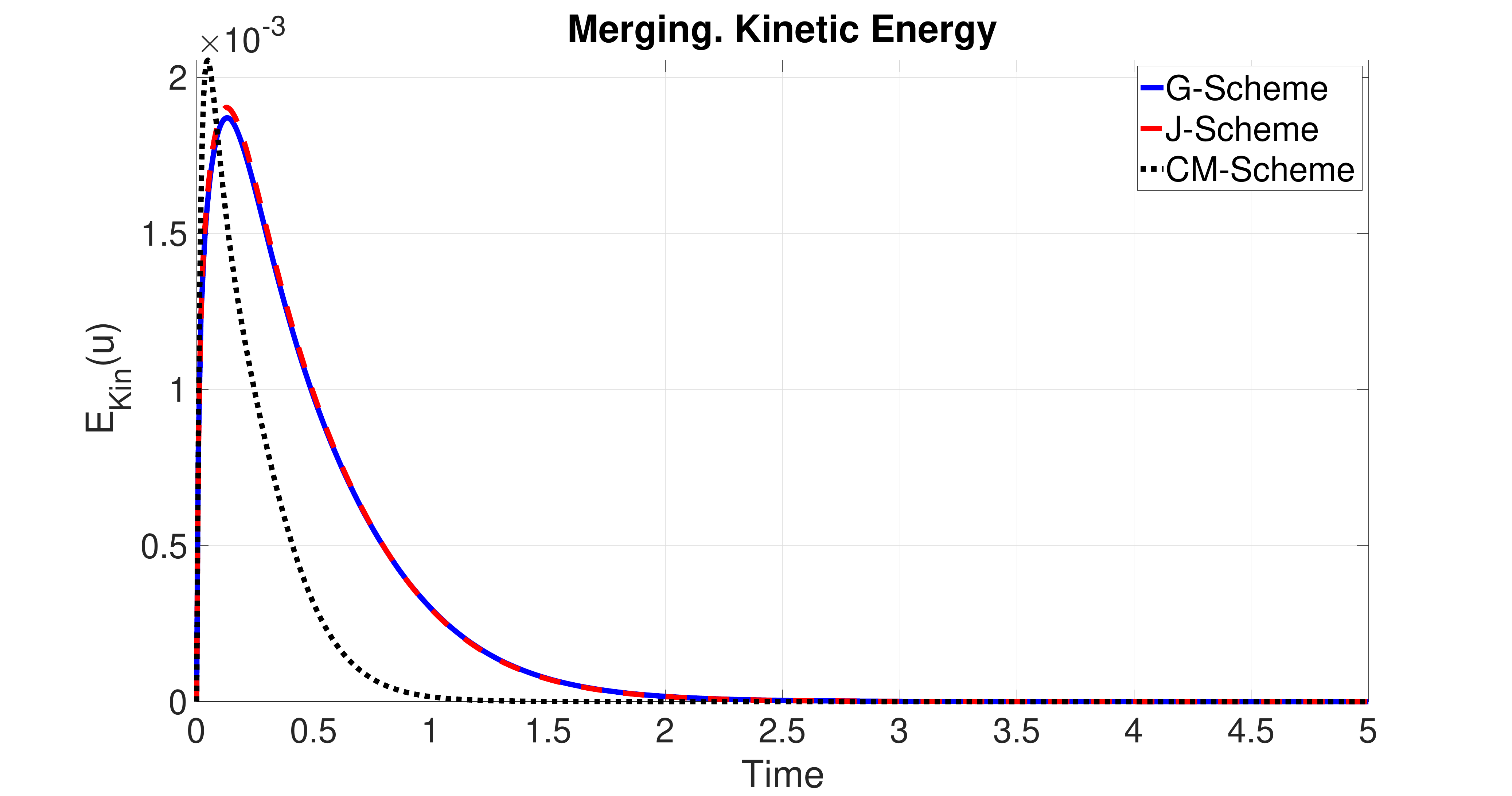}
\includegraphics[scale=0.1125]{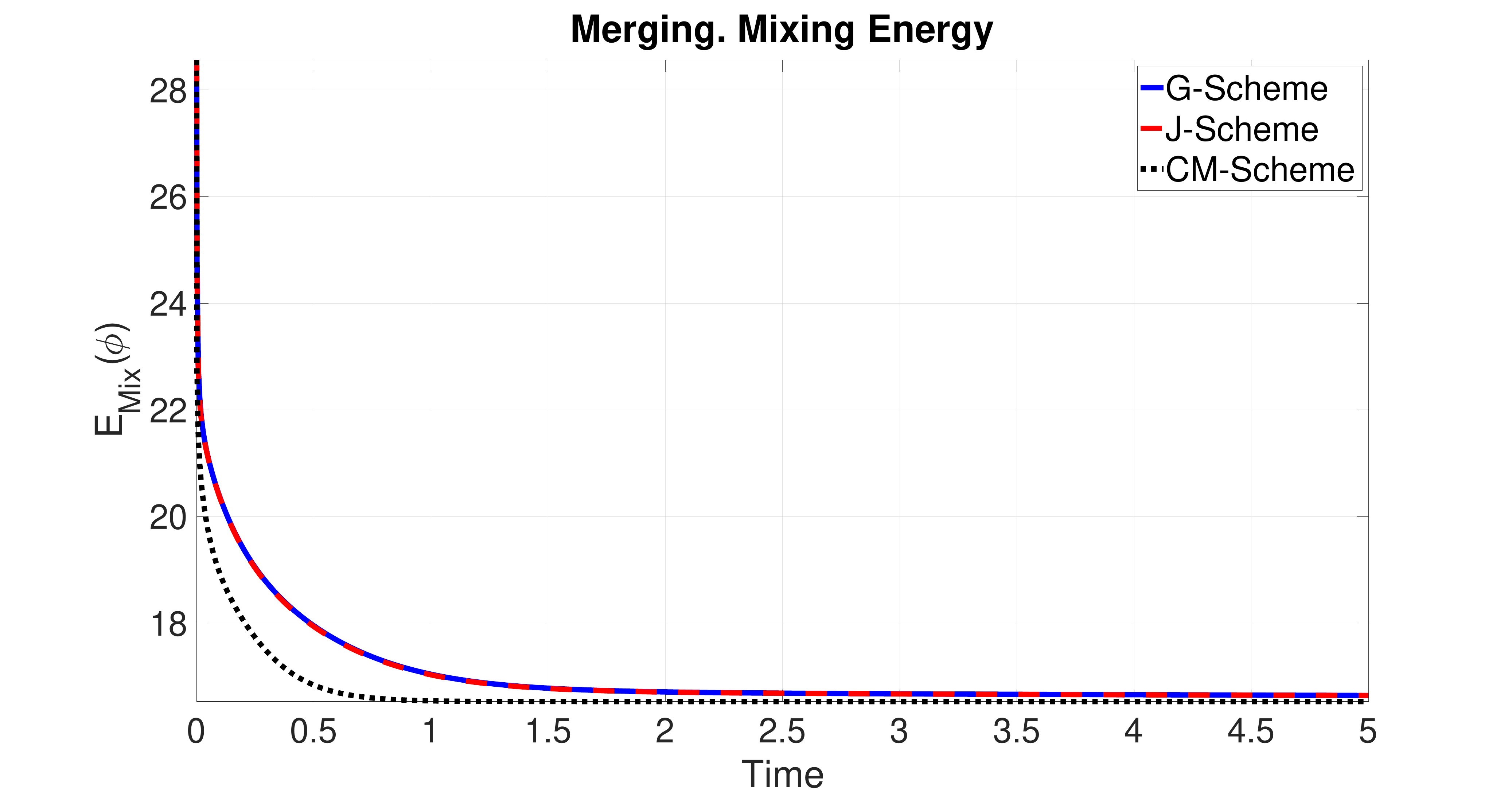}
\end{center}
\caption{Example II. Merging droplets. Evolution of the energies and the volume of the system}\label{fig:Ex2_Merging_energies}
\end{figure}

\begin{figure}[H]
\begin{center}
\includegraphics[scale=0.1125]{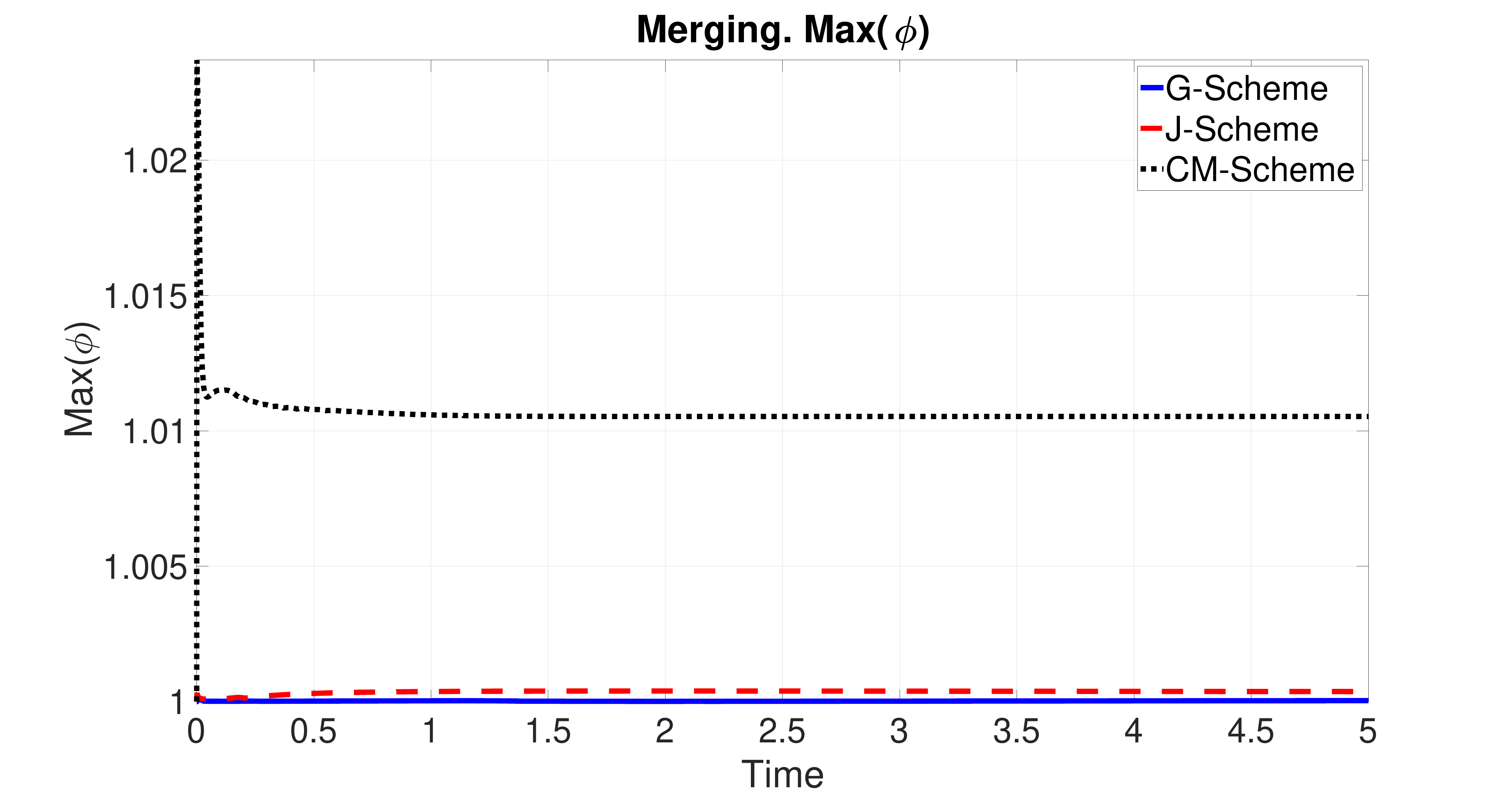}
\includegraphics[scale=0.1125]{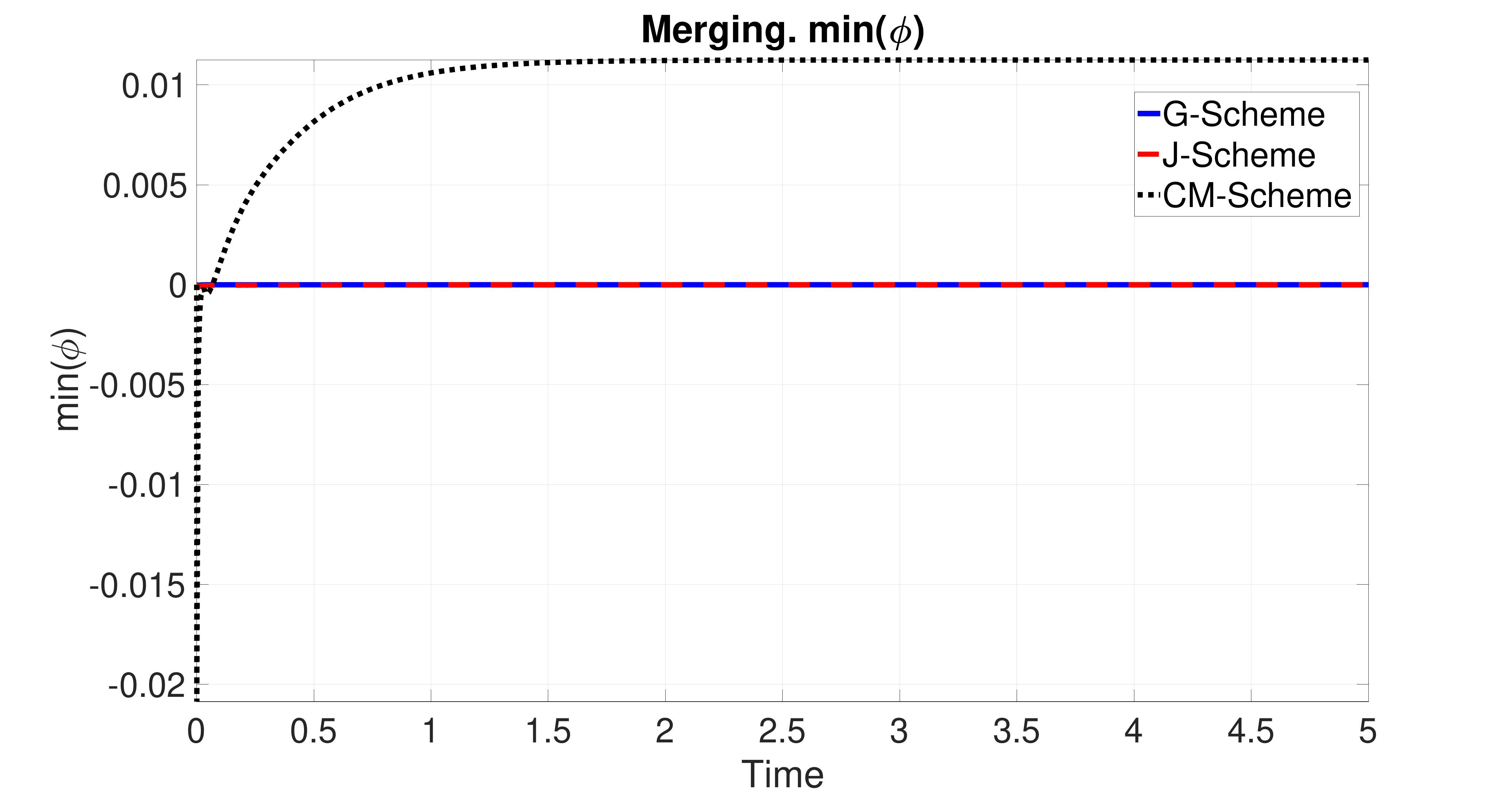}
\\
\includegraphics[scale=0.1125]{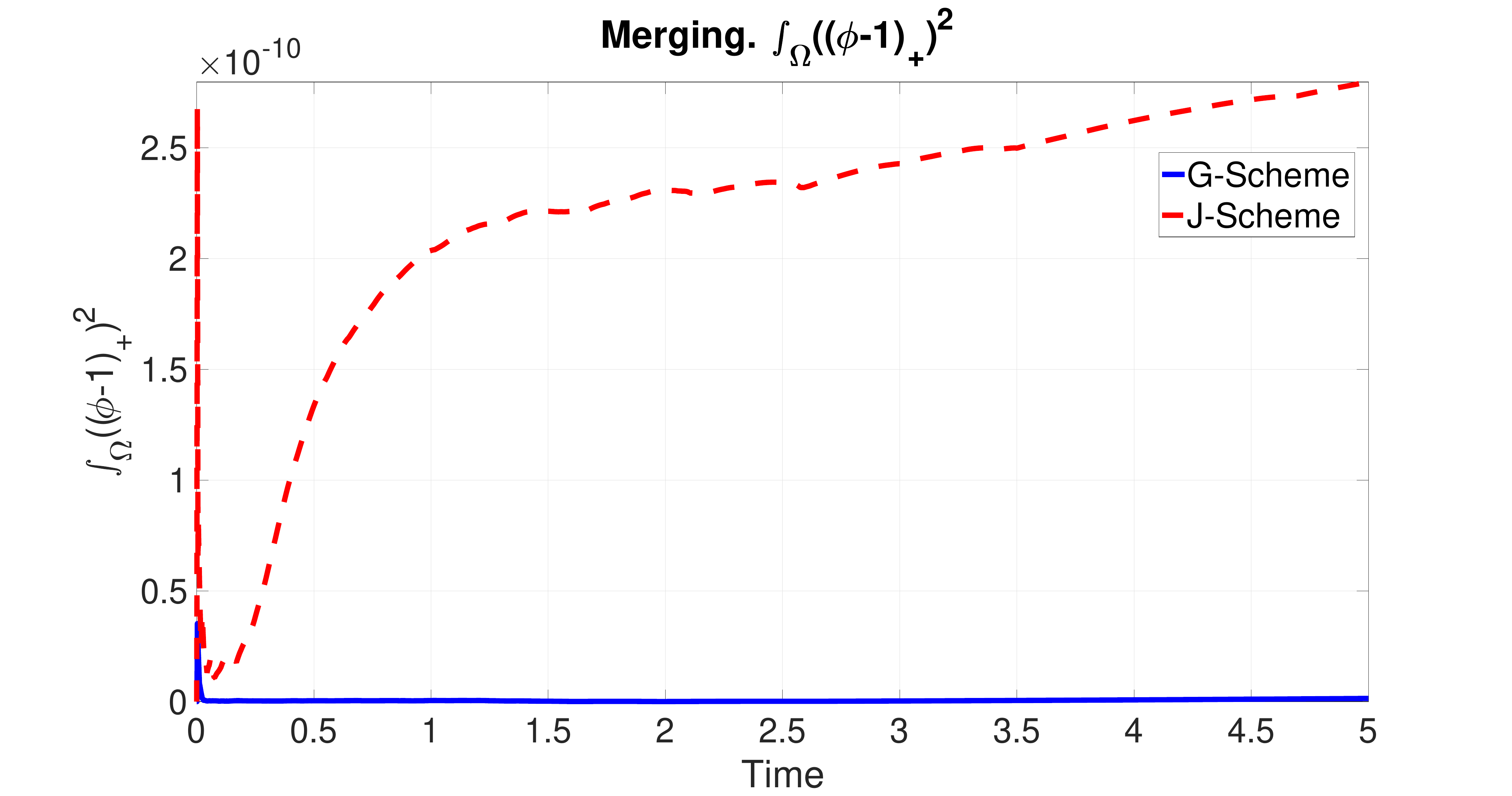}
\includegraphics[scale=0.1125]{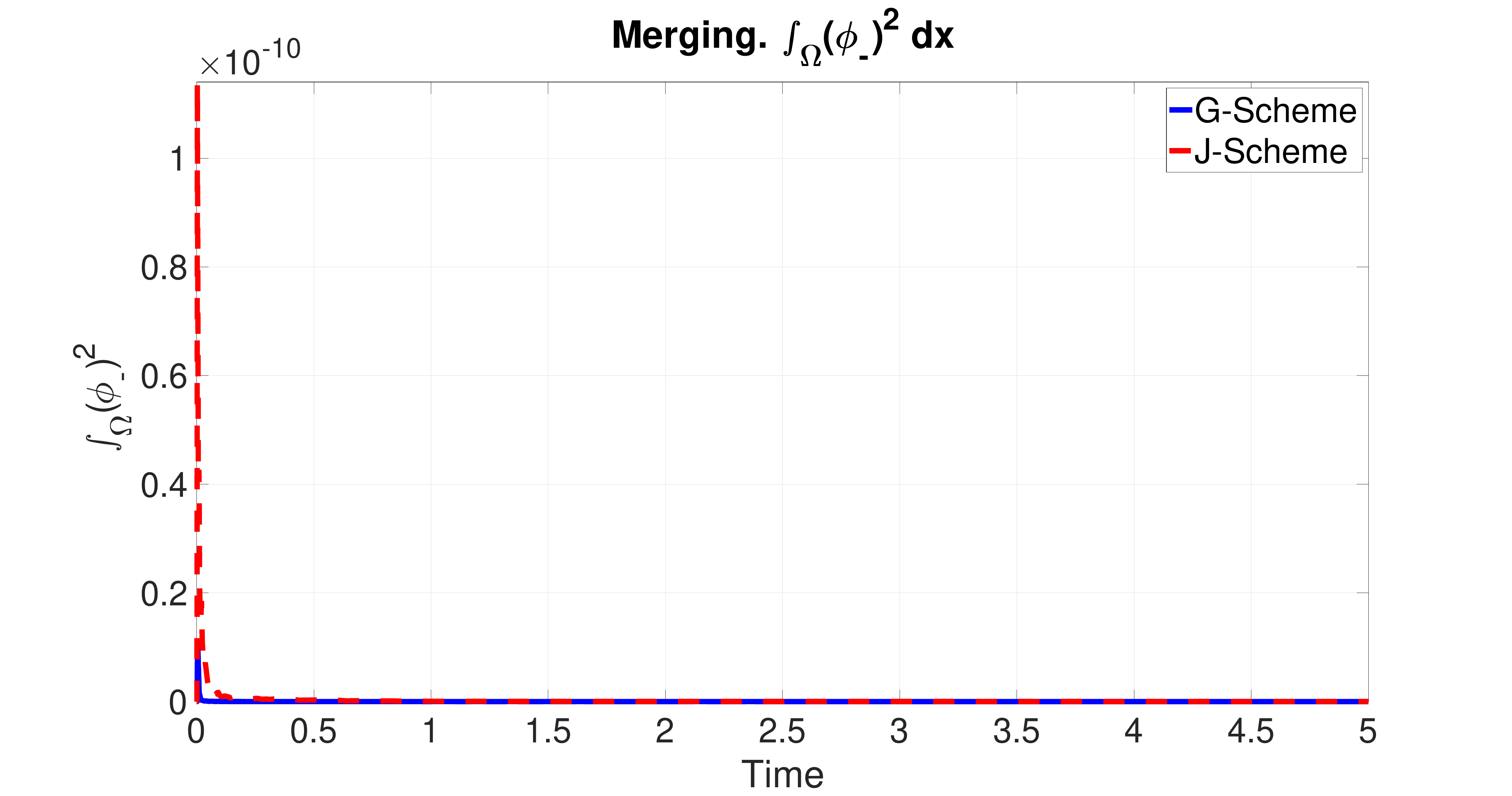}
\end{center}
\caption{Example II. Merging droplets.  Evolution of the bounds of $\phi$.}\label{fig:Ex2_Merging_bounds}
\end{figure}

\subsubsection{Coarsening effect}

In this case we have designed an initial condition resembling two droplets of a fluid immersed in another fluid, whose boundaries are far away from each other. The dynamics of the system will depend on the choice of the mobility term, producing the coarsening effect when the mobility is constant (see Figure~\ref{fig:ExIICoarCM}), i.e. the two droplets will combine to arrive to an equilibrium configuration with a single larger droplet. The coarsening effect will not occur when the mobility is non-constant (or at least it does not occur in the time interval that we have simulated), as it can be observed in Figures \ref{fig:ExIICoarG} and \ref{fig:ExIICoarJ} ($G_\varepsilon$-scheme and $J_\varepsilon$-scheme, respectively).

\begin{figure}[H]
\begin{center}
\includegraphics[scale=0.1125]{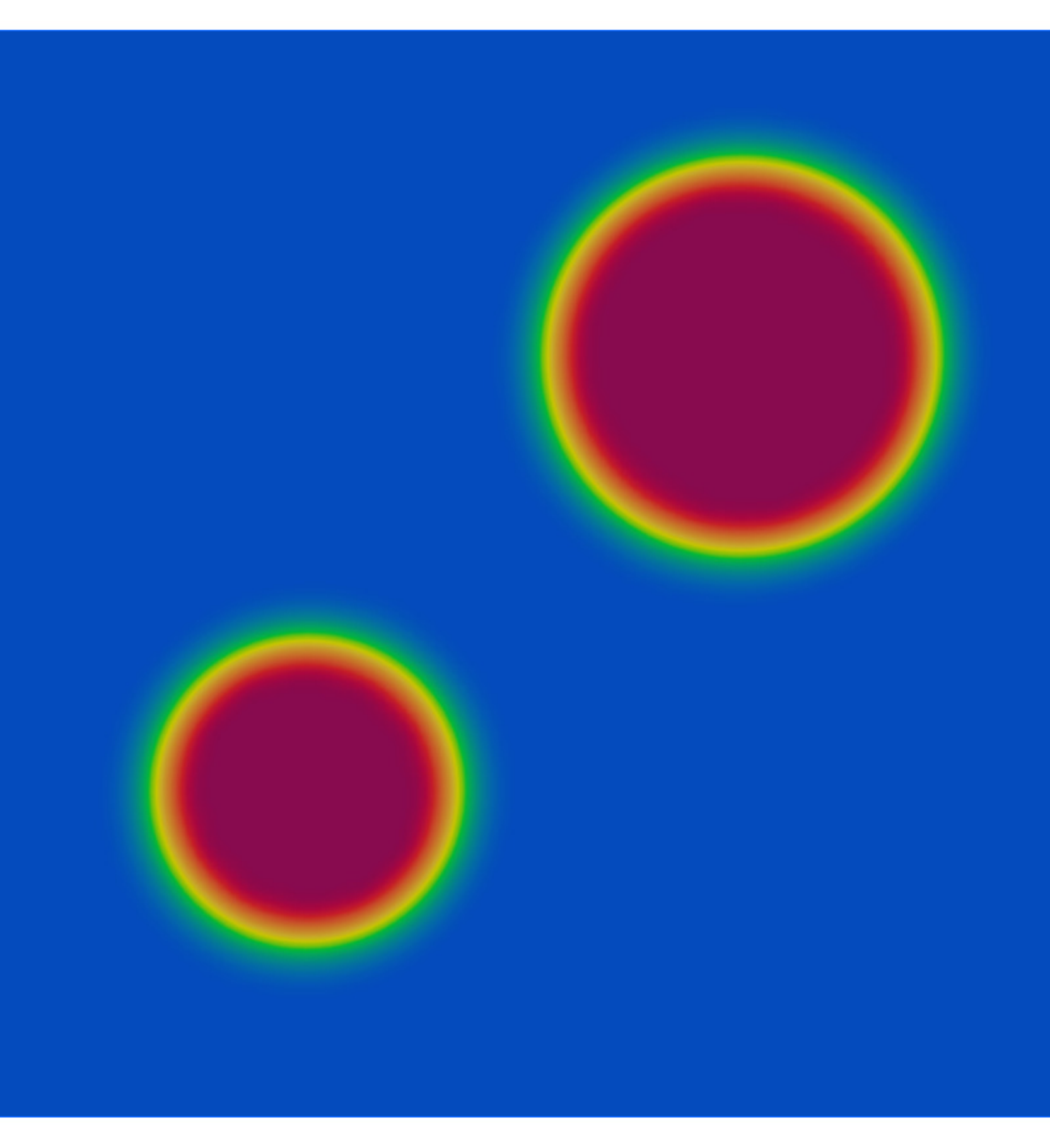}
\includegraphics[scale=0.1125]{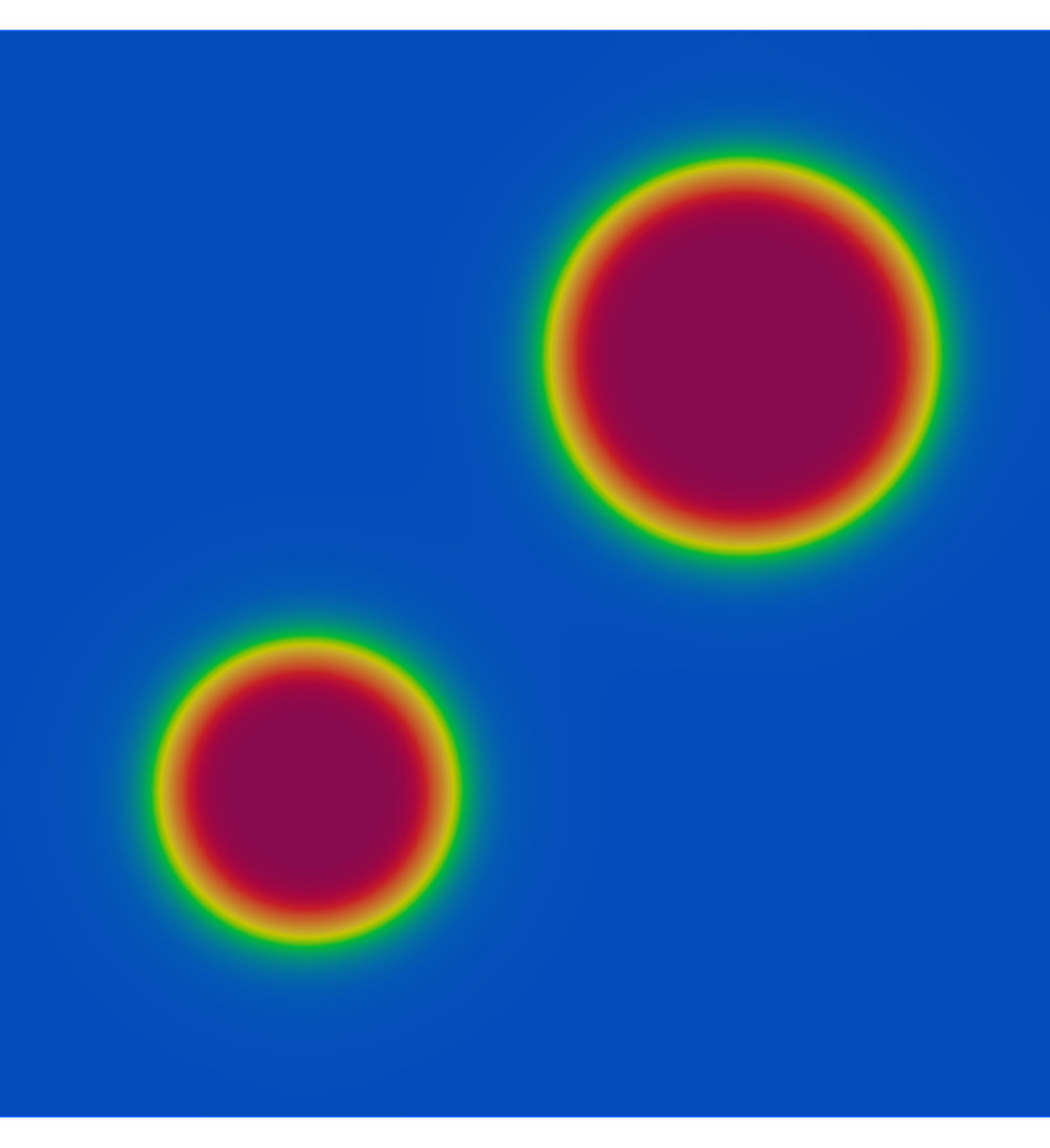}
\includegraphics[scale=0.1125]{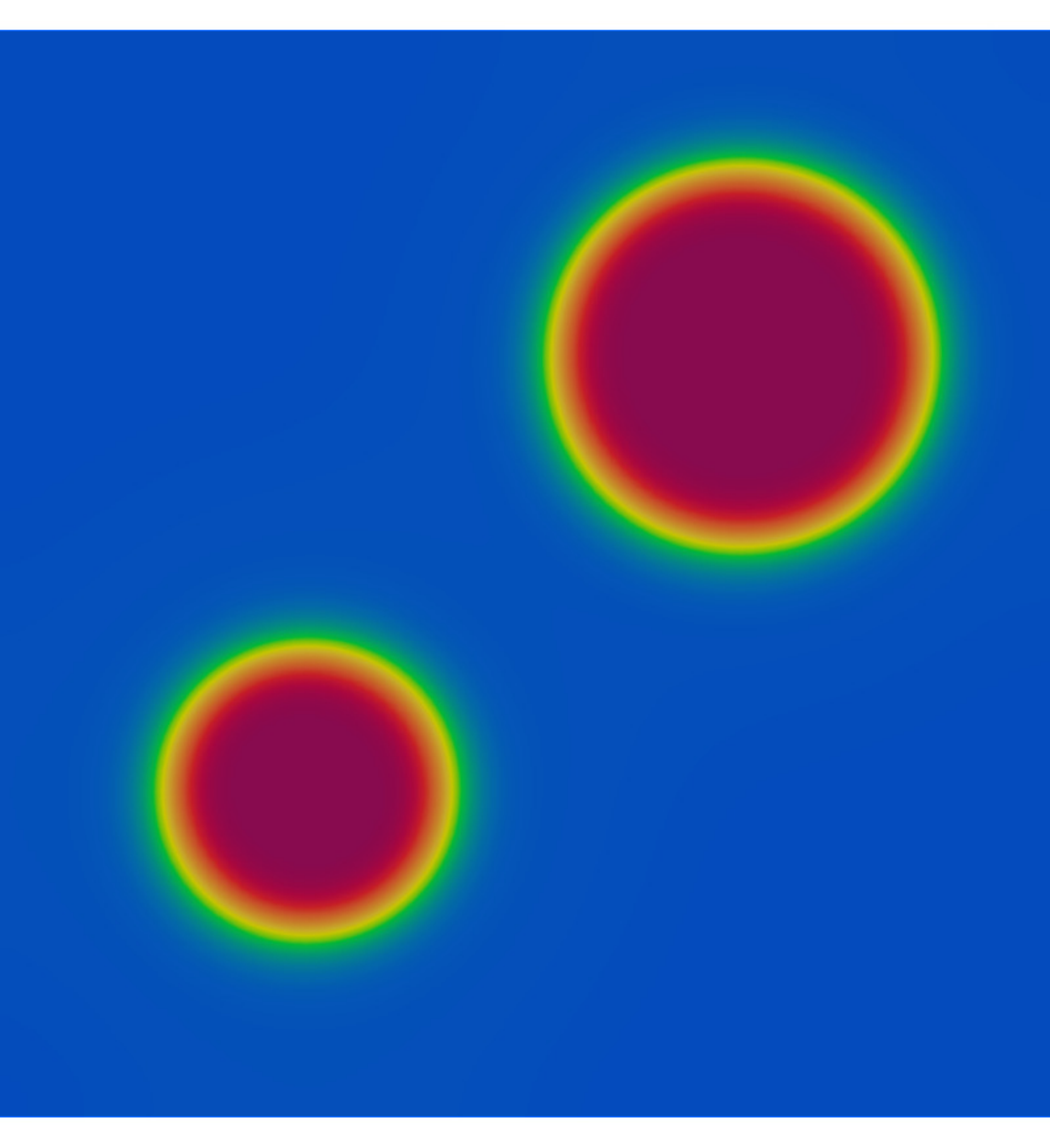}
\includegraphics[scale=0.1125]{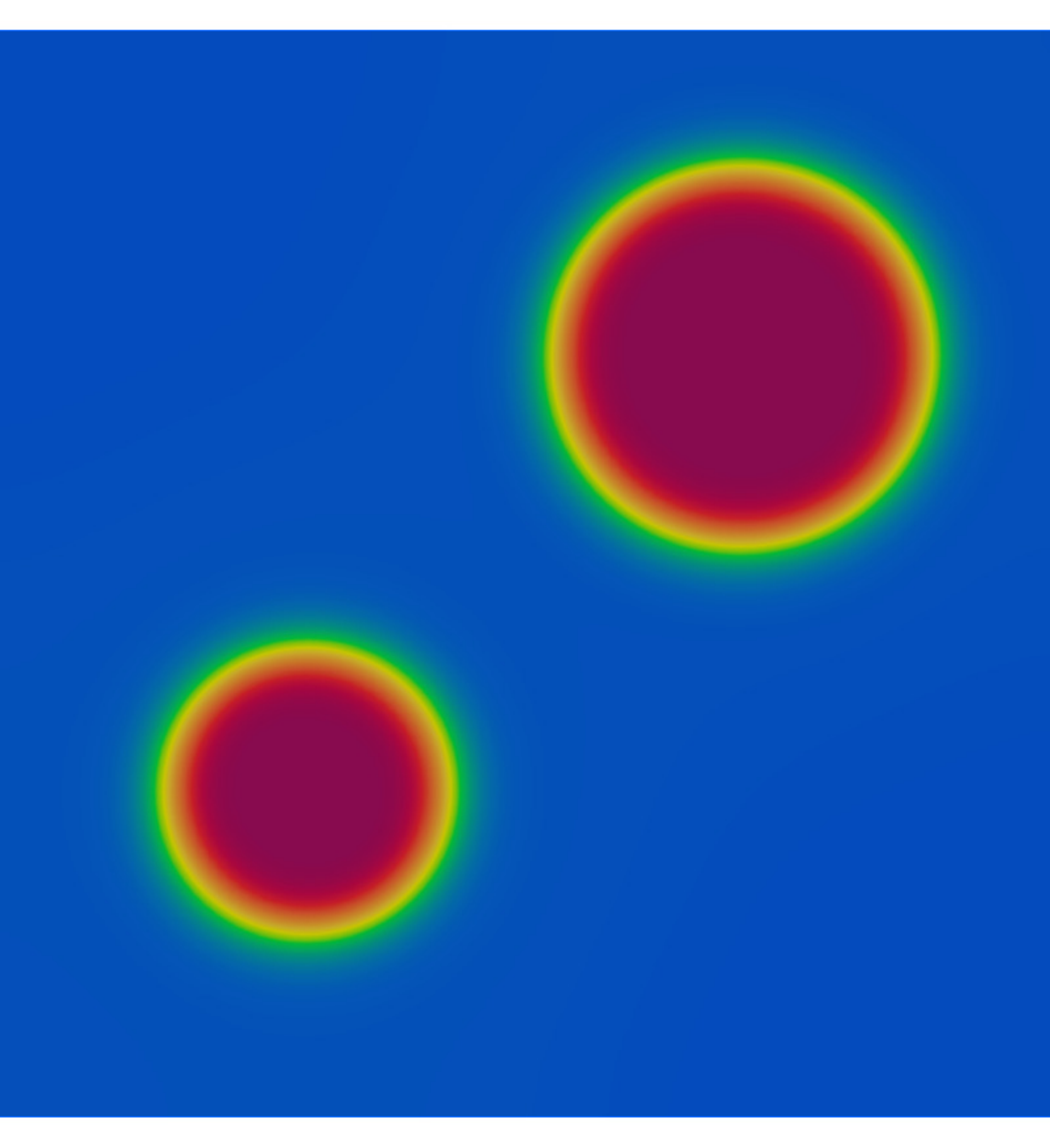}
\includegraphics[scale=0.1125]{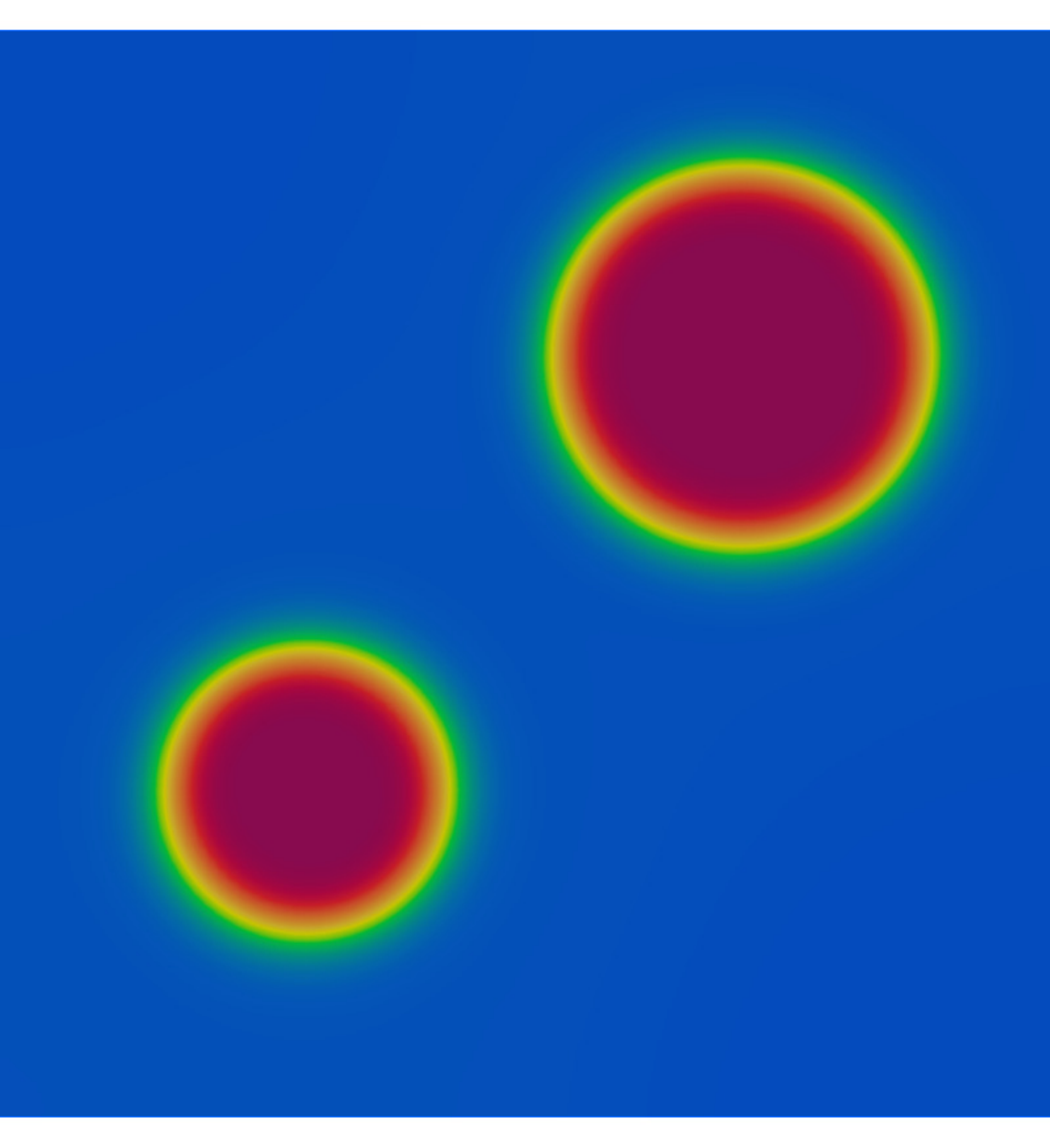}
\\
\includegraphics[scale=0.1125]{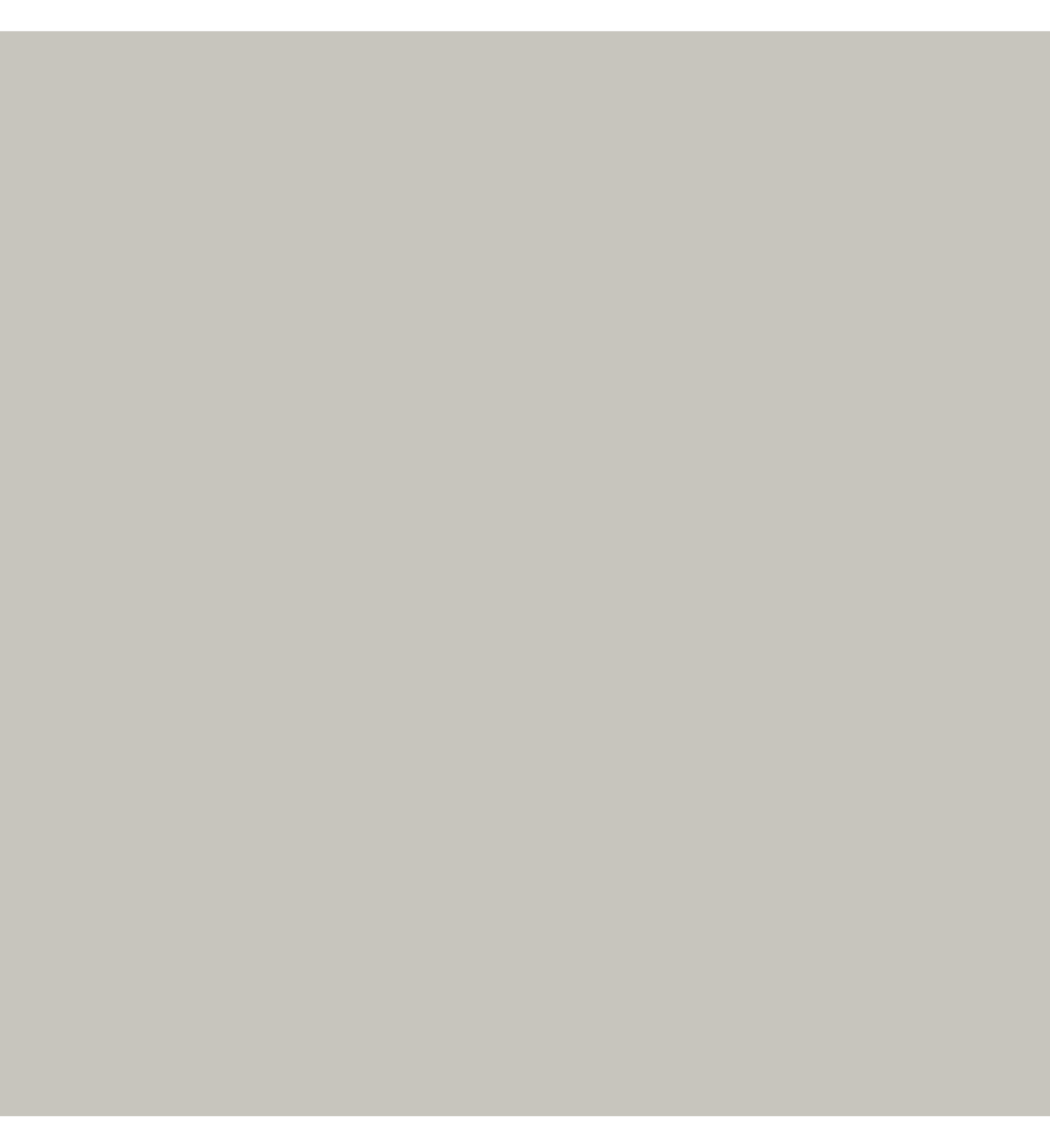}
\includegraphics[scale=0.1125]{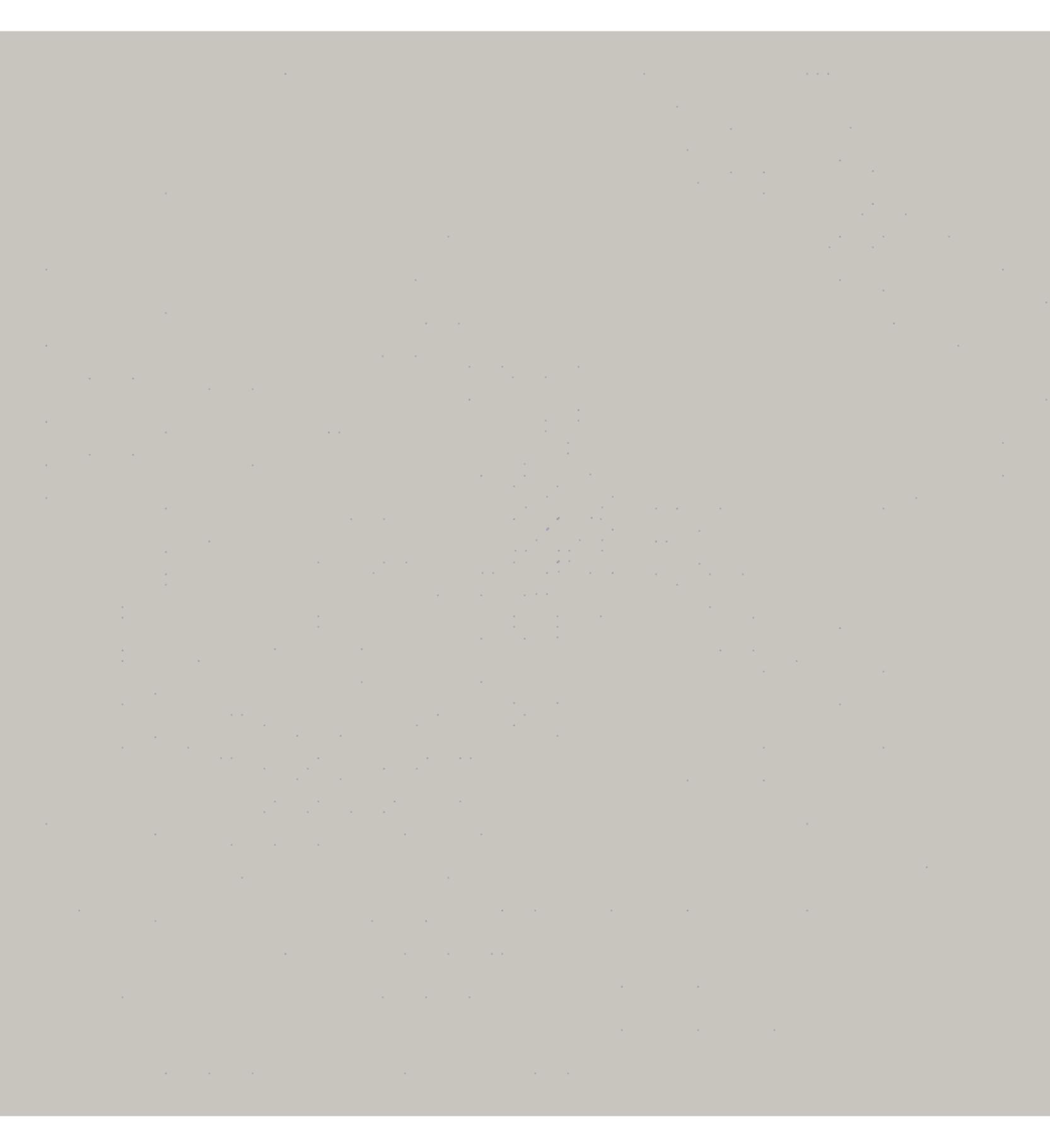}
\includegraphics[scale=0.1125]{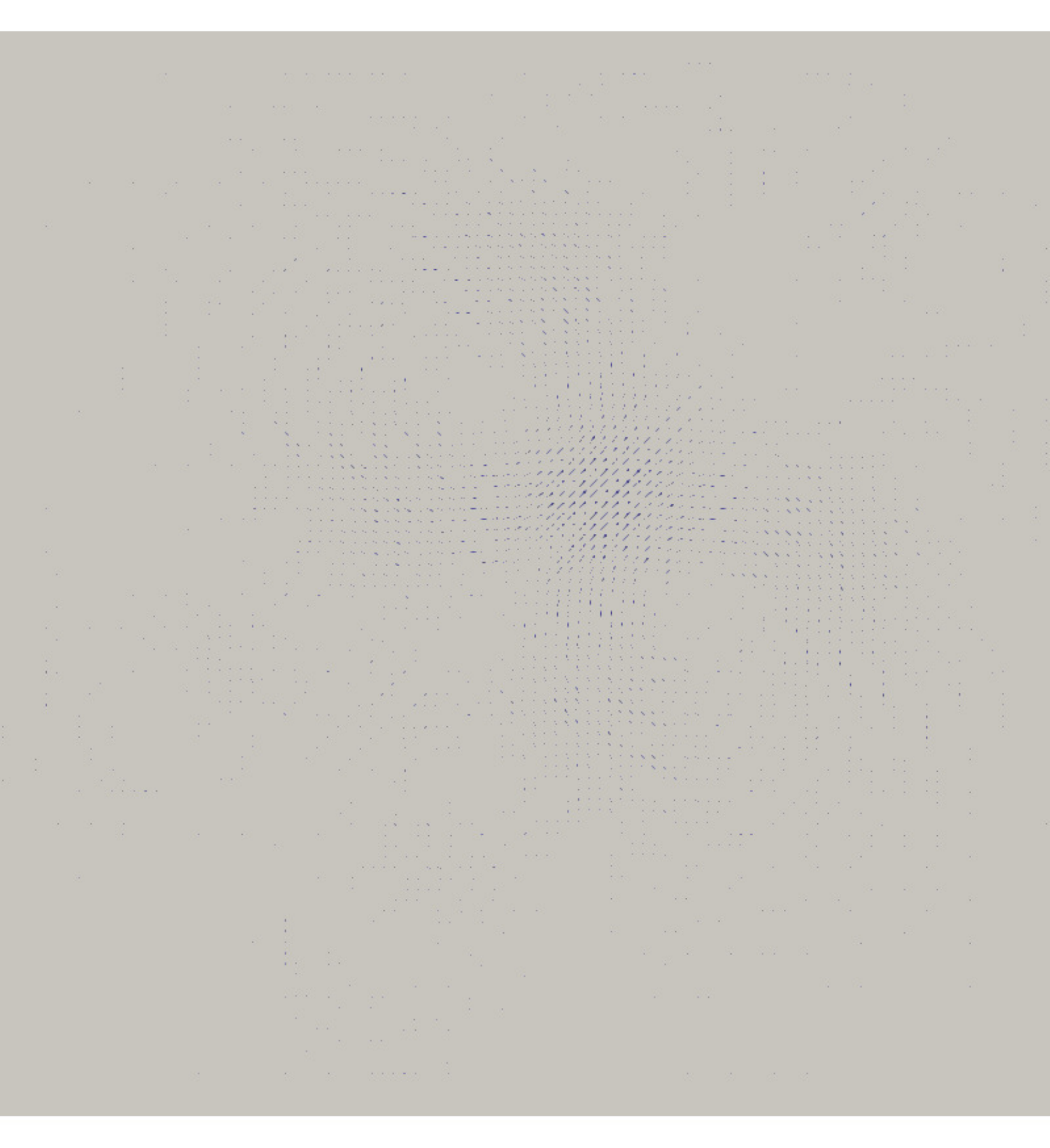}
\includegraphics[scale=0.1125]{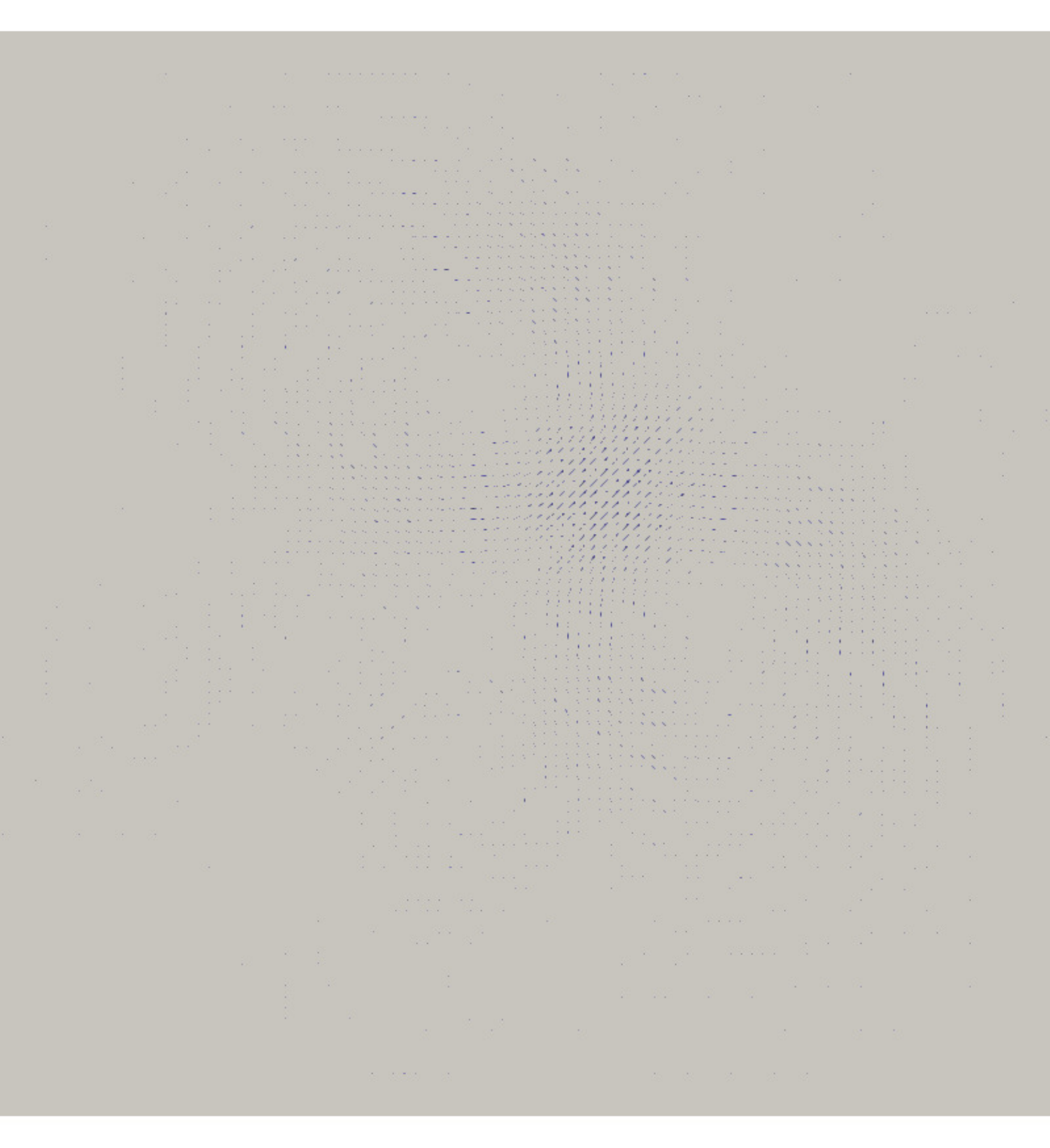}
\includegraphics[scale=0.1125]{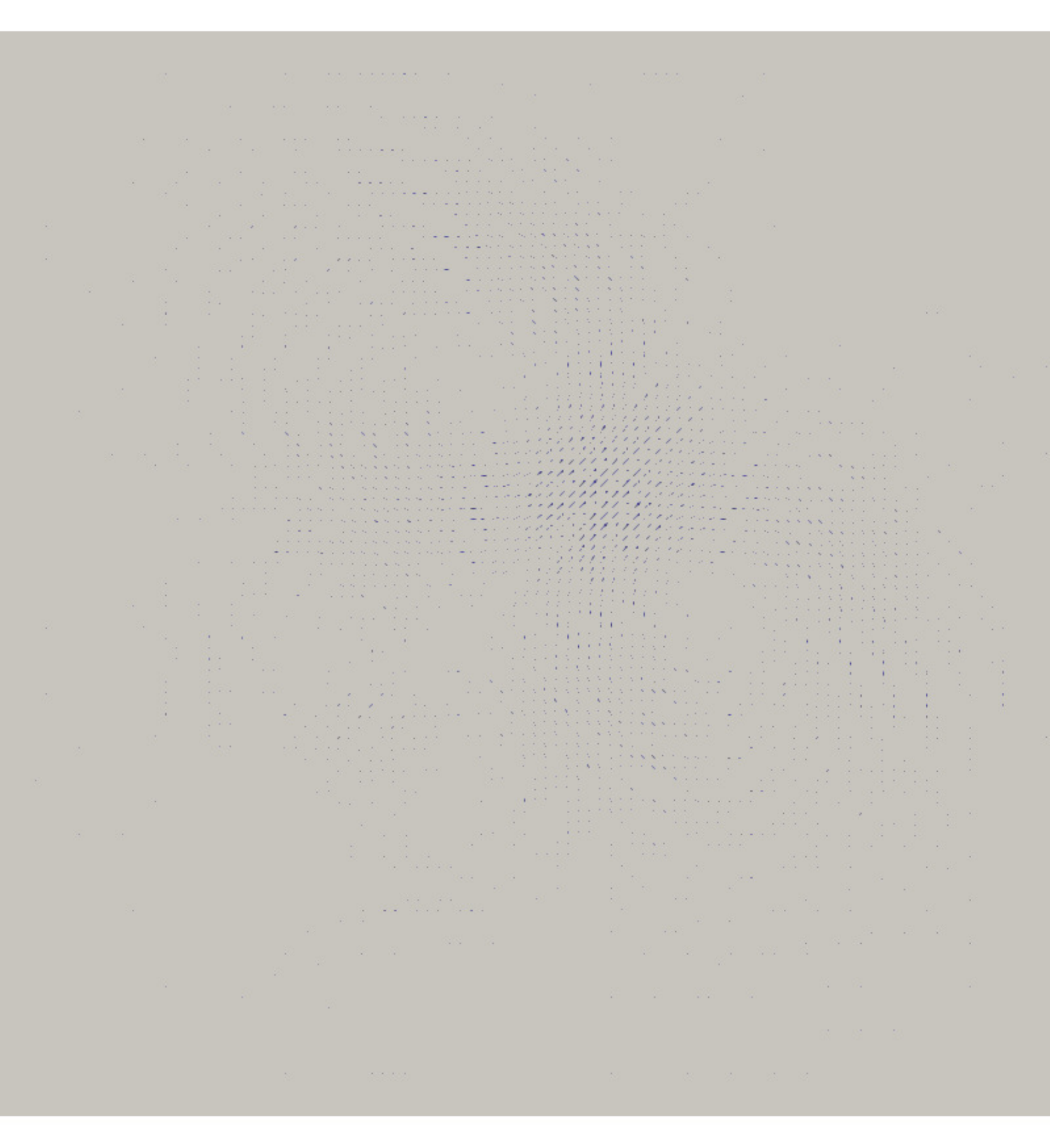}
\end{center}
\caption{Example II. Coarsening effect. Evolution in time of $\phi$ and $\u$ for $G_\varepsilon$-scheme at times  at times $t=0, 1, 2.5, 4$ and $5$.}\label{fig:ExIICoarG}
\end{figure}

\begin{figure}[H]
\begin{center}
\includegraphics[scale=0.1125]{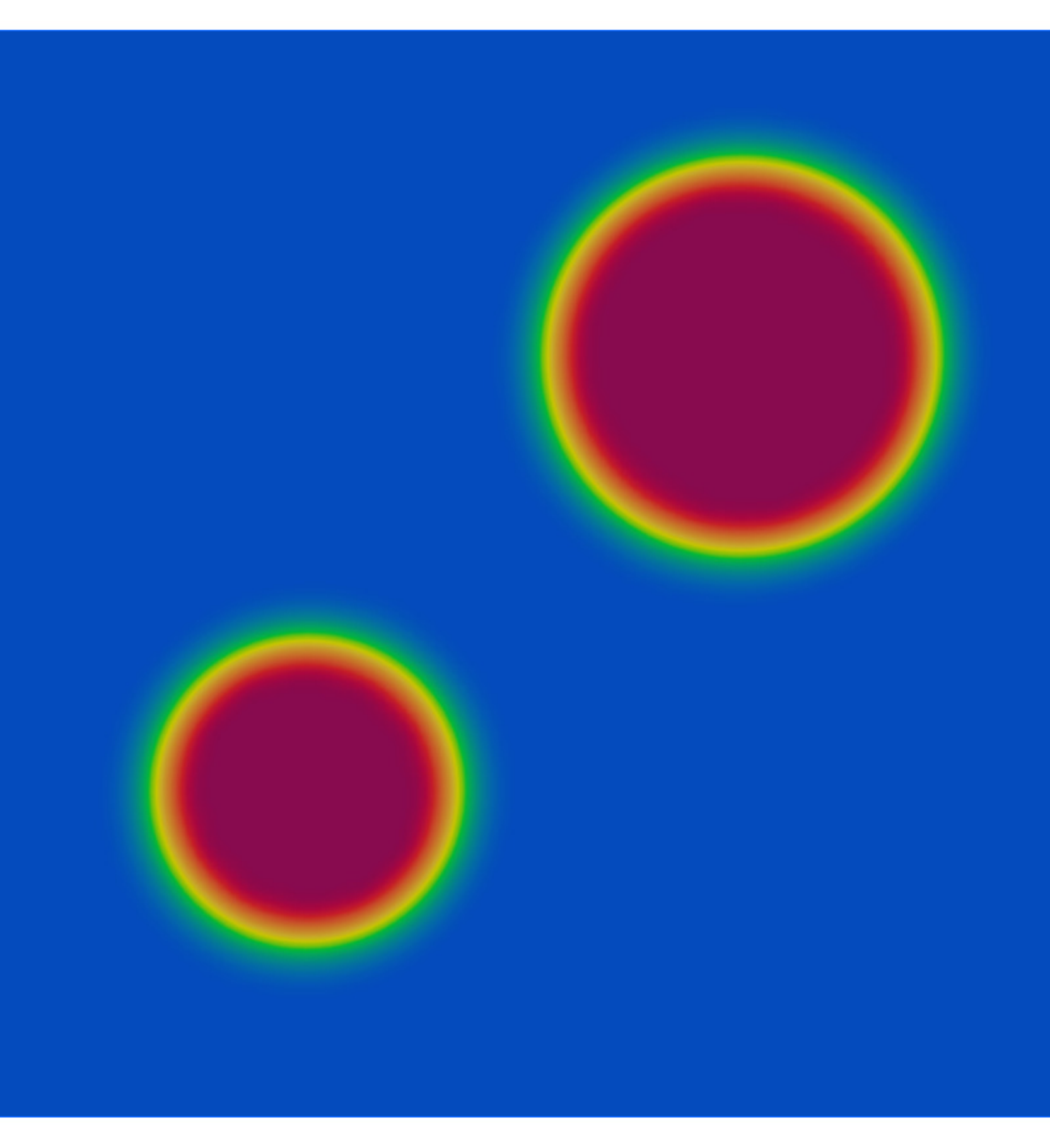}
\includegraphics[scale=0.1125]{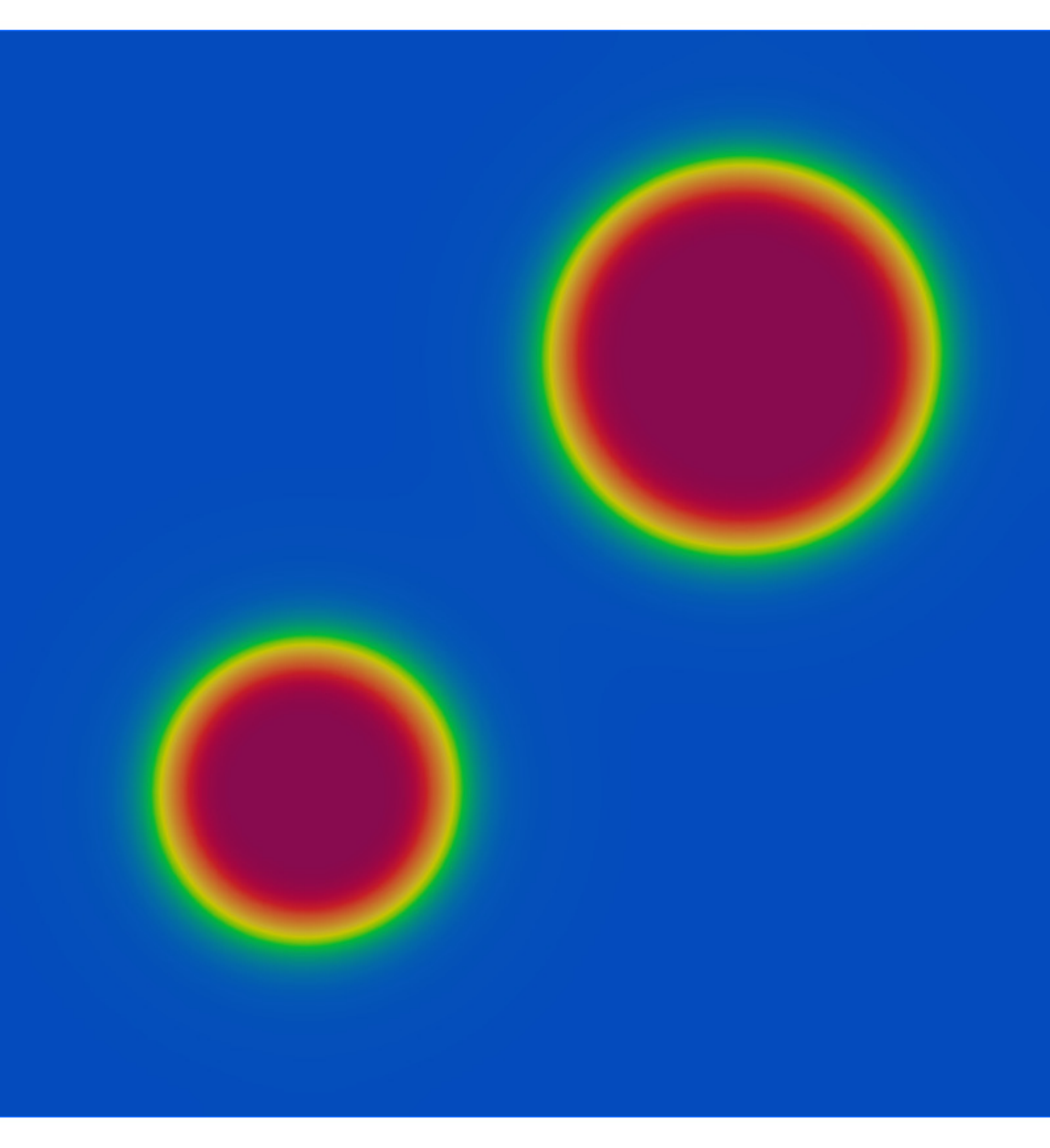}
\includegraphics[scale=0.1125]{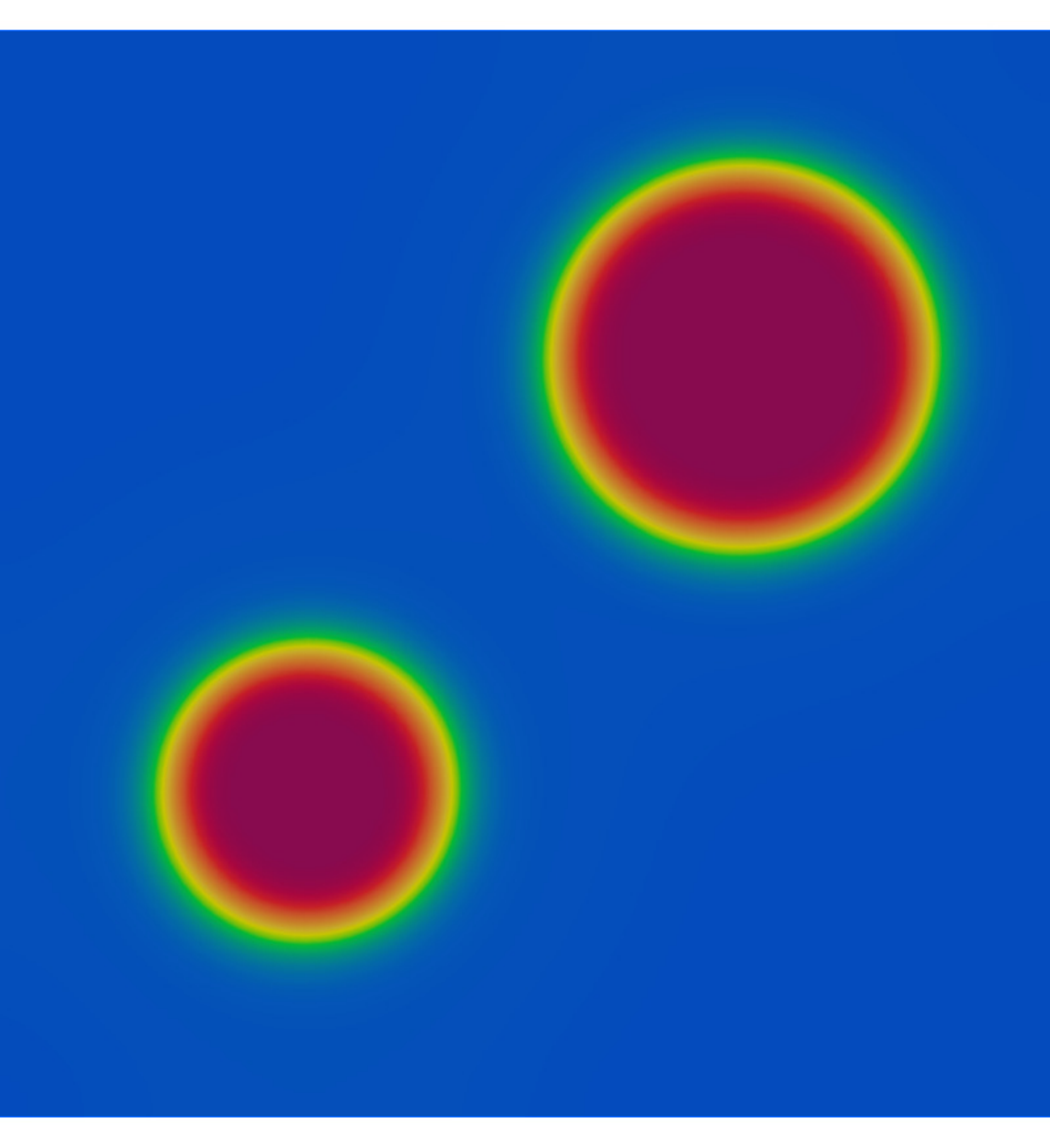}
\includegraphics[scale=0.1125]{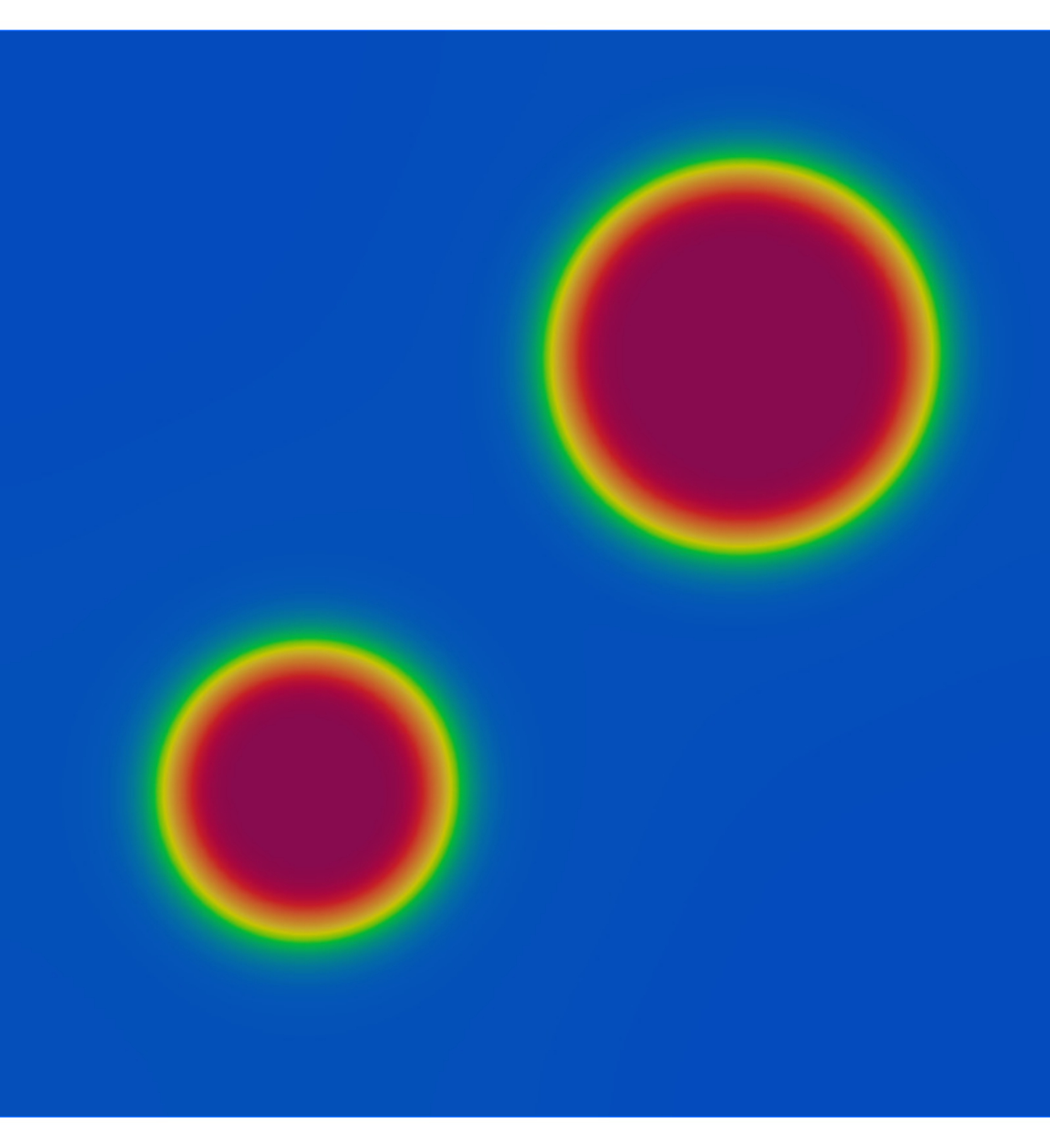}
\includegraphics[scale=0.1125]{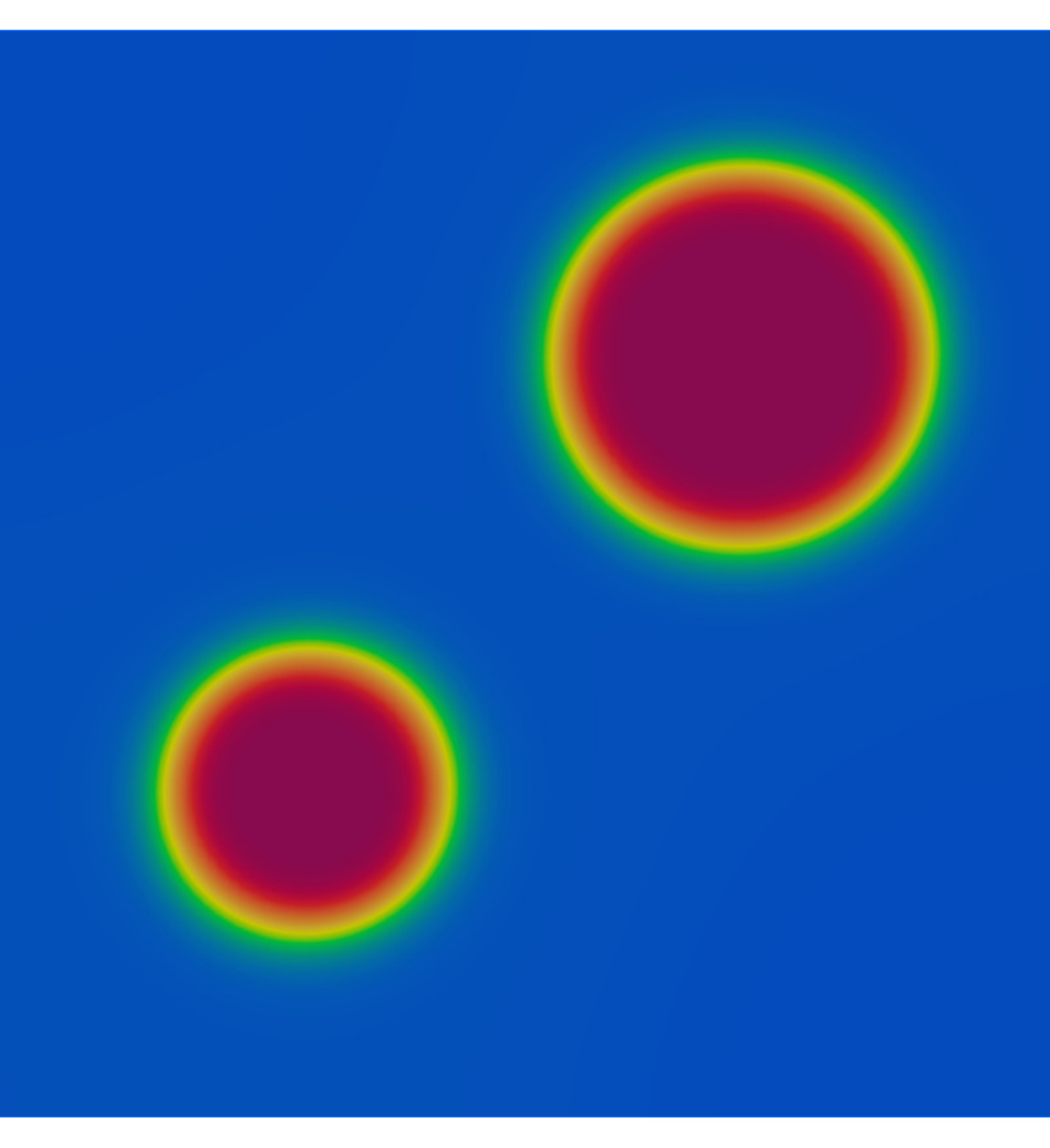}
\\
\includegraphics[scale=0.1125]{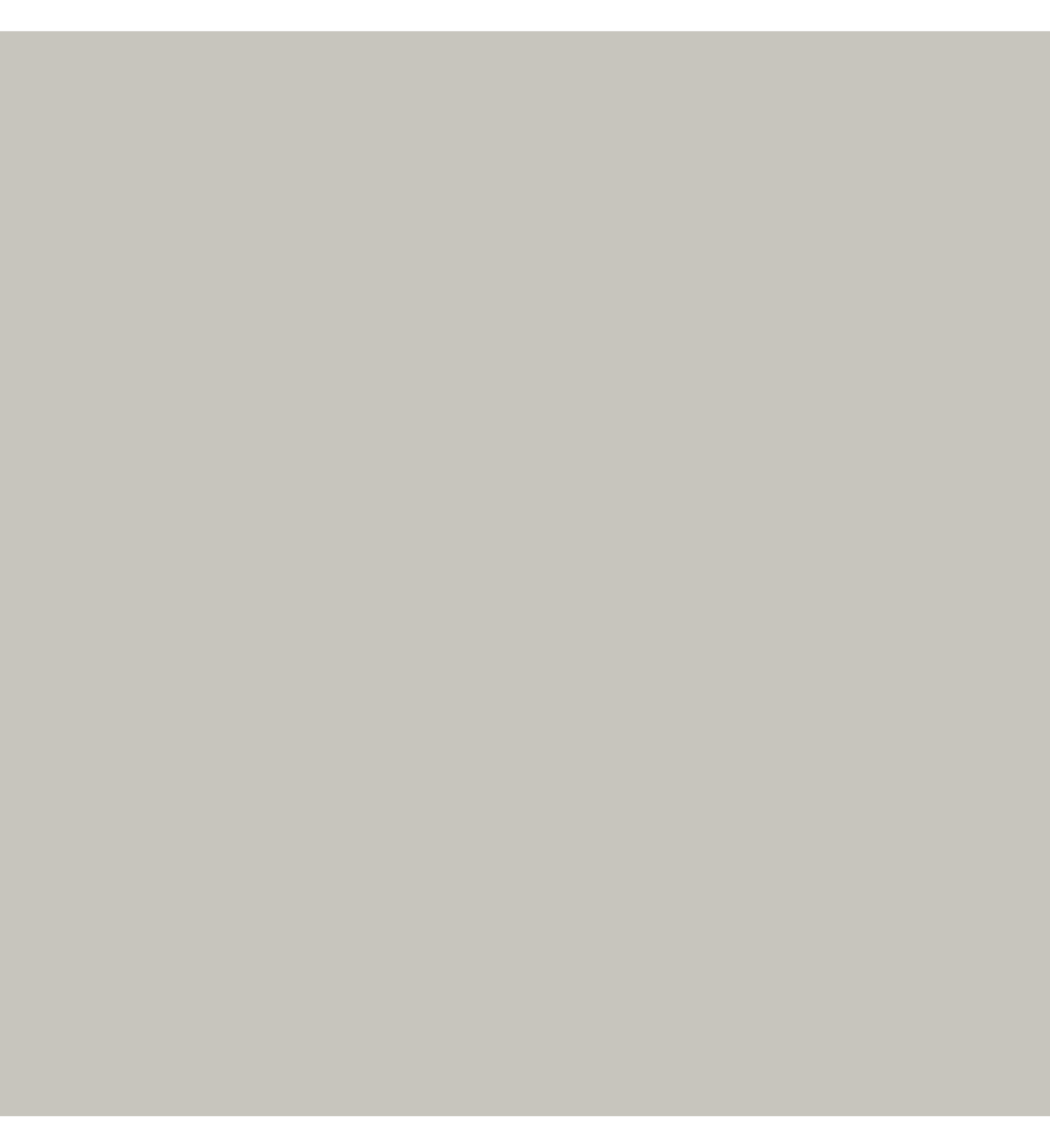}
\includegraphics[scale=0.1125]{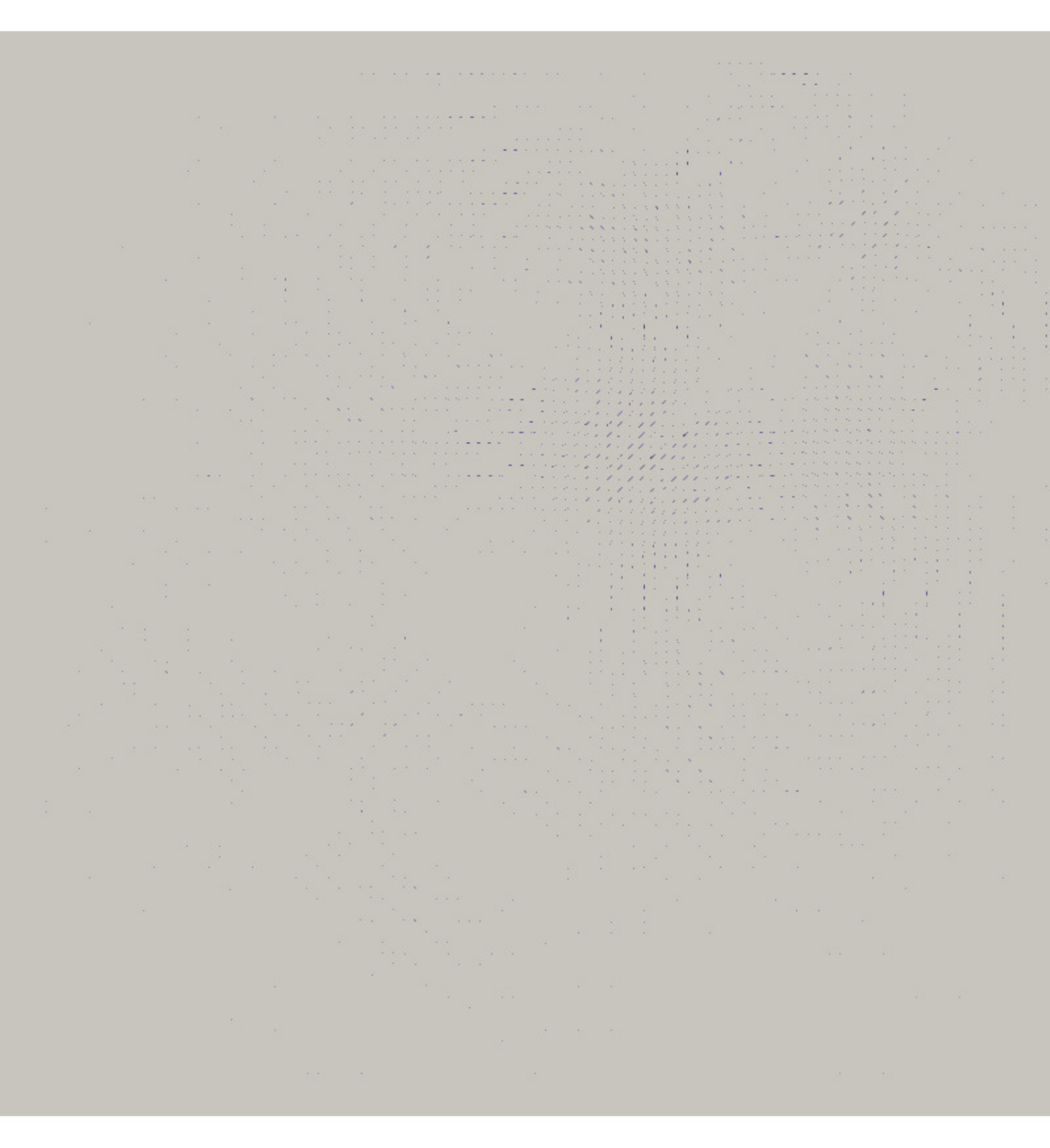}
\includegraphics[scale=0.1125]{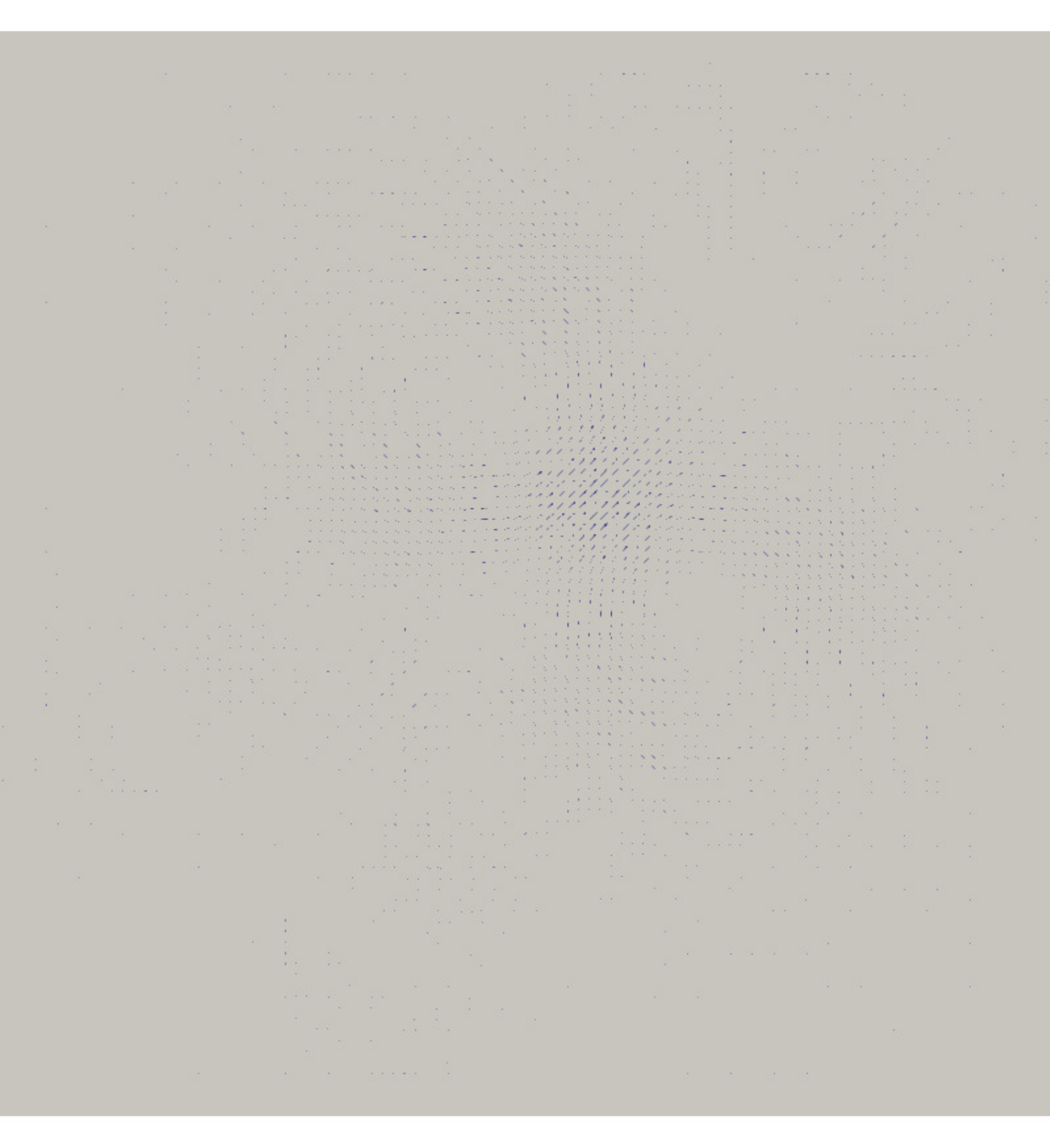}
\includegraphics[scale=0.1125]{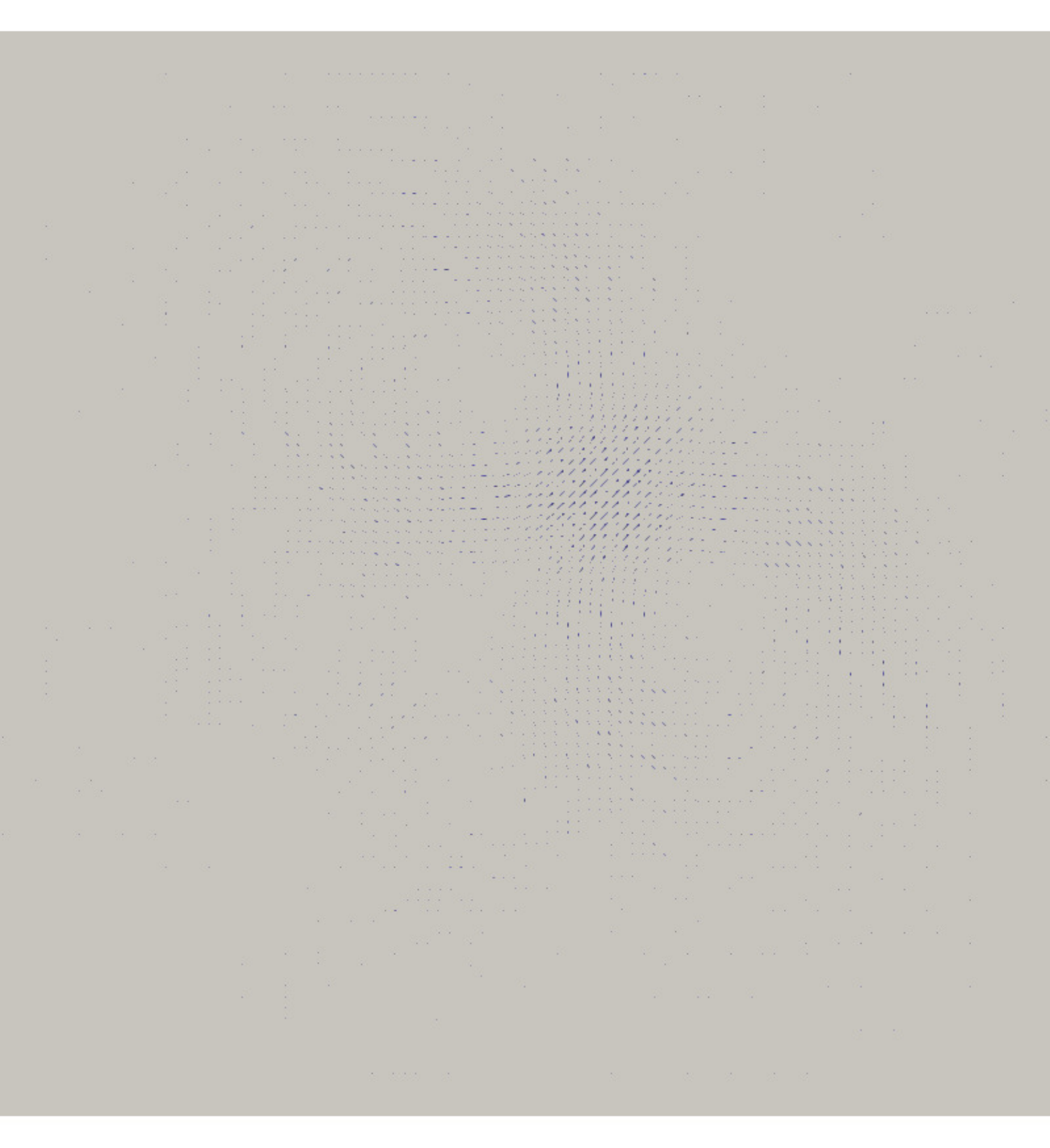}
\includegraphics[scale=0.1125]{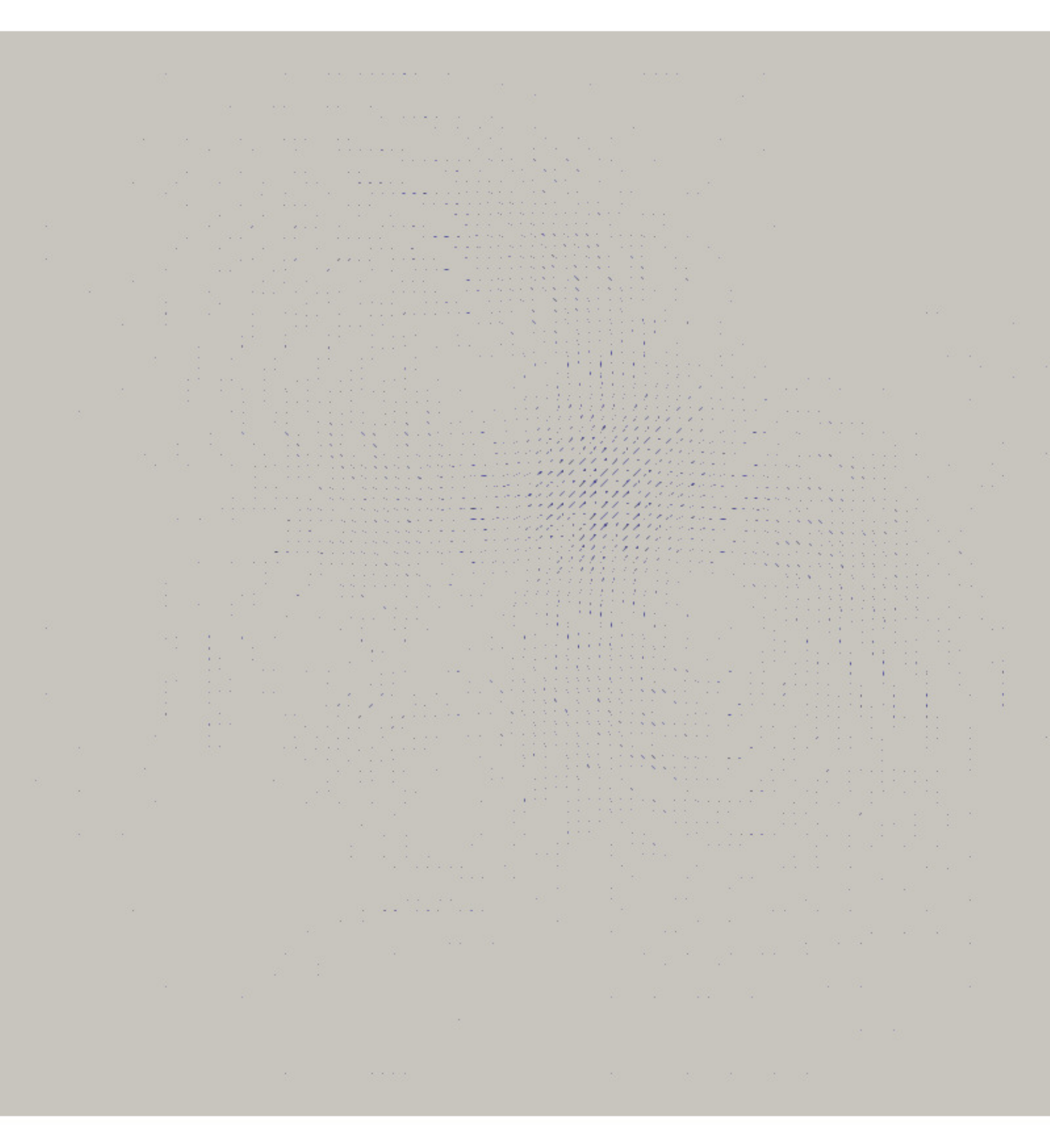}
\end{center}
\caption{Example II. Coarsening effect. Evolution in time of $\phi$ and $\u$ for $J_\varepsilon$-scheme at times  at times $t=0, 1, 2.5, 4$ and $5$.}\label{fig:ExIICoarJ}
\end{figure}

\begin{figure}[H]
\begin{center}
\includegraphics[scale=0.1125]{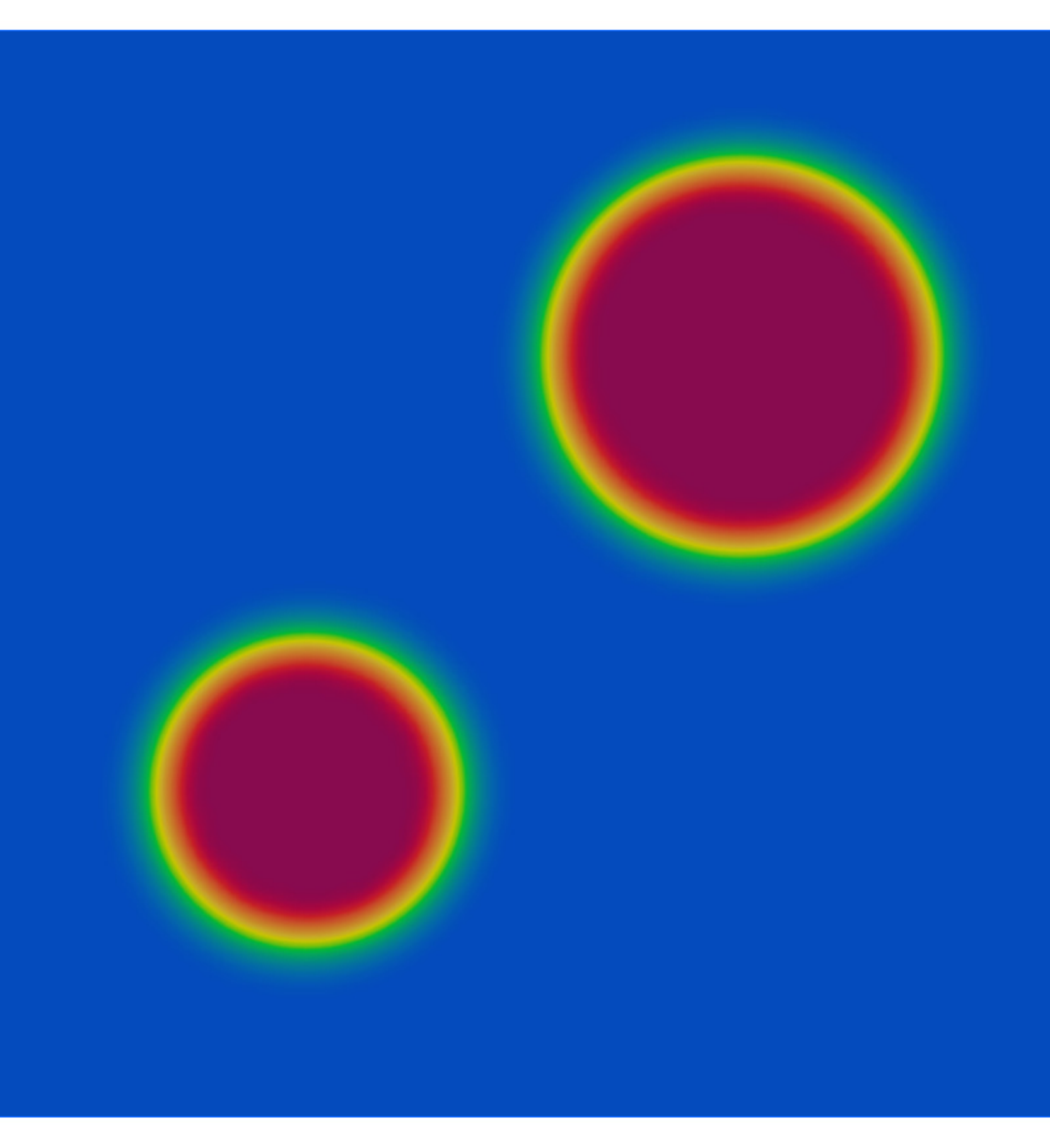}
\includegraphics[scale=0.1125]{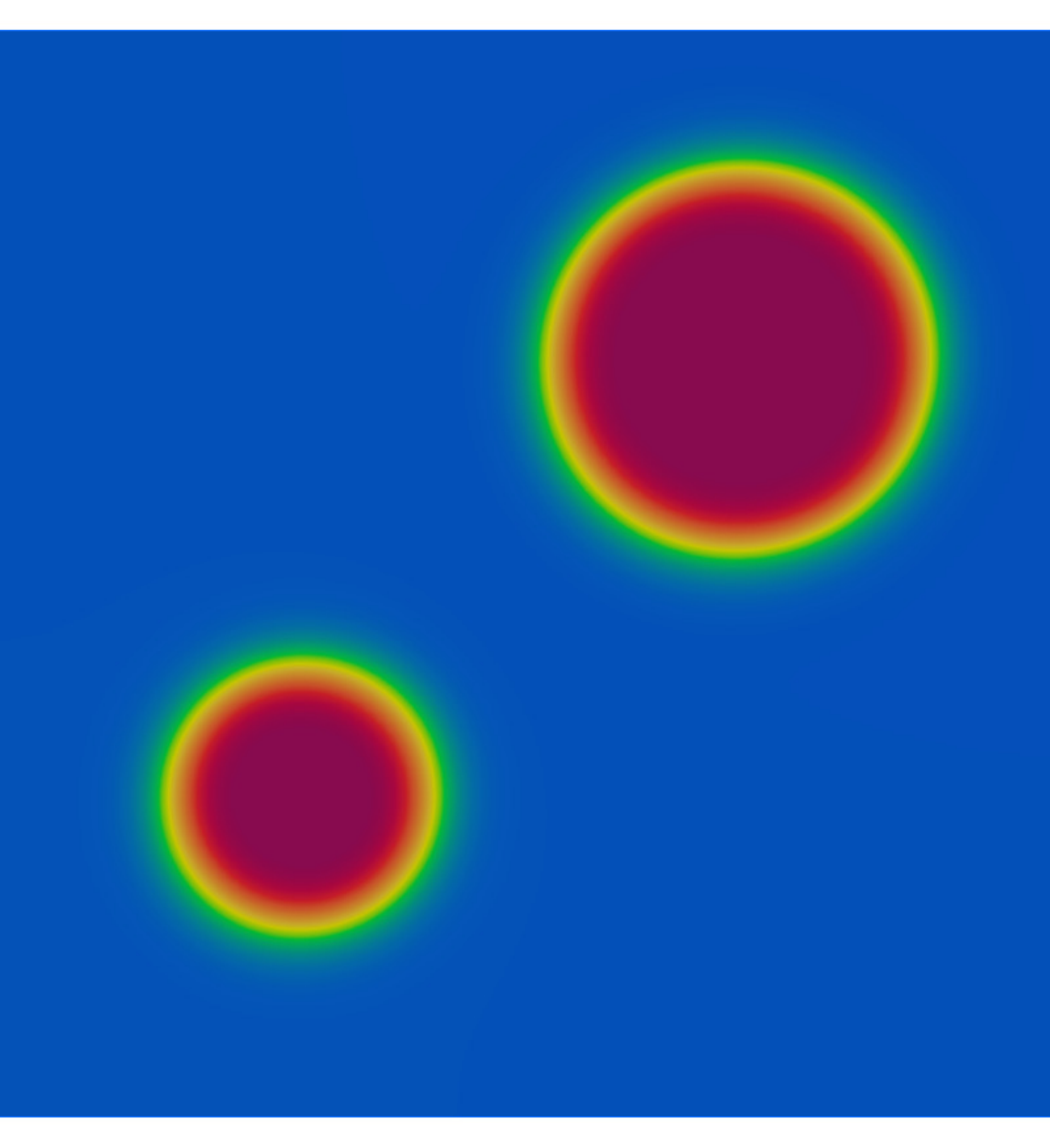}
\includegraphics[scale=0.1125]{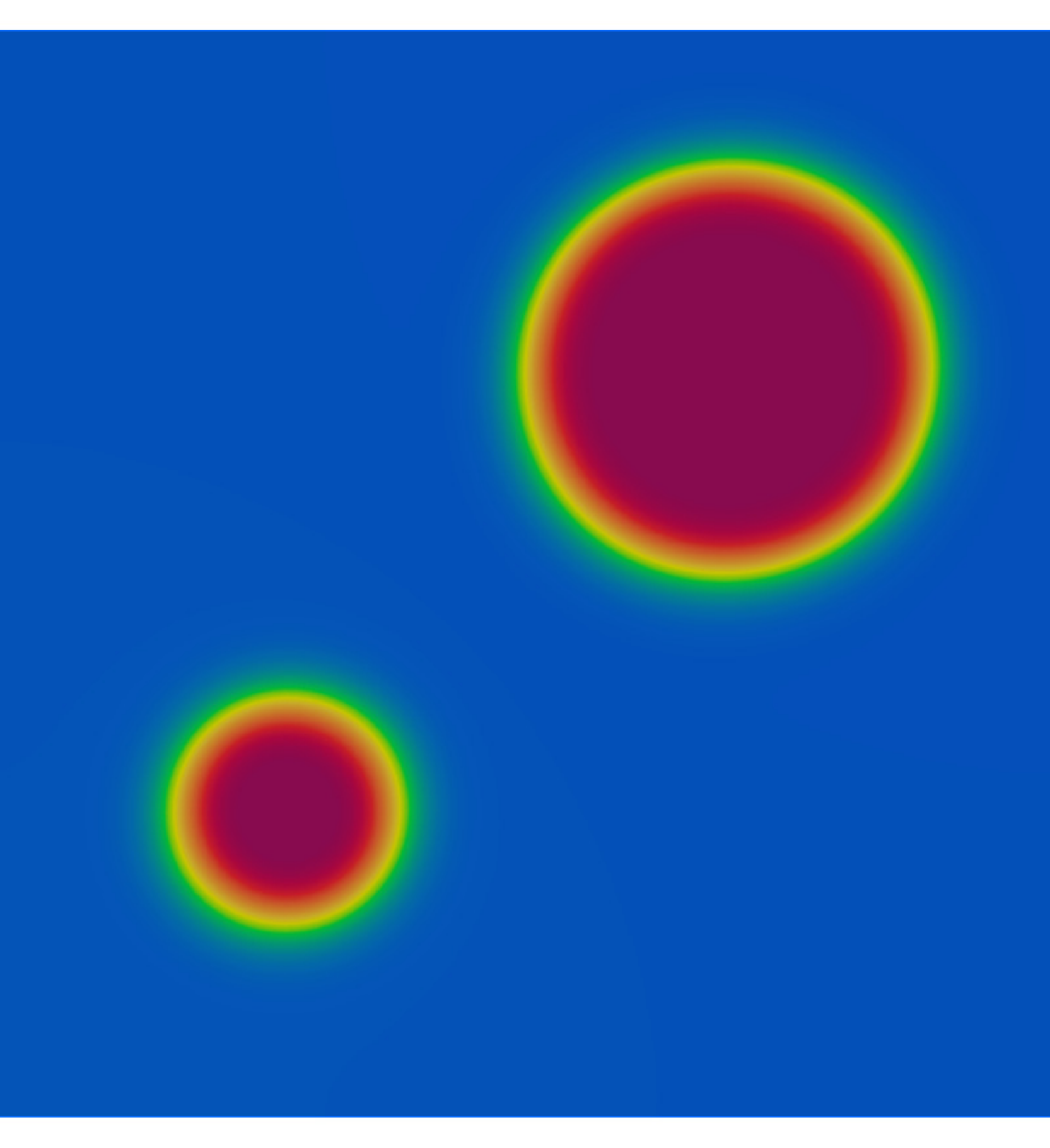}
\includegraphics[scale=0.1125]{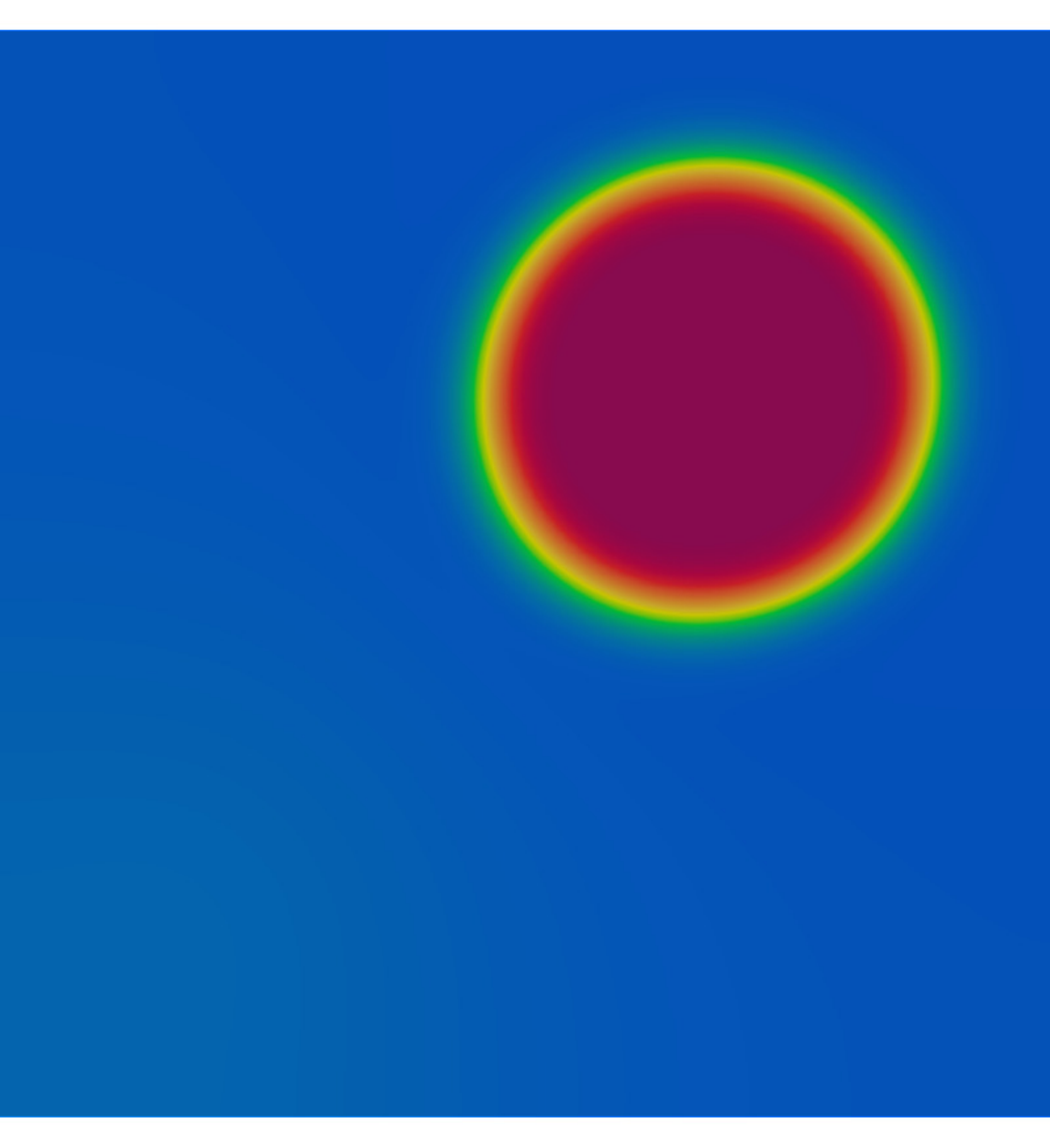}
\includegraphics[scale=0.1125]{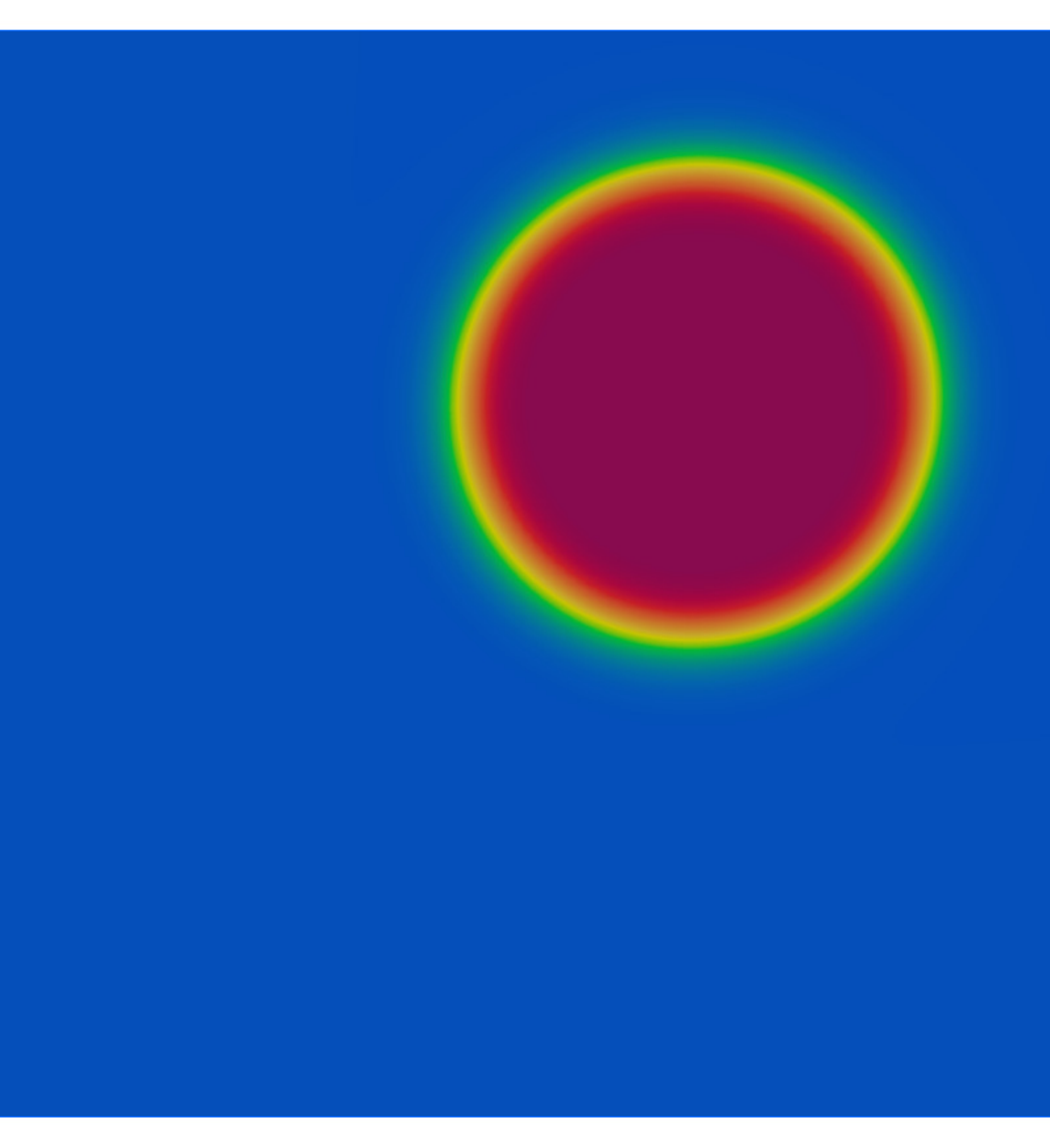}
\\
\includegraphics[scale=0.1125]{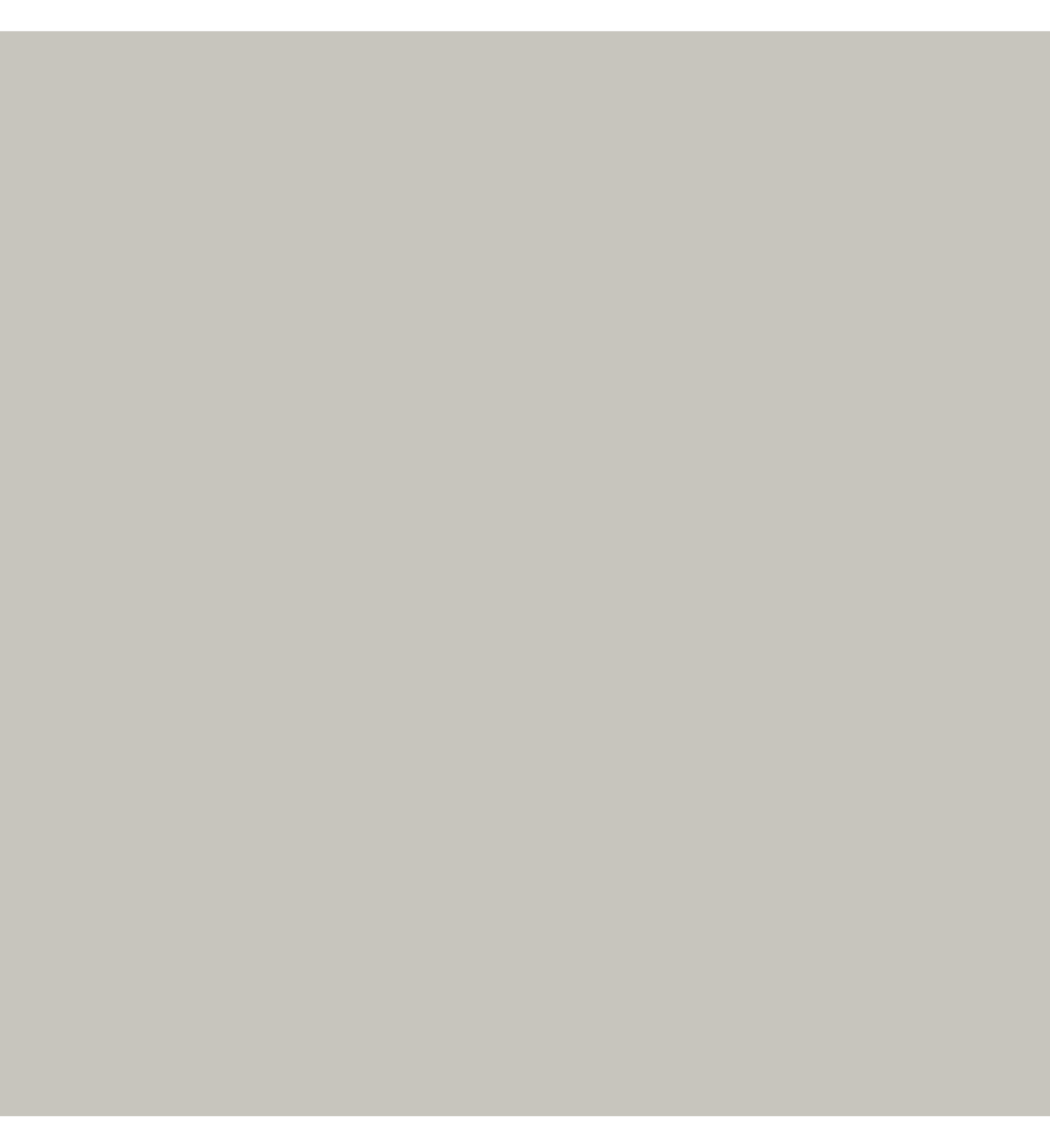}
\includegraphics[scale=0.1125]{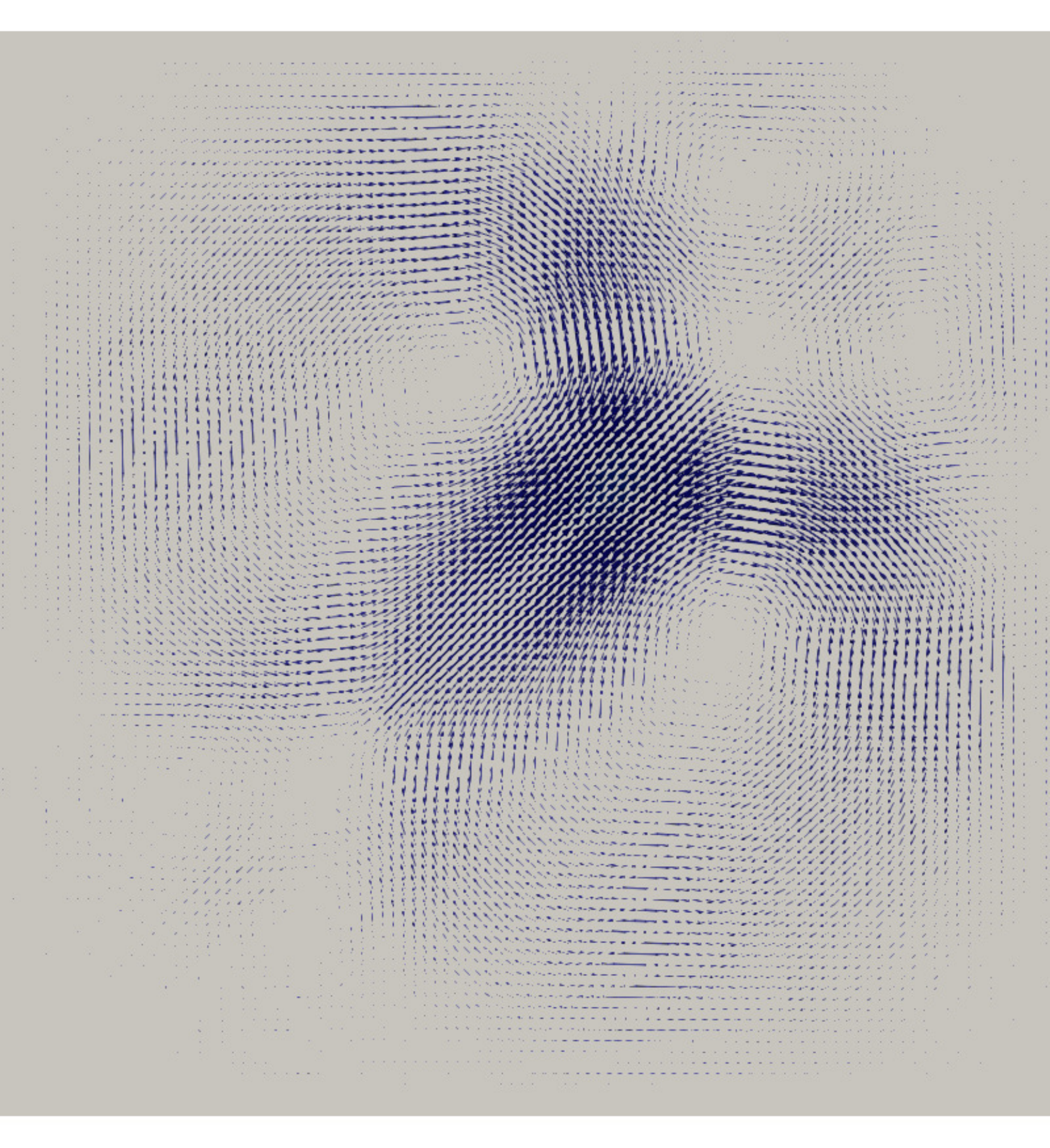}
\includegraphics[scale=0.1125]{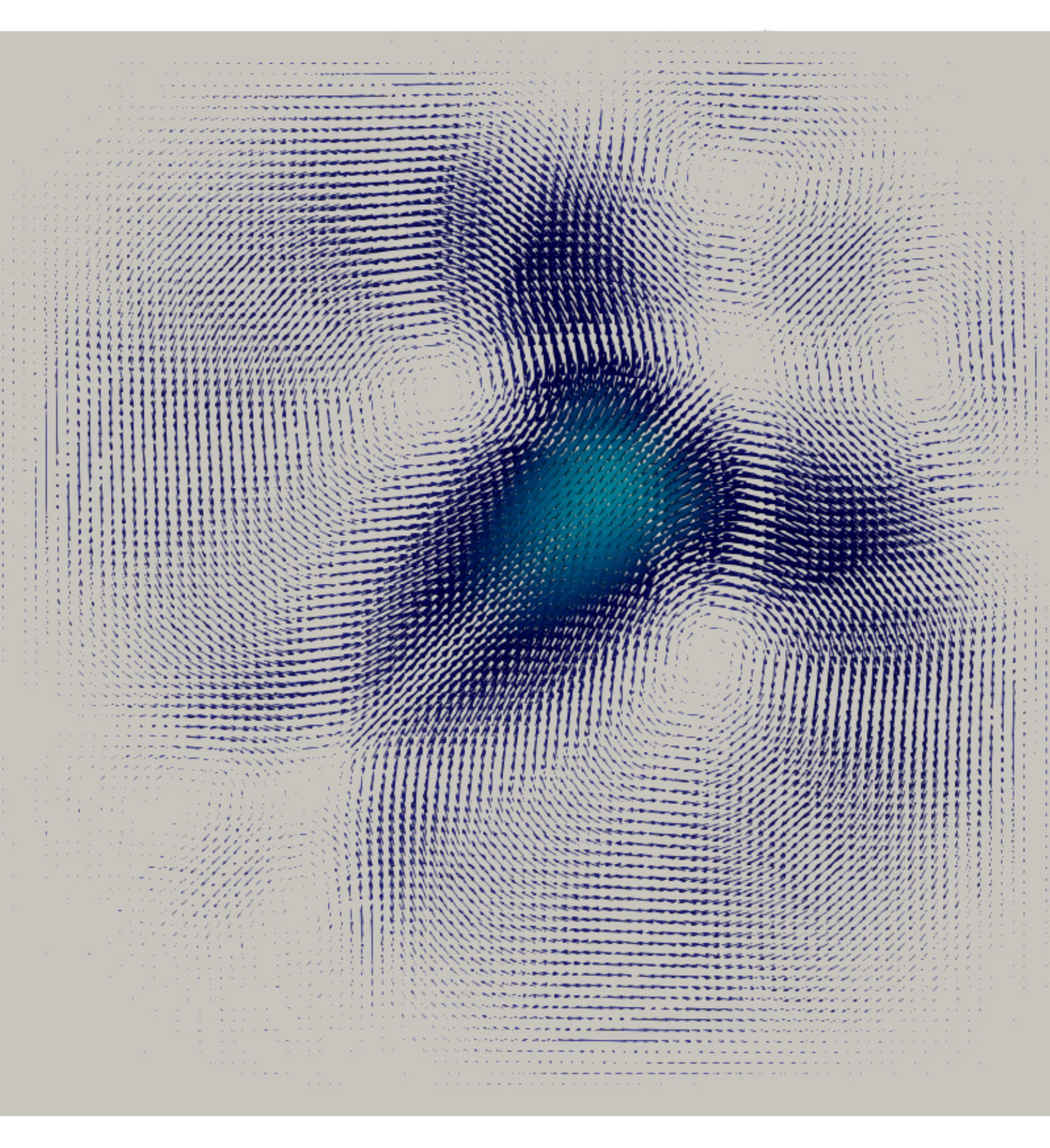}
\includegraphics[scale=0.1125]{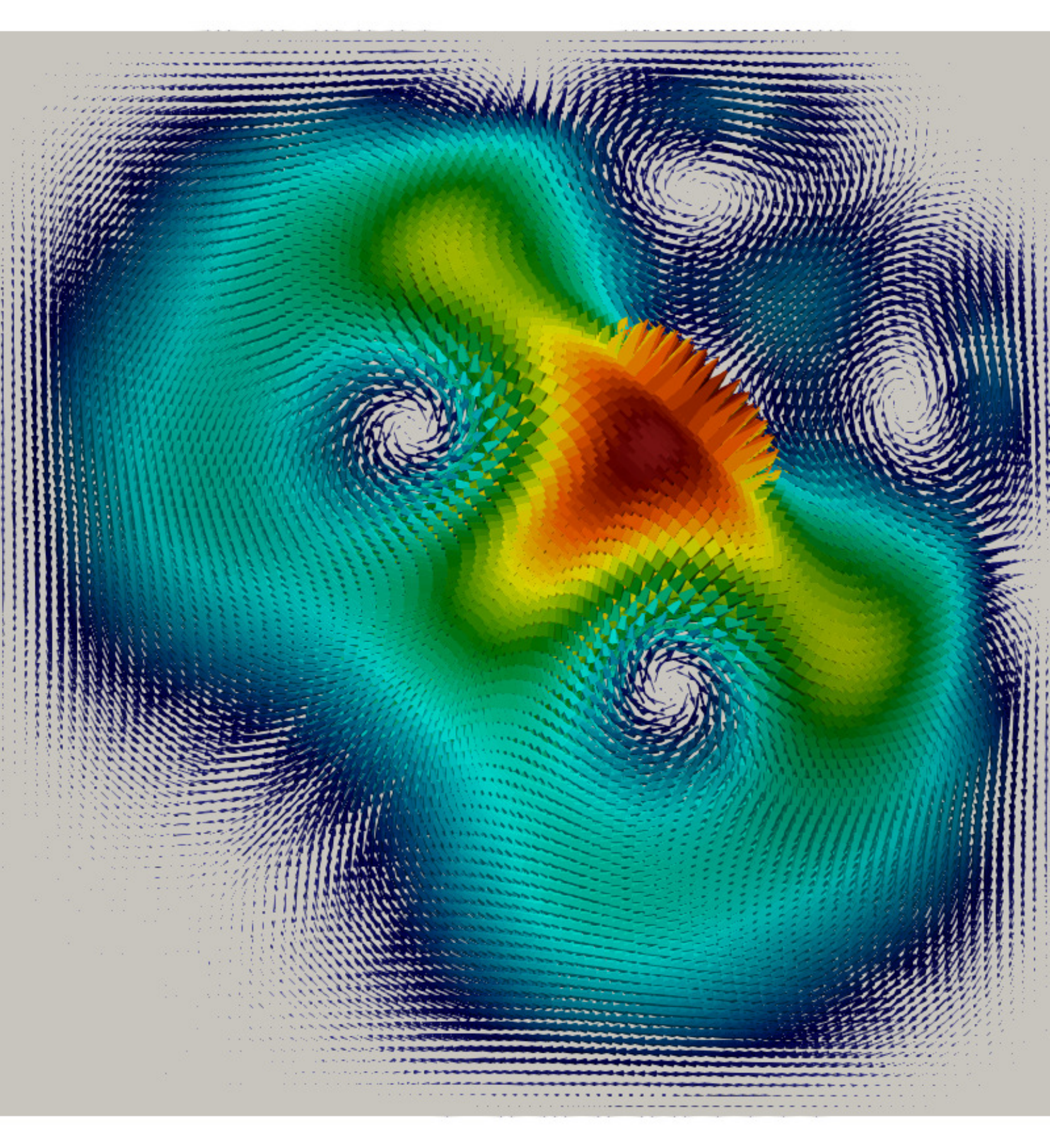}
\includegraphics[scale=0.1125]{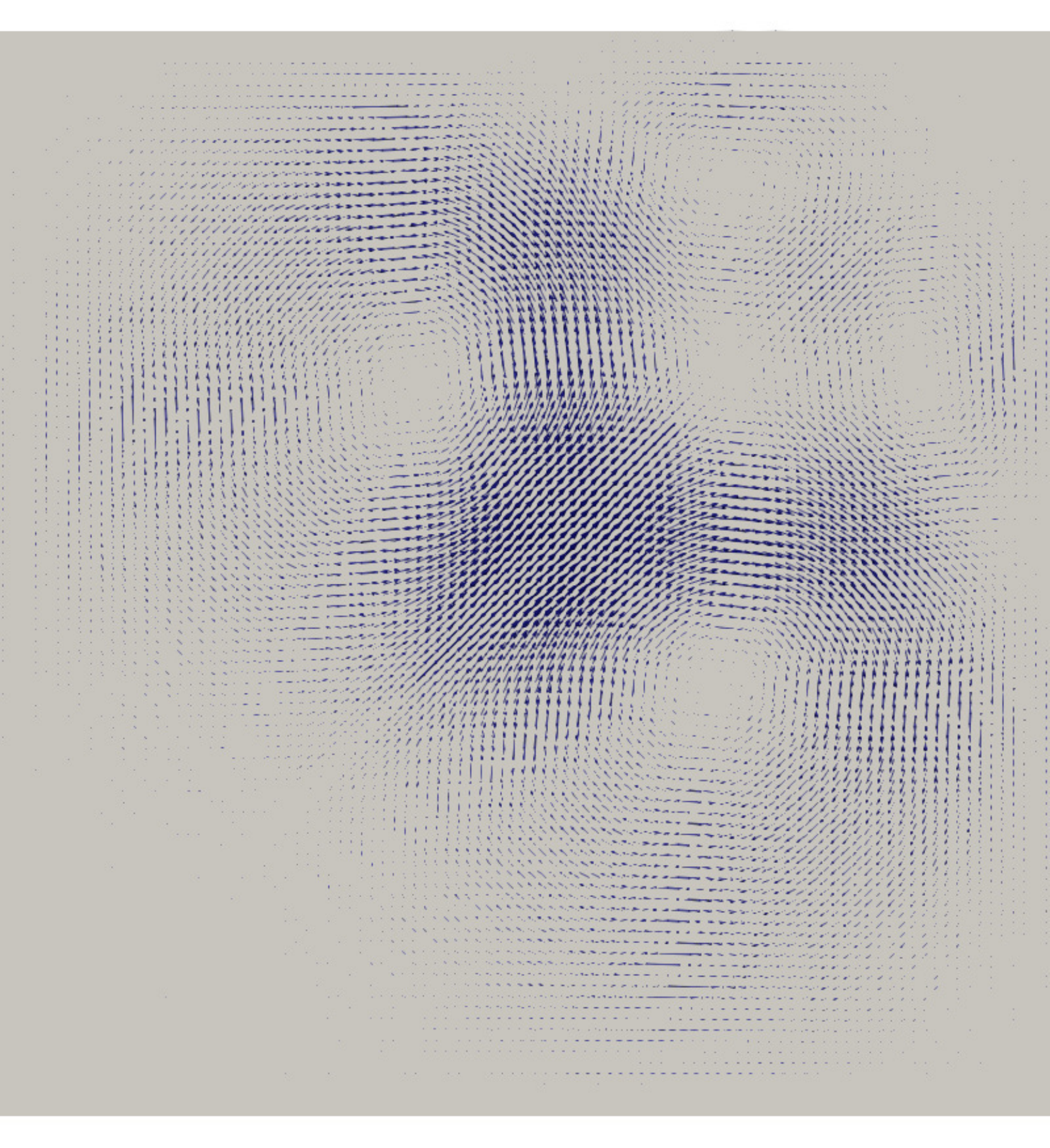}
\end{center}
\caption{Example II. Coarsening effects. Evolution in time of $\phi$ and $\u$ for CM-scheme at times  at times $t=0, 1, 2.5, 4$ and $5$.}\label{fig:ExIICoarCM}
\end{figure}

In Figure~\ref{fig:Ex2_Coarsening_energies} we present the evolution in time of the different energies and the evolution of the volume.  As expected, all the schemes achieve the dissipation of the free energy as well as the conservation of volume. The interesting part is that the behaviors are completely different between the constant and non-constant mobilities. In the constant mobility case, the elimination of the smaller ball produces a steep decrease in the mixing energy, which induces movement in the fluid part and therefore it produces the rising of the kinetic energy. Once the smaller ball is gone, the kinetic energy also decreases until the system reaches the equilibrium configuration. In the non-constant mobility case, there is initially a decrease in the mixing energy which is produced by the system adjusting the given interface in the initial condition to the desired one. This produces a small movement in the fluid part which leads to a very subtle increase of the kinetic energy. After that, the energy decreases until reaching the equilibrium configuration.

Additionally, in Figure~\ref{fig:Ex2_Coarsening_bounds} we present the evolution in time of the maximum and minimum of $\phi$ in the domain $\Omega$ as well as the evolution of $\int_\Omega(\phi_-)^2 d\x$ and $\int_\Omega((\phi-1)_+)^2 d\x$. As expected, CM-scheme does not enforce the variable $\phi$ to remain in the interval $[0,1]$ while the $G_\varepsilon$-scheme and $J_\varepsilon$-scheme are very close to achieving it. This is a consequence of the evolution $\int_\Omega(\phi_-)^2 d\x$ and $\int_\Omega((\phi-1)_+)^2 d\x$ which are clearly bounded in terms of $\varepsilon$. Furthermore, the obtained bound for $G_\varepsilon$-scheme is slightly better than the one for $J_\varepsilon$-scheme (as expected from Remarks~\ref{rem:boundG} and \ref{rem:boundJ}).

\begin{figure}[H]
\begin{center}
\includegraphics[scale=0.1125]{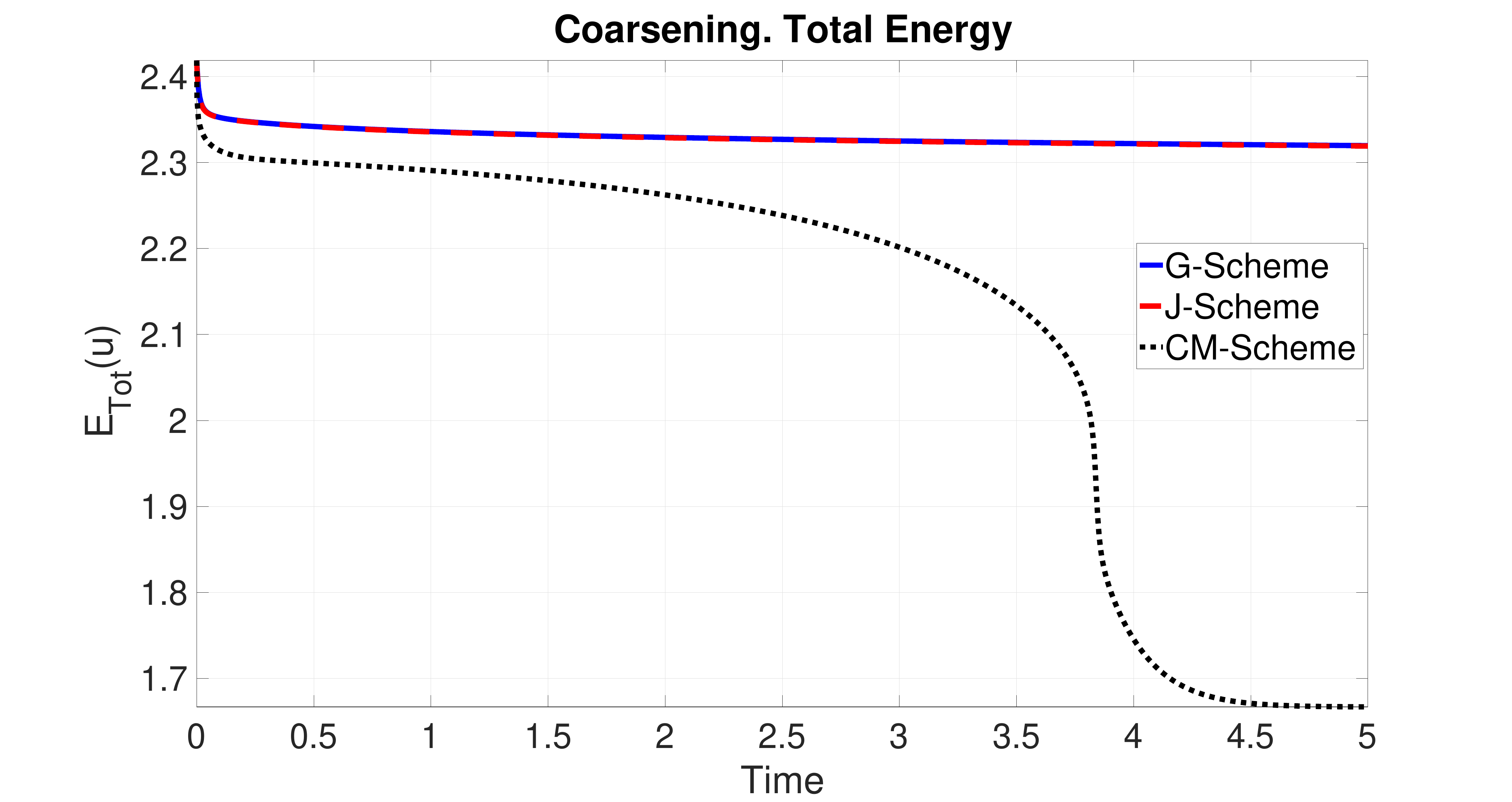}
\includegraphics[scale=0.1125]{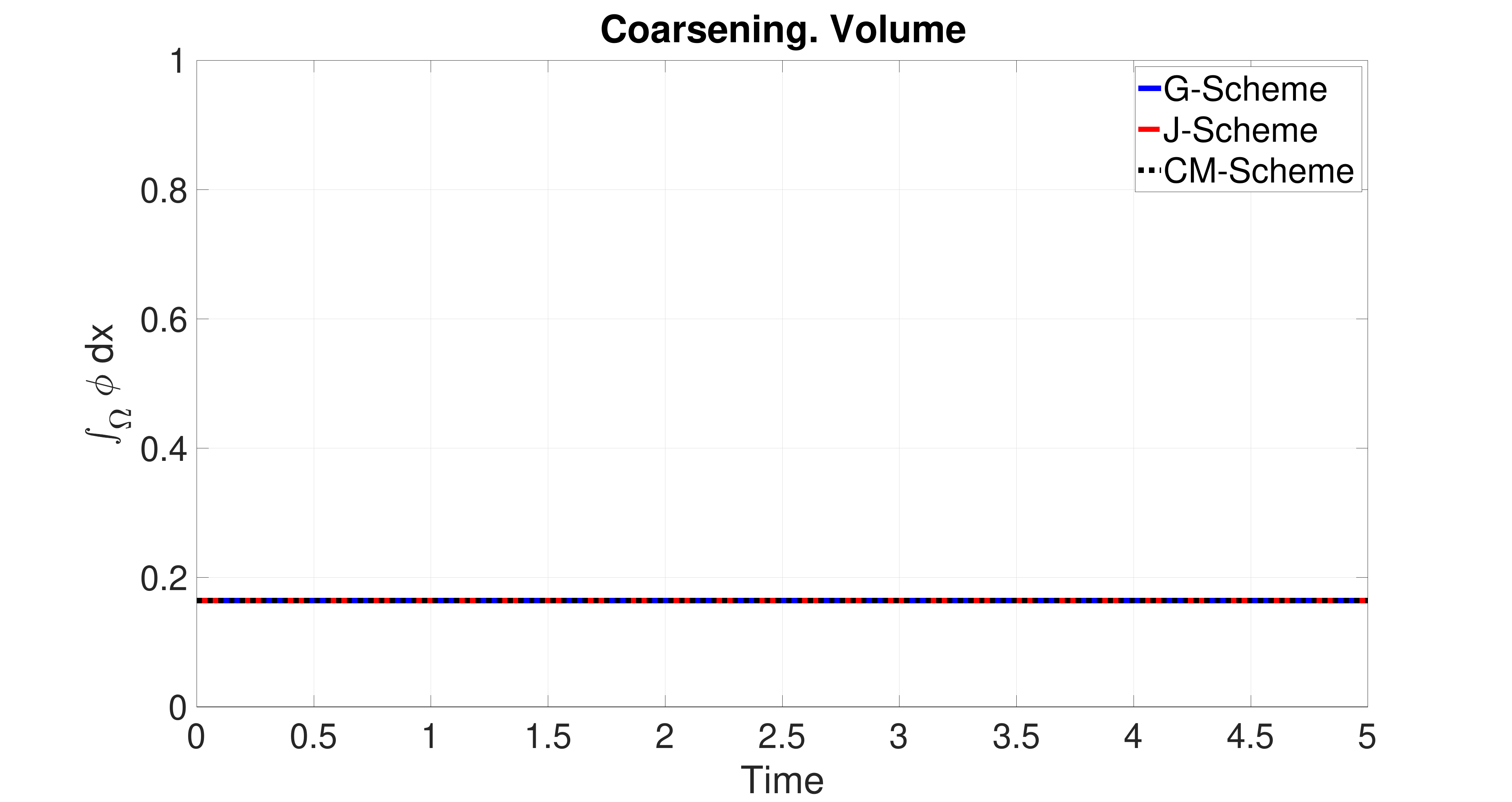}
\\
\includegraphics[scale=0.1125]{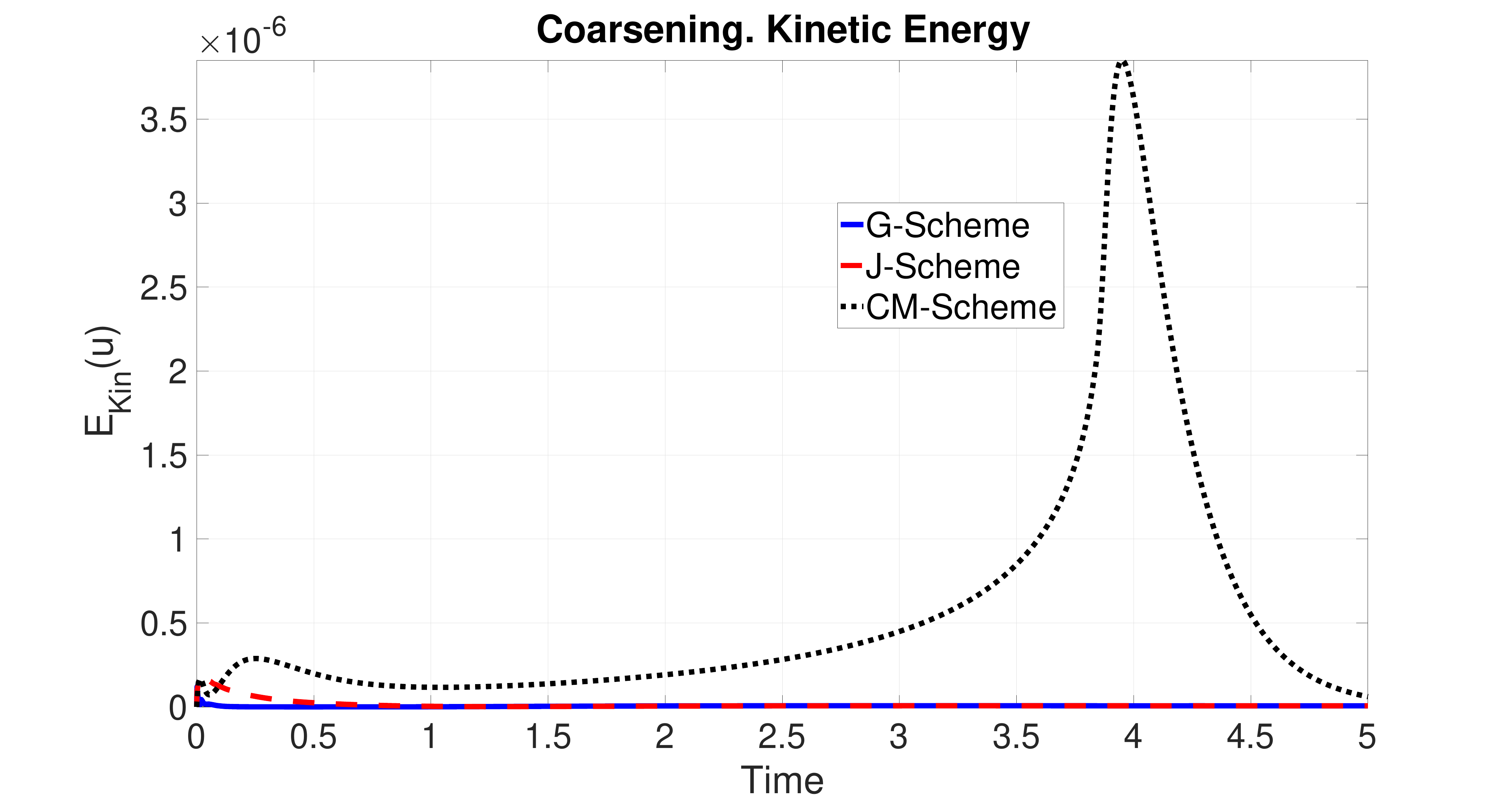}
\includegraphics[scale=0.1125]{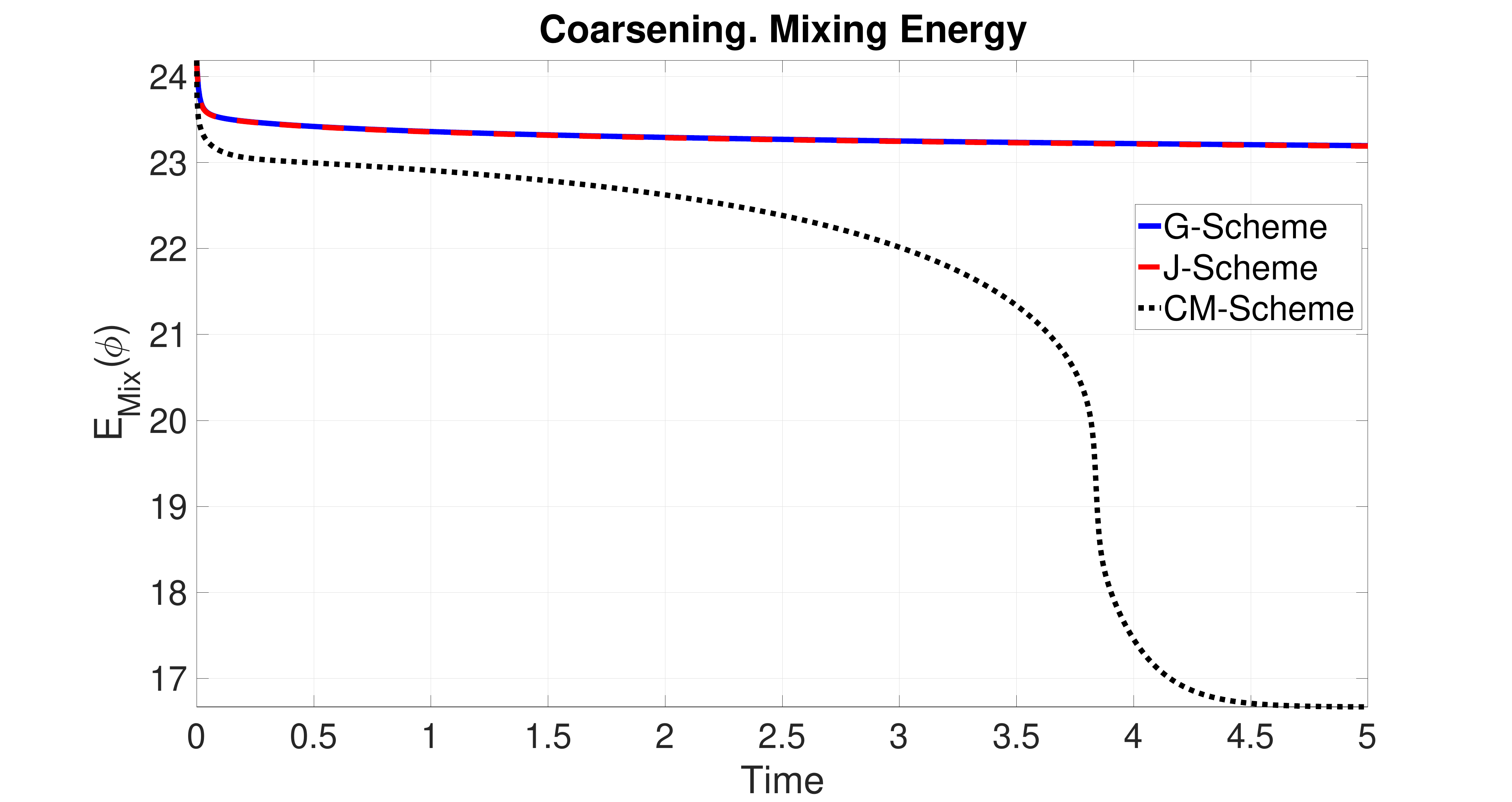}
\end{center}
\caption{Example II. Coarsening effect. Evolution of the energies and the volume of the system}\label{fig:Ex2_Coarsening_energies}
\end{figure}

\begin{figure}[H]
\begin{center}
\includegraphics[scale=0.1125]{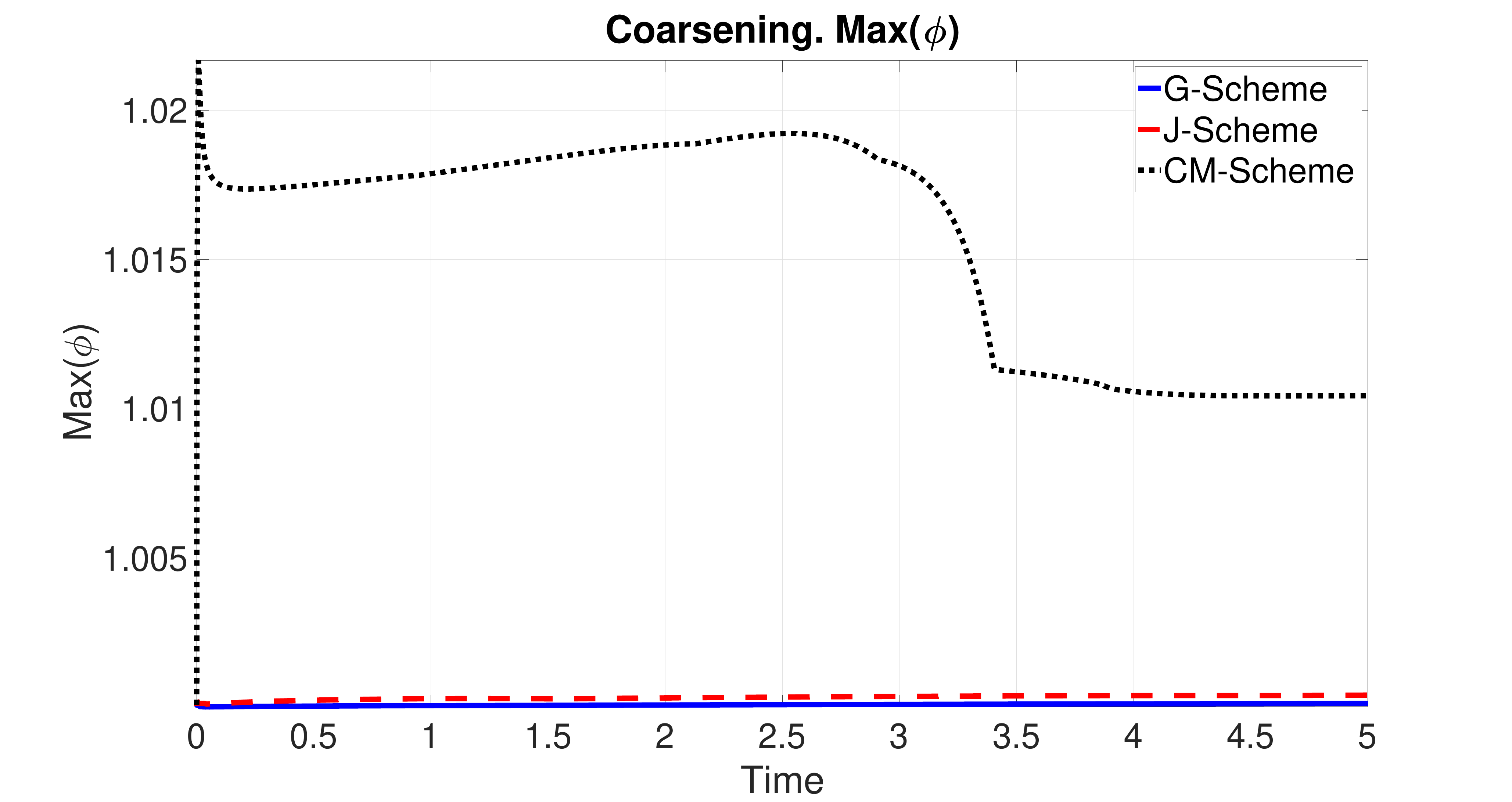}
\includegraphics[scale=0.1125]{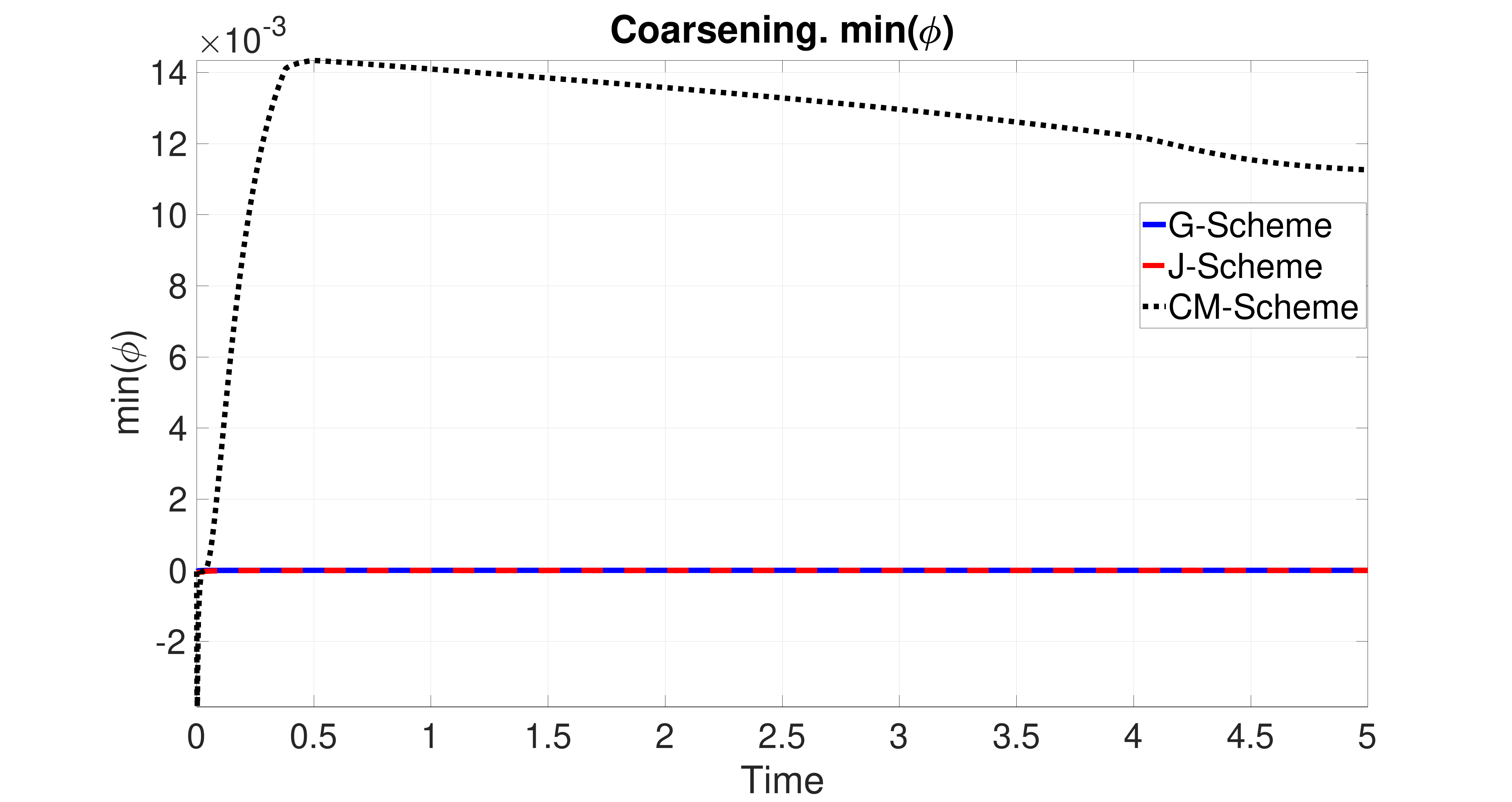}
\\
\includegraphics[scale=0.1125]{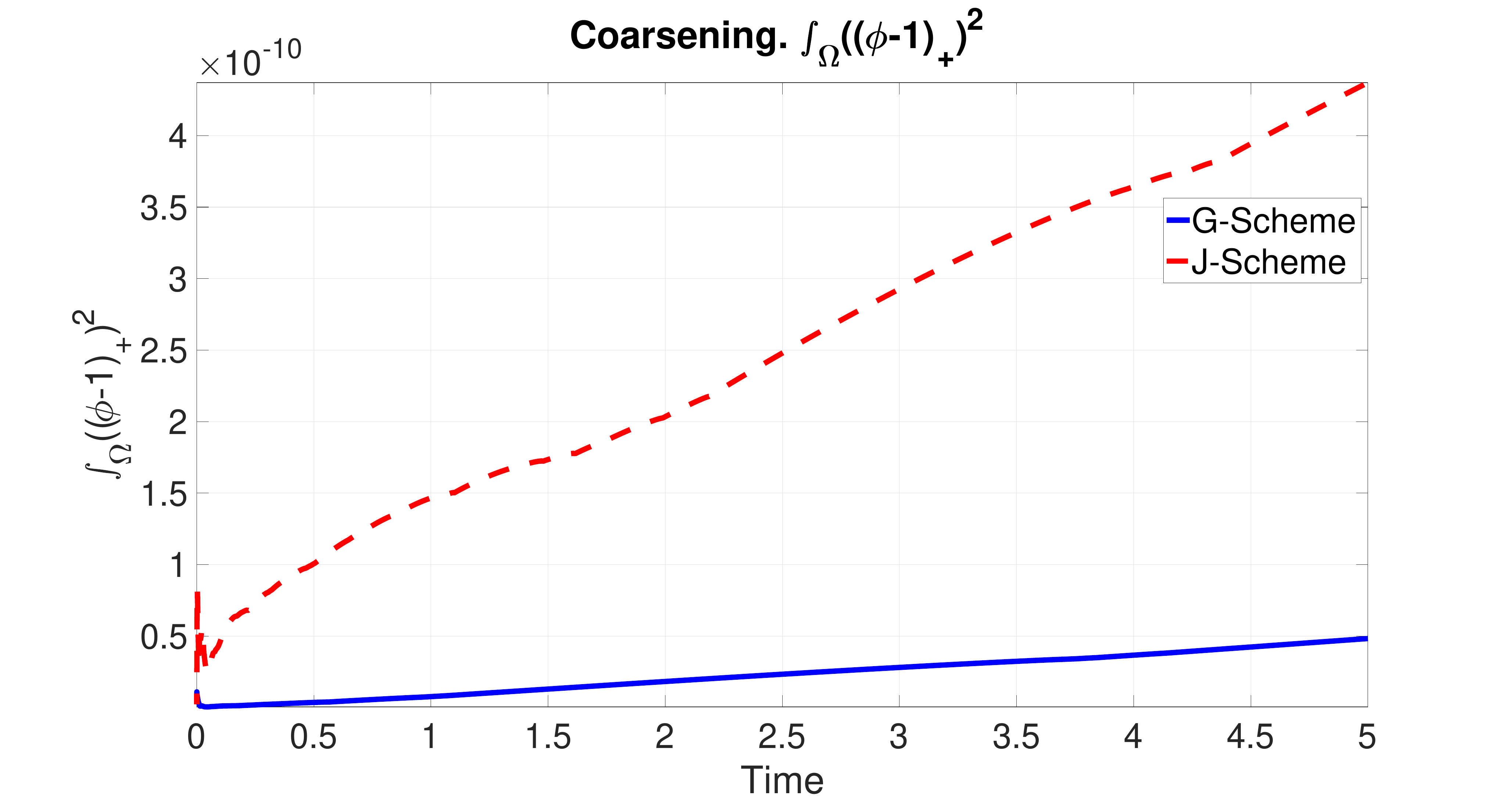}
\includegraphics[scale=0.1125]{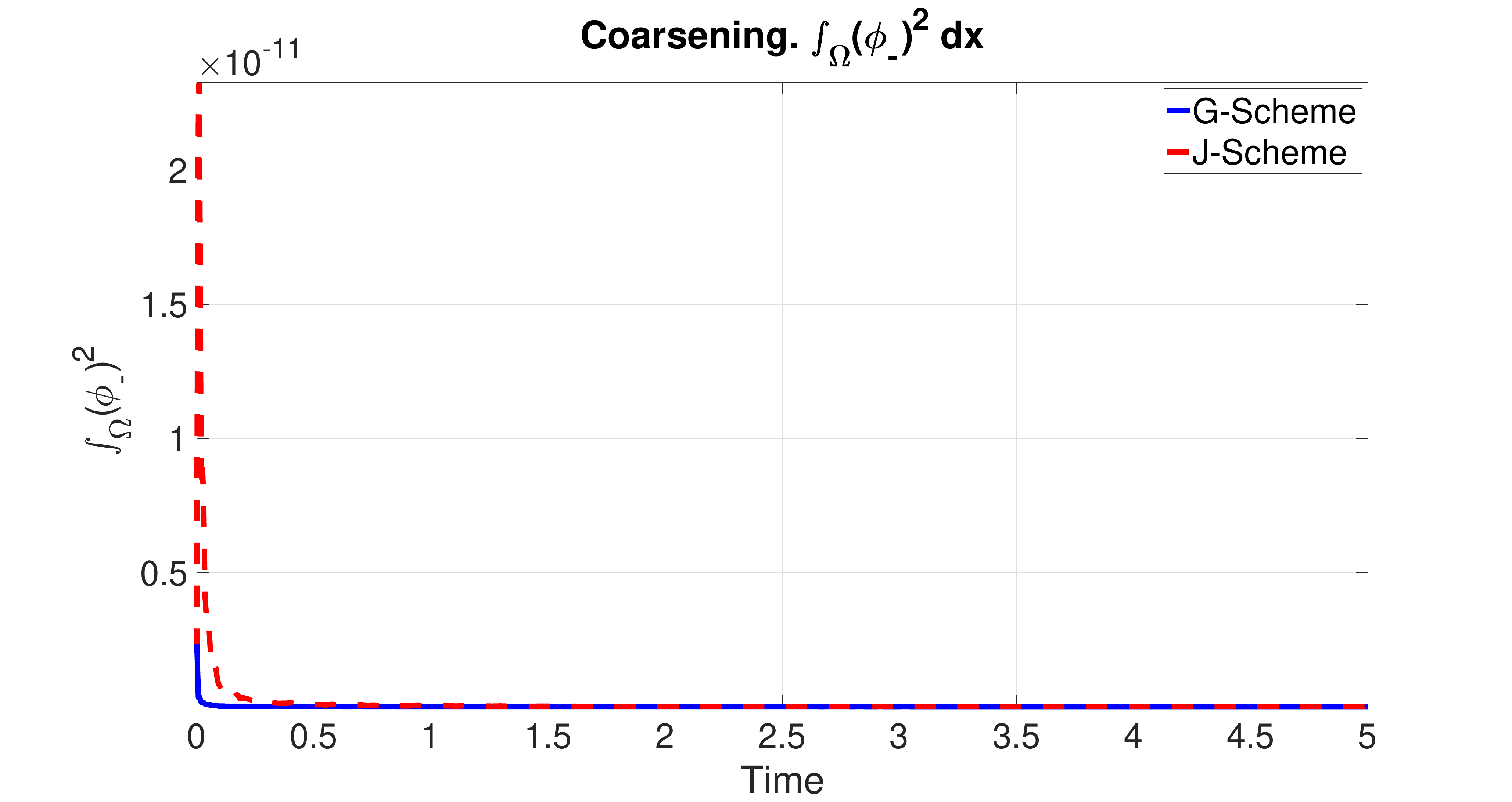}
\end{center}
\caption{Example II. Coarsening effect.  Evolution of the bounds of $\phi$.}\label{fig:Ex2_Coarsening_bounds}
\end{figure}


\subsection{Example III. Rotating fluids}

In this part of the section we illustrate the validity of the numerical schemes to capture more realistic (and complicated) dynamics. To this end we have studied the dynamics produced by a rotating fluid, in the sense that there is an external force applied to the system through the boundary conditions (note that with the introduction of an external force, no dissipation of the total energy is expected anymore). In particular, 
\beq
\u(\x,0)
=\u(\x,t)\Big|_{\partial\Omega}
=2\pi(\cos(\pi y)\sin(\pi x)^2,-2\cos(\pi x)\sin(\pi x)\sin(\pi y))^T\,.
\eeq

We can observe the complete dynamics of the phase field function $\phi$ in Figures~\ref{fig:ExIIIDynG}, \ref{fig:ExIIIDynJ} and \ref{fig:ExIIIDynCM} ($G_\varepsilon$-scheme, $J_\varepsilon$-scheme and CM-scheme, respectively).
The dynamics of $G_\varepsilon$-scheme and $J_\varepsilon$-scheme coincide, but they differ from the dynamics of the constant mobility case computed using CM-scheme. In all the cases the balls rotate, merge, elongate and separate until the system reaches a configuration that does not change much in time. In any case, the imposition of this external force  result in very non-trivial dynamics for both mobilities (constant and non-constant).

\begin{figure}[H]
\begin{center}
\includegraphics[scale=0.1125]{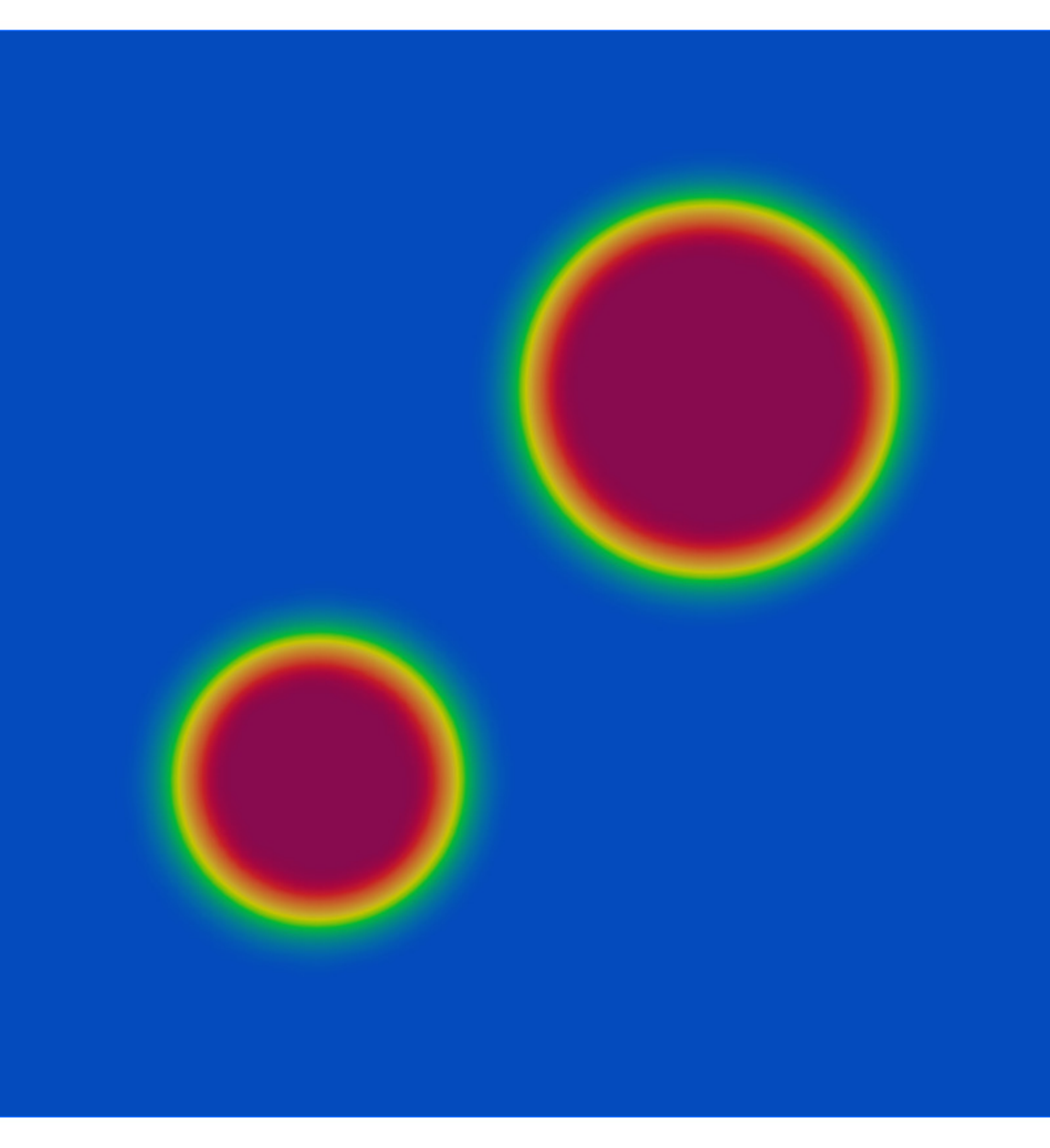}
\includegraphics[scale=0.1125]{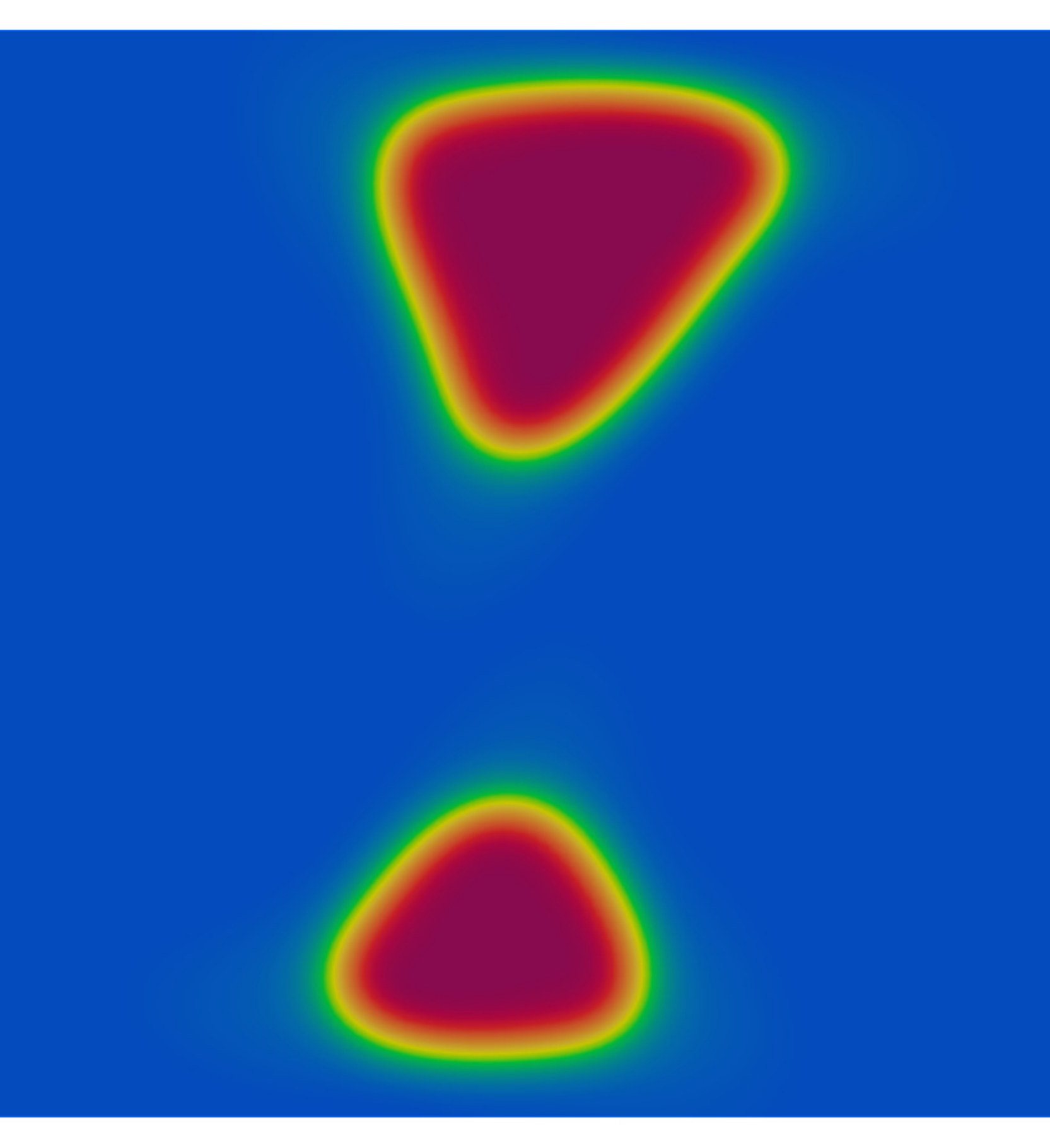}
\includegraphics[scale=0.1125]{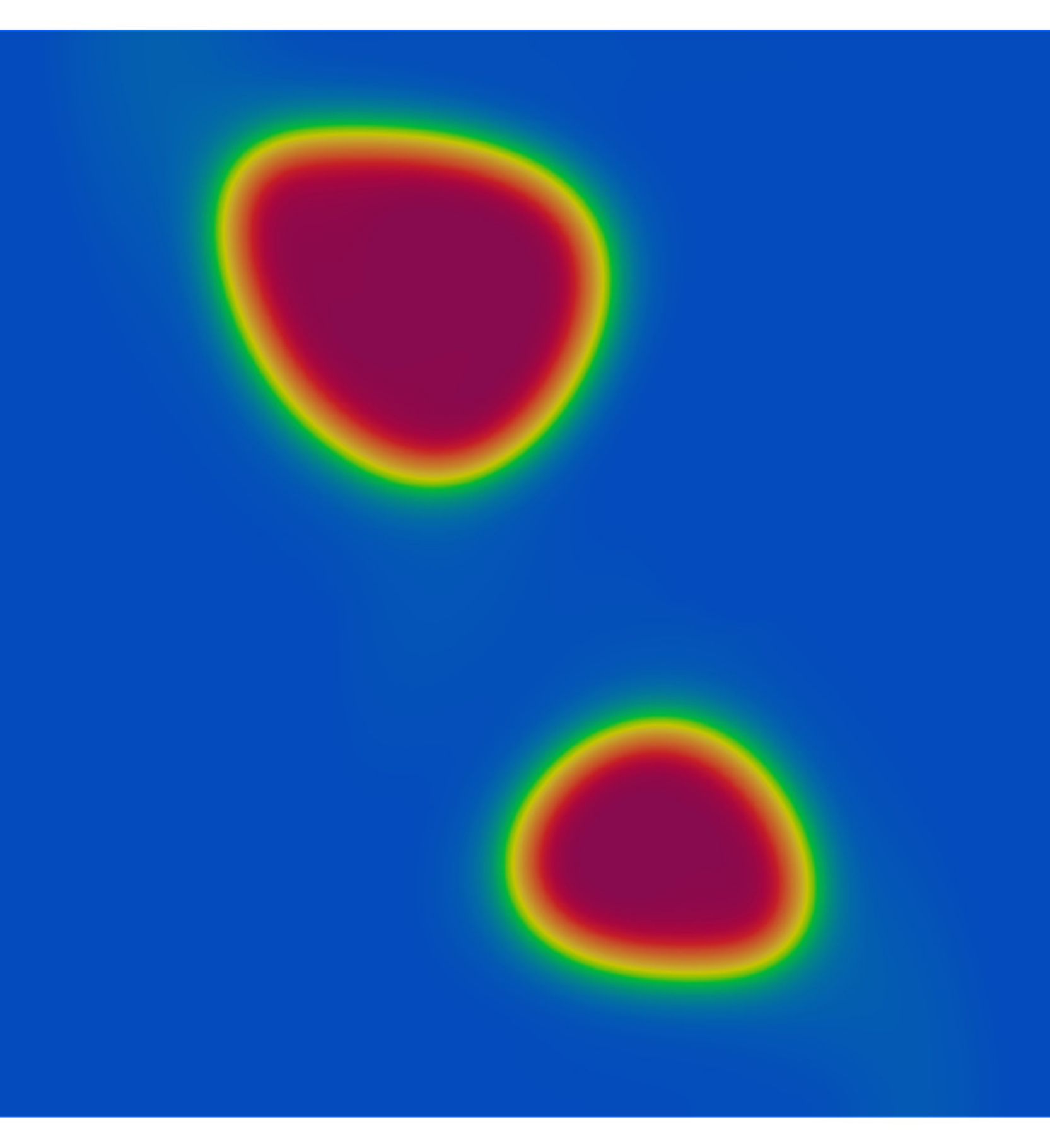}
\includegraphics[scale=0.1125]{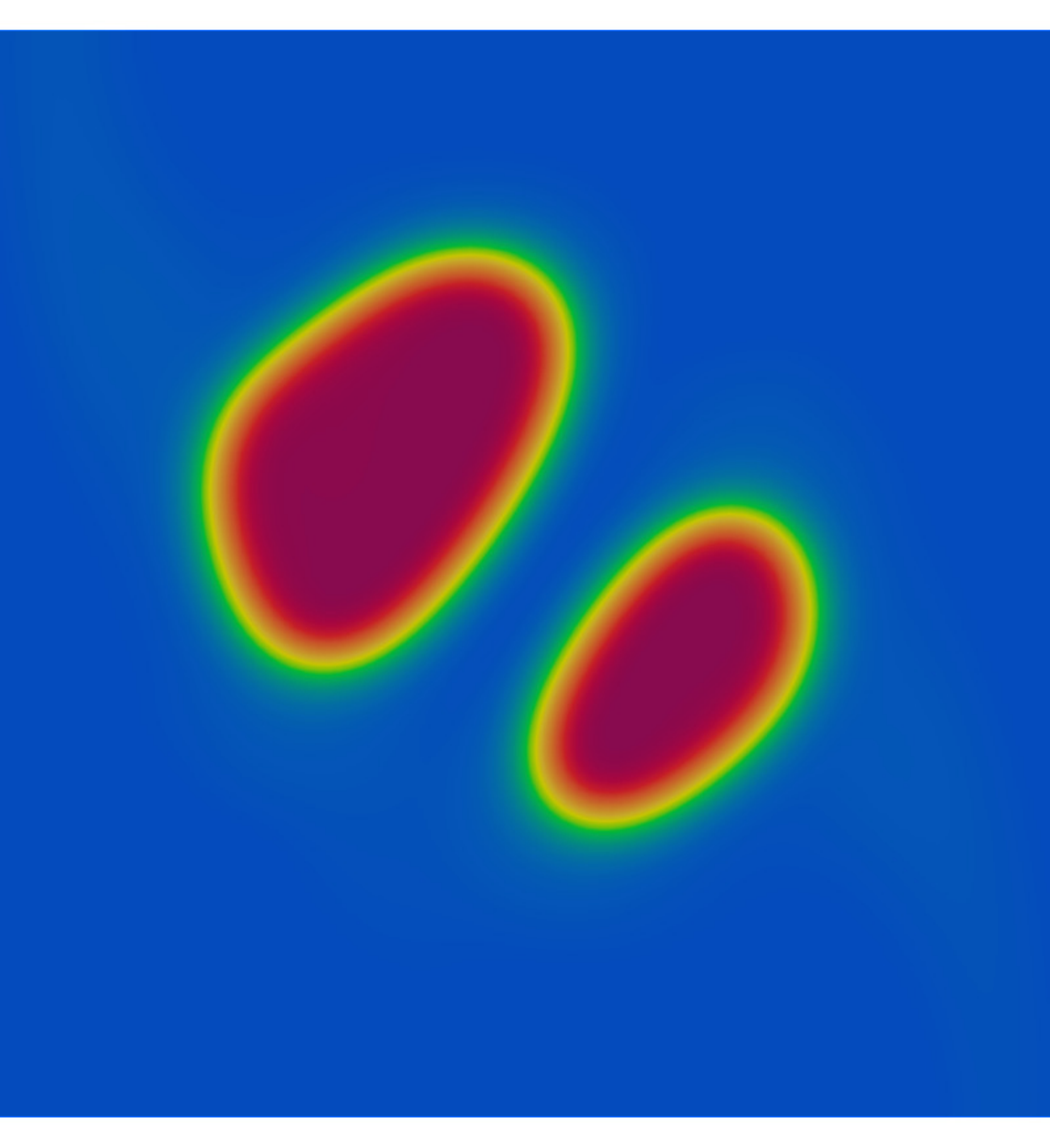}
\includegraphics[scale=0.1125]{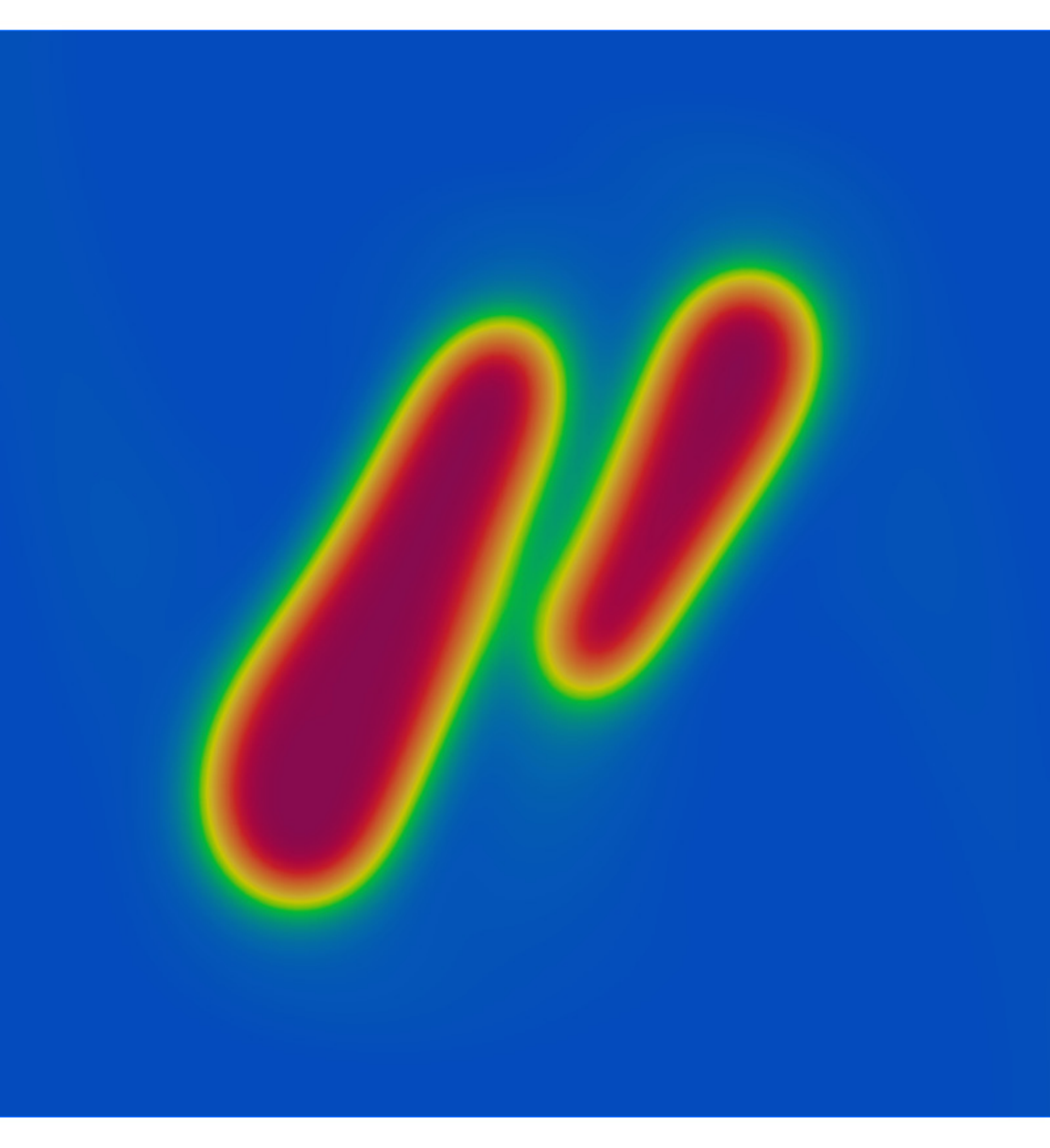}
\\
\includegraphics[scale=0.1125]{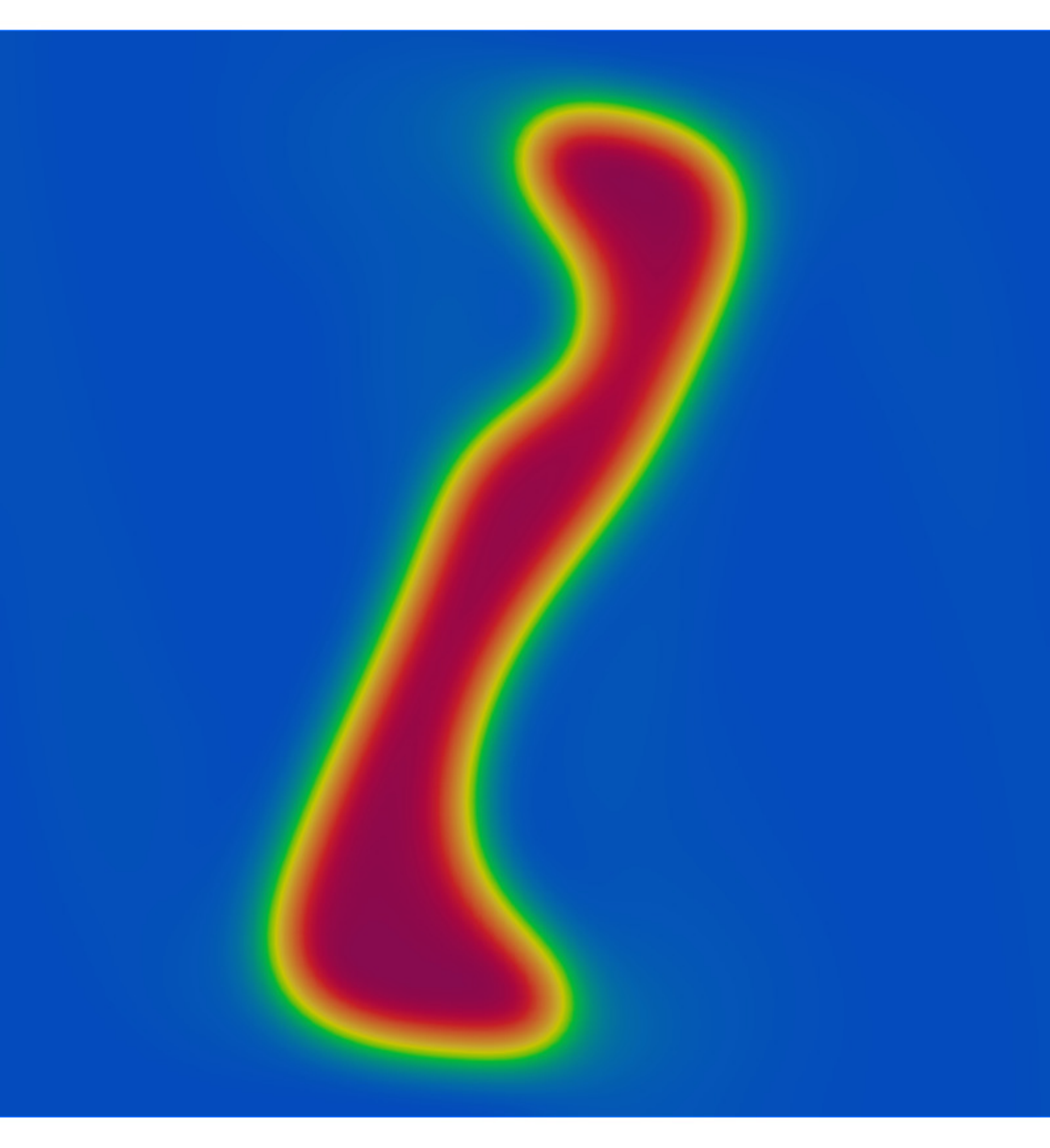}
\includegraphics[scale=0.1125]{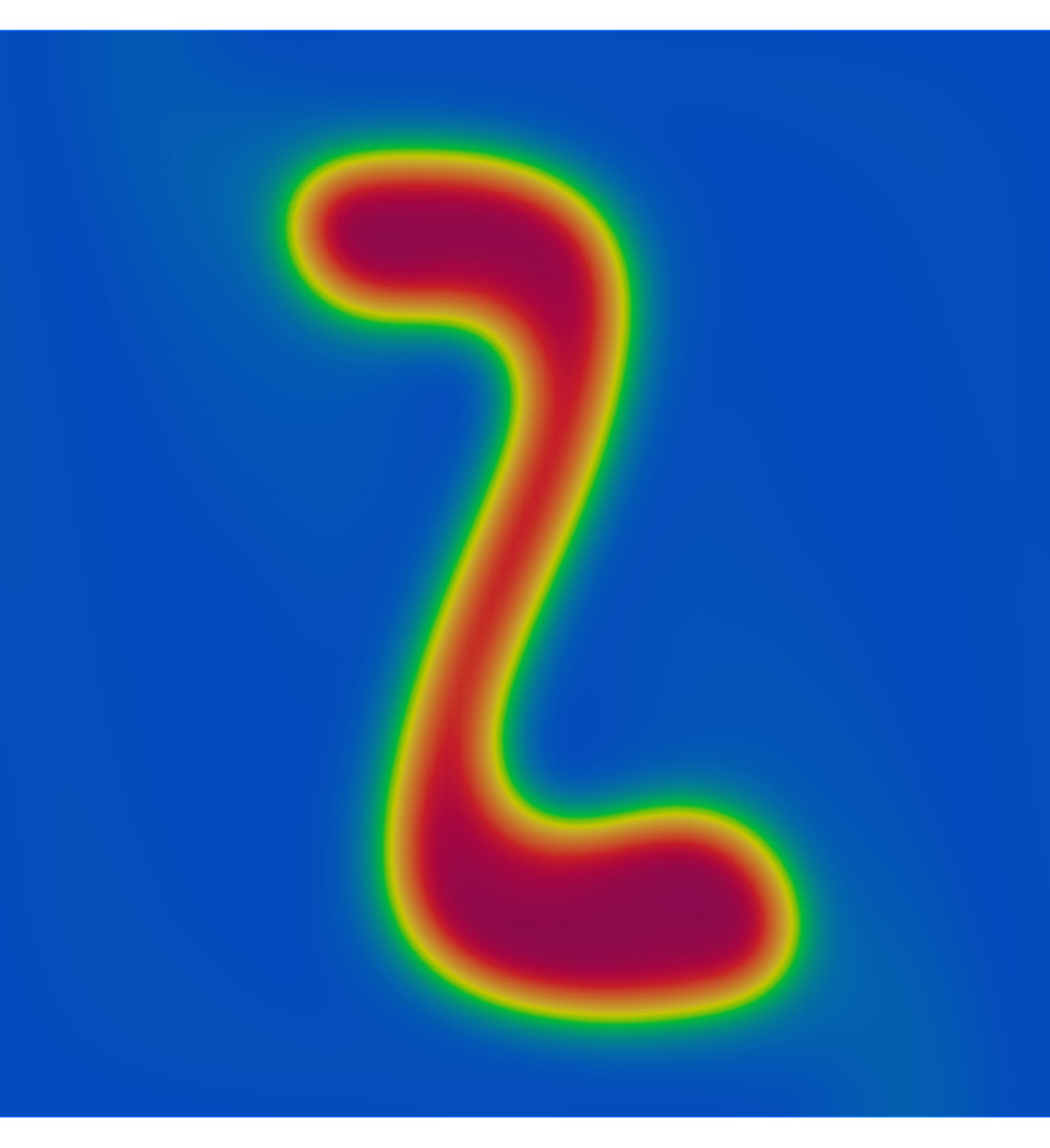}
\includegraphics[scale=0.1125]{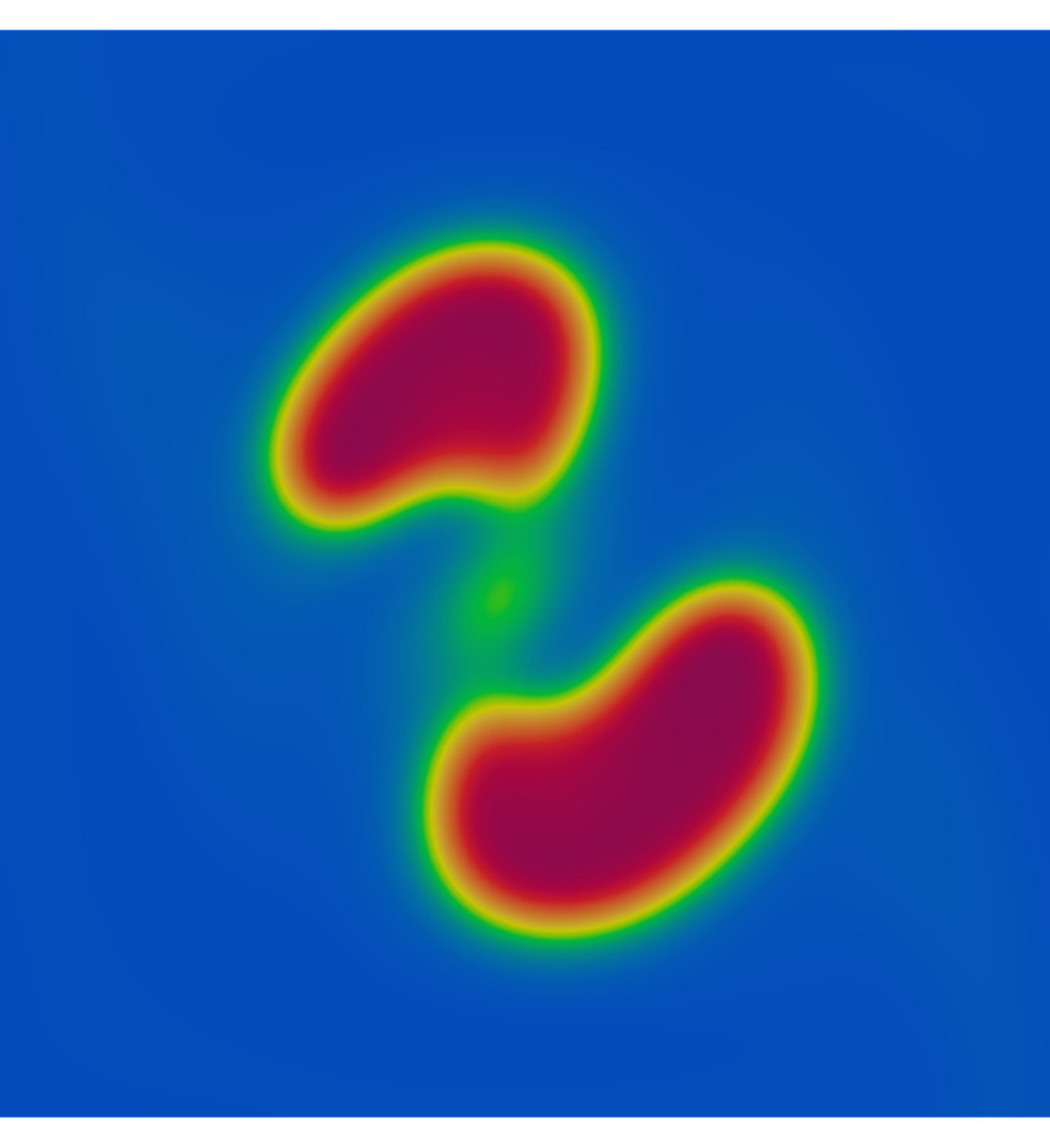}
\includegraphics[scale=0.1125]{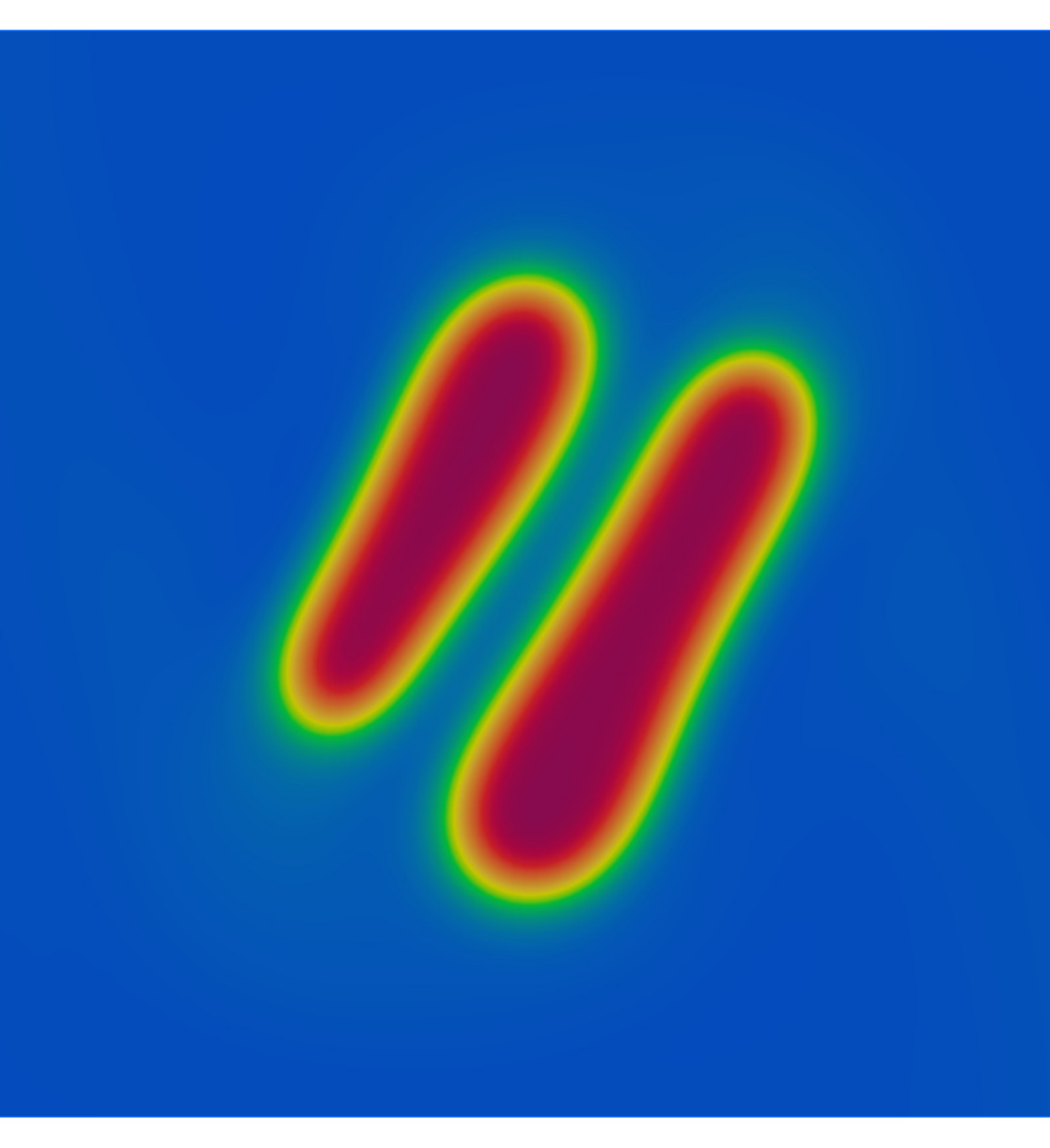}
\includegraphics[scale=0.1125]{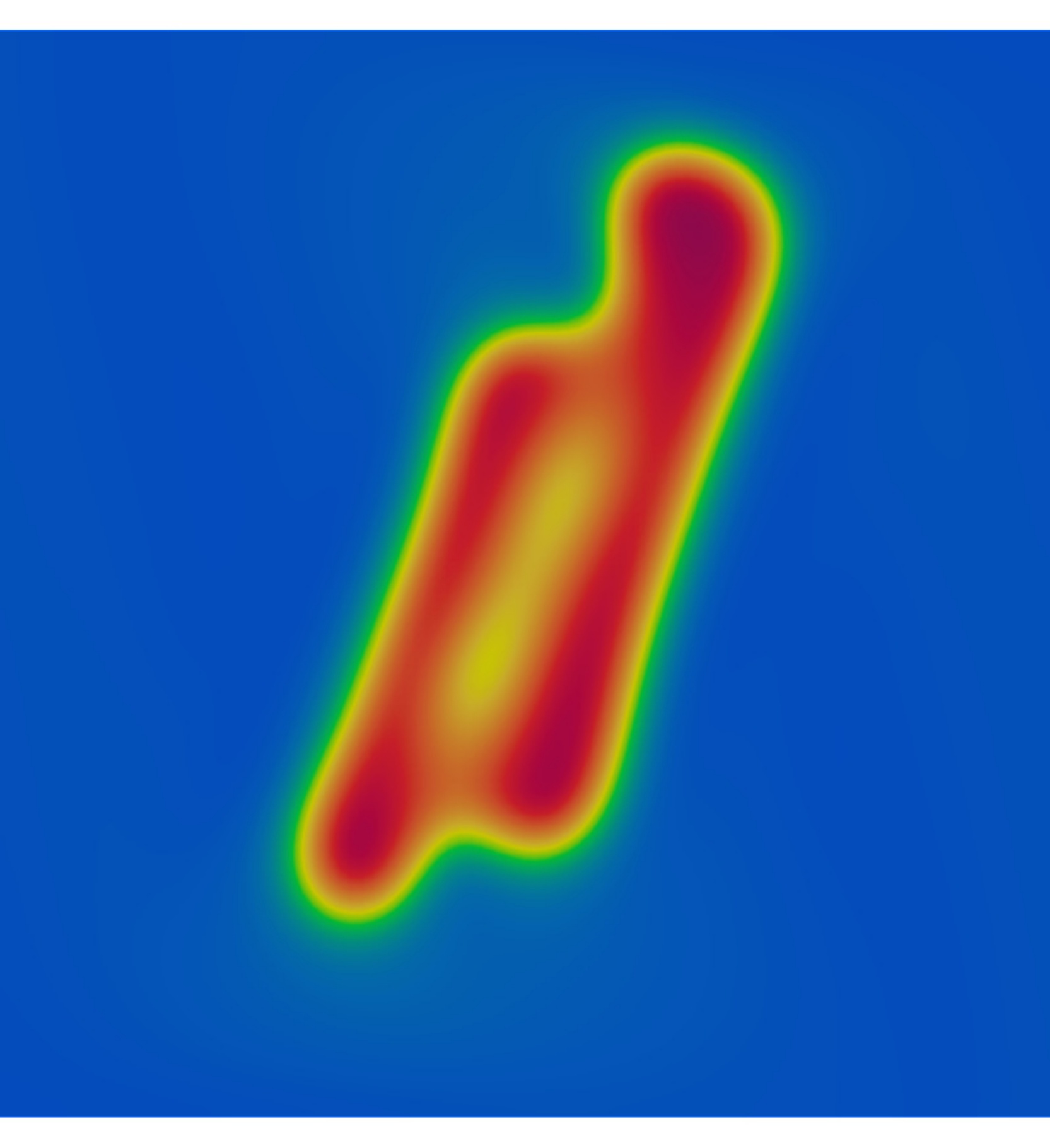}
\\
\includegraphics[scale=0.1125]{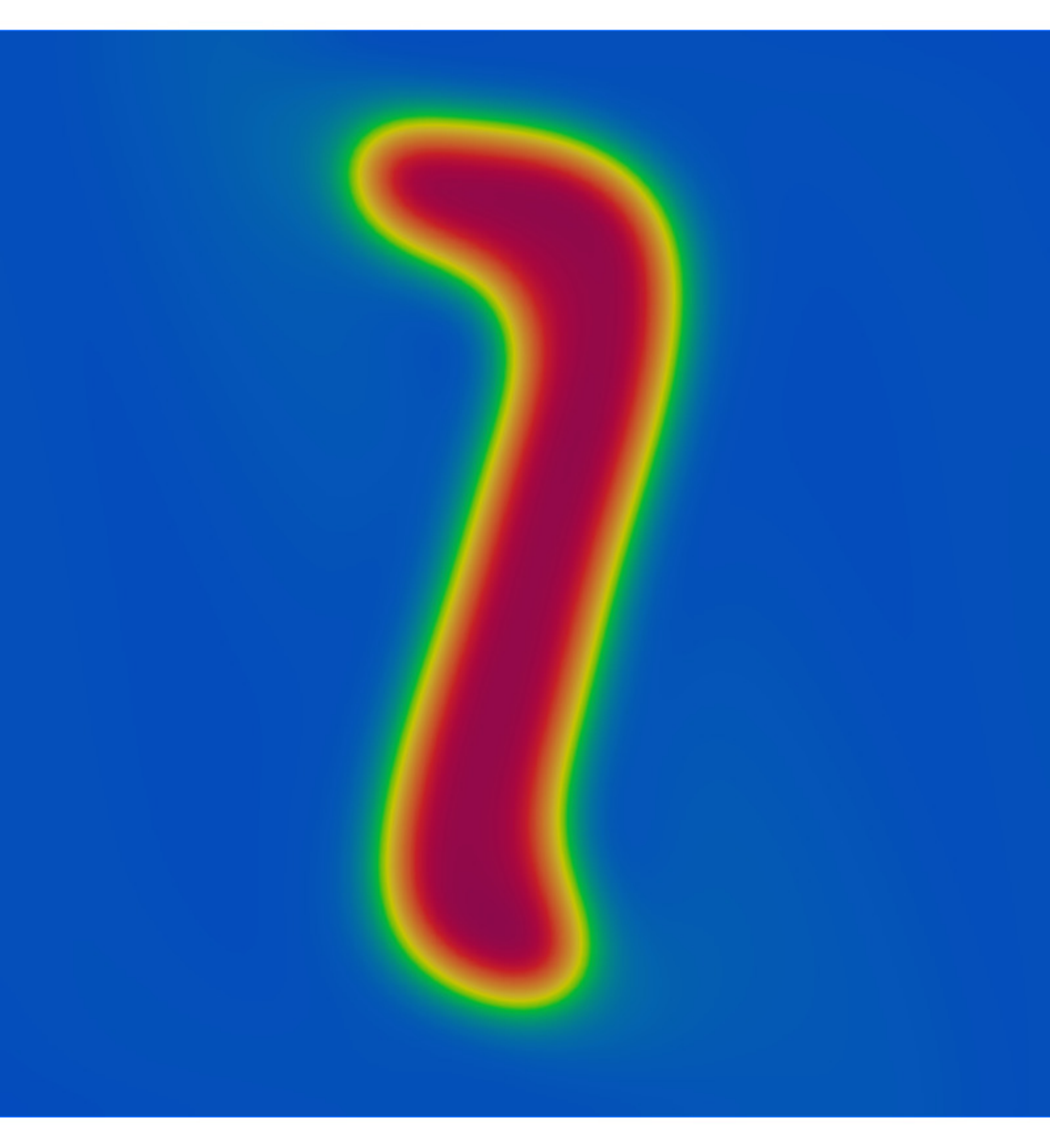}
\includegraphics[scale=0.1125]{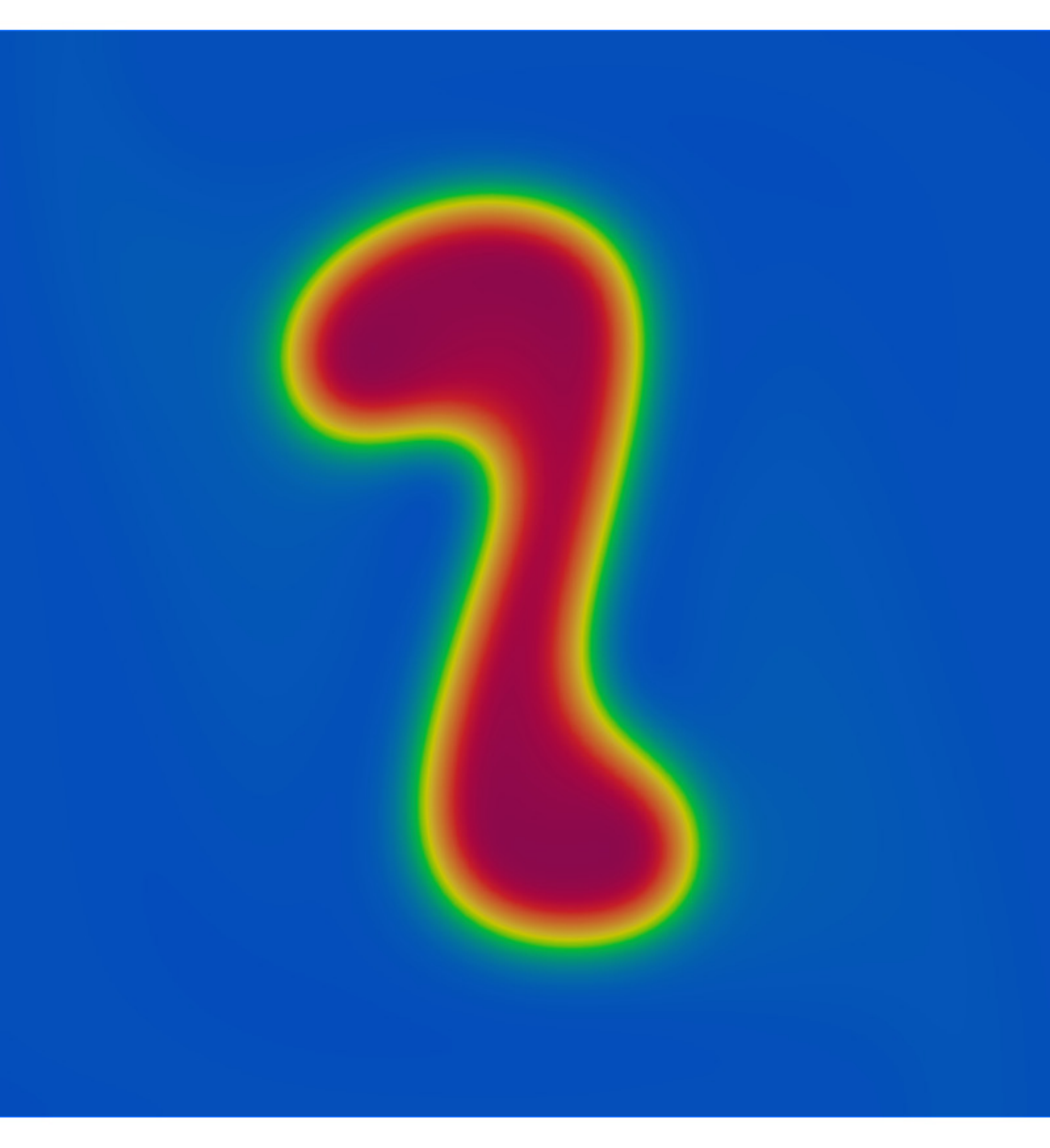}
\includegraphics[scale=0.1125]{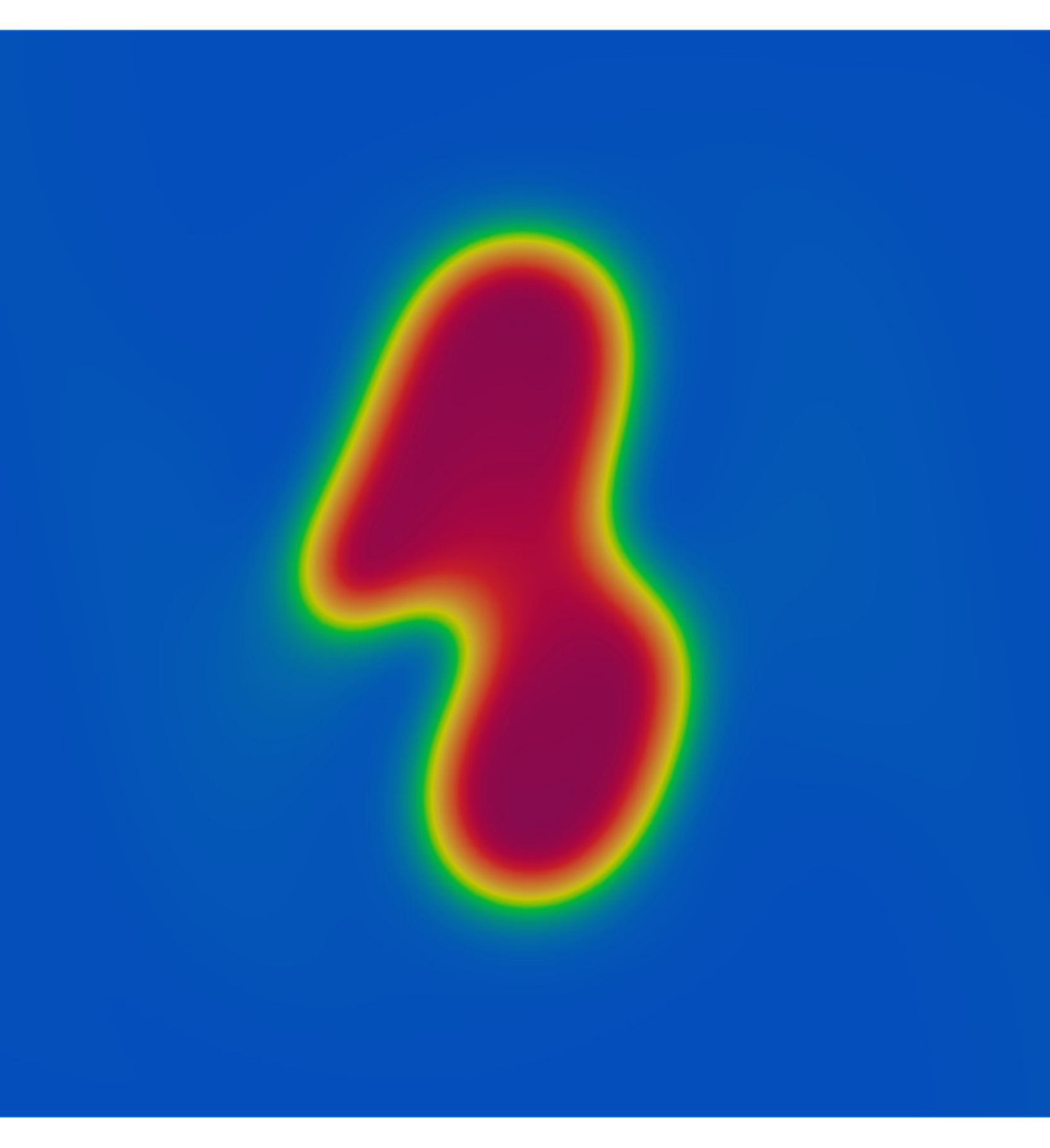}
\includegraphics[scale=0.1125]{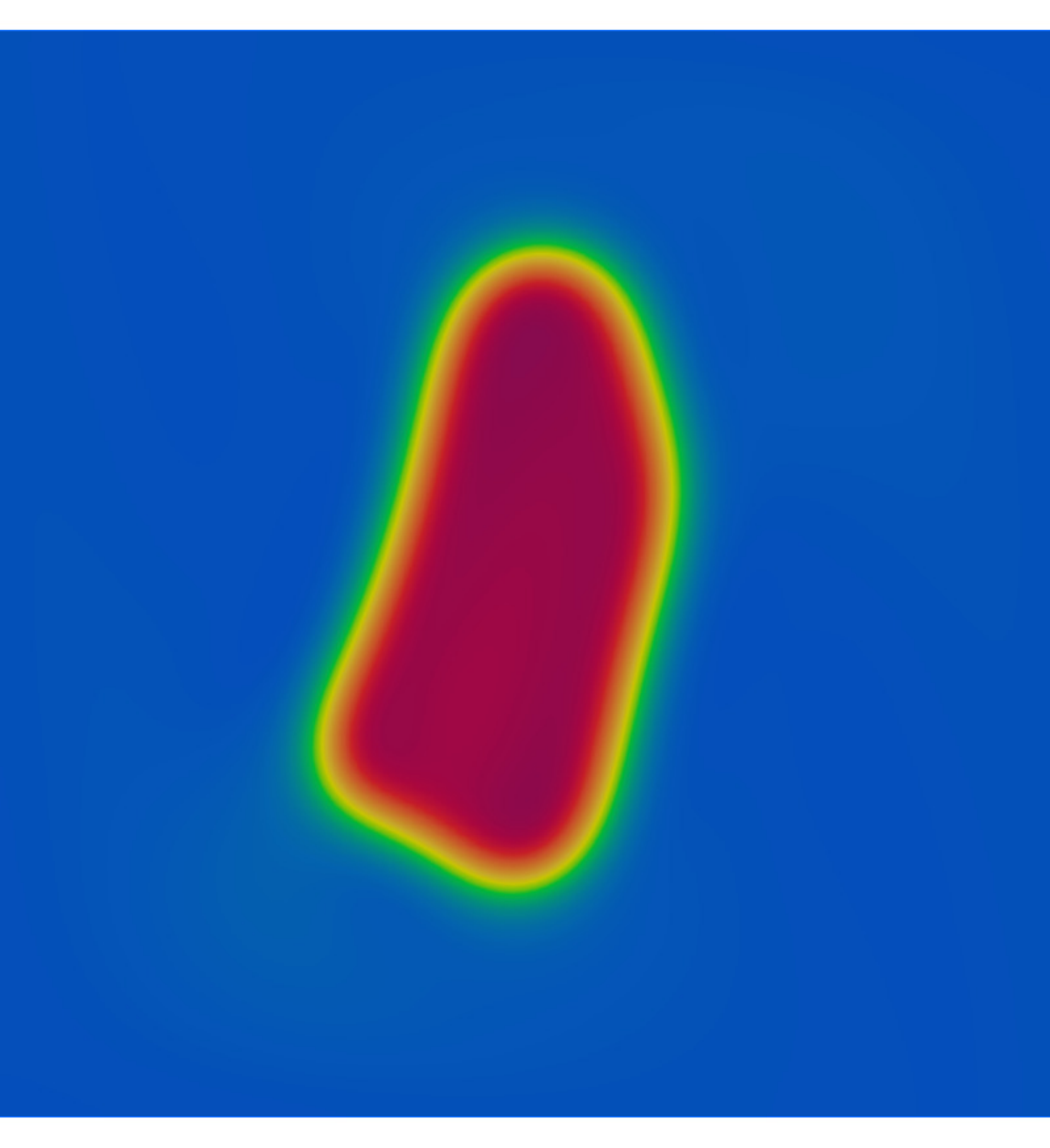}
\includegraphics[scale=0.1125]{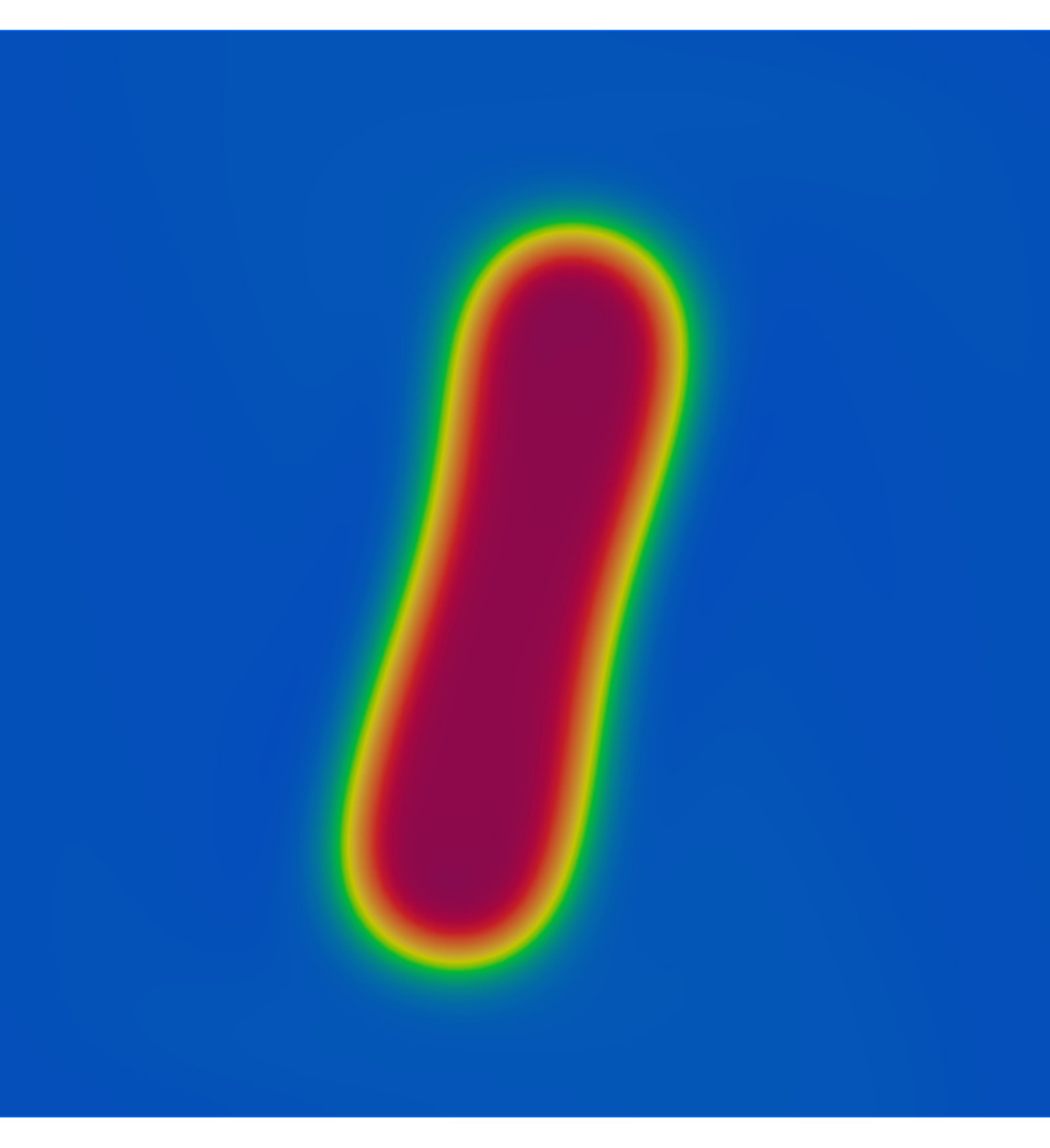}
\\
\includegraphics[scale=0.1125]{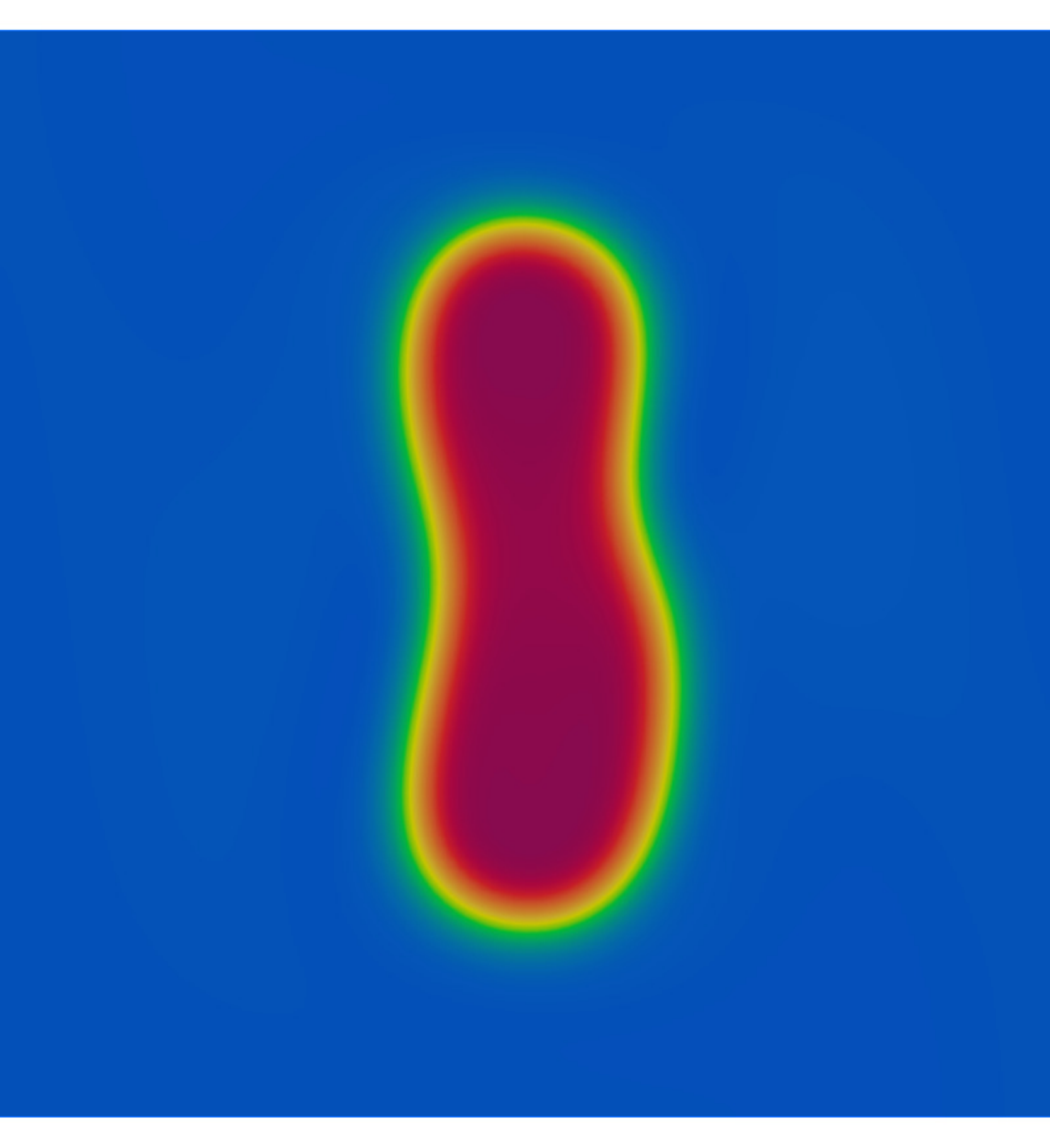}
\includegraphics[scale=0.1125]{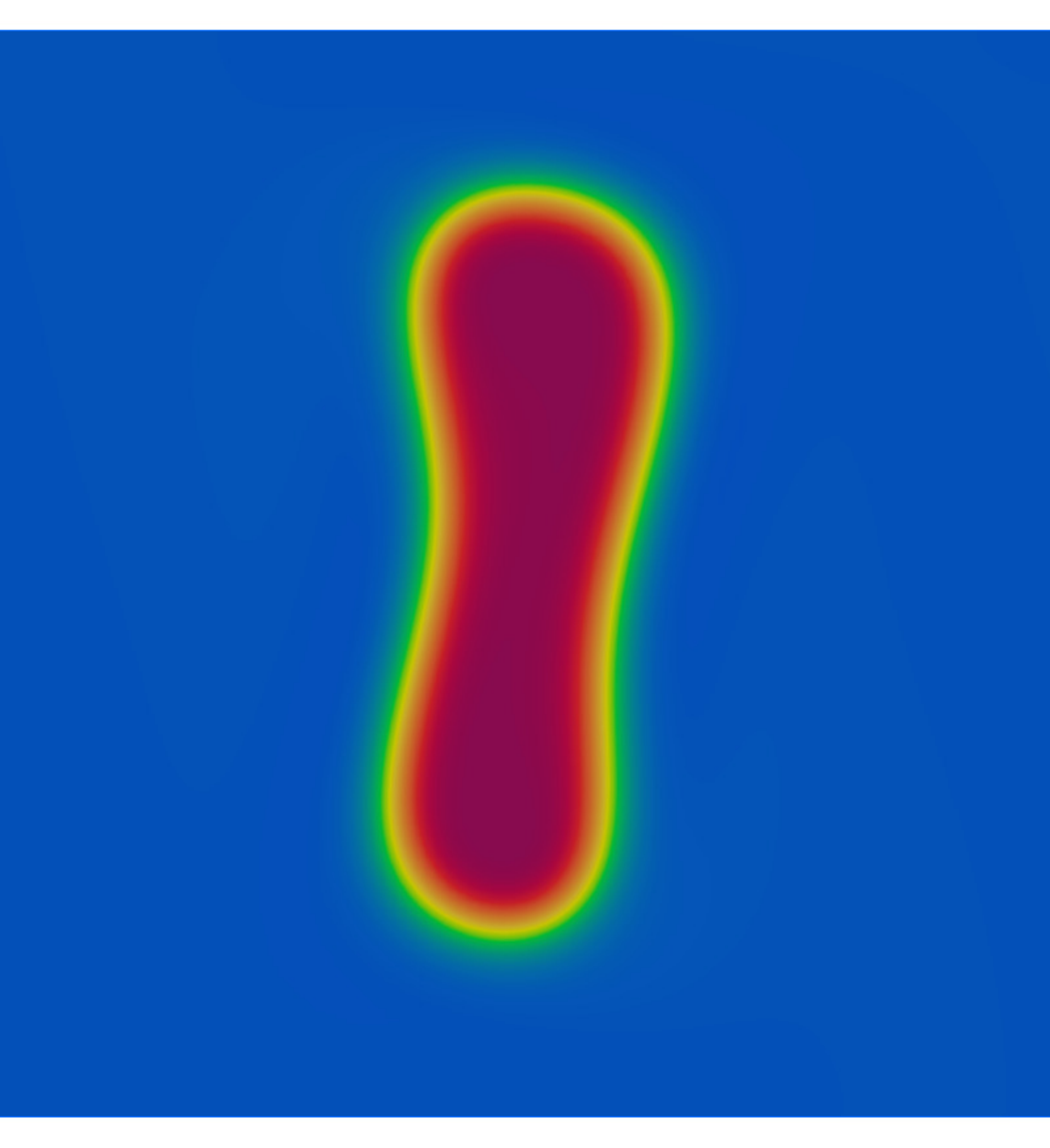}
\includegraphics[scale=0.1125]{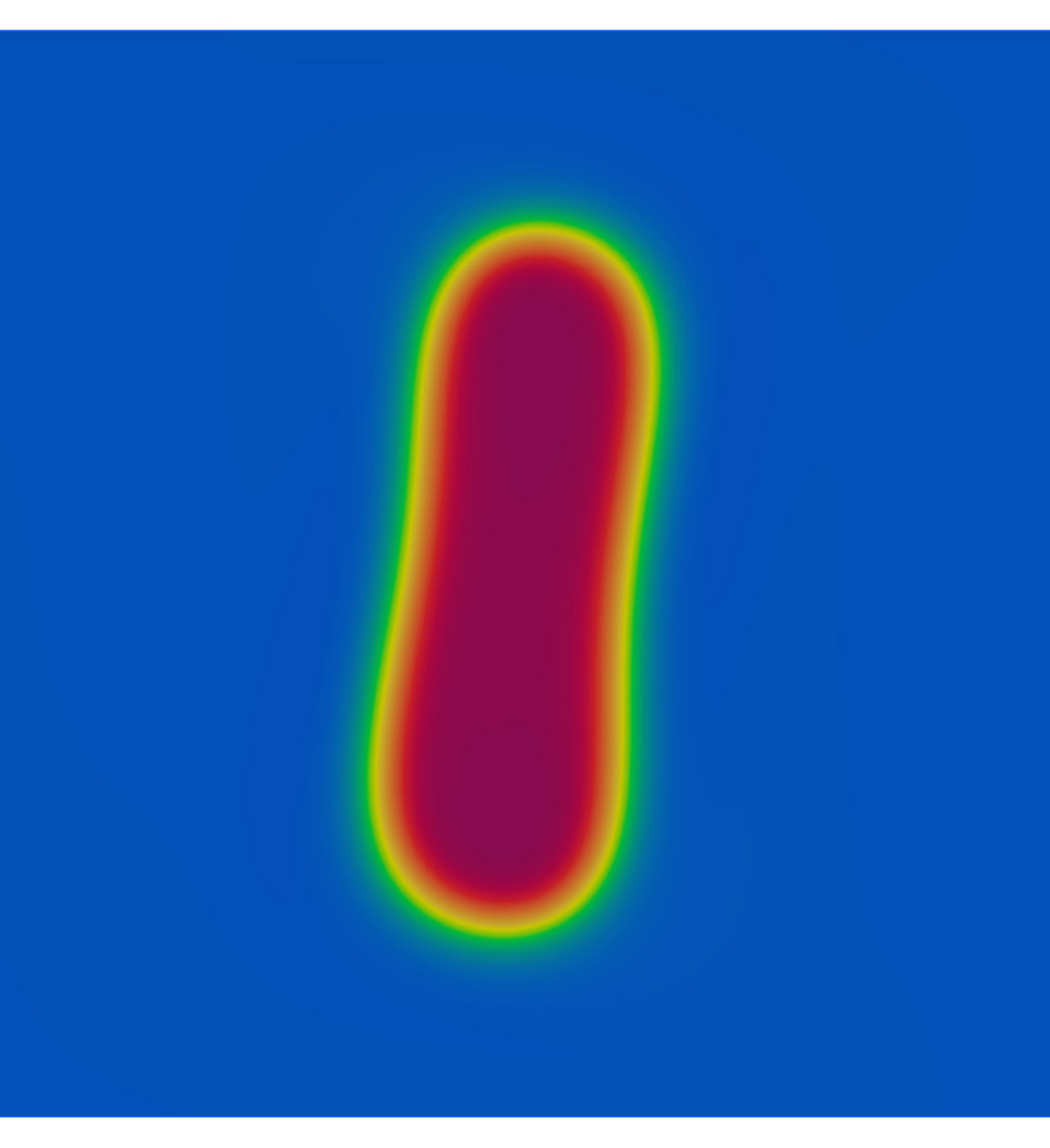}
\includegraphics[scale=0.1125]{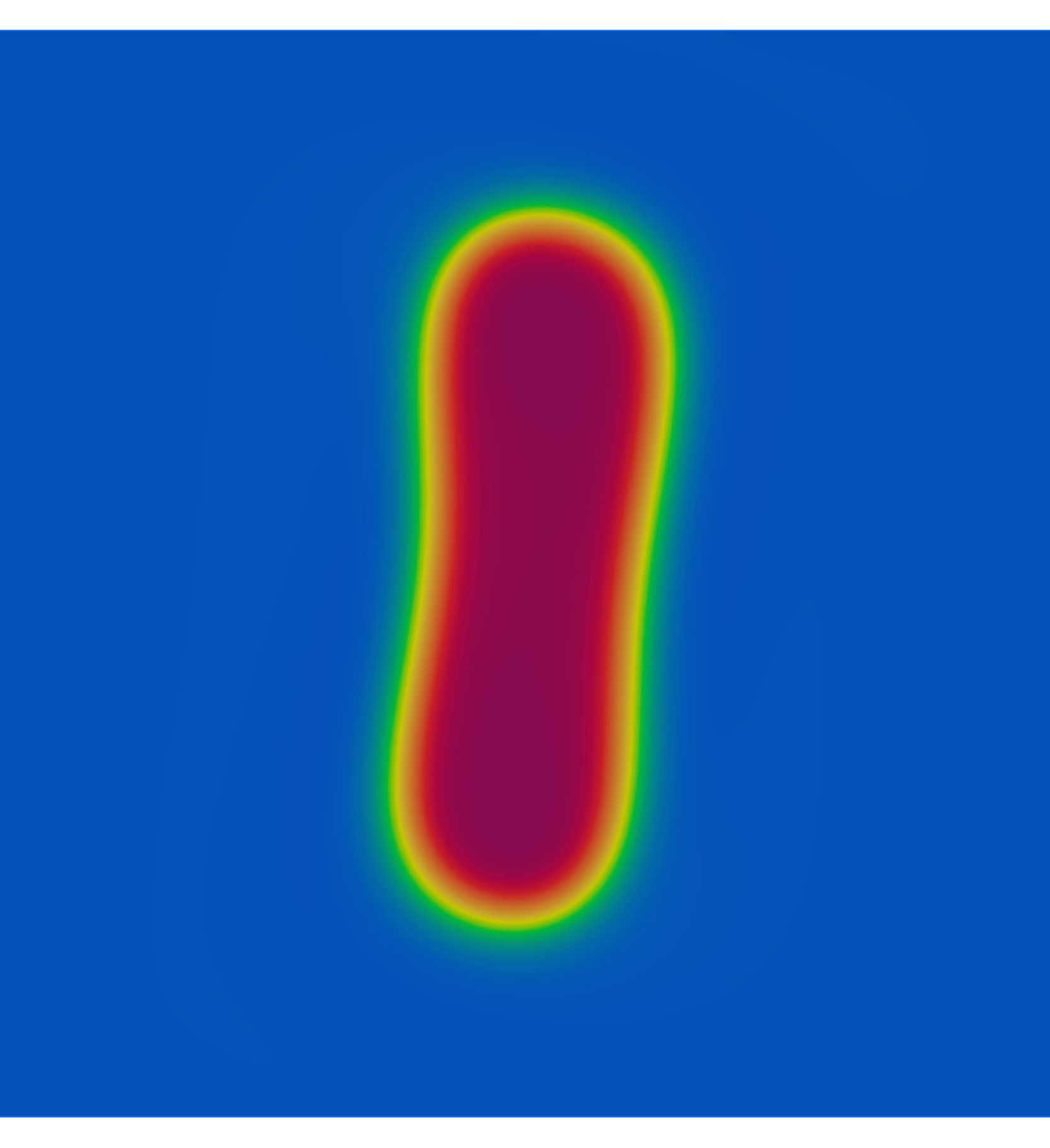}
\includegraphics[scale=0.1125]{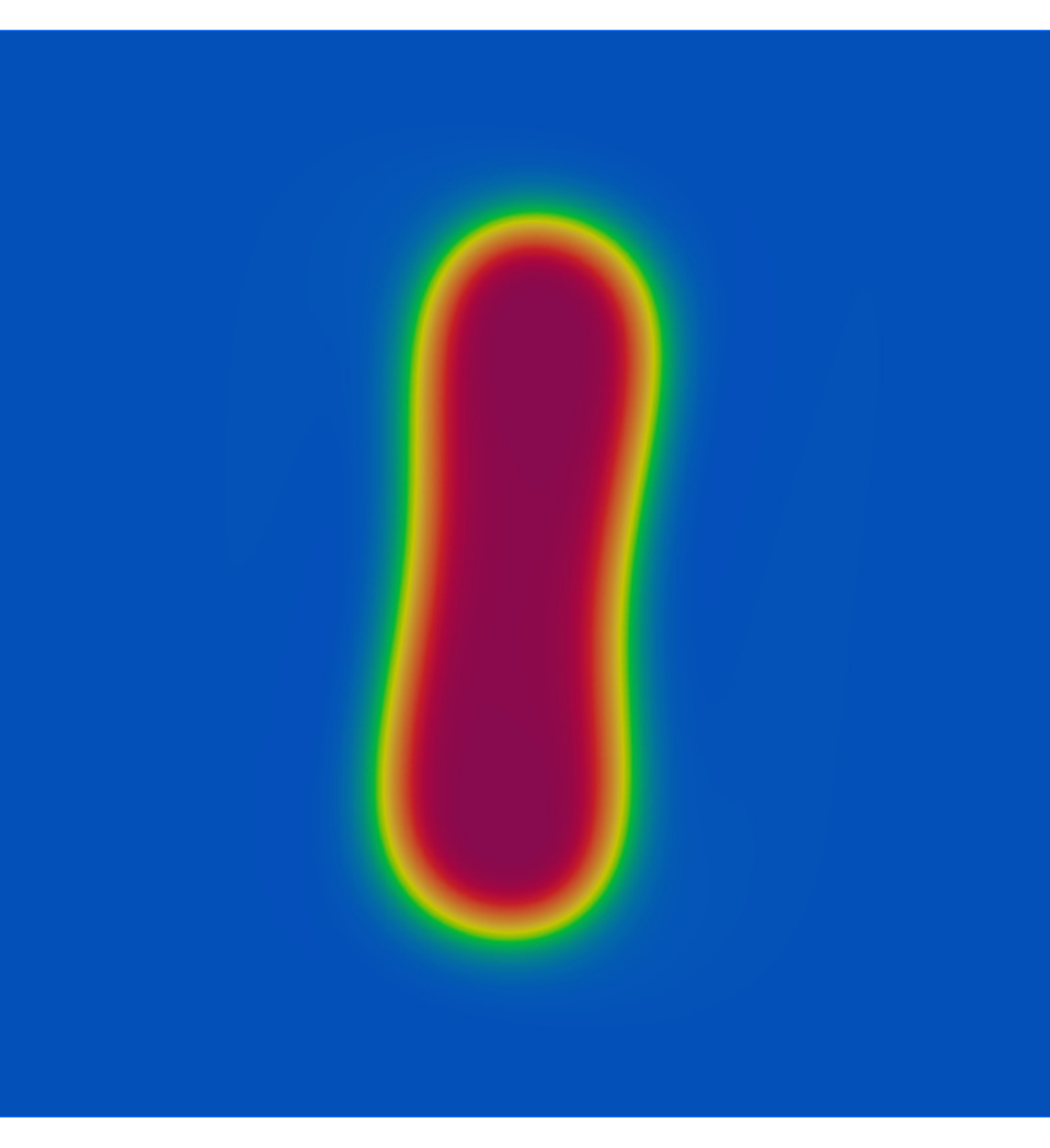}
\end{center}
\caption{Example III. Rotating fluids. Evolution in time of $\phi$ for $G_\varepsilon$-scheme at times $t=0, 0.15, 0.3, 0.45, 0.6, 0.75, 0.9, 1.05, 1.2, 1.35, 1.5, 1.65, 1.8, 1.95, 2.1, 2.5, 3.0, 3.5, 4.25$ and $5$.}\label{fig:ExIIIDynG}
\end{figure}

\begin{figure}[H]
\begin{center}
\includegraphics[scale=0.1125]{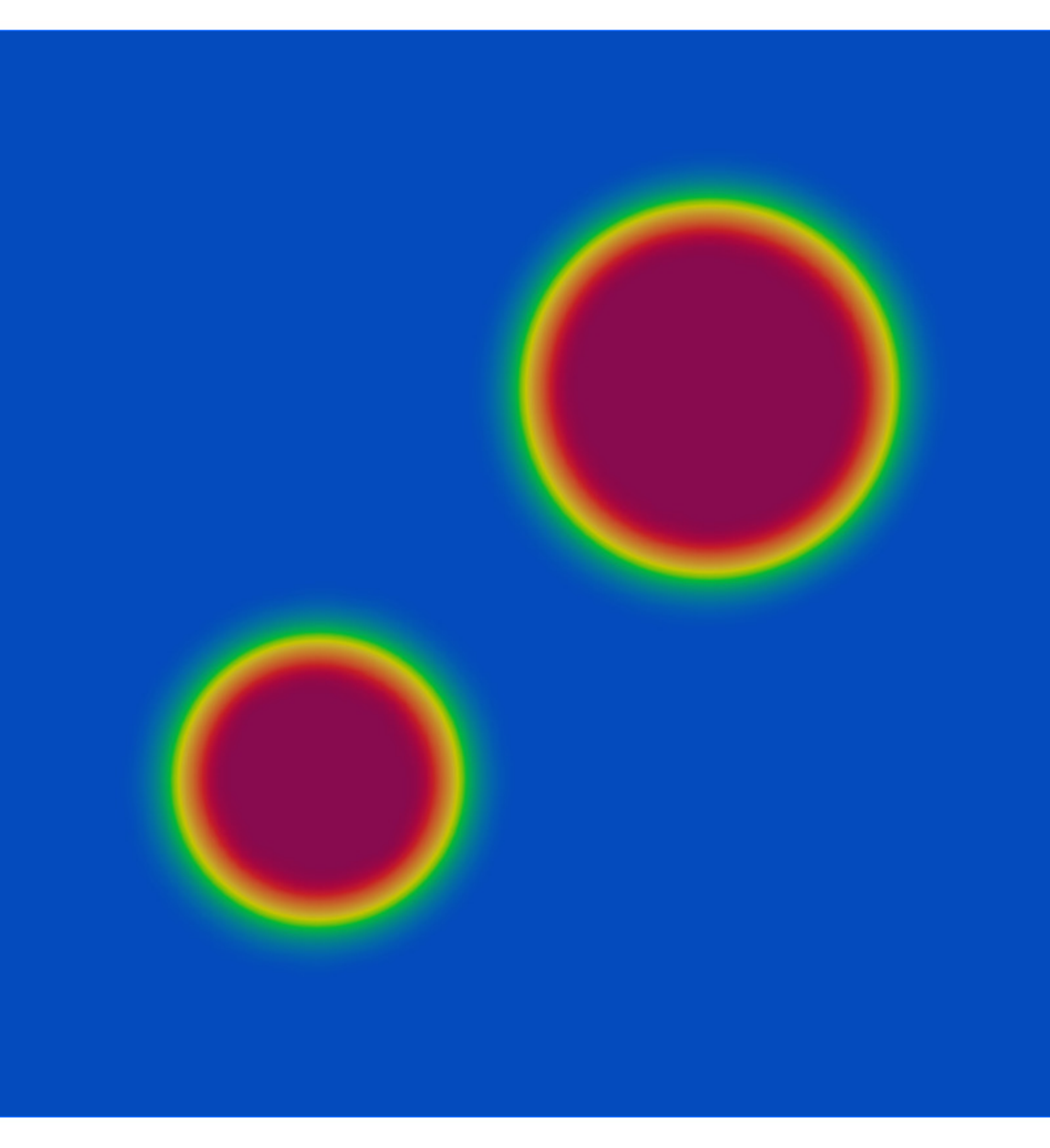}
\includegraphics[scale=0.1125]{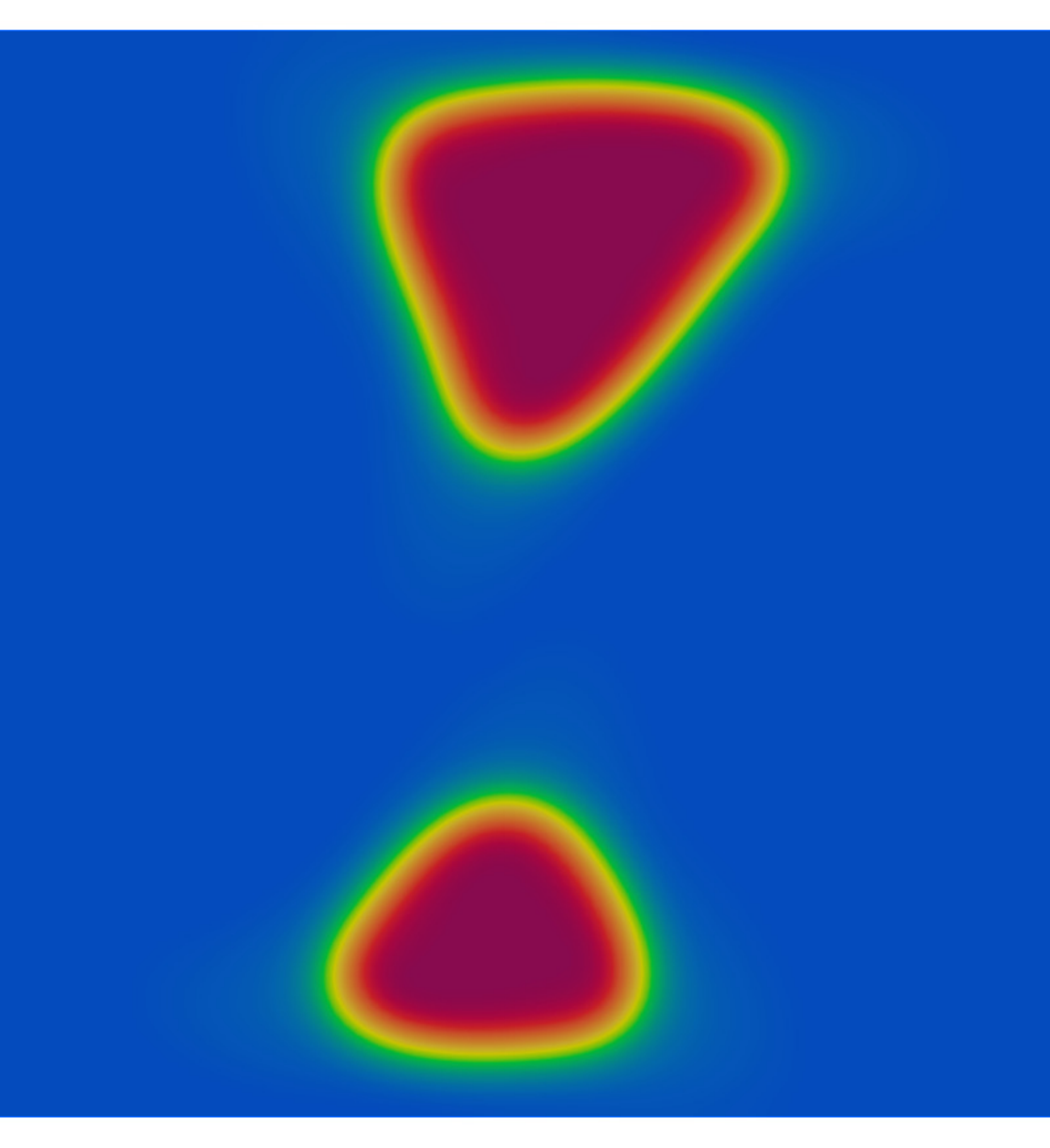}
\includegraphics[scale=0.1125]{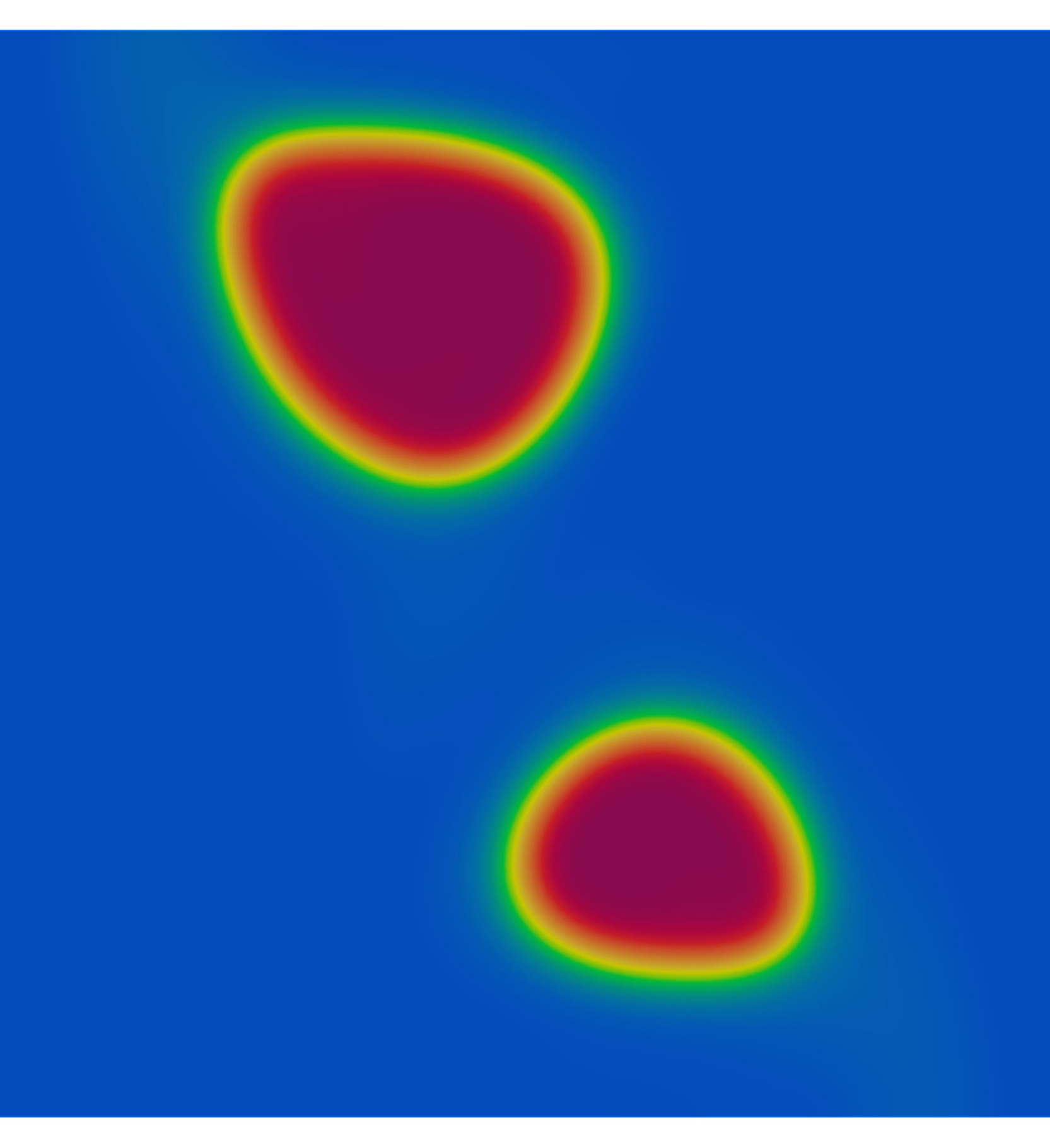}
\includegraphics[scale=0.1125]{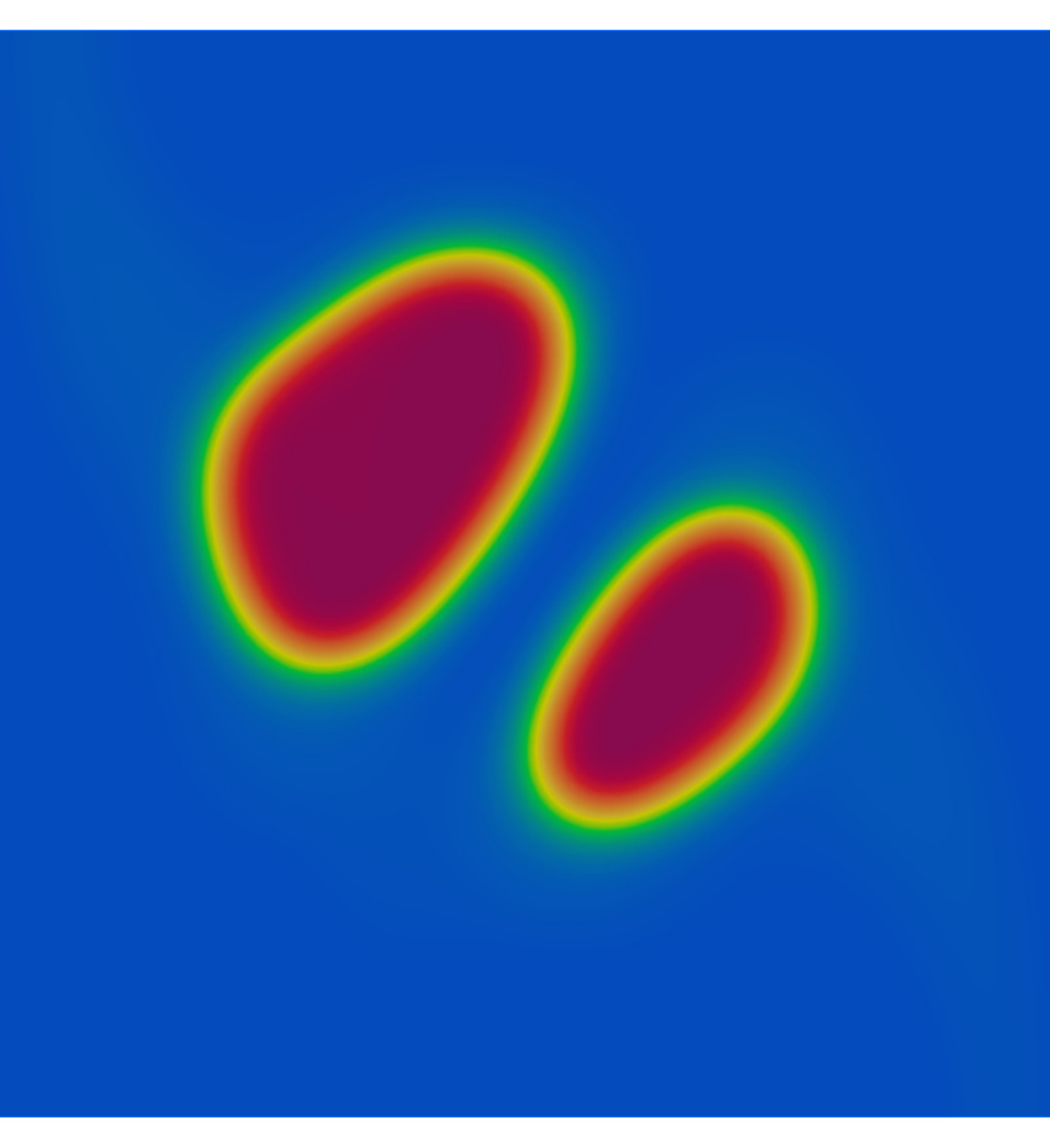}
\includegraphics[scale=0.1125]{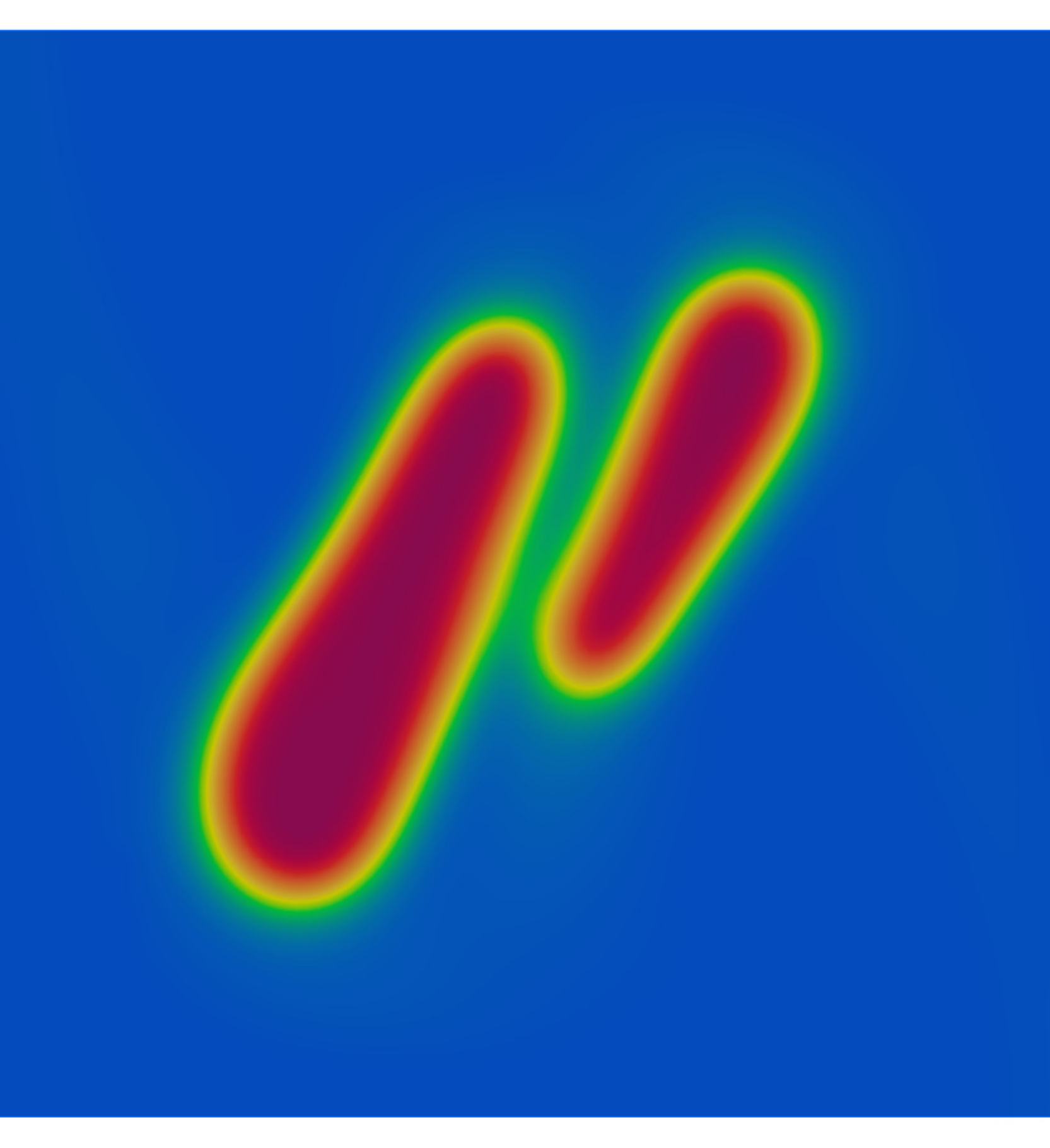}
\\
\includegraphics[scale=0.1125]{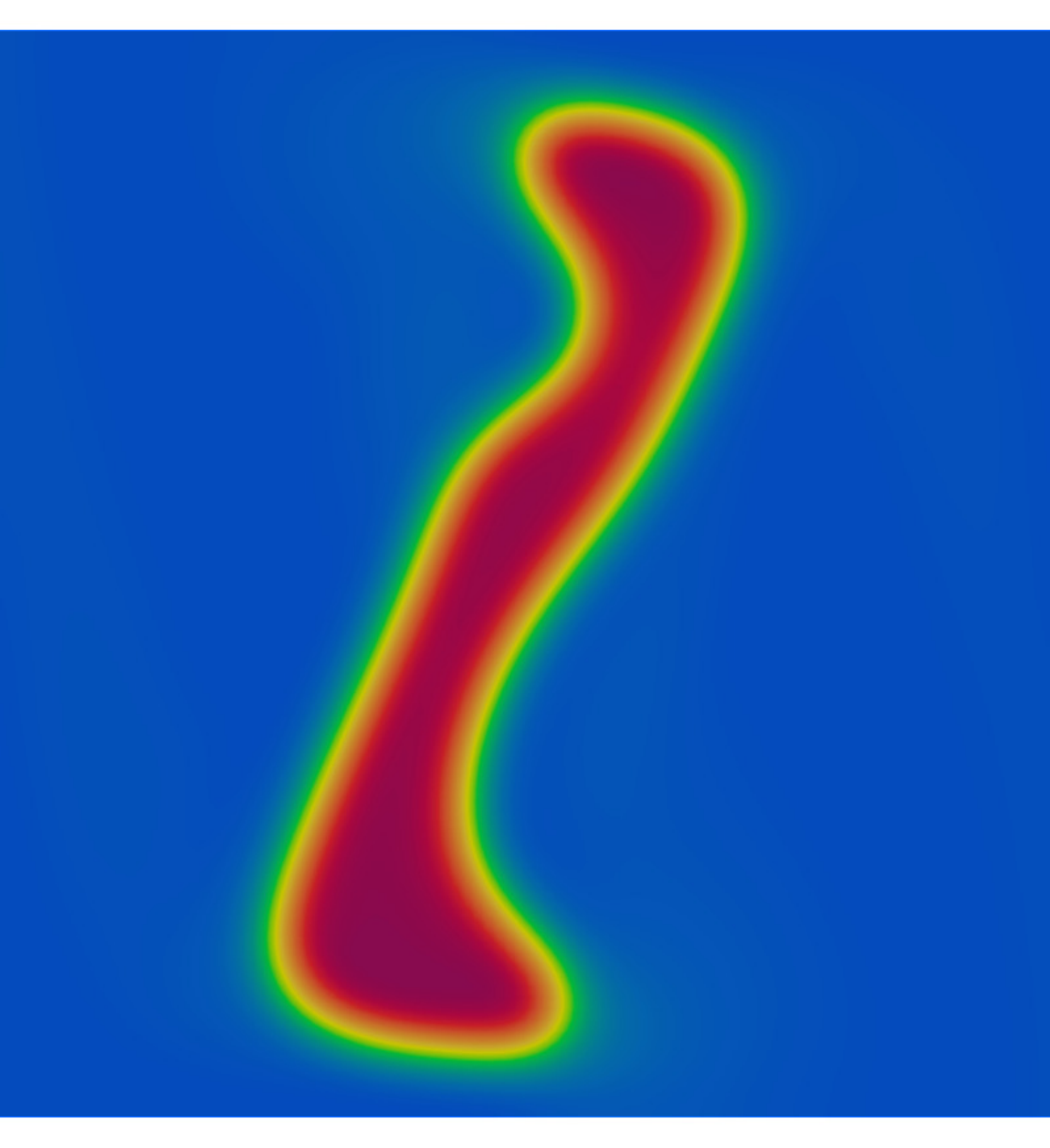}
\includegraphics[scale=0.1125]{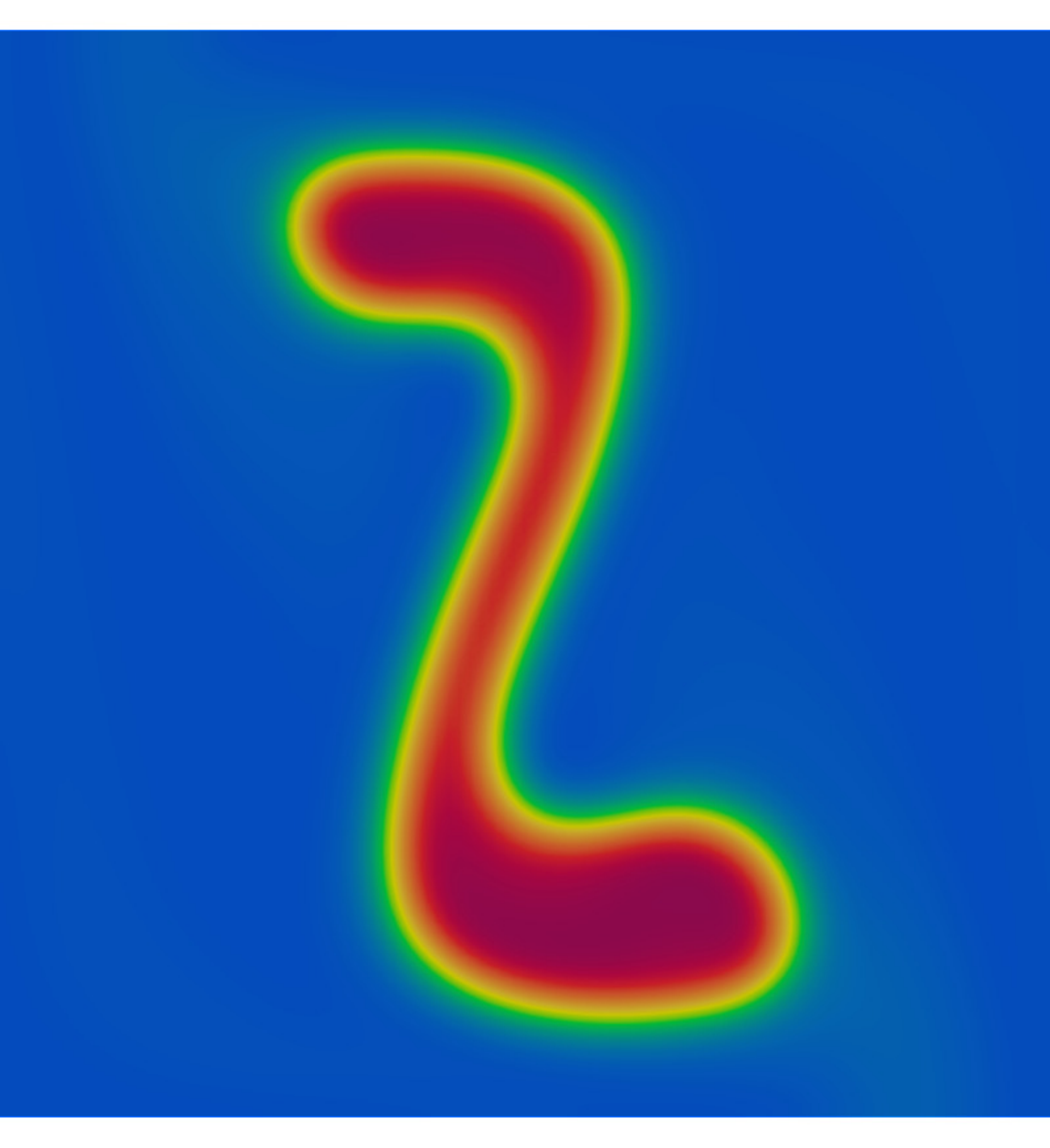}
\includegraphics[scale=0.1125]{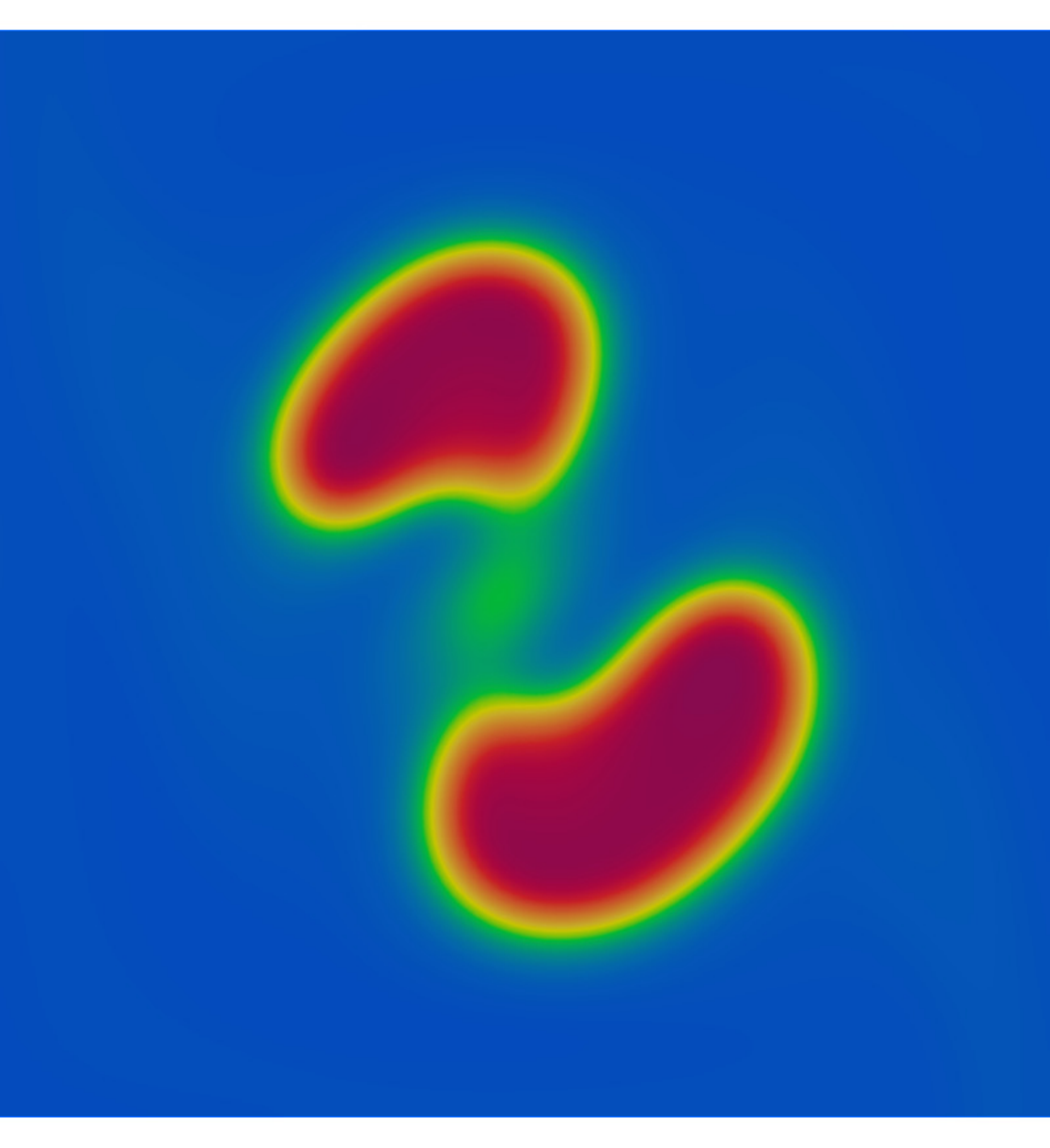}
\includegraphics[scale=0.1125]{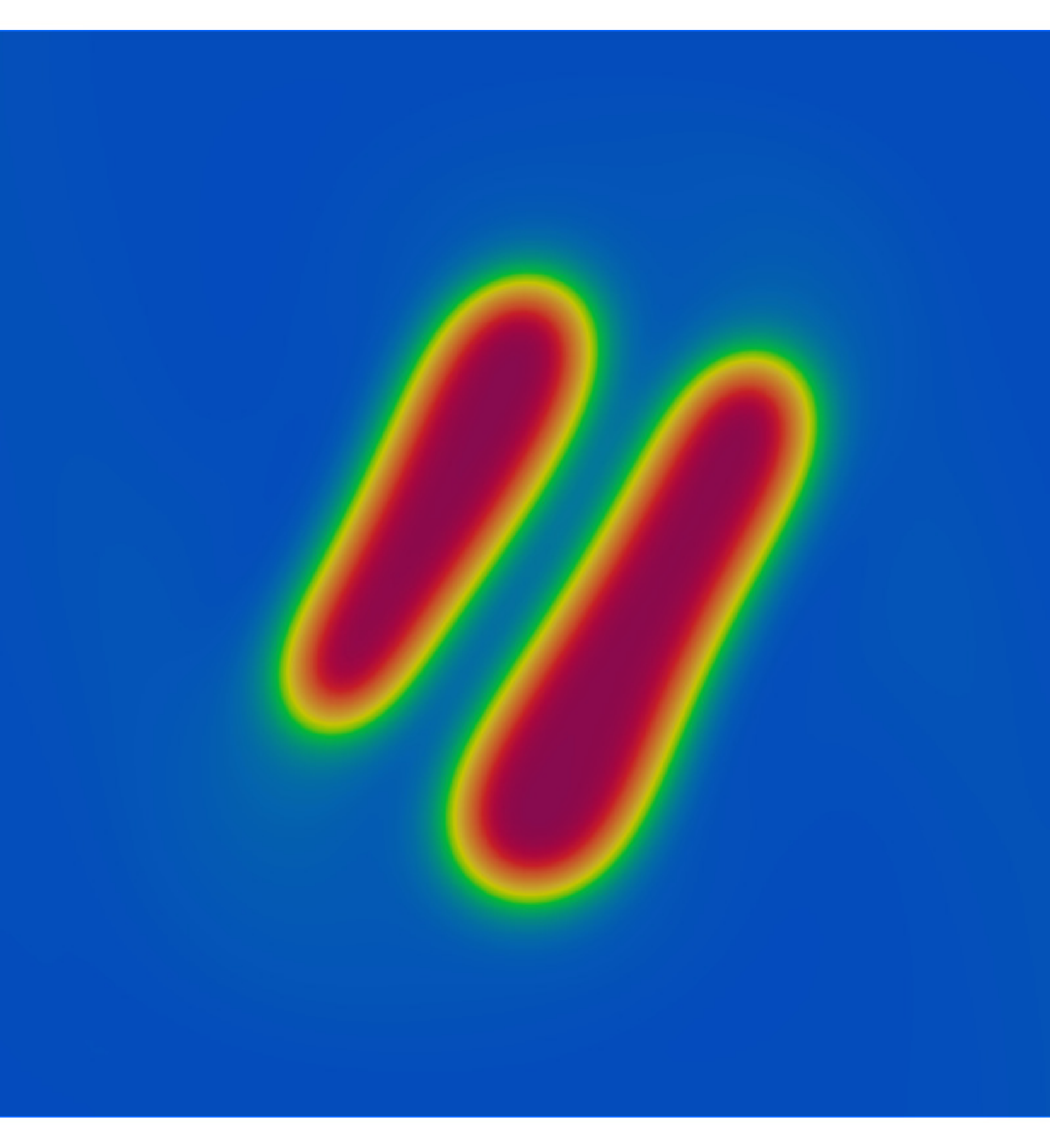}
\includegraphics[scale=0.1125]{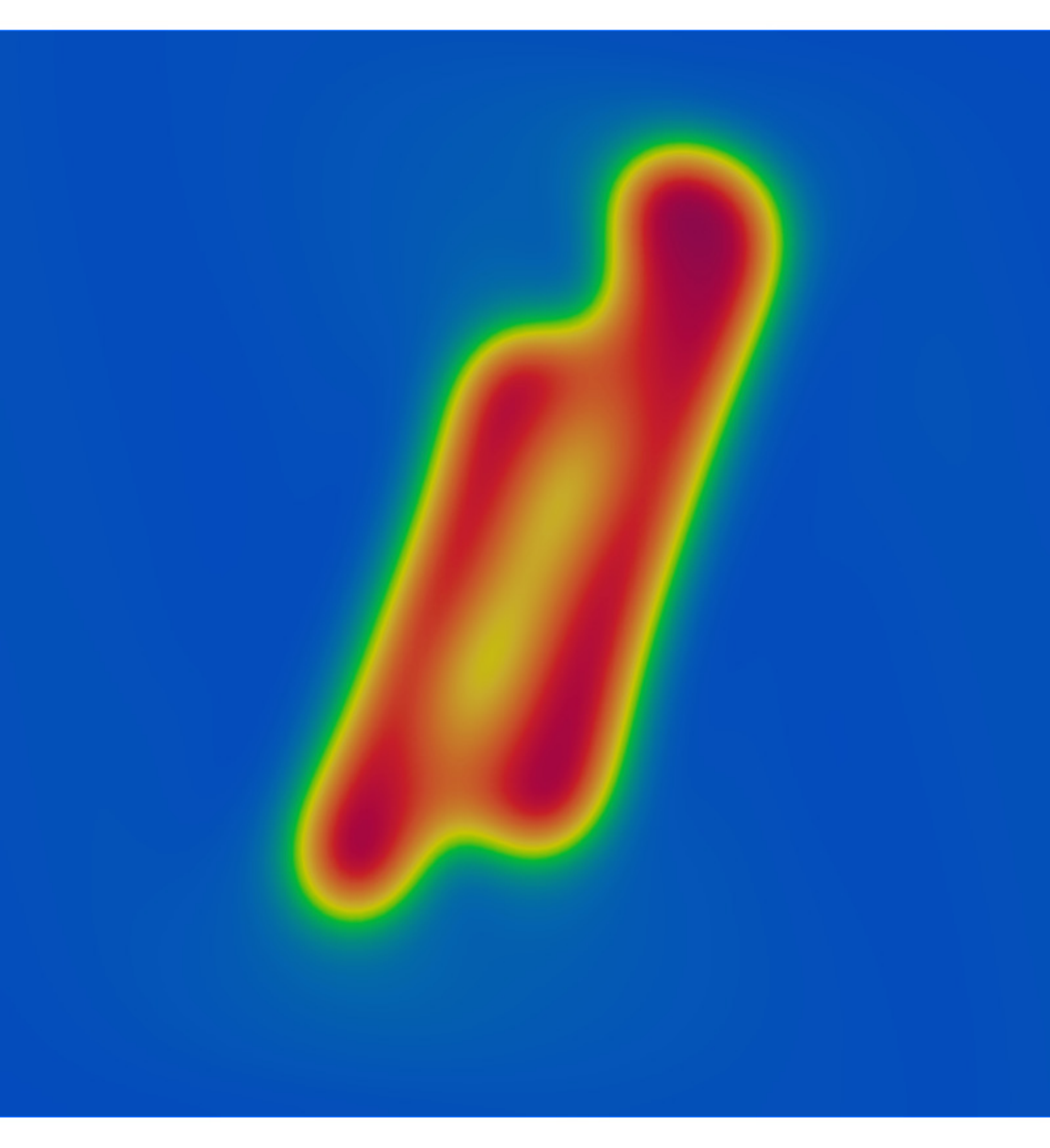}
\\
\includegraphics[scale=0.1125]{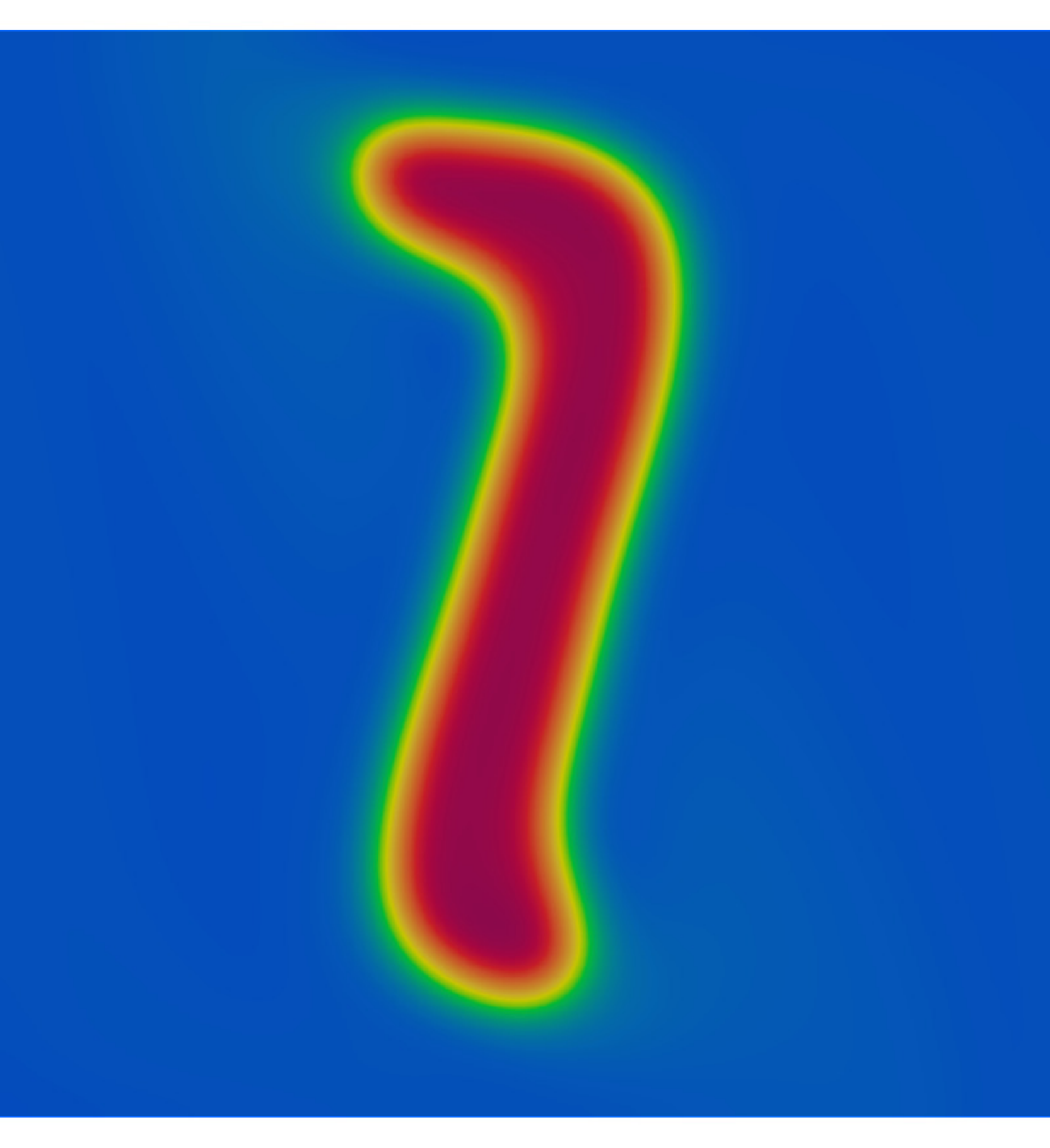}
\includegraphics[scale=0.1125]{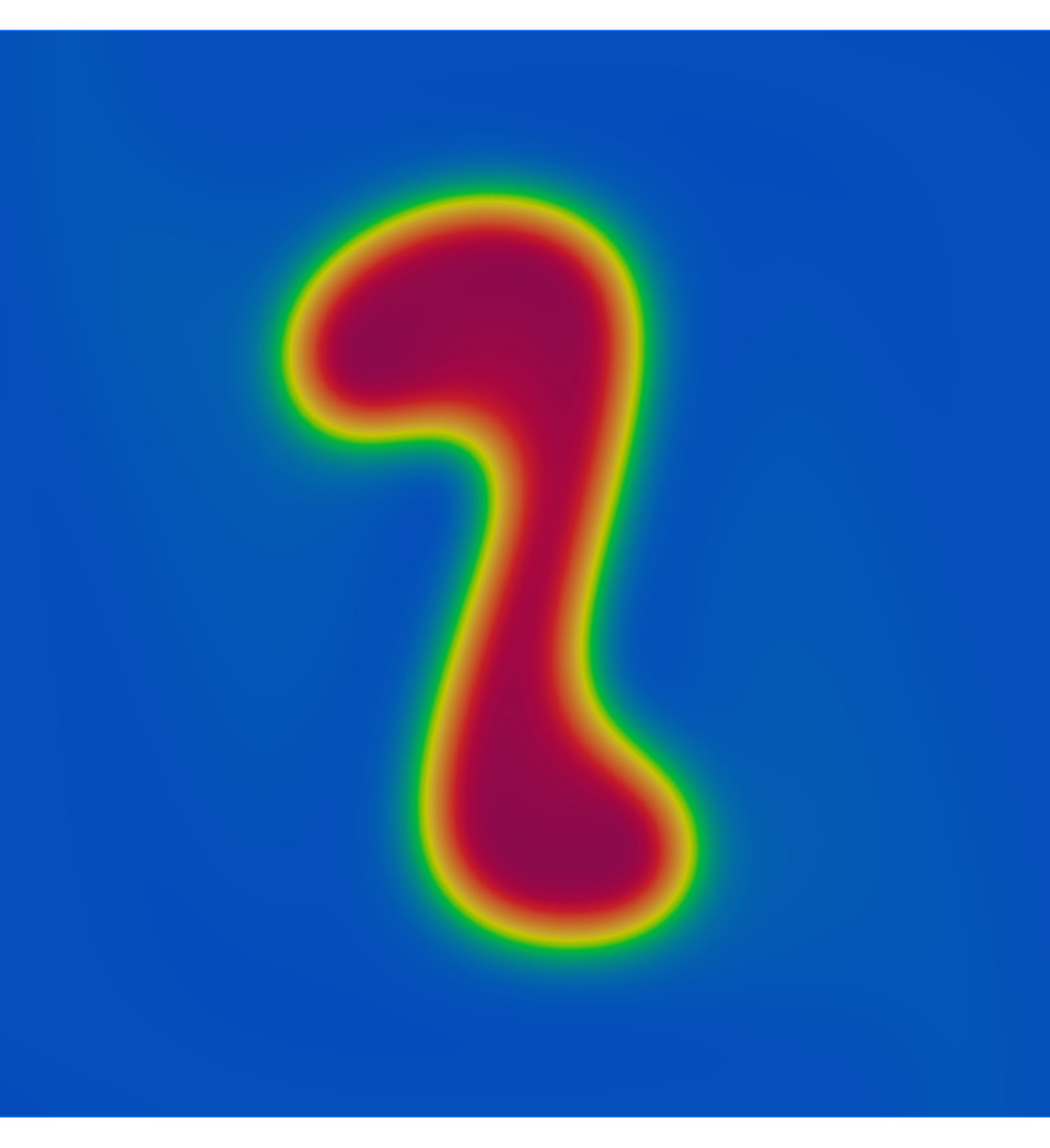}
\includegraphics[scale=0.1125]{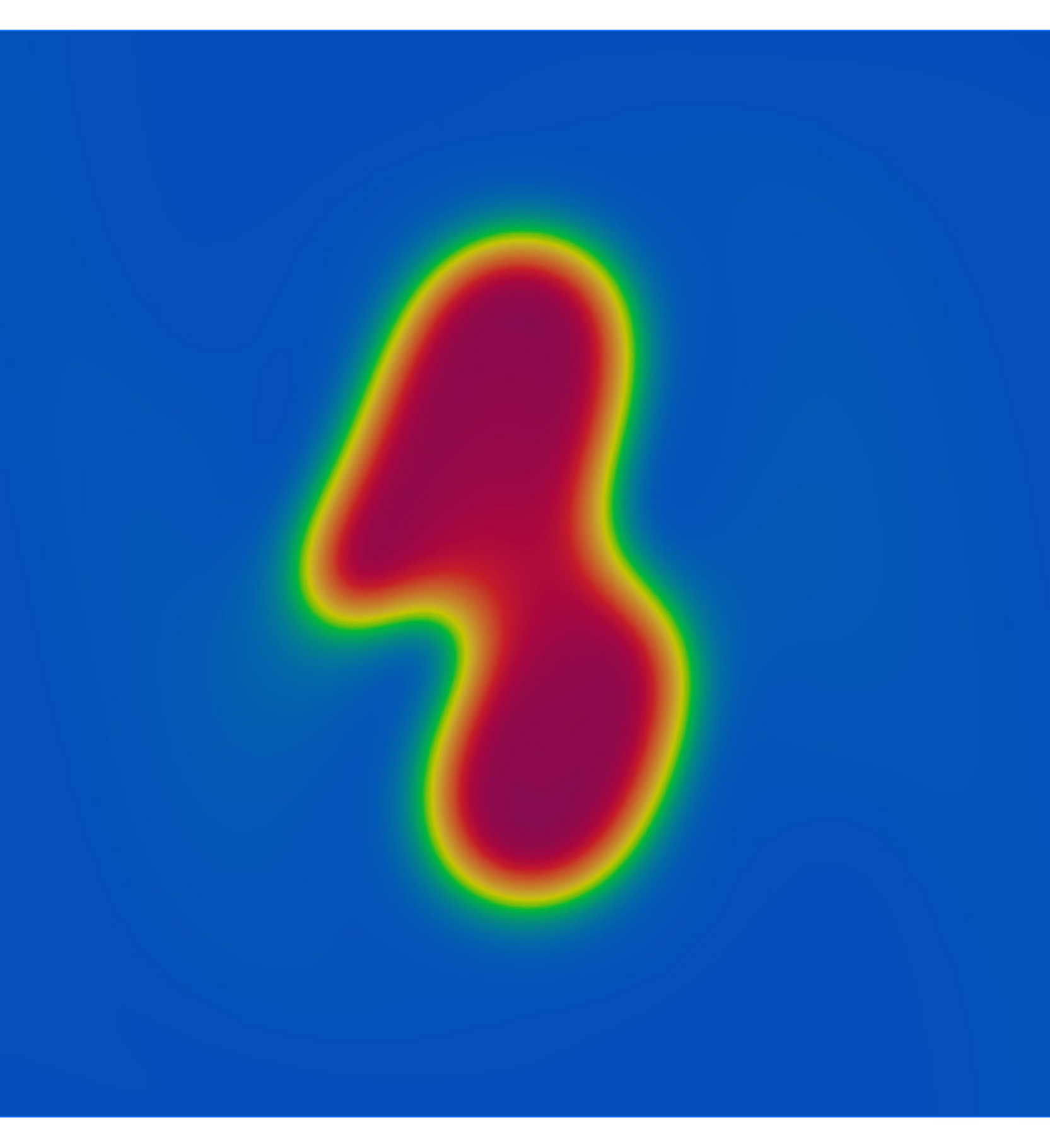}
\includegraphics[scale=0.1125]{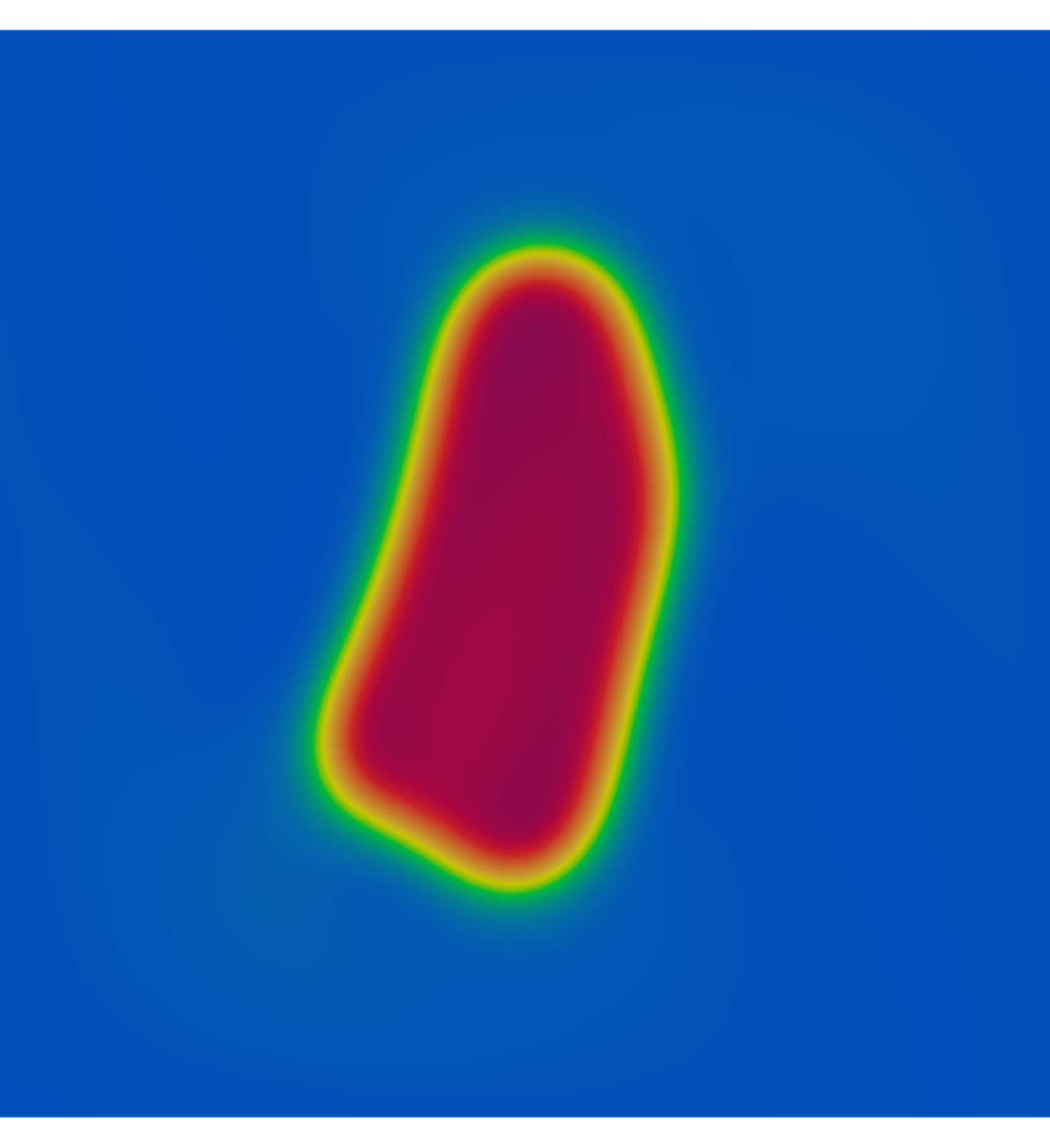}
\includegraphics[scale=0.1125]{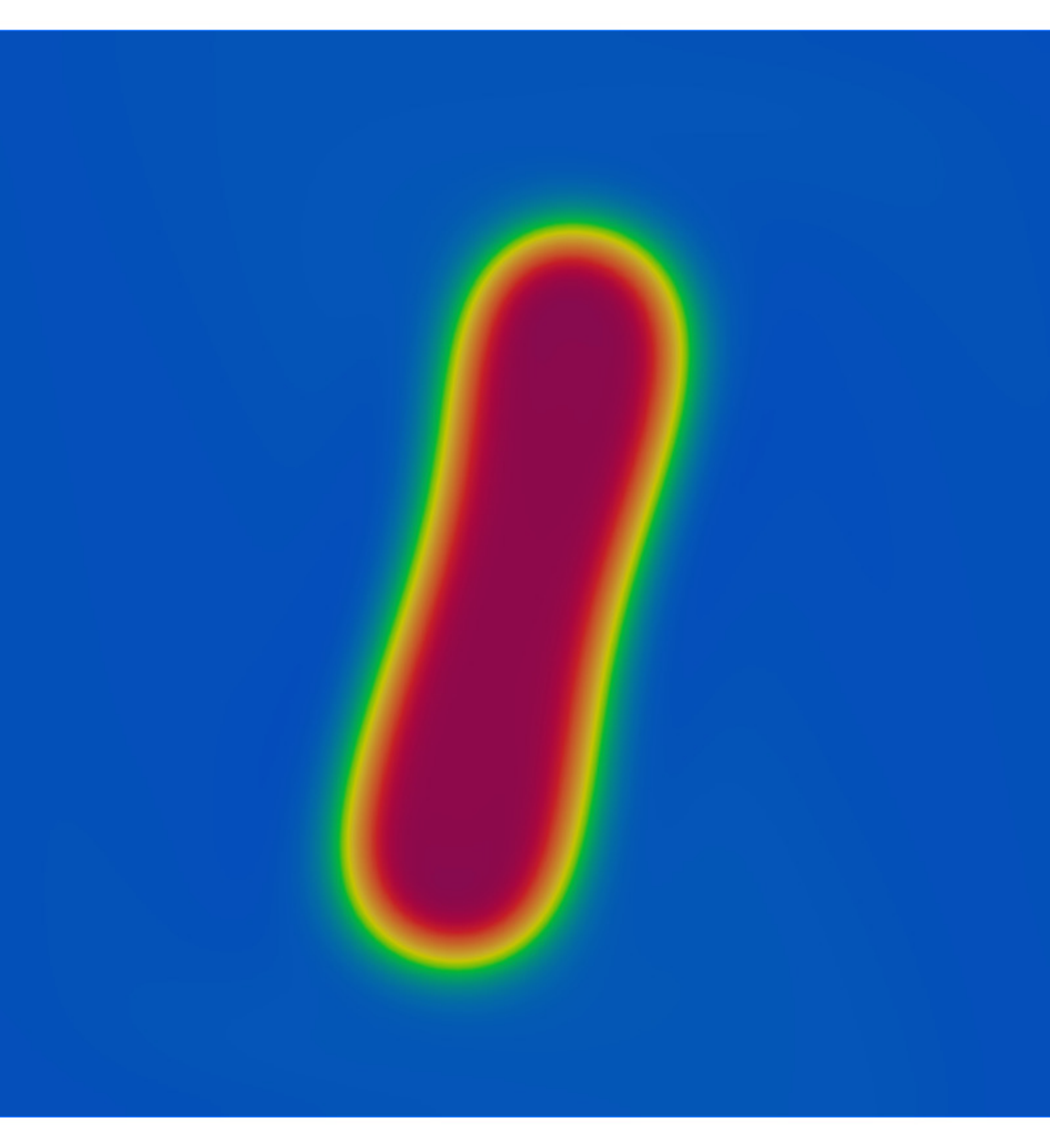}
\\
\includegraphics[scale=0.1125]{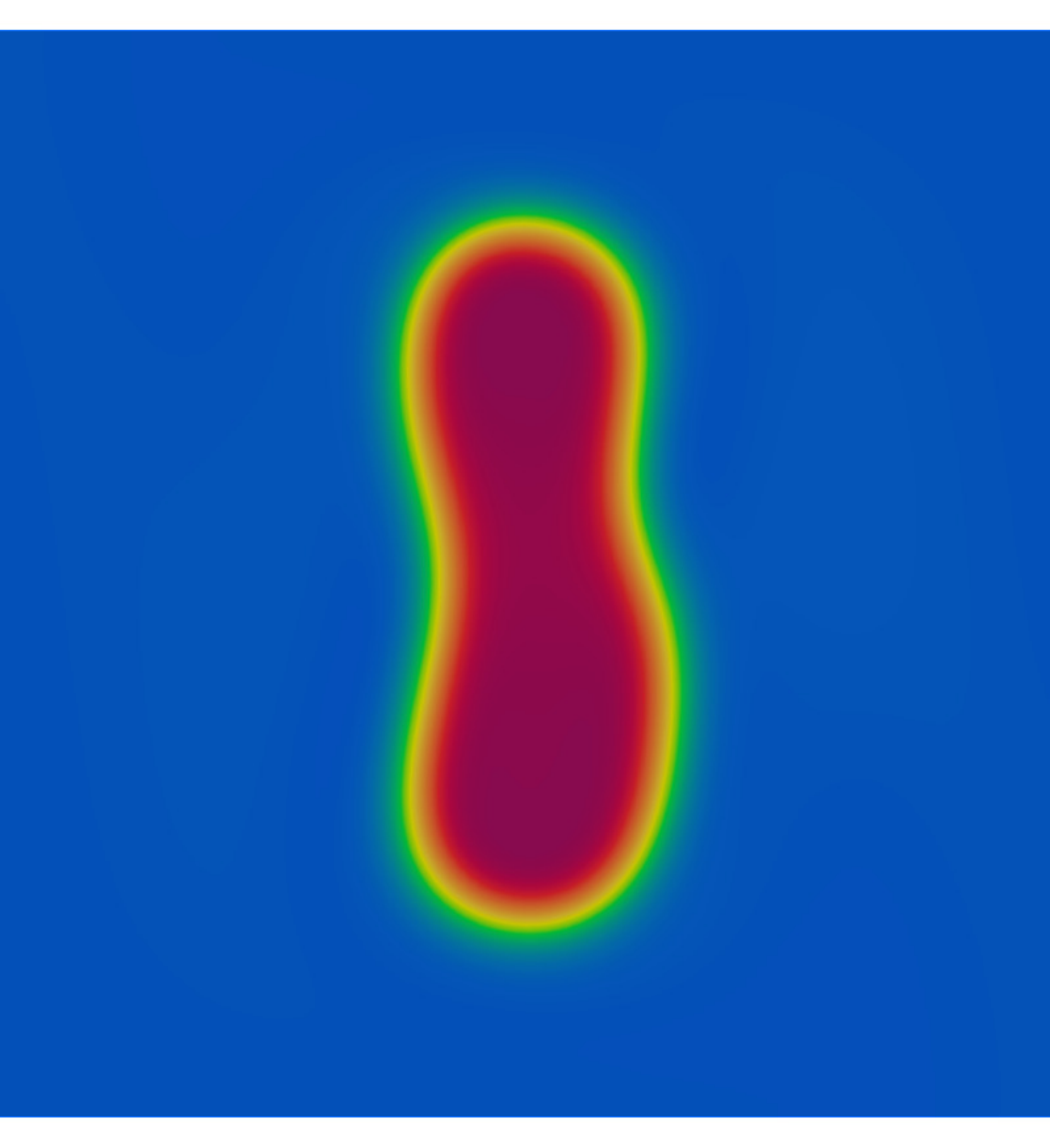}
\includegraphics[scale=0.1125]{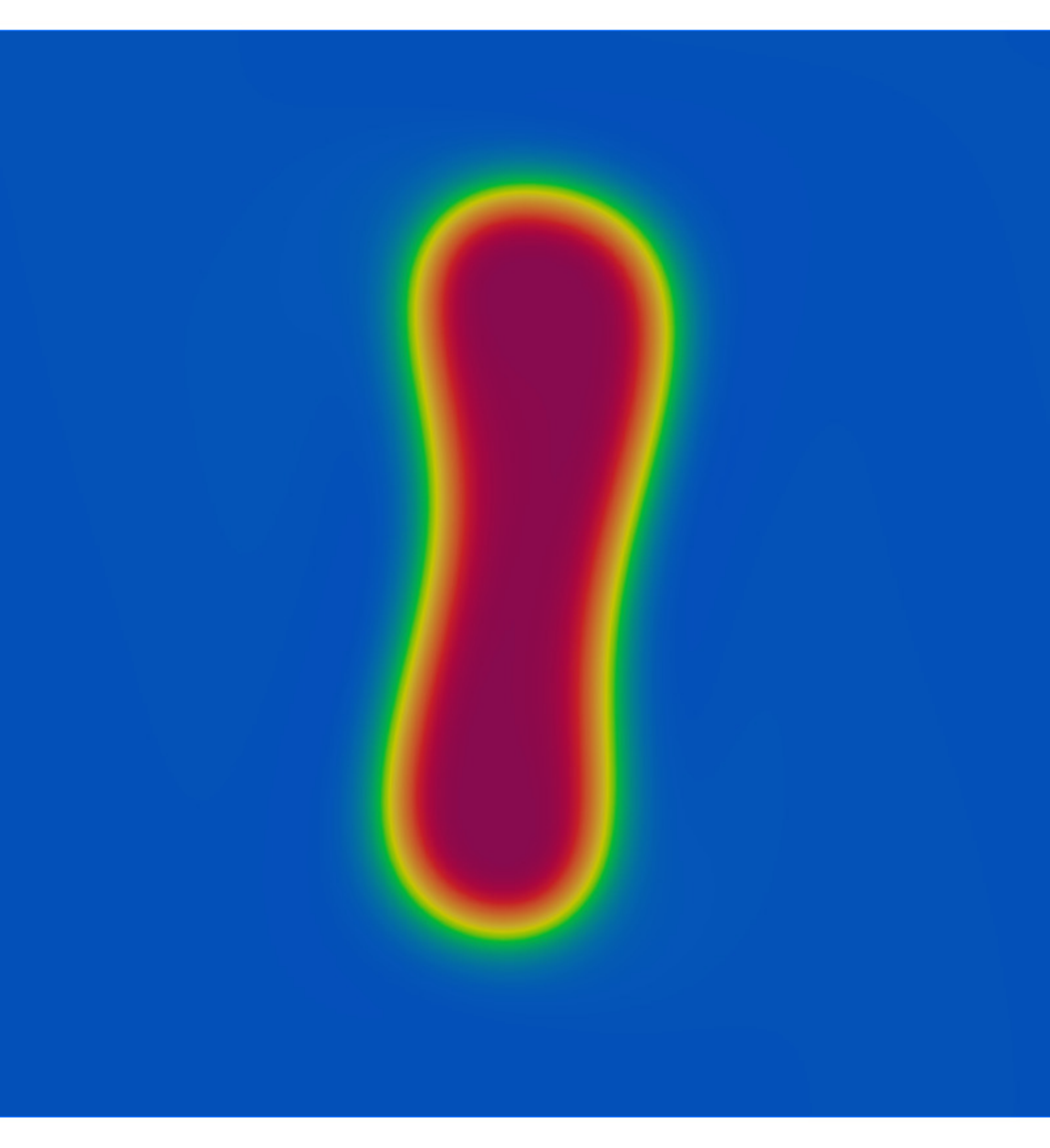}
\includegraphics[scale=0.1125]{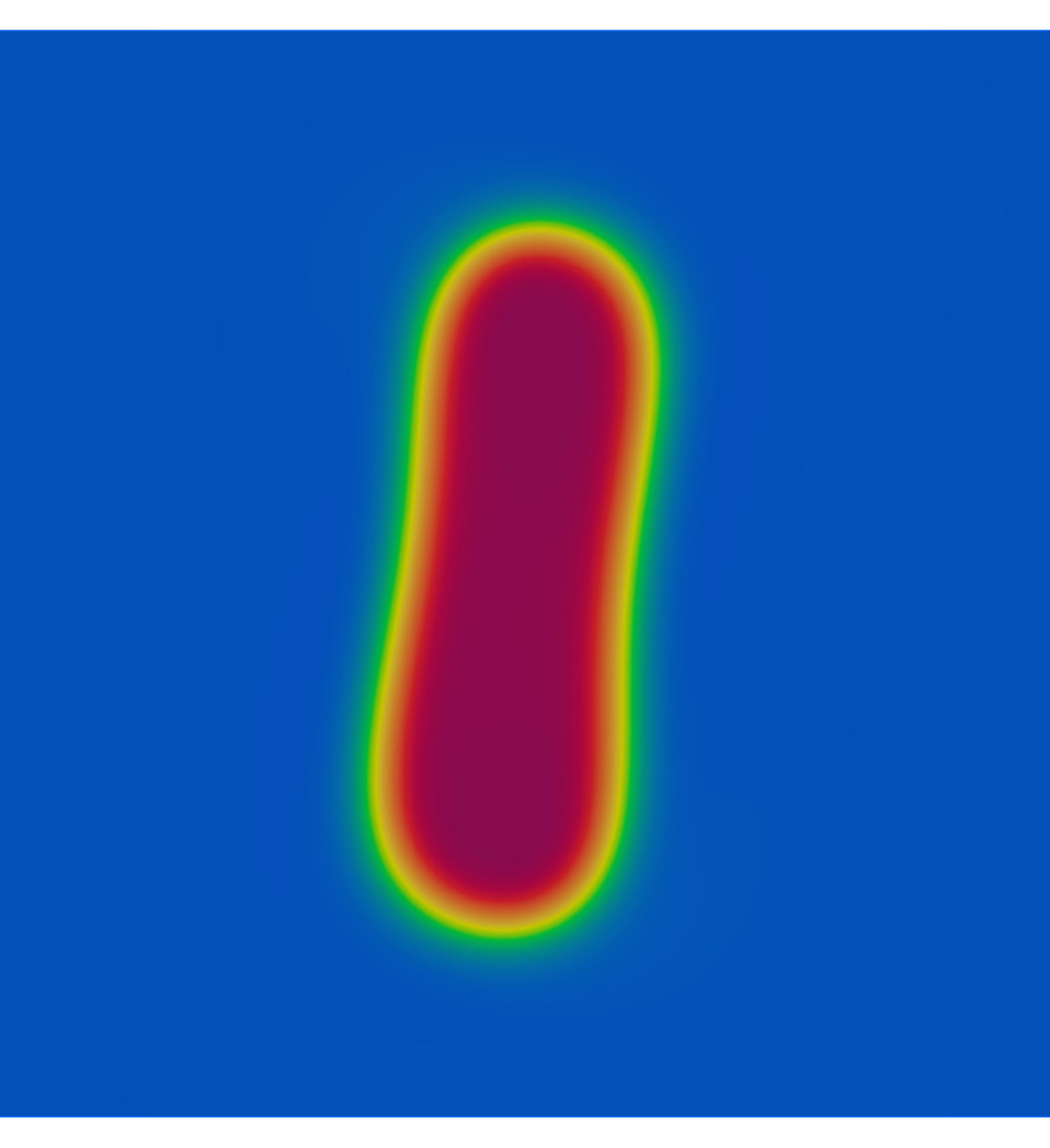}
\includegraphics[scale=0.1125]{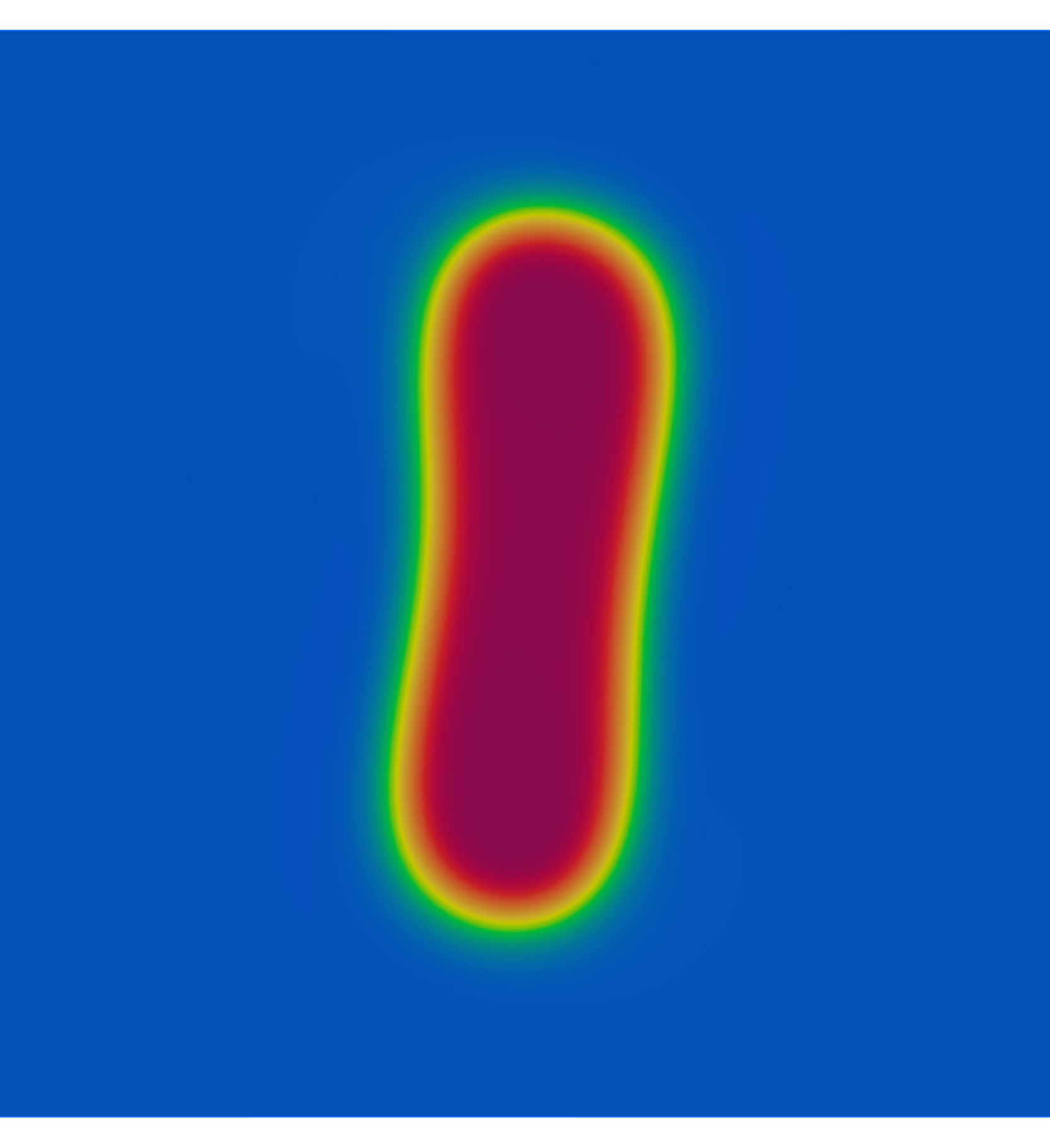}
\includegraphics[scale=0.1125]{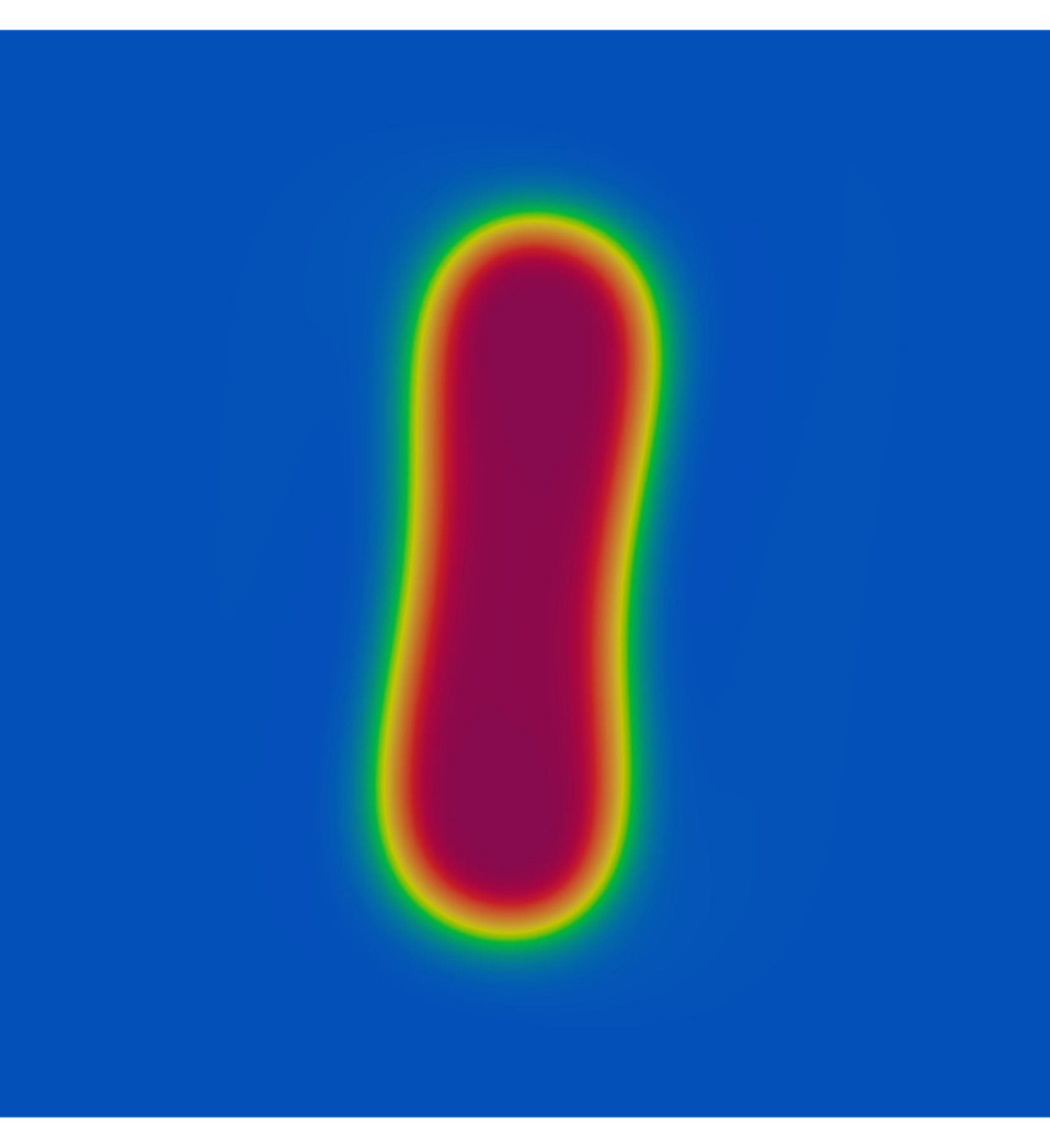}
\end{center}
\caption{Example III. Rotating fluids. Evolution in time of $\phi$ for $J_\varepsilon$-scheme at times $t=0, 0.15, 0.3, 0.45, 0.6, 0.75, 0.9, 1.05, 1.2, 1.35, 1.5, 1.65, 1.8, 1.95, 2.1, 2.5, 3.0, 3.5, 4.25$ and $5$.}\label{fig:ExIIIDynJ}
\end{figure}

\begin{figure}[H]
\begin{center}
\includegraphics[scale=0.1125]{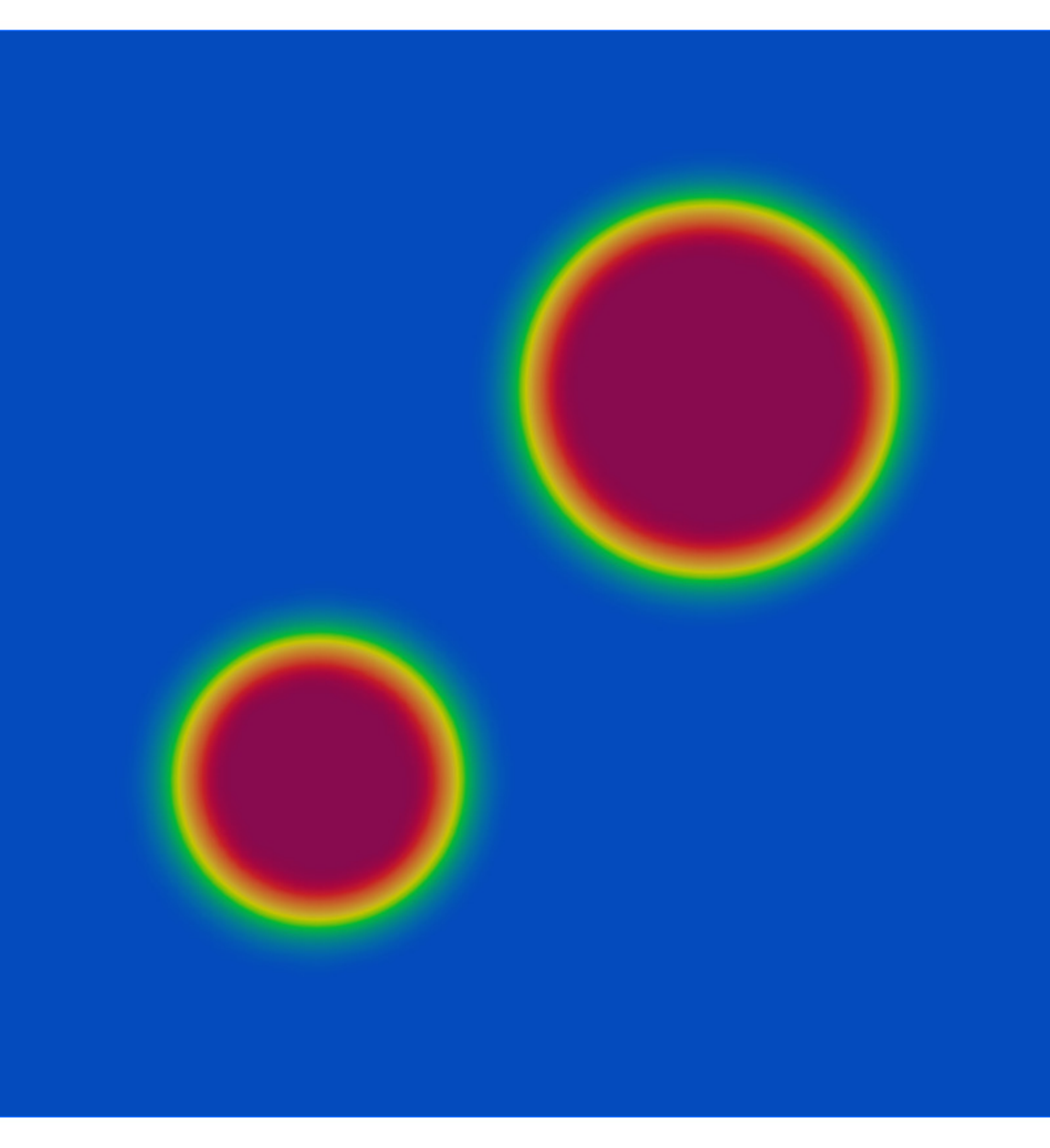}
\includegraphics[scale=0.1125]{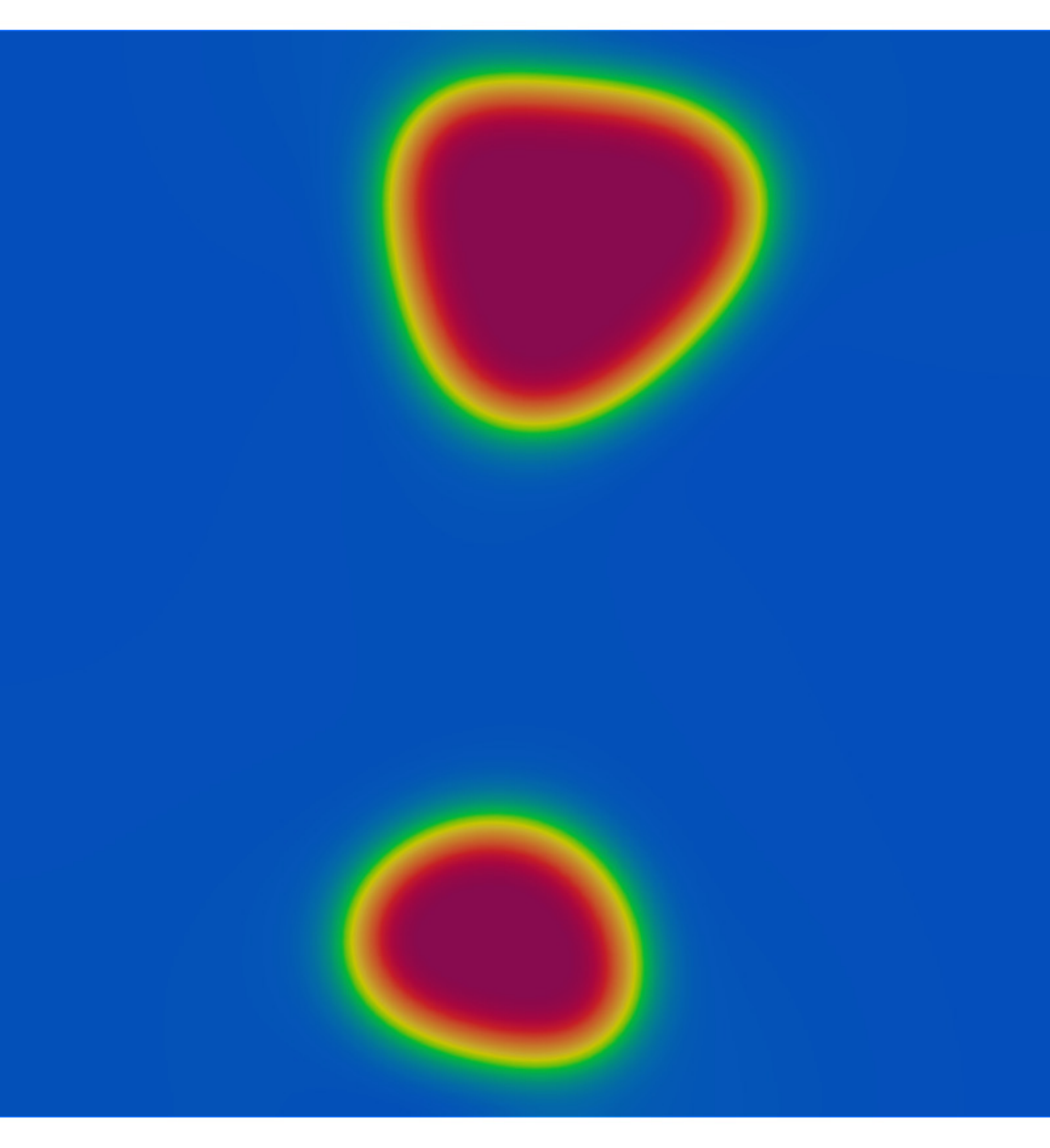}
\includegraphics[scale=0.1125]{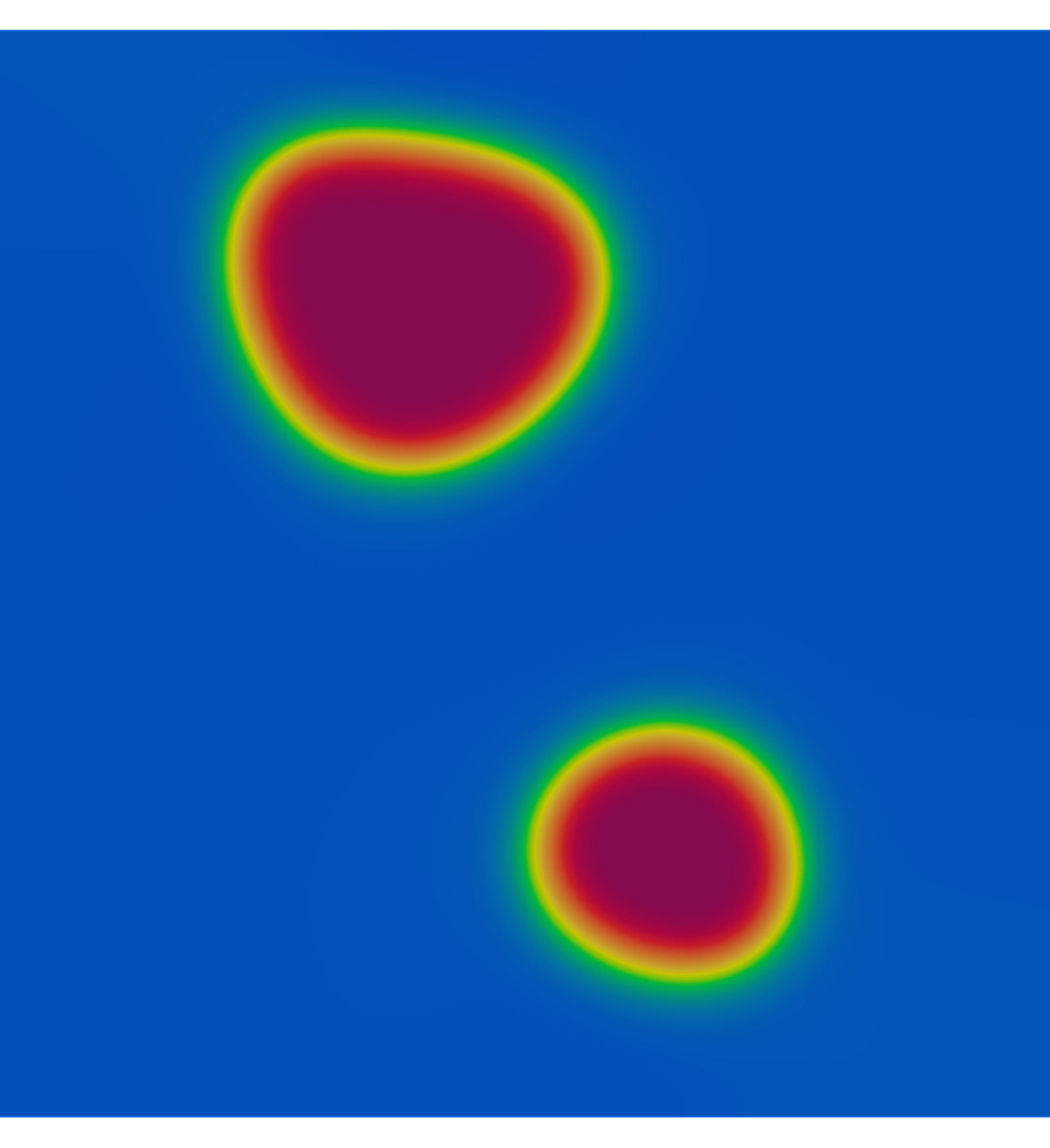}
\includegraphics[scale=0.1125]{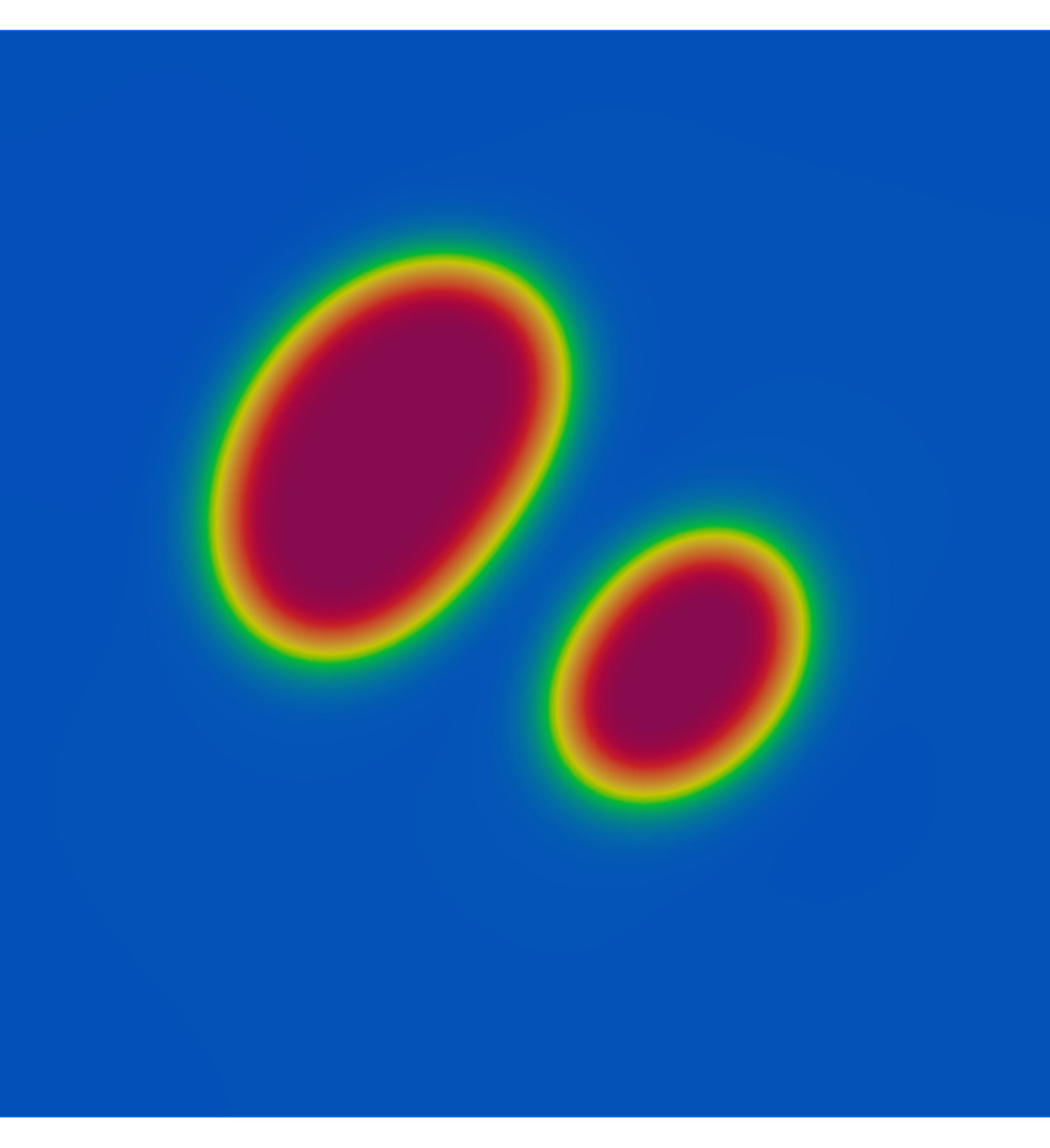}
\includegraphics[scale=0.1125]{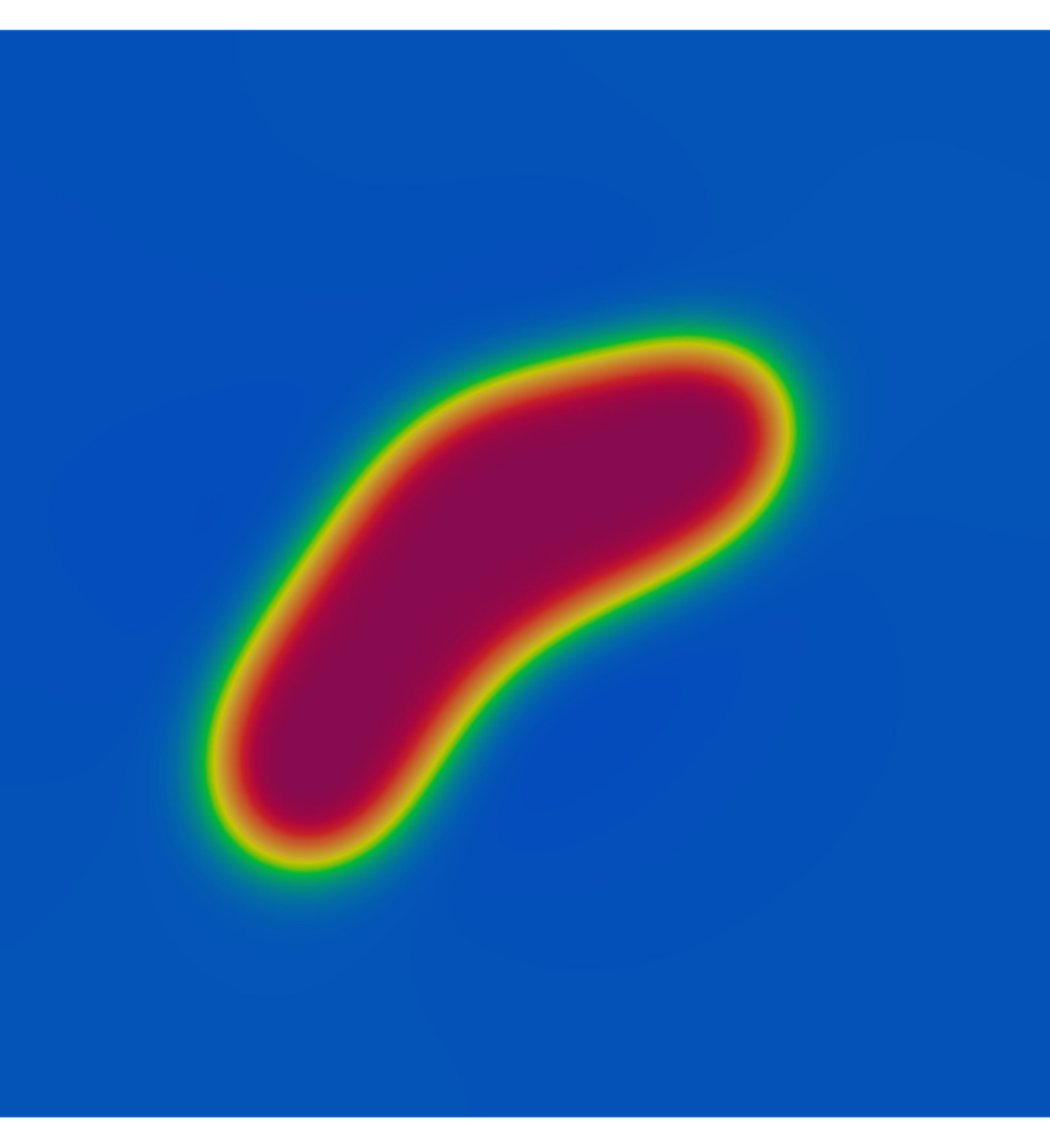}
\\
\includegraphics[scale=0.1125]{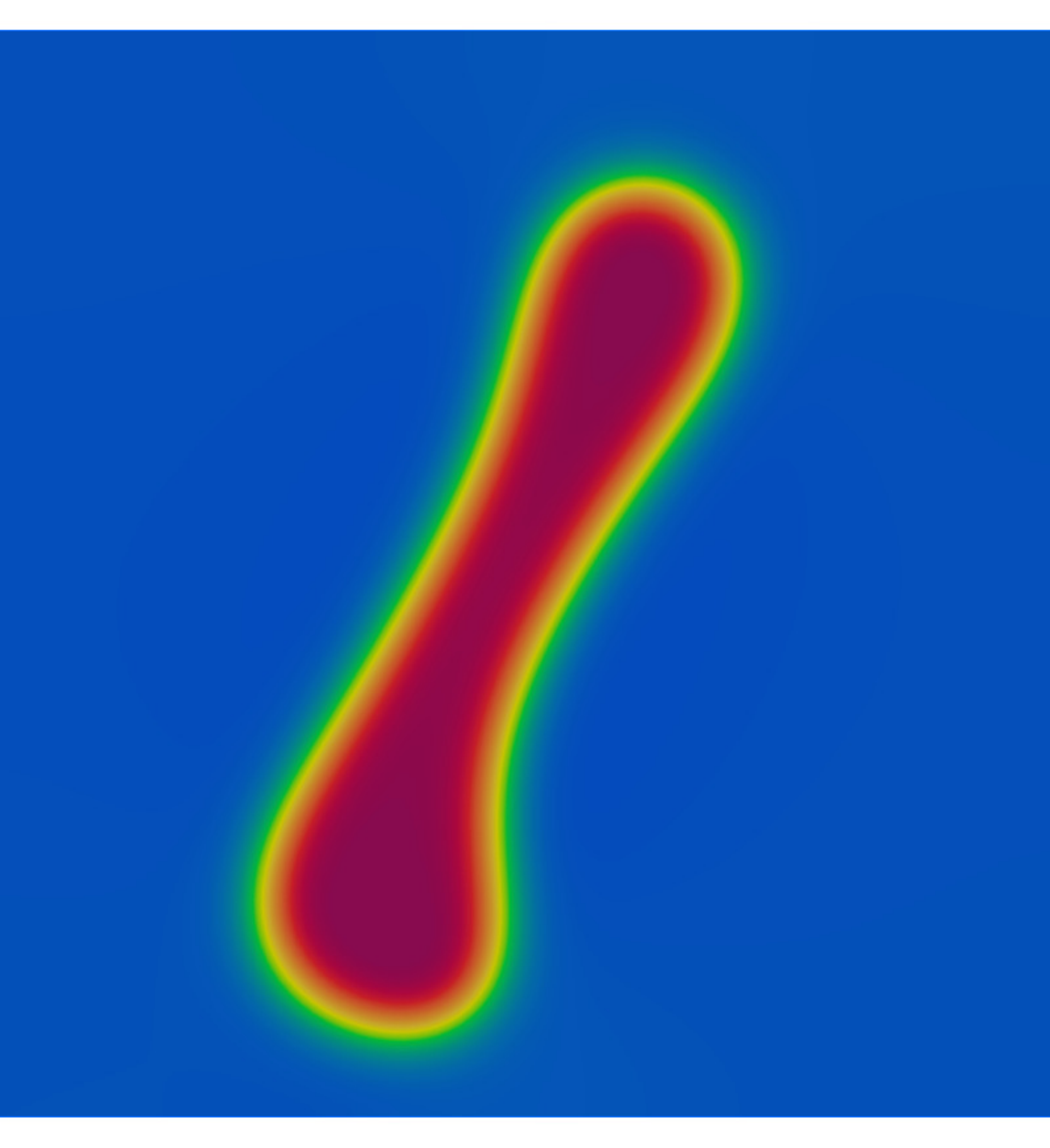}
\includegraphics[scale=0.1125]{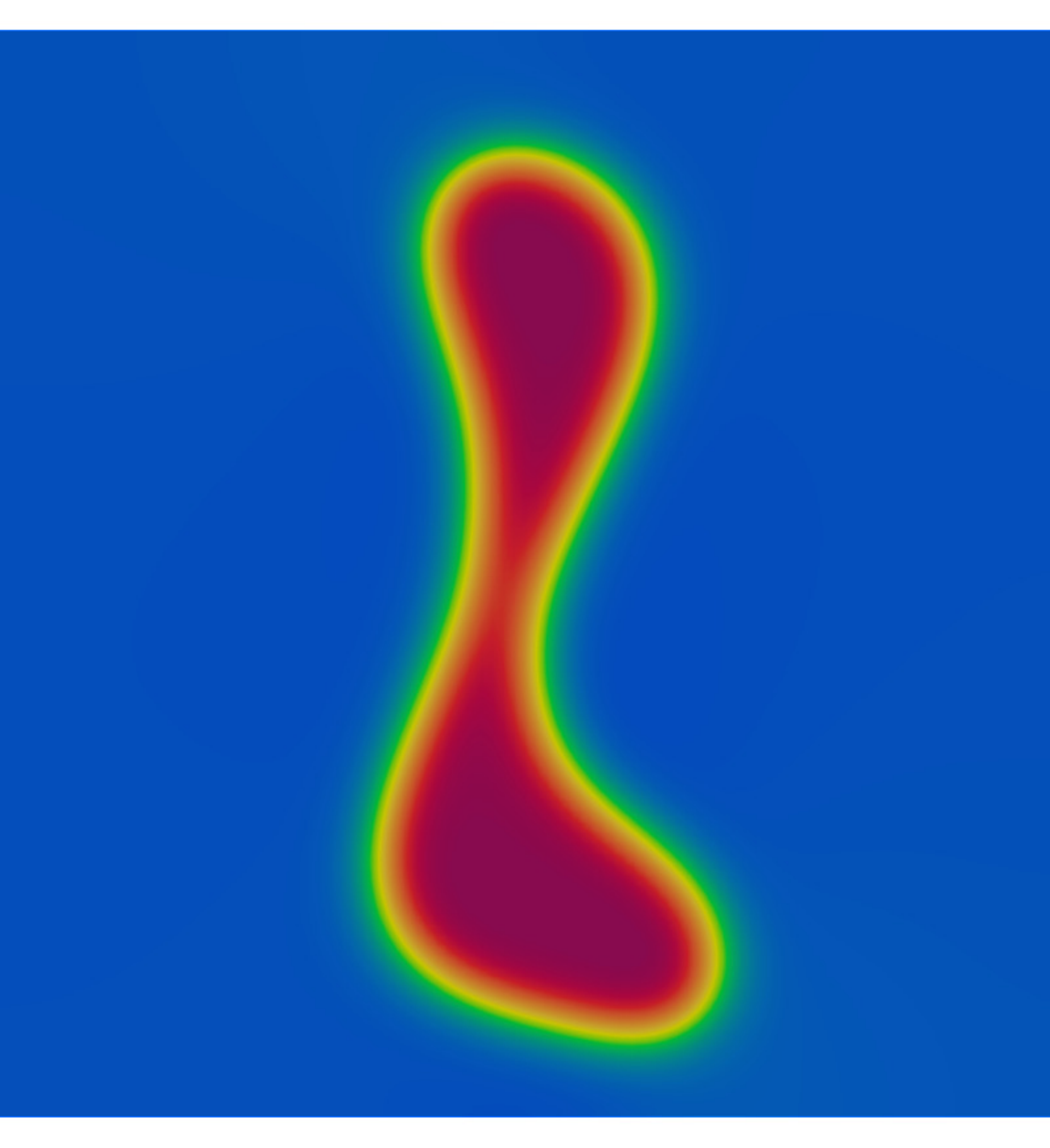}
\includegraphics[scale=0.1125]{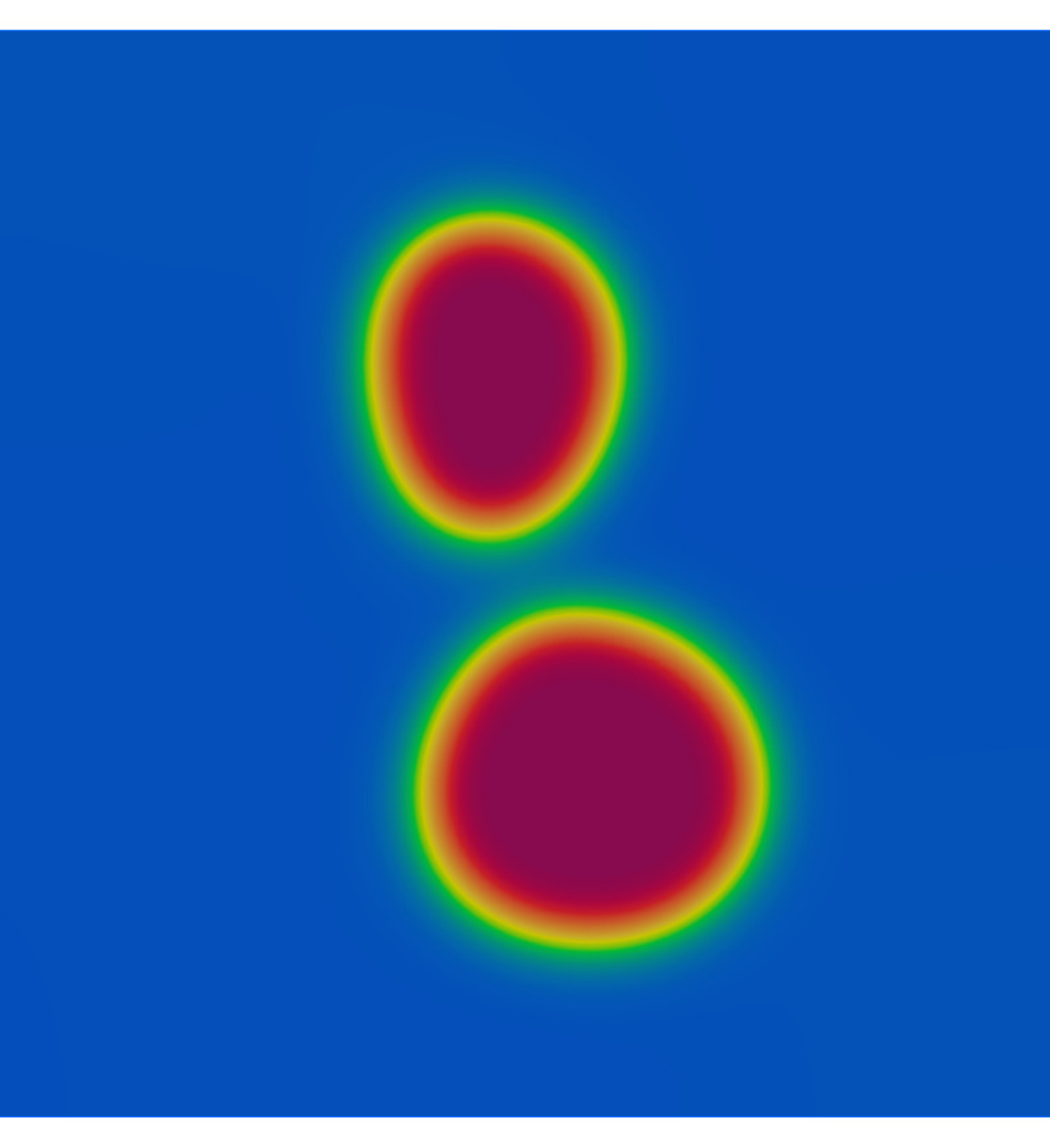}
\includegraphics[scale=0.1125]{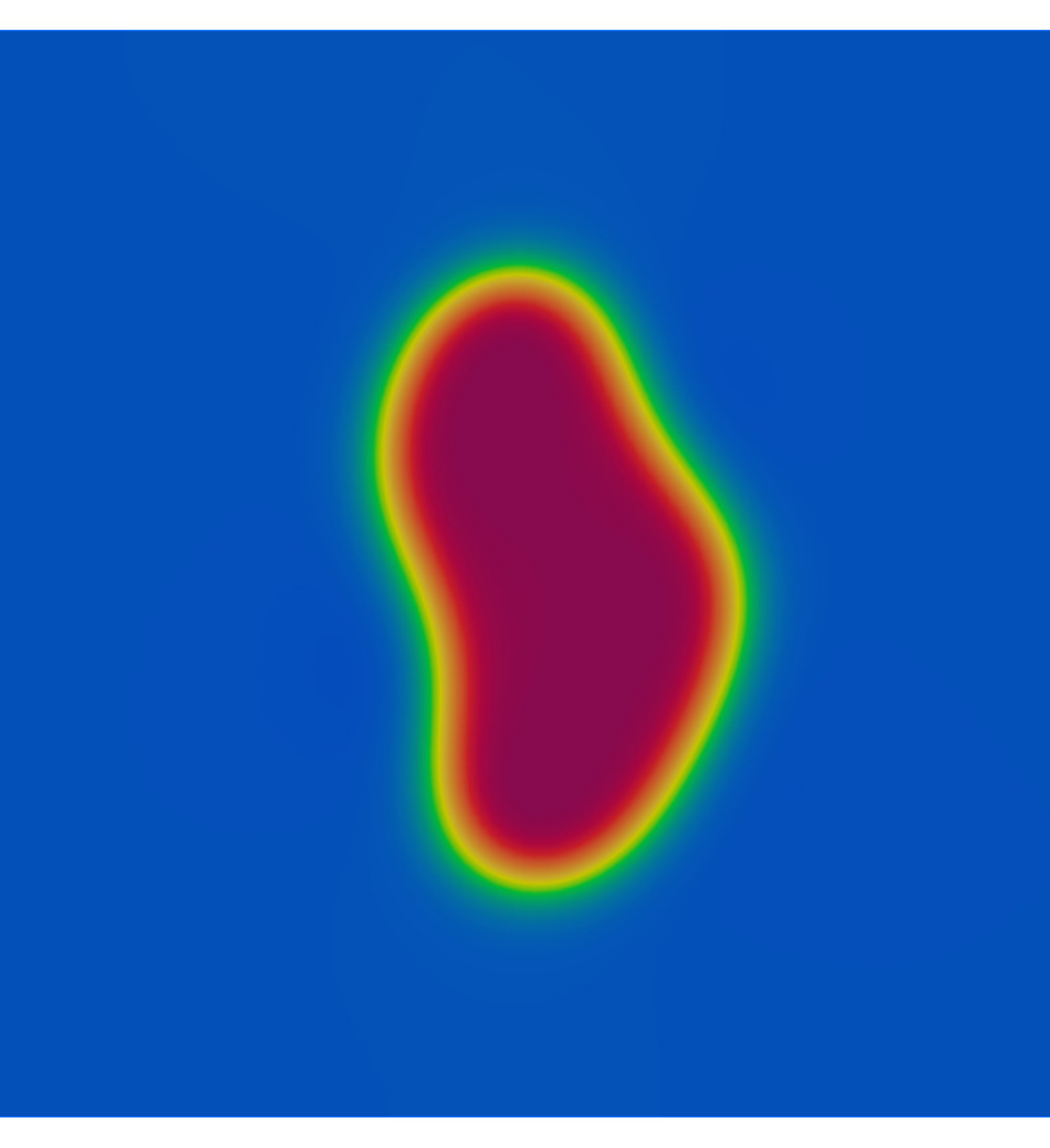}
\includegraphics[scale=0.1125]{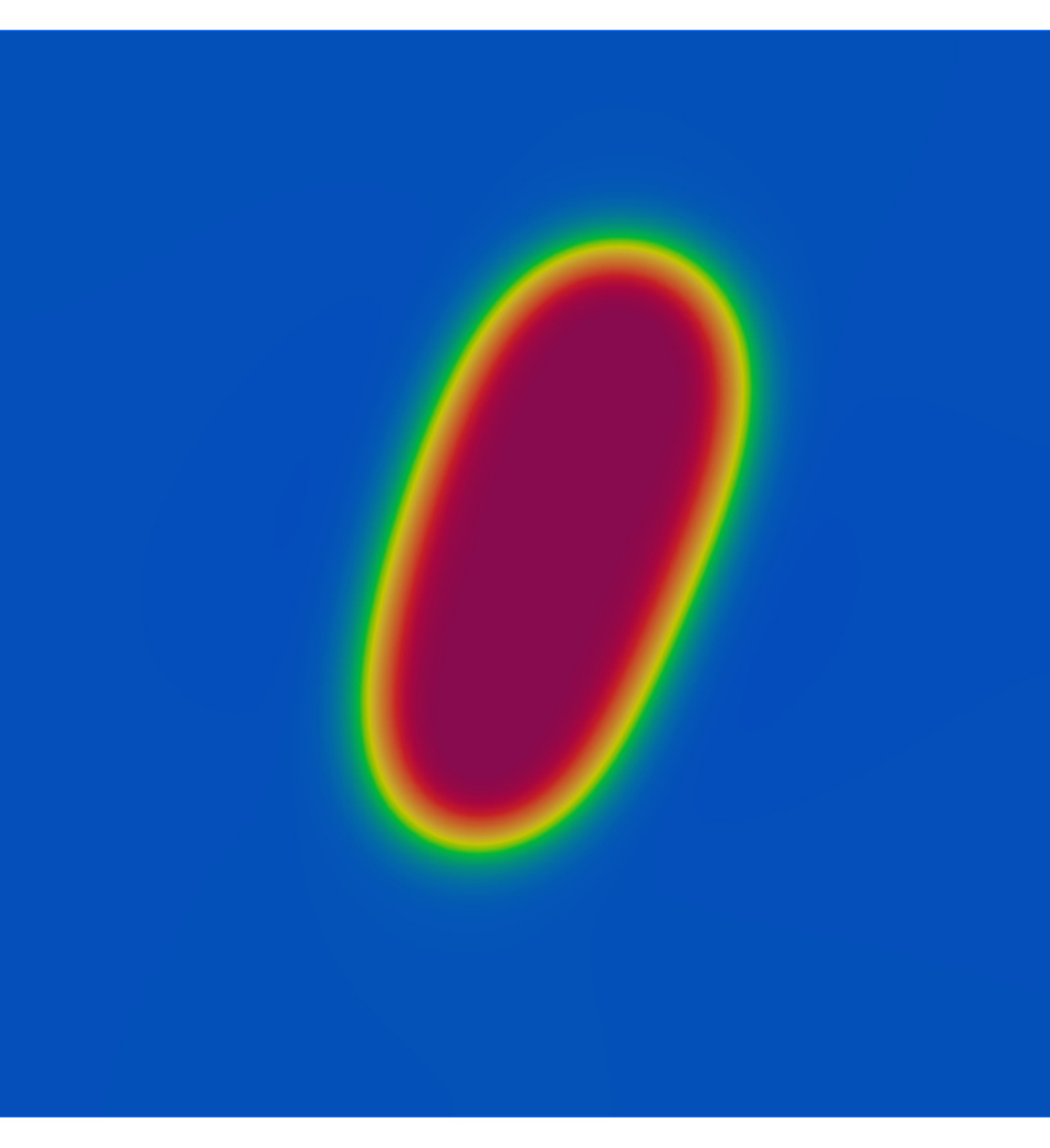}
\\
\includegraphics[scale=0.1125]{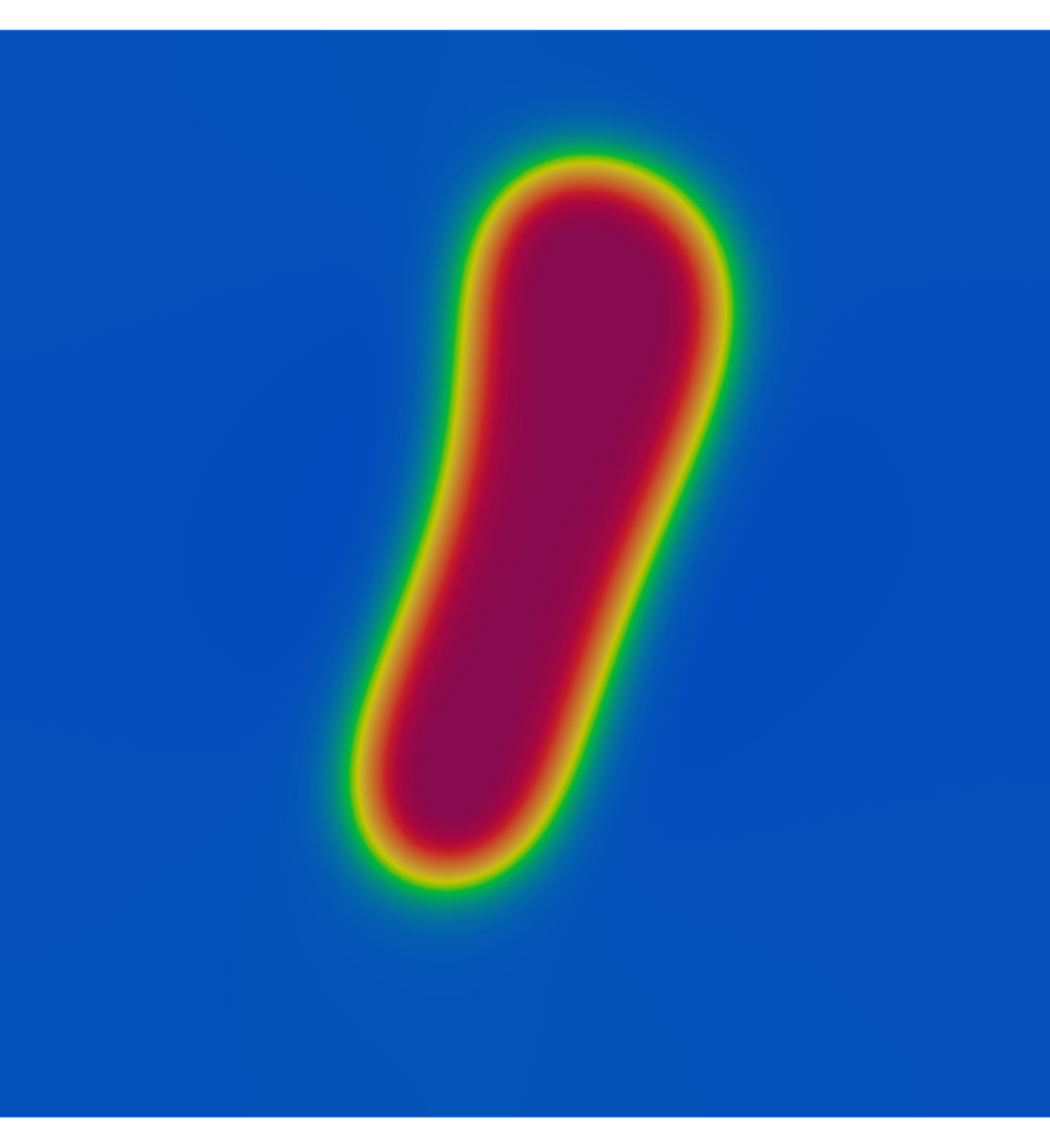}
\includegraphics[scale=0.1125]{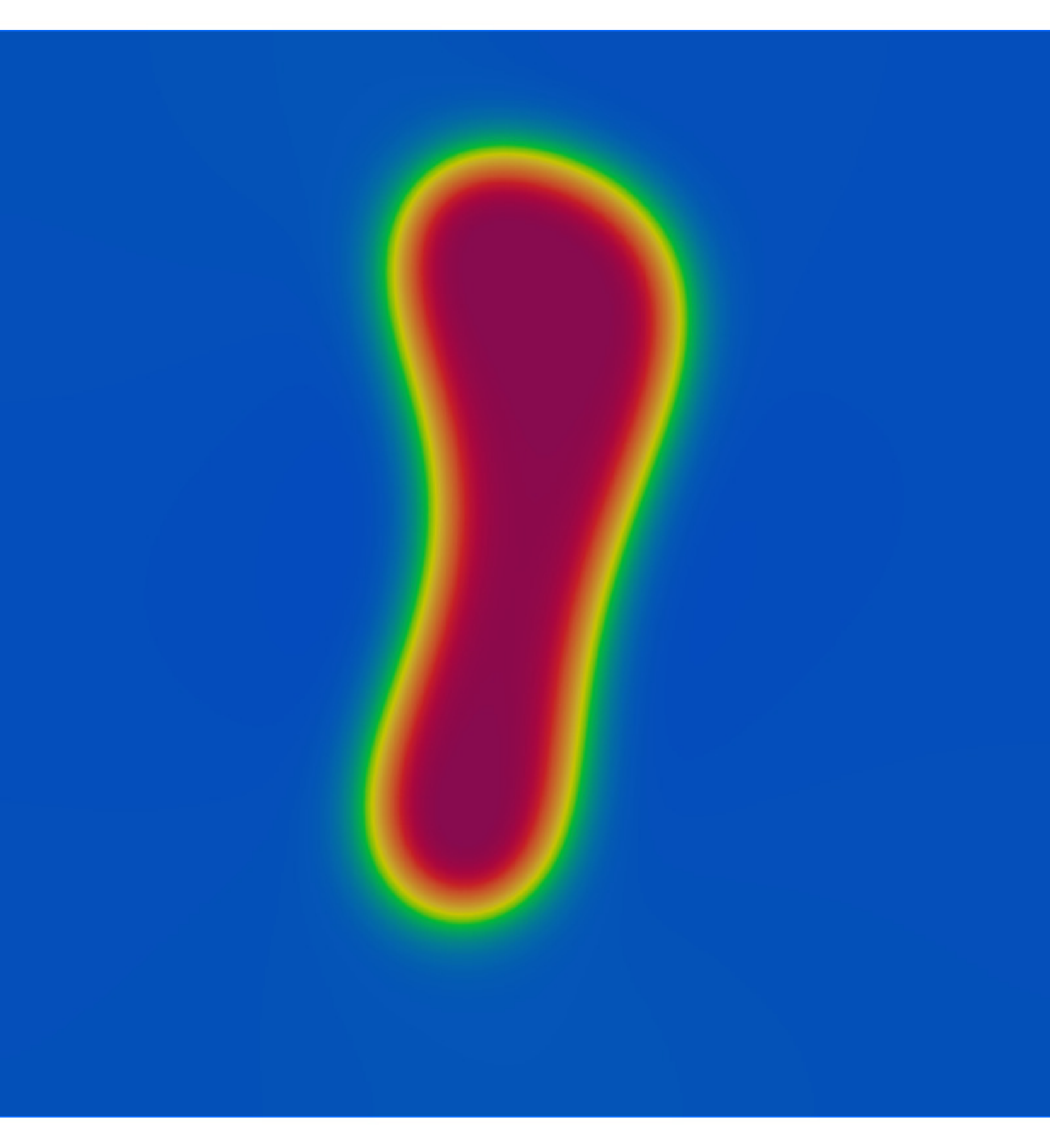}
\includegraphics[scale=0.1125]{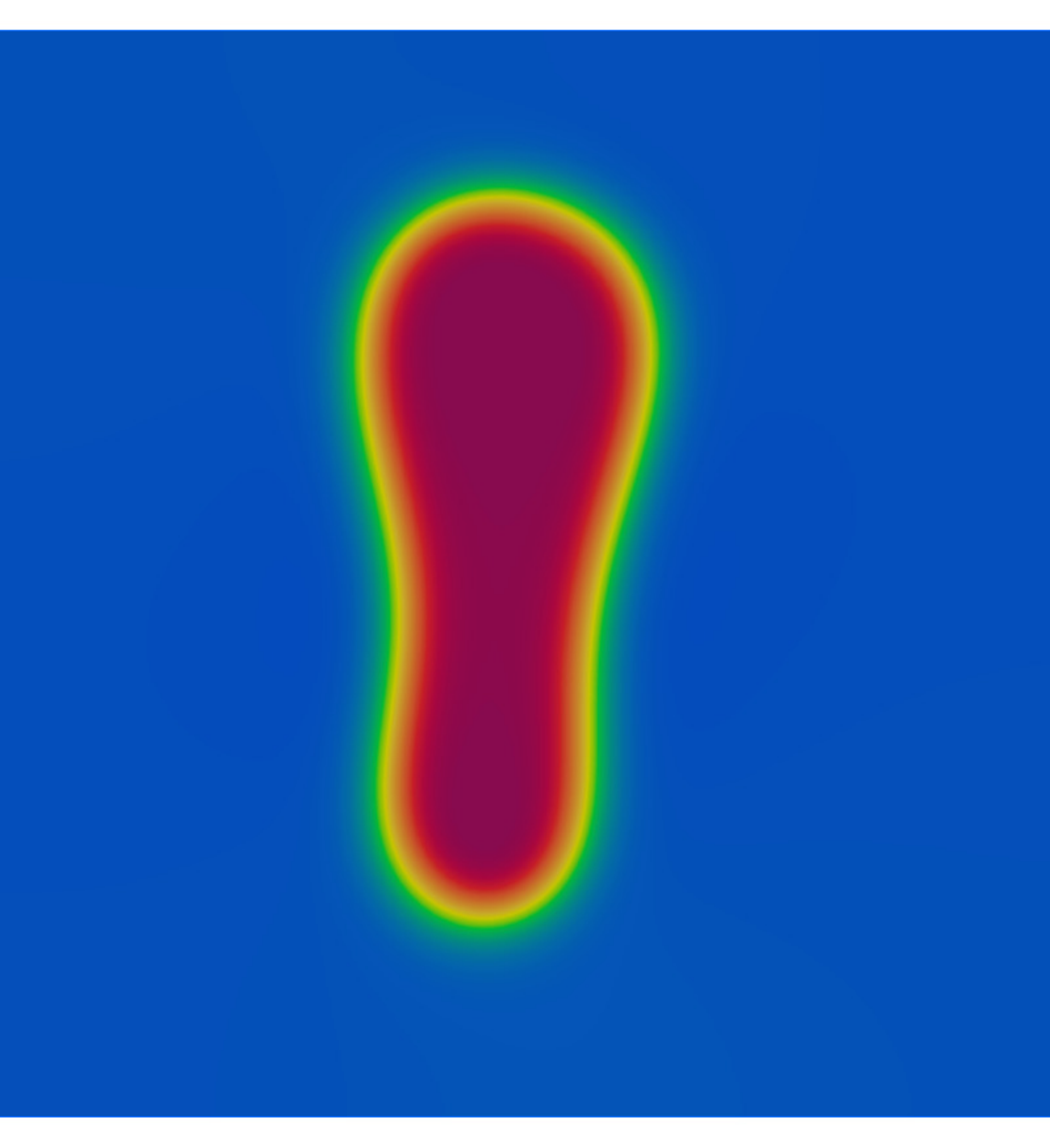}
\includegraphics[scale=0.1125]{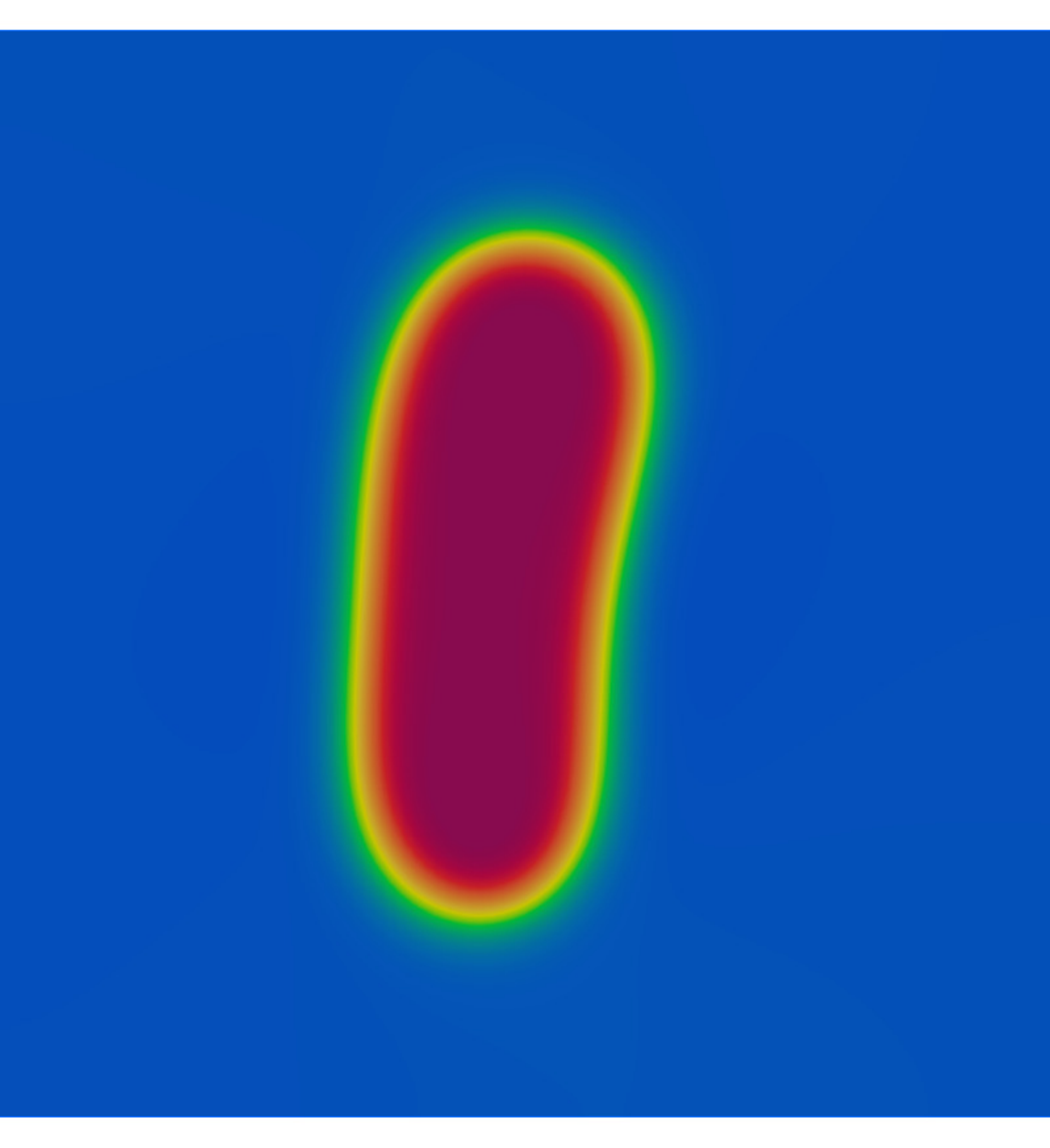}
\includegraphics[scale=0.1125]{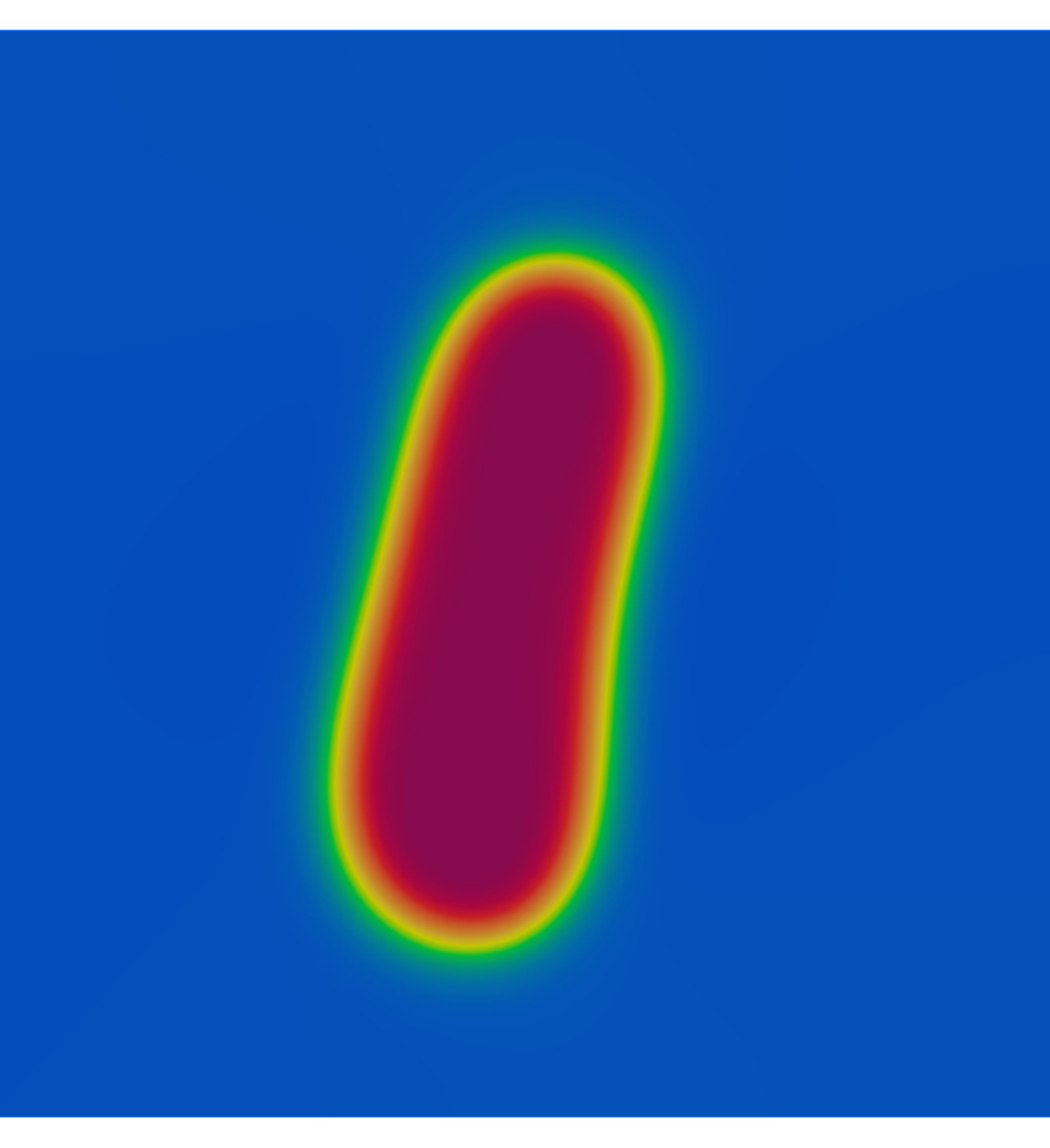}
\\
\includegraphics[scale=0.1125]{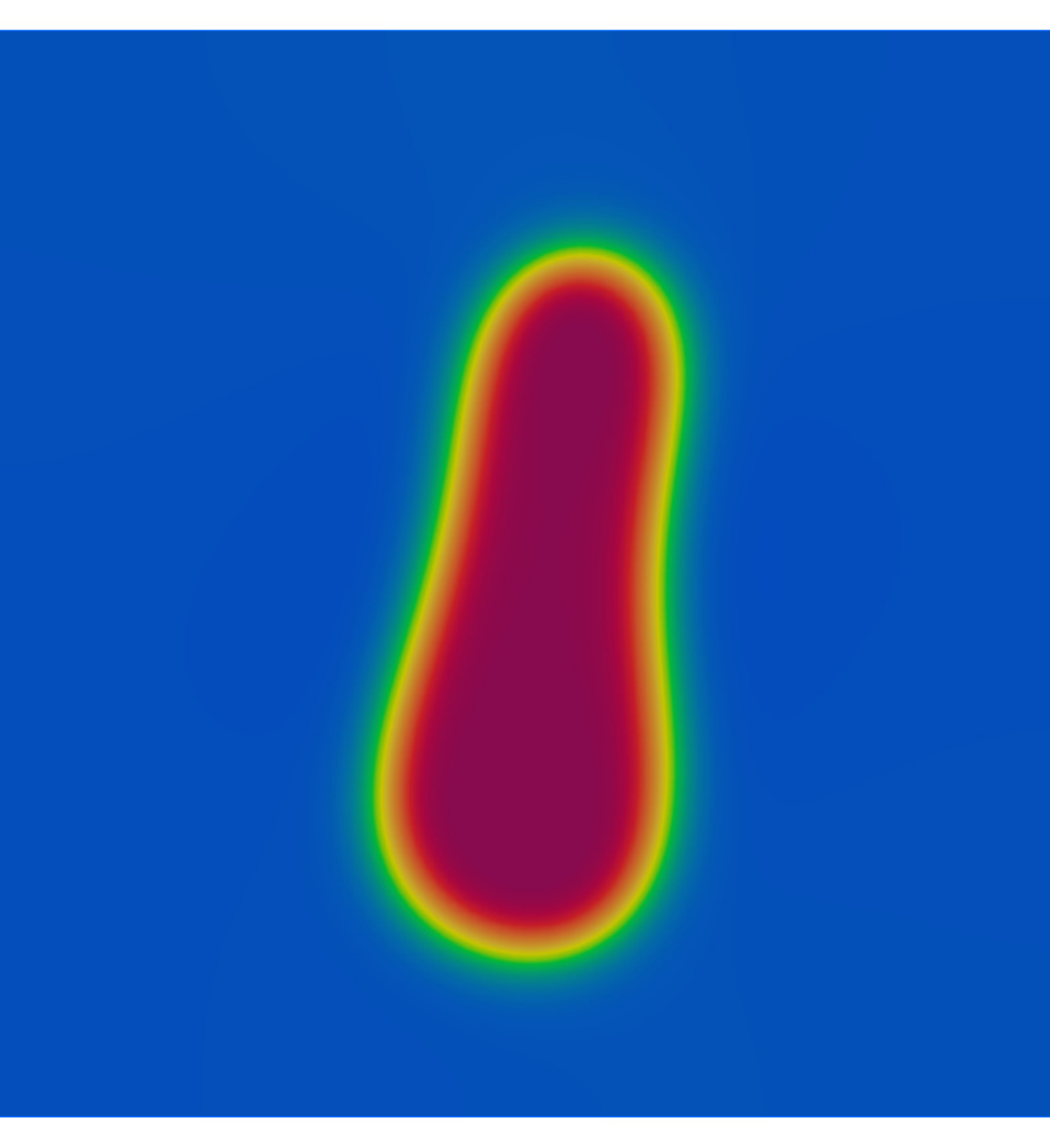}
\includegraphics[scale=0.1125]{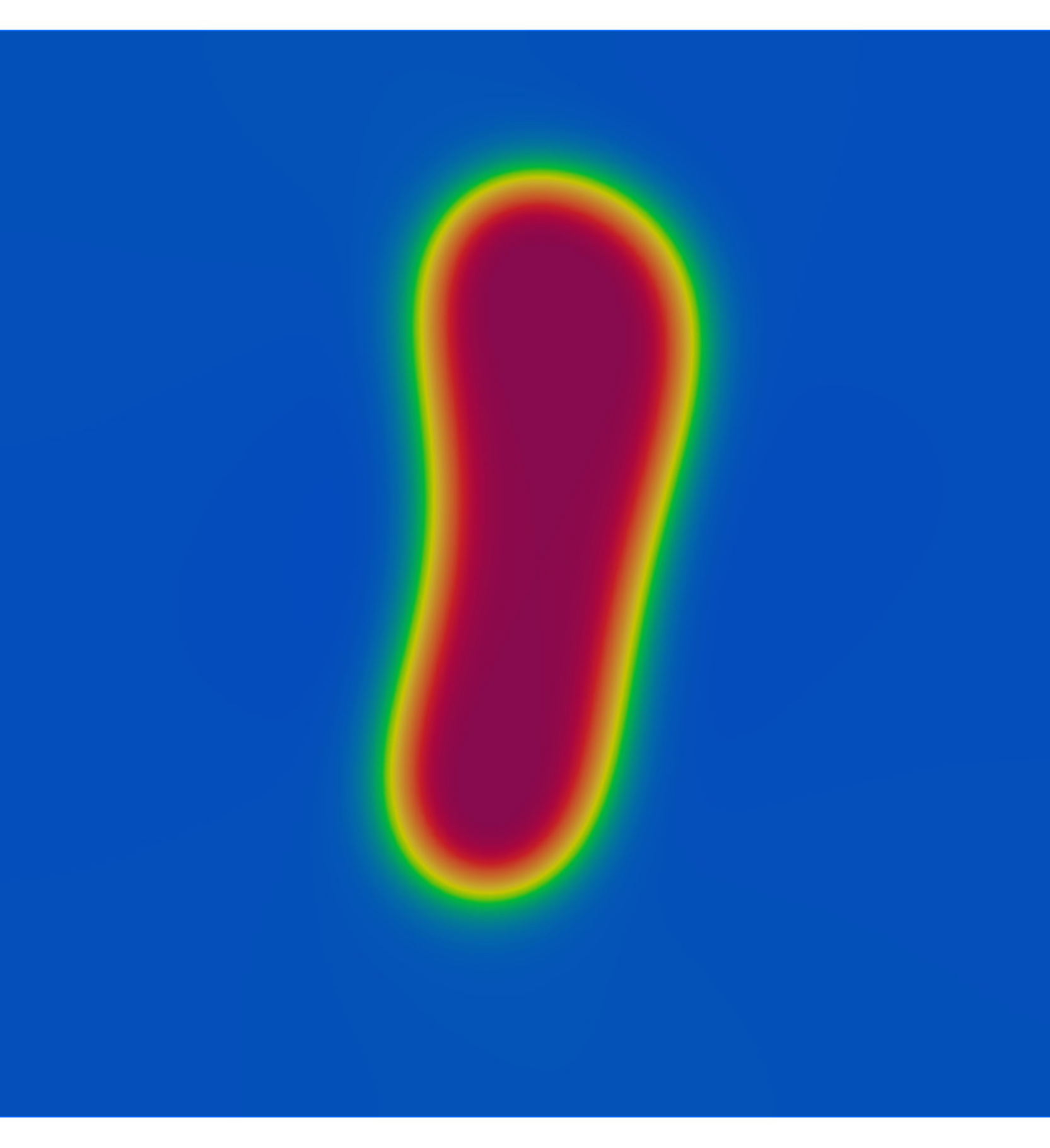}
\includegraphics[scale=0.1125]{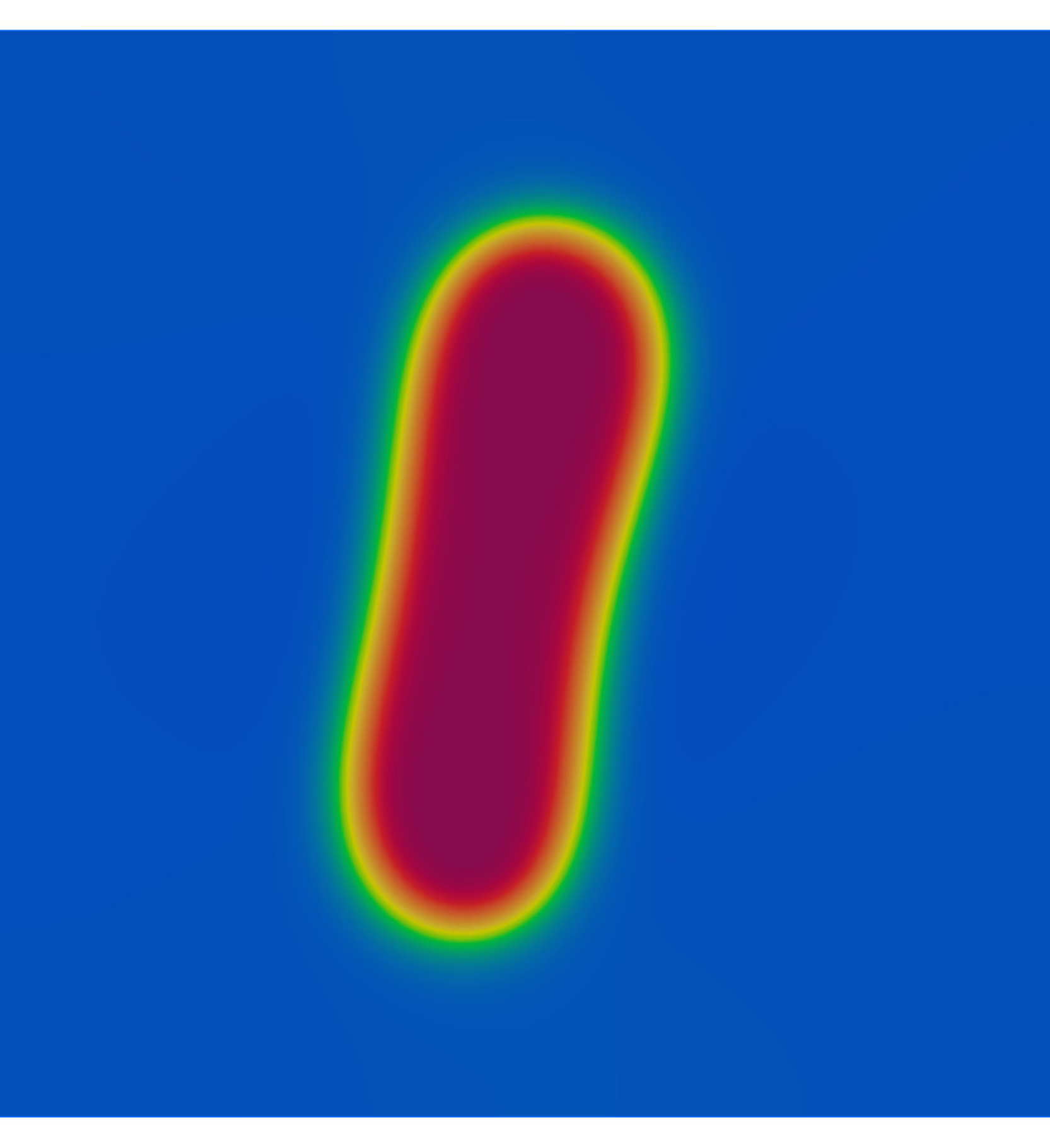}
\includegraphics[scale=0.1125]{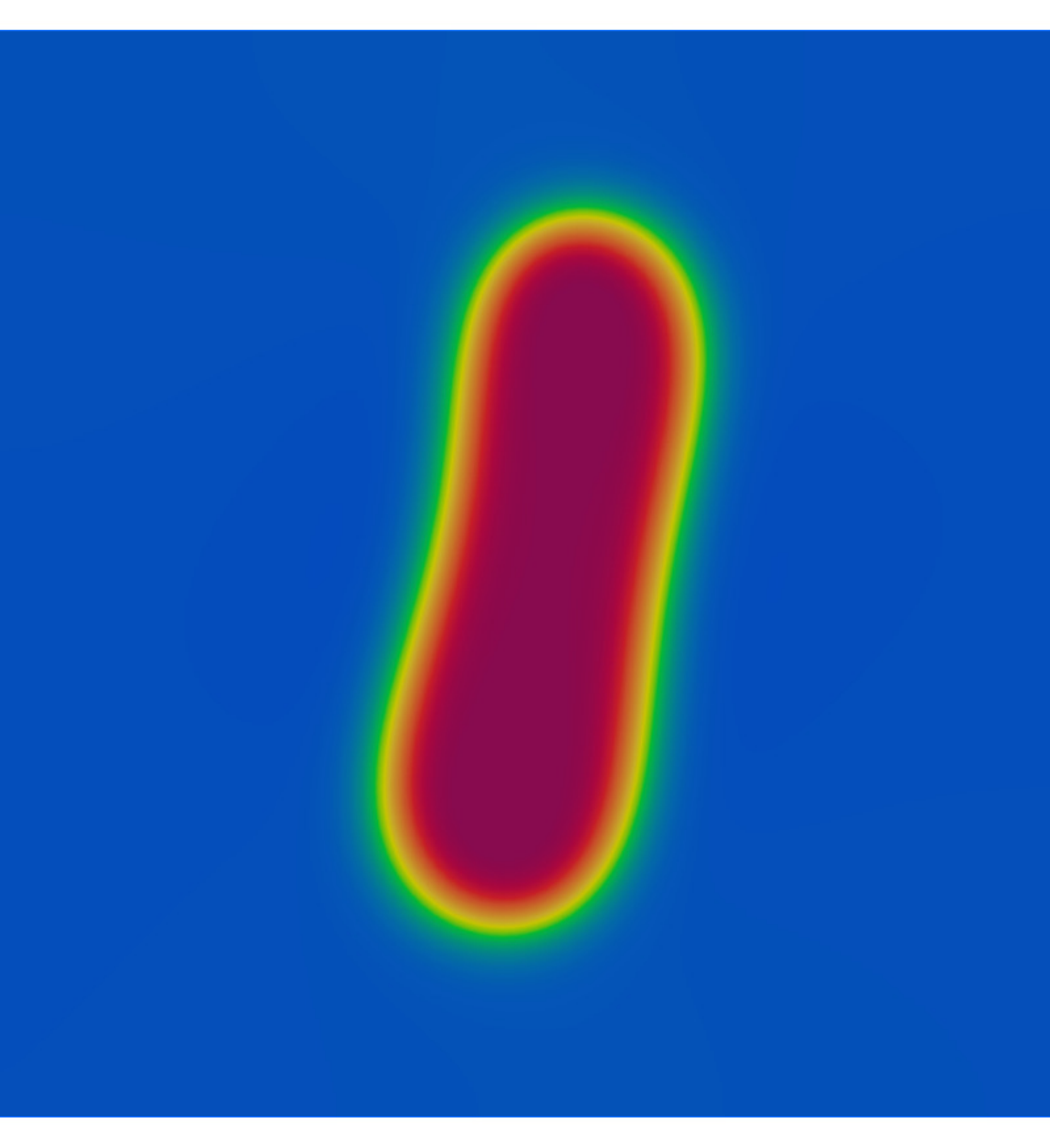}
\includegraphics[scale=0.1125]{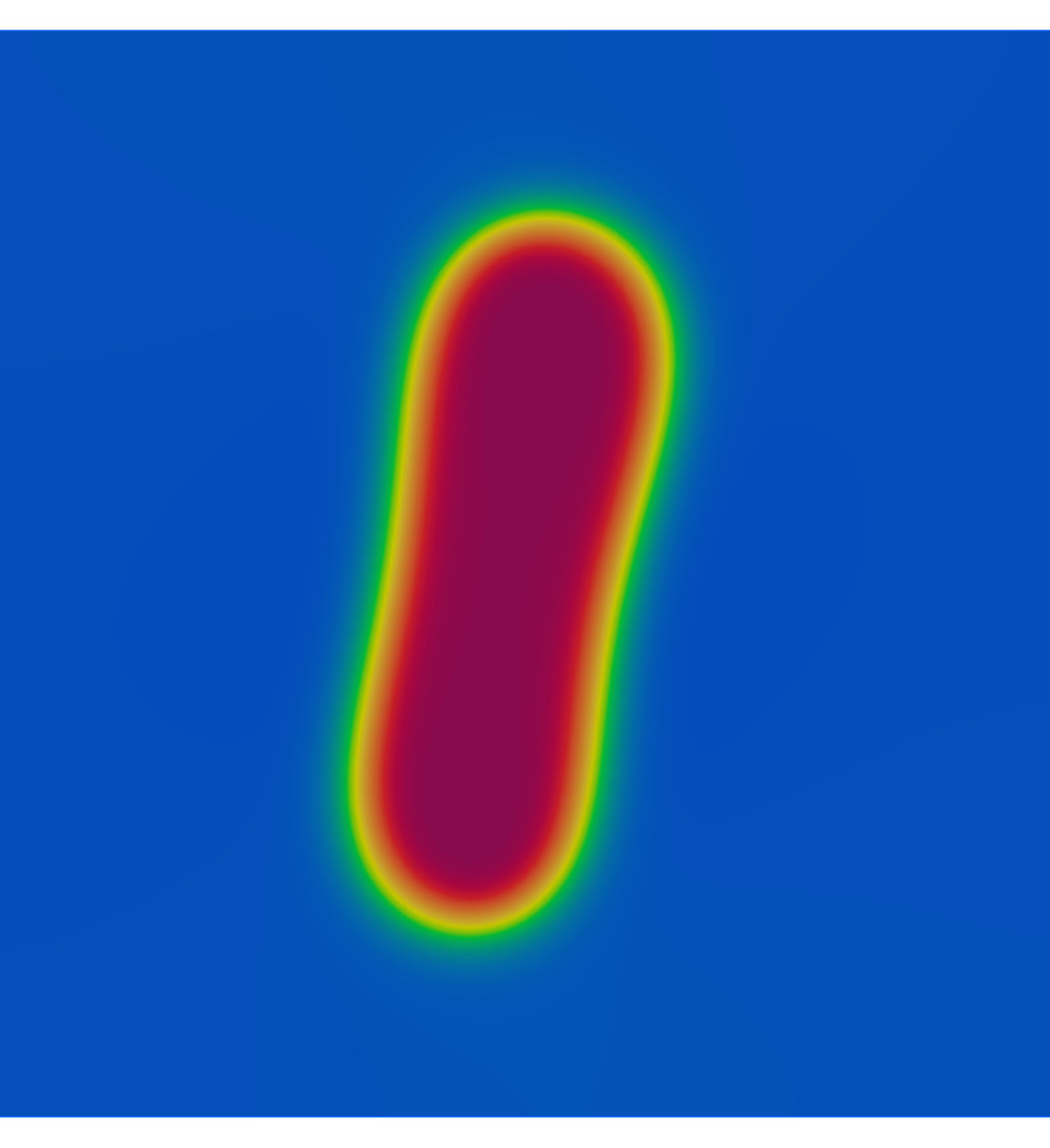}
\end{center}
\caption{Example III. Rotating fluids. Evolution in time of $\phi$ for CM-scheme at times $t=0, 0.15, 0.3, 0.45, 0.6, 0.75, 0.9, 1.05, 1.2, 1.35, 1.5, 1.65, 1.8, 1.95, 2.1, 2.5, 3.0, 3.5, 4.25$ and $5$.}\label{fig:ExIIIDynCM}
\end{figure}

We present the evolution in time of the different energies and the evolution of the volume for this example in Figure~\ref{fig:Ex3_energies}.  
As expected, the volume is conserved in the three cases but in this example the total energy is not decreasing in time due to the introduction of an external force in the system. Moreover, we present information about the evolution in time of the bounds of $\phi$ in Figure~\ref{fig:Ex3_bounds}. As expected, CM-scheme does not enforce the variable $\phi$ to remain in the interval $[0,1]$ while the $G_\varepsilon$-scheme and $J_\varepsilon$-scheme are very close to achieving it, even in this challenging setting when there is a lot of movement in the system. As in the previous examples, this is a consequence of Remarks~\ref{rem:boundG} and \ref{rem:boundJ}.

\begin{figure}[H]
\begin{center}
\includegraphics[scale=0.1125]{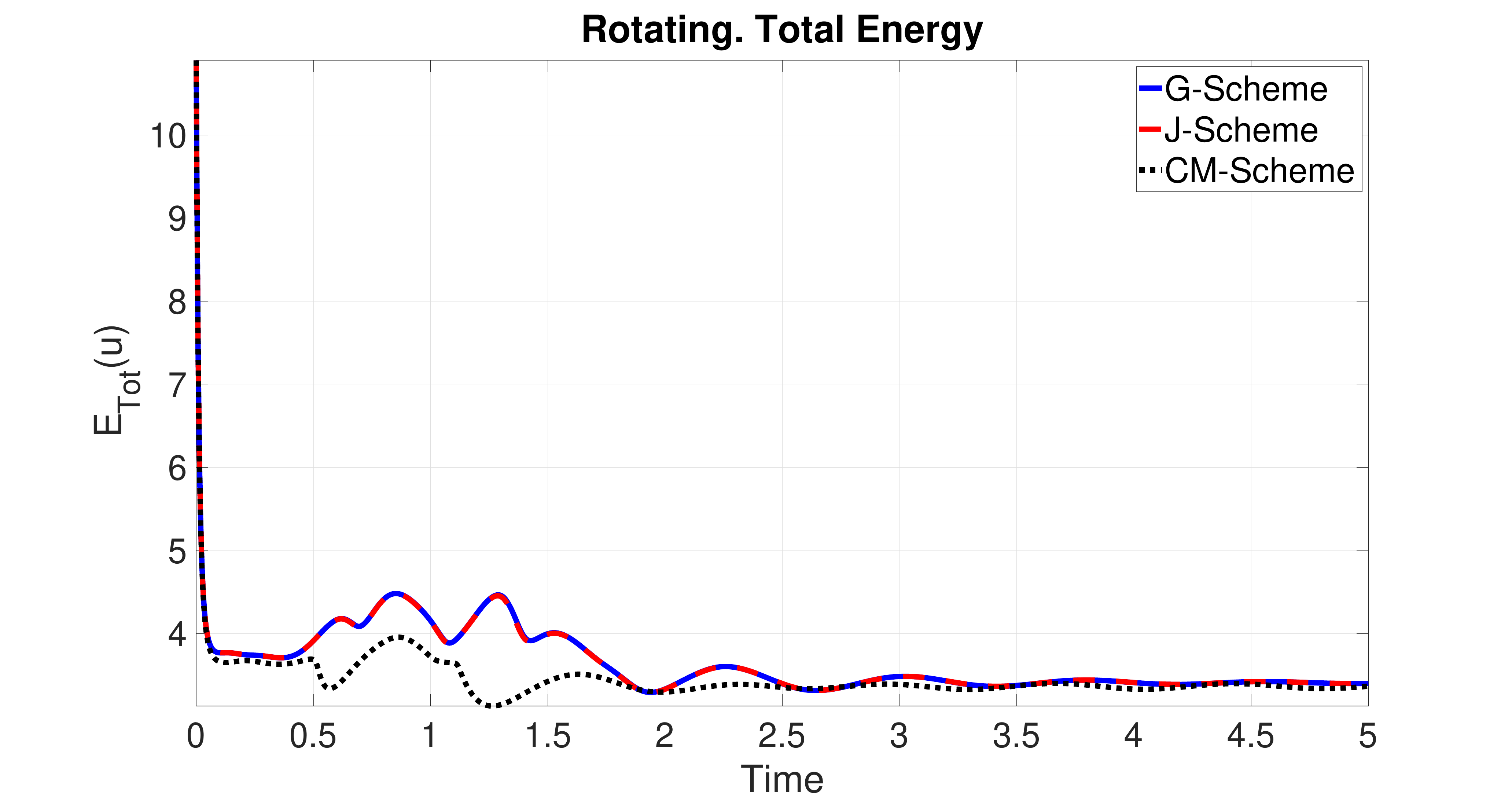}
\includegraphics[scale=0.1125]{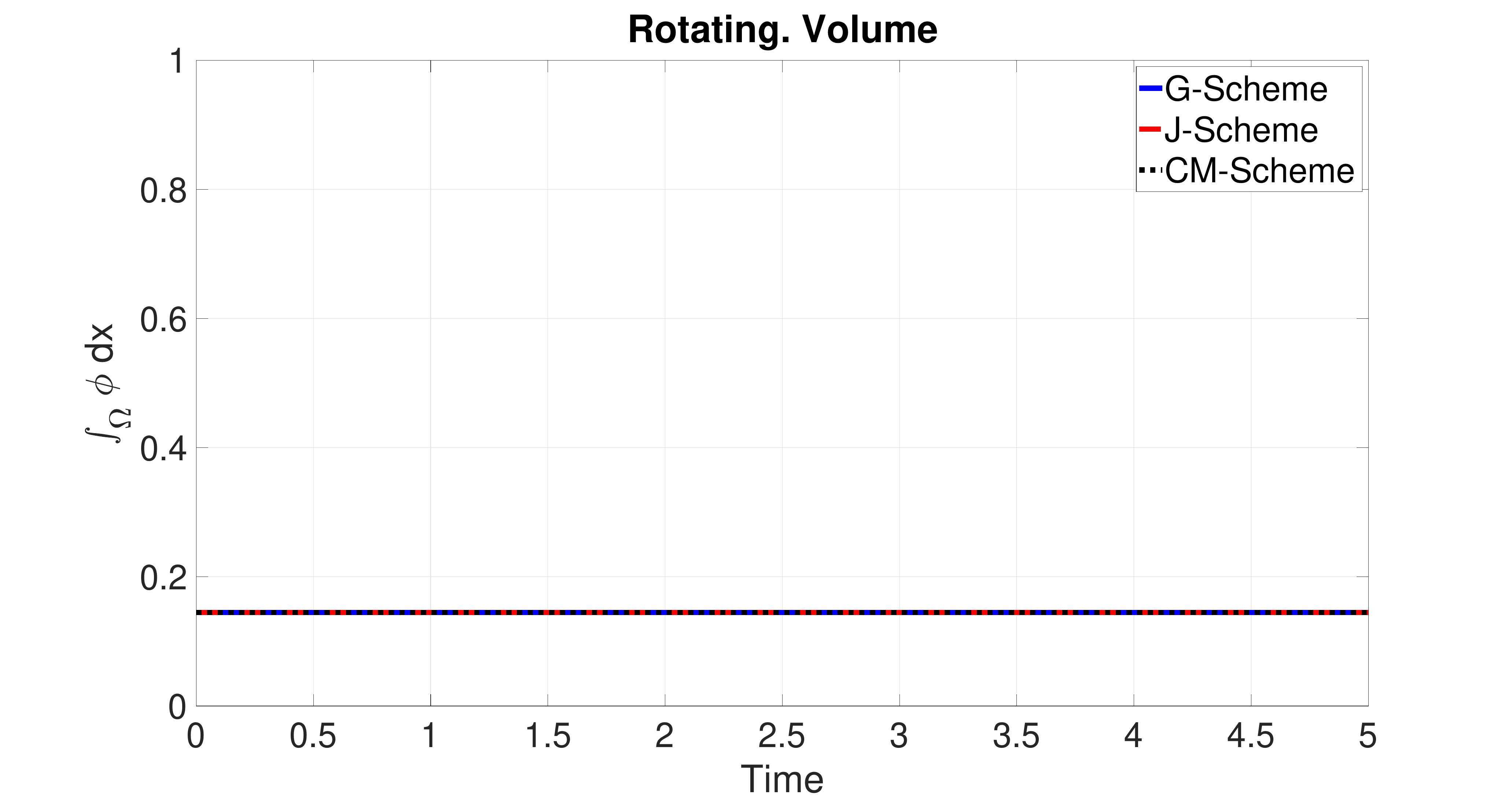}
\\
\includegraphics[scale=0.1125]{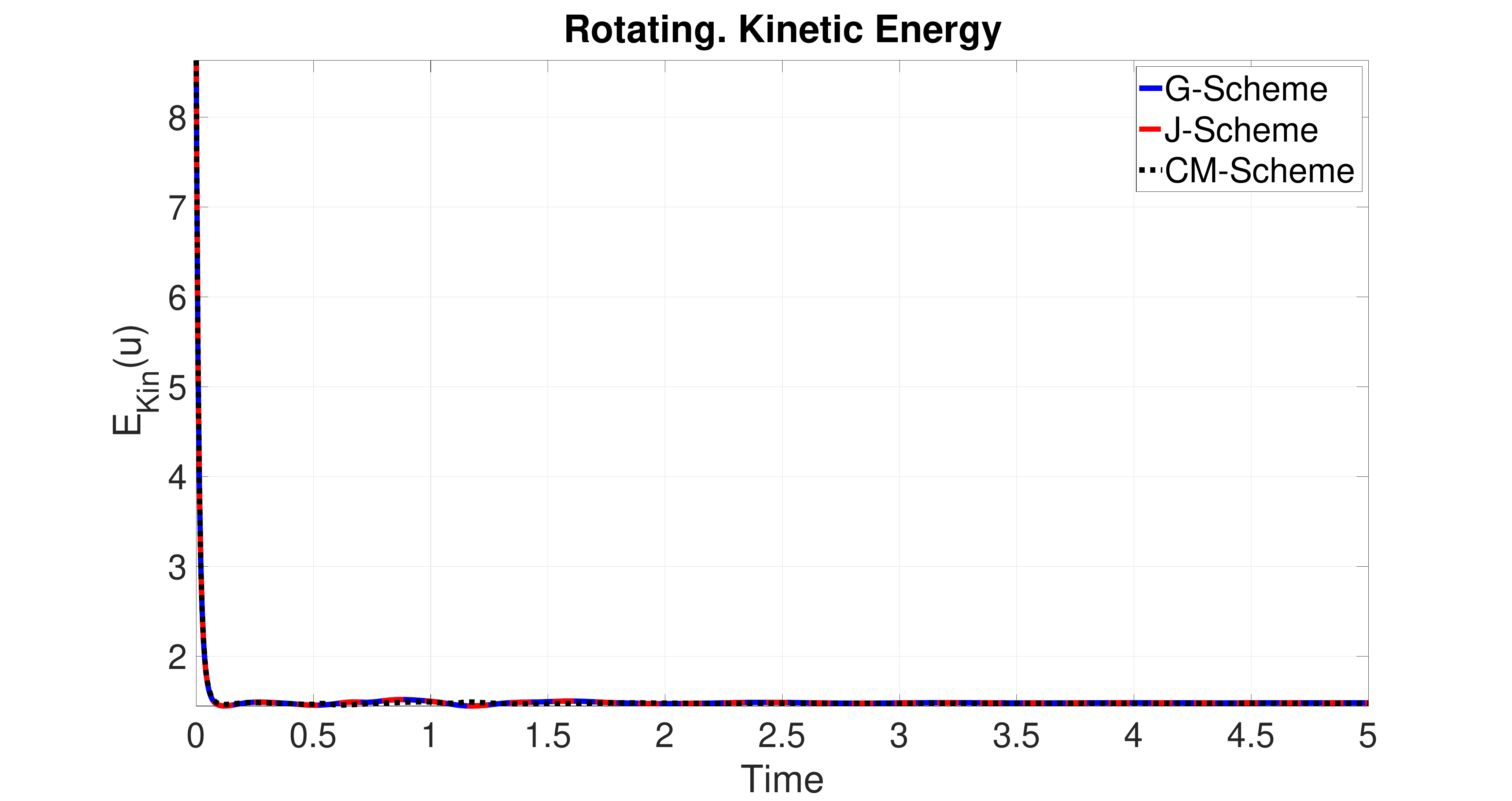}
\includegraphics[scale=0.1125]{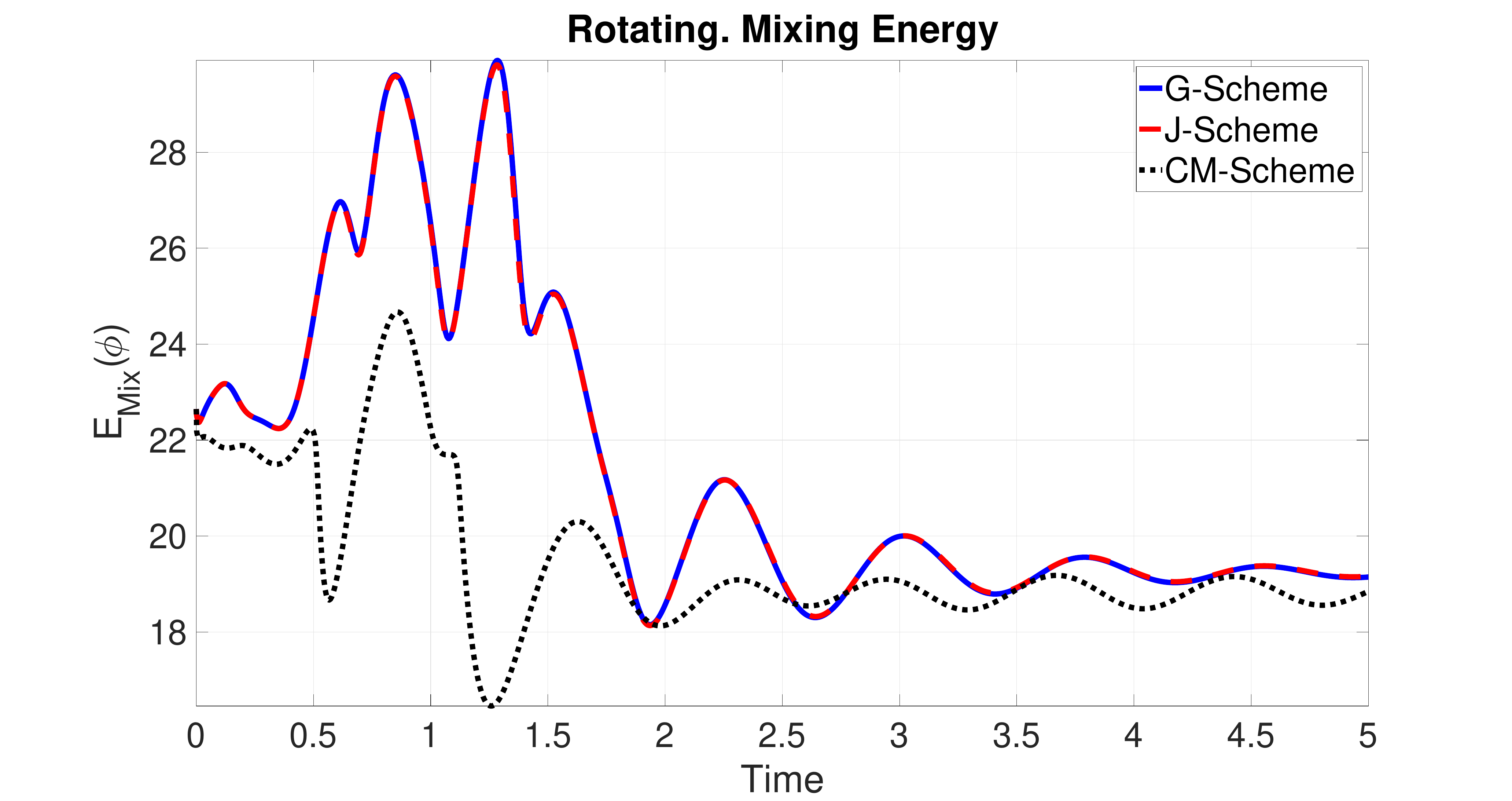}
\end{center}
\caption{Example III. Rotating fluids. Evolution of the energies and the volume of the system}\label{fig:Ex3_energies}
\end{figure}

\begin{figure}[H]
\begin{center}
\includegraphics[scale=0.1125]{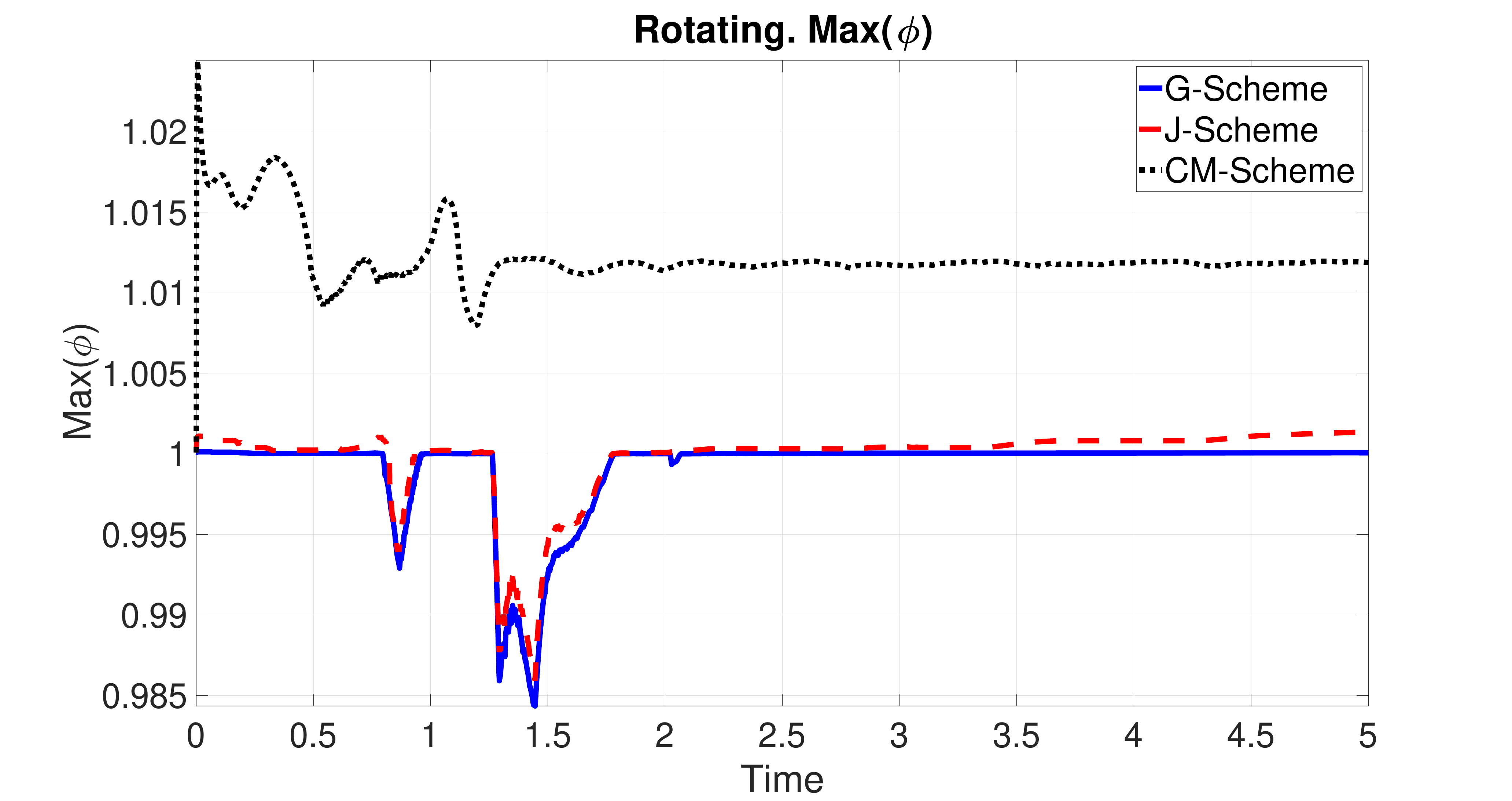}
\includegraphics[scale=0.1125]{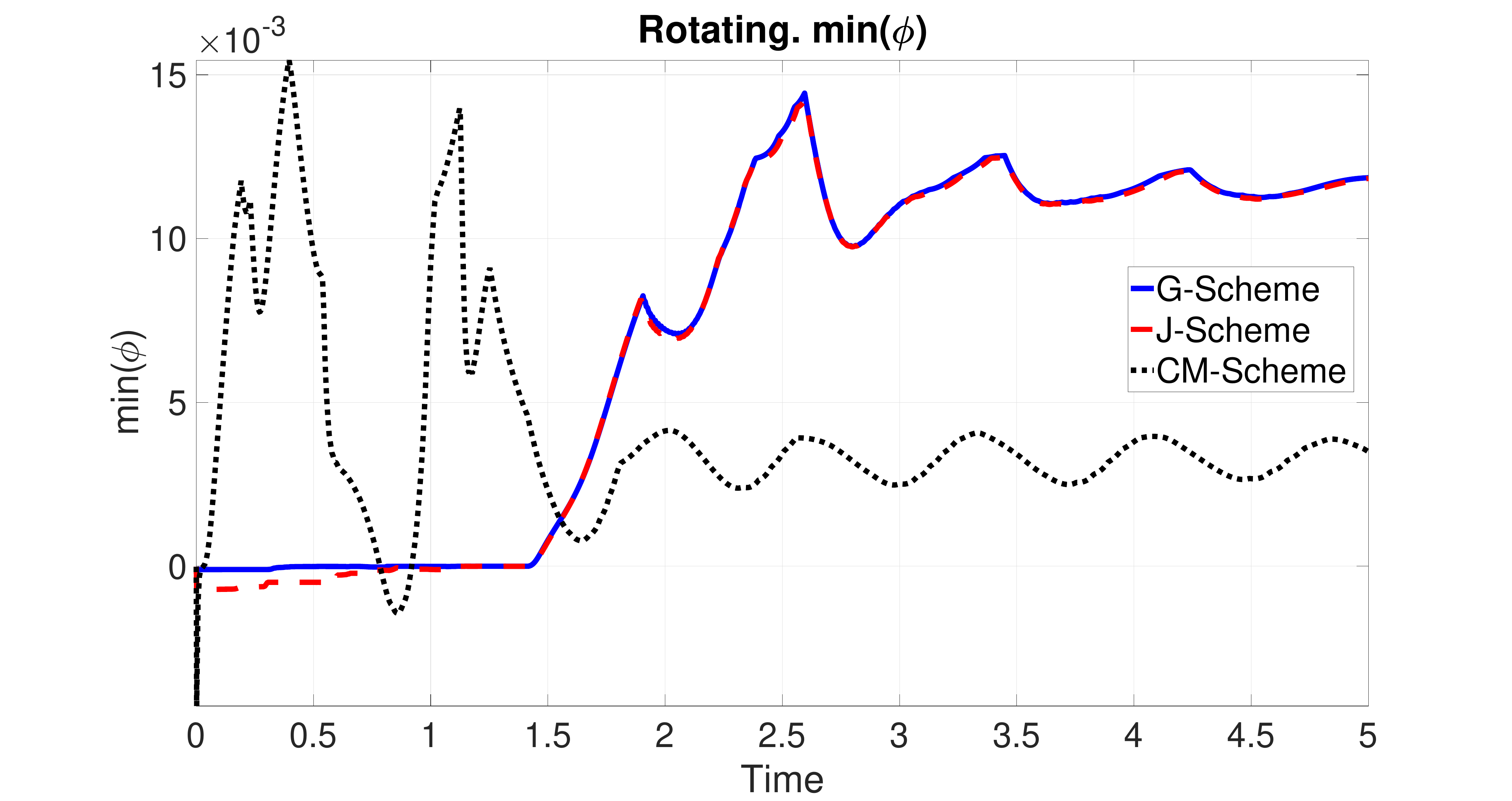}
\\
\includegraphics[scale=0.1125]{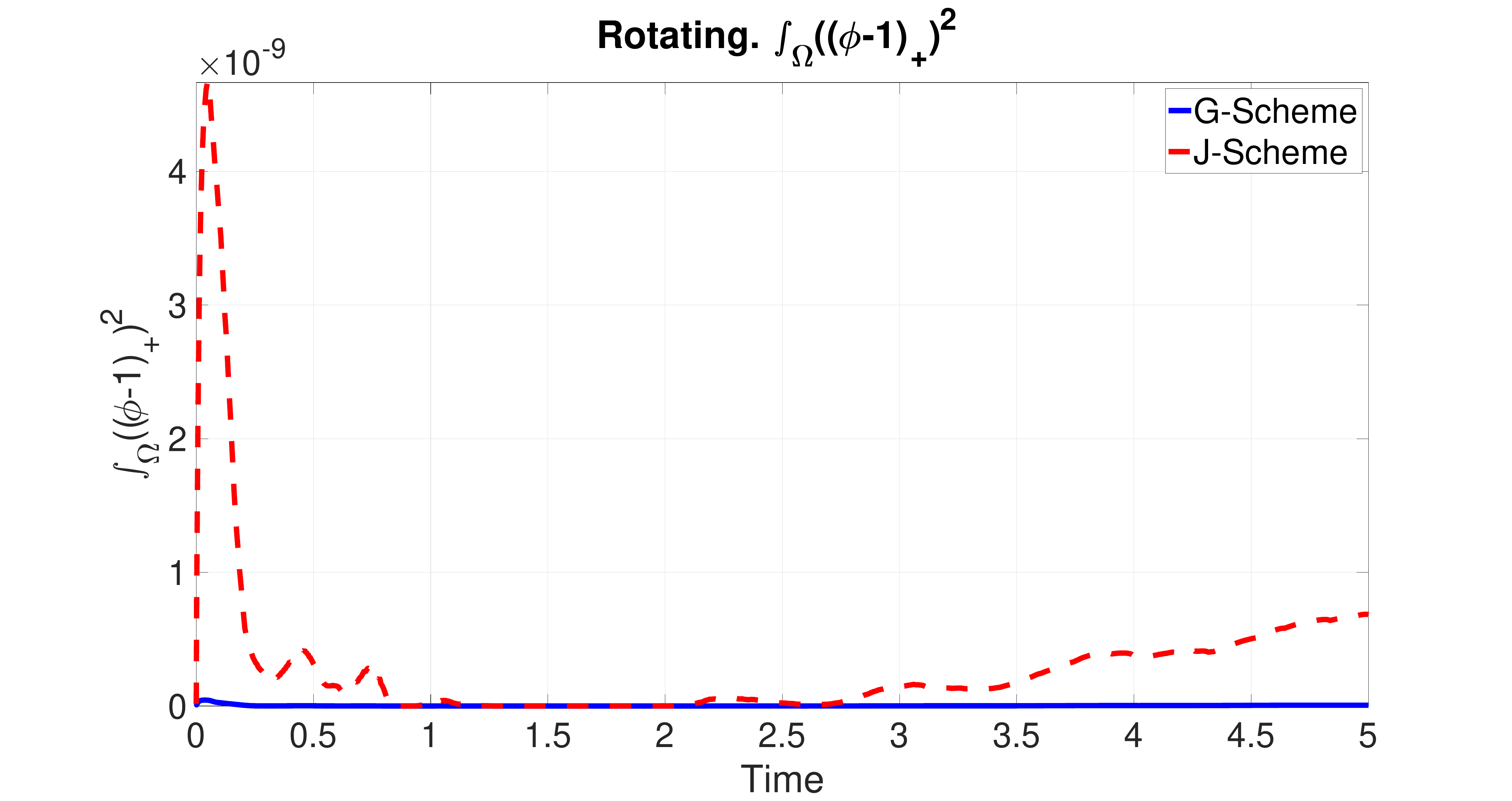}
\includegraphics[scale=0.1125]{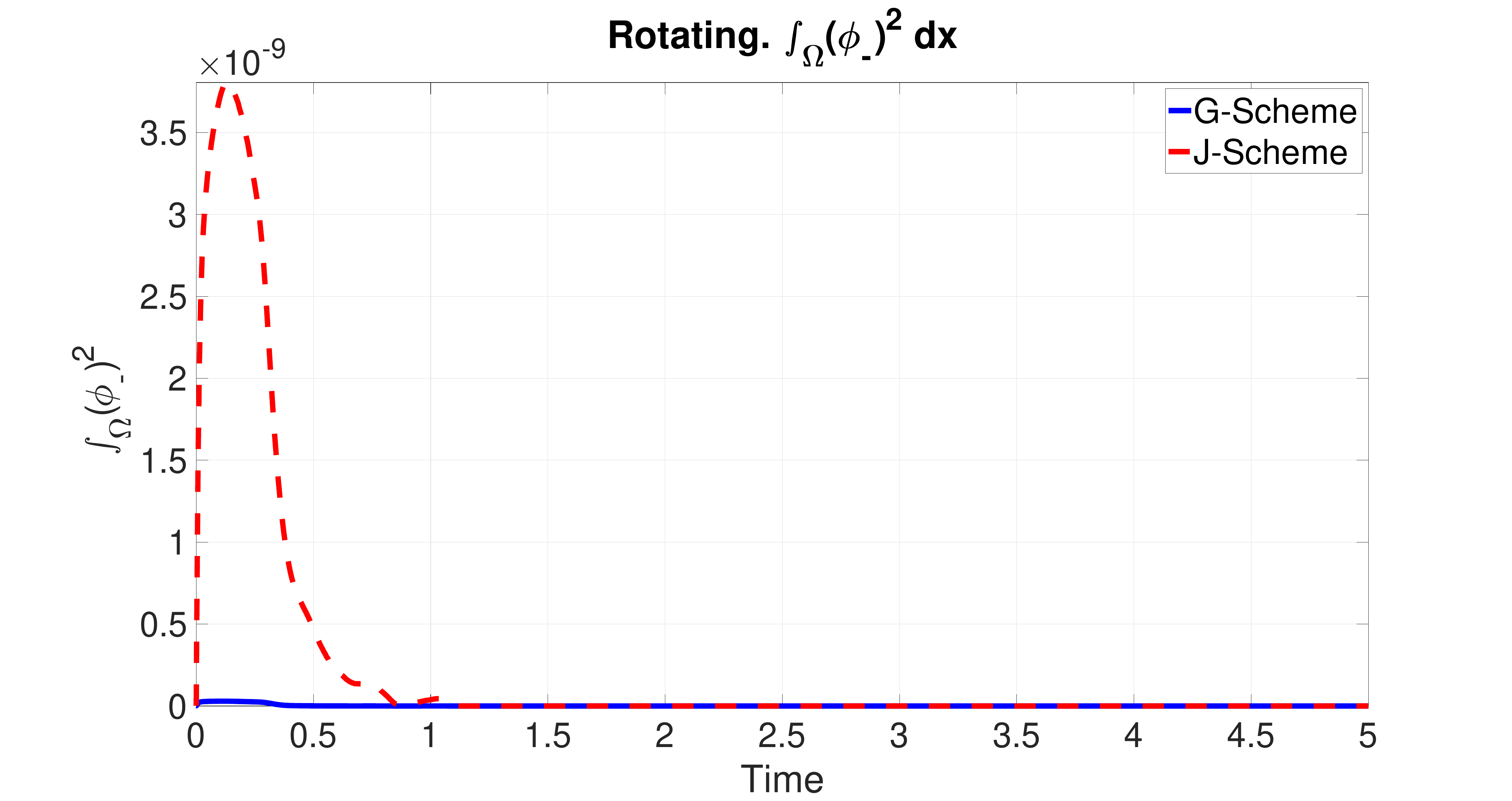}
\end{center}
\caption{Example III. Rotating fluids. Evolution of the bounds of $\phi$.}\label{fig:Ex3_bounds}
\end{figure}

\section{Conclusions}\label{sec:conclusions}

In this work we have derived two new numerical schemes to approximate the Navier-Stokes-Cahn-Hilliard system with degenerate mobility: $G_\varepsilon$-scheme and $J_\varepsilon$-scheme. In particular we show that both schemes are conservative, energy-stable and they satisfy bounds on $\phi$ that lead to approximate maximum principles. In particular, the derived bounds on the phase variable $\phi$ are associated with two different singular potentials, one for each scheme.

We have reported the results of several numerical simulations using the presented  schemes, showing that these schemes: (1) are accurate and satisfy the derived bounds for $\phi$; (2) achieve energy-stability in cases where no external forces are applied to the system; (3) are able to capture interesting dynamics in cases with external forces applied to the system, while maintaining the bounds on $\phi$. Moreover, we have compared the results with the case of considering a constant mobility term, illustrating that both models produce different dynamics.


\begin{thebibliography}{99}
%

\bibitem{Abeletal}  Abels,~H., Depner,~D. \& Garcke, H. 
2013. 
On an incompressible Navier-Stokes/Cahn-Hilliard system with degenerate mobility.
\textit{Ann. Inst. Henri Poincare (C) Anal. Non Lineaire.} \textbf{30}, 1175-1190.

\bibitem{DKR23}  Acosta-Soba,~D., Guill\'en-Gonz\'alez,~F. \& Rodr\'iguez-Galv\'an,~J.R. 
2023. 
A structure-preserving upwind DG scheme for a degenerate phase-field tumor model.
\textit{Computers and Mathematics with Applications} \textbf{152}, 317-333. 

\bibitem{DKR22}  Acosta-Soba,~D., Guill\'en-Gonz\'alez,~F. \& Rodr\'iguez-Galv\'an,~J.R. 
2023. 
An upwind DG scheme preserving the maximum principle for the convective Cahn-Hilliard model.
\textit{Numerical Algorithms} \textbf{92}, 1589-1619. 

\bibitem{paraview}  Ayachit,~U.
2015. 
The ParaView Guide: A Parallel Visualization Application, .
\textit{Kitware} ISBN 9781930934306

\bibitem{Barrett98}  Barrett,~J.W., Blowey,~J.F. \& Garcke,~H.
1998.  
Finite element approximation of a fourth order nonlinear degenerate parabolic equation.
\textit{Numer. Math.} \textbf{80}, 525-556.


\bibitem{Cahn1958} Cahn,~J.W. \& Hilliard,~J.E.
1958. 
Free energy of a nonuniform system. I. Interfacial free energy.
\textit{J. Chem. Phys.} \textbf{28},  258-267.
 
\bibitem{Chenetal19}
Chen,~W., Wang,~C., Wang,~X. \& Wise,~S.M.
2019. 
Positivity-preserving, energy stable numerical schemes for the Cahn-Hilliard equation with logarithmic potential. 
\textit{Journal of Computational Physics X} \textbf{3} 100031. 
 
\bibitem{CiarletRaviart73} Ciarlet,~P.G. \& Raviart,~P.A. 
1973. 
Maximum principle and uniform convergence for the finite element method.
\textit{Comput. Methods Appl. Mech. Engrg.} \textbf{2},  17-31.


\bibitem{Copetti92} Copetti,~M.I.M. \& Elliot,~C.M. 
1992. 
Numerical analysis of the Cahn-Hilliard equation with a logarithmic free energy.
\textit{Numer. Math.} \textbf{63},  39-65.


\bibitem{ElliotGarcke1996} Elliot,~C.M. \& Garcke,~H.
1996. 
On the Cahn-Hilliard equation with degenerate mobility.
\textit{SIAM J. Math. Anal} \textbf{27},  404-423.

\bibitem{Eyre} Eyre,~D.J.
An Unconditionally Stable One-Step Scheme for Gradient System,
\\ http://www.math.utah.edu/~eyre/research/methods/stable.ps, unpublished.

\bibitem{Frigerietal19} Frigeri,~S., Gal,~C.G., Grasselli,~M. \& Sprekels,~J. 
2019. 
Two-dimensional nonlocal Cahn-Hilliard-Navier-Stokes systems with variable viscosity, degenerate mobility and singular potential
\textit{Nonlinearity} \textbf{32}, 678-727

\bibitem{Frigerietal15} Frigeri,~S., Grasselli,~M. \& Rocca,~E. 
2015. 
A diffuse interface model for two-phase incompressible flows with non-local interactions and non-constant mobility
\textit{Nonlinearity} \textbf{28}, 1257-1293

\bibitem{Frigerietal20} Frigeri,~S., Grasselli,~M. \& Sprekels,~J. 
2015. 
Optimal Distributed Control of Two-Dimensional Nonlocal Cahn-Hilliard-Navier-Stokes Systems with Degenerate Mobility and Singular Potential
\textit{Applied Mathematics \& Optimization} \textbf{81}, 899-931

\bibitem{GiraultRaviart86} Girault,~V.,  Raviart,~P.A. 
Finite Element Methods for Navier-Stokes Equations: Theory and Algorithms, Springer-Verlag, Berlin, 1986.


\bibitem{KMAD19} Guill\'en-Gonz\'alez,~F., Rodr\'iguez-Bellido, M.A. \& Rueda-G\'omez~D.A.  
2019. 
Unconditionally energy stable fully discrete schemes for a chemo-repulsion model
\textit{Mathematics of Computation} \textbf{88} 2069-2099.


\bibitem{KMAD21} Guill\'en-Gonz\'alez,~F., Rodr\'iguez-Bellido, M.A. \& Rueda-G\'omez, D.A.; 
2021. 
A chemorepulsion model with superlinear production: analysis of the continuous problem and two approximately positive and energy-stable schemes. 
\textit{Advances in Comput. Math.} \textbf{47} 87

\bibitem{KMAD22} Guill\'en-Gonz\'alez,~F., Rodr\'iguez-Bellido, M.A. \& Rueda-G\'omez, D.A.; 
2022. 
Comparison of two finite element schemes for a chemo-repulsion system with quadratic production. 
\textit{Applied Numerical Mathematics} \textbf{173} 193-210.




\bibitem{KGCHDM} Guill\'en-Gonz\'alez,~F. \& Tierra,~G. 
2024. 
Energy-stable and boundedness preserving numerical schemes for the Cahn-Hilliard equation with degenerate mobility.
\textit{Appl. Numer. Math.} \textbf{196}  62-82.

\bibitem{GuillenTierra} Guill\'en-Gonz\'alez,~F. \& Tierra~G.  
2024. 
Finite Element numerical schemes for a chemo-attraction and consumption model.
\textit{Journal of Computational and Applied Mathematics} \textbf{441},  115676.

\bibitem{GPV96} Gurtin,~D., Polignone,~D. \&  Vi\~nals,~J. 
1996. 
Two-phase binary fluids and immiscible fluids described
by an order parameter.
\textit{Math. Models Methods Appl. Sci} \textbf{6},  815-831.

\bibitem{freefem} Hecht,~F. 
2012. 
New development in FreeFem+.
\textit{J. Numer. Math.} \textbf{20}  251-265.


\bibitem{HH77} Hohenberg,~P.P. \&  Halperin,~B.I. 
1977. 
Theory of dynamic critical phenomena.
\textit{Rev. Mod. Phys} \textbf{49}  435-479.

\bibitem{Kay2009} Kay,~D., Styles,~V. \& S\"uli E.
2009. 
Discontinuous Galerkin finite element approximation of the Cahn-Hilliard equation with convection.
\textit{SIAM J. Numer. Anal.} \textbf{47} 2660-2685.

\bibitem{LiuFrankRiviere} Liu,~C., Frank,~F. \& Rivi\`ere B.M.
2019. 
Numerical error analysis for nonsymmetric interior penalty discontinuous Galerkin method of Cahn-Hilliard equation.
\textit{Numer. Methods Partial Differential Equations} \textbf{35} 1509-1537.


\bibitem{matlab} MATLAB R2023a version 9.14.0, The MathWorks Inc., Natick, Massachusetts, 2023.

\bibitem{Tierra}  Tierra,~G. \&  Guill\'en-Gonz\'alez,~F.
2015. 
Numerical methods for solving the Cahn-Hilliard equation and its applicability to related energy-based models.
\textit{Arch. Comput. Methods Eng.} \textbf{22},  269-289.

\bibitem{XiaXuShu} Xia,~Y., Xu,~Y. \& Shu,~C.-W.
2007. 
Local discontinuous Galerkin methods for the Cahn-Hilliard
type equations
\textit{J. Comput. Phys.} \textbf{227}, 472-491.



\end{thebibliography}
\end{document}